\documentclass{article}
\usepackage{algorithm}
\usepackage{algpseudocode}
\usepackage[export]{adjustbox}
\usepackage{indentfirst}
\usepackage{amsthm}
\usepackage[utf8]{inputenc}
\usepackage{mathtools}   
\usepackage{amsfonts}
\usepackage{array}
\usepackage{booktabs}
\usepackage{multirow}
\usepackage{makecell}
\newtheorem{theorem}{Theorem}
\newtheorem{corollary}{Corollary}
\newtheorem{lemma}{Lemma}
\newtheorem{prop}{Proposition}
\newtheorem{definition}{Definition}[section]
\newtheorem{example}{Example}
\usepackage{float}

\newtheorem{remark}{\textbf{Remark}}
\usepackage{etoolbox}
\usepackage{tikz} 
\usepackage[colorlinks]{hyperref}
\usepackage{dynkin-diagrams}
\usepackage{blindtext}
\usepackage{xspace}
\usepackage{multicol}
\usepackage{enumitem}
\makeatletter
\newcommand{\mathleft}{\@fleqntrue\@mathmargin0pt}
\newcommand{\mathcenter}{\@fleqnfalse}
\usepackage{geometry}
\usepackage{lipsum}
\usepackage{graphicx}
\usepackage{bbm}
\newcommand{\indep}{\rotatebox[origin=c]{90}{$\models$}}
\makeatother
\usepackage{fullpage}

\newcommand{\iid}{i.i.d.{\ }}

\title{Robust Estimation via Robust Optimization}
\author{}
\date{}

\usepackage{titlesec}
\usepackage{tocloft}

\titleformat{\paragraph}
  {\normalfont\normalsize\bfseries}{\theparagraph}{1em}{}
\titlespacing*{\paragraph}{0pt}{3.25ex plus 1ex minus .2ex}{1.5ex plus .2ex}

\renewcommand{\theparagraph}{\thesubsubsection.\arabic{paragraph}} 



\begin{document}

\begin{center}
{\bf{\LARGE{On the Benefits of Accelerated Optimization in Robust and Private Estimation}}} \\
\vspace*{.25in}
\begin{tabular}{ccc}
{\large{Laurentiu Marchis}} & \hspace*{.5in}  & {\large{Po-Ling Loh}}\\
{\large{\texttt{lam223@cam.ac.uk}}} & \hspace*{.5in} & {\large{\texttt{pll28@cam.ac.uk}}}
\end{tabular}
\begin{center}
Statistical Laboratory \\
Department of Pure Mathematics and Mathematical Statistics \\
University of Cambridge
\end{center}
\vspace*{.2in}
June 2025
\vspace*{.2in}
\end{center}

\begin{abstract}
We study the advantages of accelerated gradient methods, specifically based on the Frank-Wolfe method and projected gradient descent, for privacy and heavy-tailed robustness. Our approaches are as follows: For the Frank-Wolfe method, our technique is based on a tailored learning rate and a uniform lower bound on the gradient of the $\ell_2$-norm over the constraint set. For accelerating projected gradient descent, we use the popular variant based on Nesterov's momentum, and we optimize our objective over $\mathbb{R}^p$. These accelerations reduce iteration complexity, translating into stronger statistical guarantees for empirical and population risk minimization. Our analysis covers three settings: non-random data, random model-free data, and parametric models (linear regression and generalized linear models). Methodologically, we approach both privacy and robustness based on noisy gradients. We ensure differential privacy via the Gaussian mechanism and advanced composition, and we achieve heavy-tailed robustness using a geometric median-of-means estimator, which also sharpens the dependency on the dimension of the covariates. Finally, we compare our rates to existing bounds and identify scenarios where our methods attain optimal convergence.
\end{abstract}


\section{Introduction}\label{sec:Introduction}

The study of differential privacy and robustness for statistical estimation and machine learning has recently attracted considerable attention, both individually and in combination. One approach to achieving privacy is output perturbation, where calibrated noise is added to the output of an estimation procedure \cite{(NEAR), BOLT_ON_PRIV, EFF_PRIV_SMOOTH}. Another key approach is gradient perturbation, where noise is added to gradients during an iterative algorithm such as gradient descent. Using composition theorems, this method produces private outputs, where each step is itself private. Talwar et al.~\cite{NOPL} proposed such an approach within the framework of the Frank-Wolfe algorithm, and analyzed convex, Lipschitz losses optimized over a convex polytope. For the Lasso, Talwar et al.\
achieved a rate of $\widetilde{O}\left(\frac{1}{(\epsilon n)^{2/3}}\right)$, which is optimal up to logarithmic factors. They also generalized their analysis to consider $L_2$-Lipschitz losses optimized over arbitrary convex sets of finite diameter.

The work of Talwar et al.\ raises questions regarding faster rates of convergence under a different geometry of the constraint set $\mathcal{C}$. An important technique which we leverage in this regard is acceleration. Section \ref{sec:{Private ERM for Distribution-Free Data: Upper and Lower Bounds}}  investigates ridge regression, taking into account the strong convexity of $\mathcal{C}$. By incorporating a relaxed and accelerated Frank-Wolfe method introduced based on \cite{ACCFW}, we show that better rates can be achieved with an appropriate learning rate, assuming a lower bound on the $\ell_2$-norm of the empirical risk gradient. We show how to establish such a bound, with high probability, under a parametric linear model. In the regime where $p \asymp m^2$ and $n \asymp \frac{m^3}{\log(m)}$, our results demonstrate the optimality of the upper bound. Using a lower bound construction inspired by \cite{NOPL}, we also show that the data conditions match those required for a lower bound on the $\ell_2$-norm of the of the empirical risk gradient. Notably, our accelerated method significantly improves performance by reducing noise requirements and lowering the iteration count $T$ from polynomial to logarithmic.

\begin{table}[!htbp]
    \centering
    \small  
    \setlength{\tabcolsep}{2.5pt}  
    \renewcommand{\arraystretch}{1.2} 
    \begin{tabular}{|c|c|c|c|}
        \hline
        \textbf{Setting and What We Are Bounding} & \textbf{FW} & \textbf{ACCFW} & \textbf{ACCFW Optimal?}\\
        \hline
        \makecell{Squared Loss, $|y_i|, ||x_i||_\infty \leq 1$, \\ Under additional assumptions \\ on the data for ACCFW, Privacy, \\ $\mathbb{E}\left[\mathcal{L}(\theta_T, \mathcal{D}_n) - \mathop{\min}\limits_{\theta \in \mathcal{C}}\mathcal{L}(\theta, \mathcal{D}_n)\right]$} &  
        \makecell{\\$\left(\frac{(\sqrt{p} + pD)D^2p}{n\epsilon}\right)^{2/3}$, \\ \\ $T \asymp \left(\frac{n\epsilon D}{1 + 2\sqrt{p}D}\right)^{2/3}$, \\ in \cite{NOPL}} & 
        \makecell{\\$\frac{(\sqrt{p} + pD)D\sqrt{p}}{n\epsilon}$, \\ \\ $T \asymp \log(n)$, \\ in Theorem \ref{theorem:UPRidge}} & \makecell{Yes, for \\ $n \asymp \frac{p^{3/2}}{\log(p)}$, \\ in Theorem \ref{theorem:LBRidge}} \\
        \hline
        \makecell{GLM, $|y_i|, ||x_i||_2 \lesssim 1$, \\ $D \uparrow ||\theta^*||_2$, $p \asymp 1$, Privacy, \\ $\mathbb{E}\left[\mathcal{L}(\theta_T, \mathcal{D}_n) - \mathop{\min}\limits_{\theta \in \mathcal{B}}\mathcal{L}(\theta, \mathcal{D}_n)\right]$ \\ for FW, $\mathcal{B} = \mathbb{B}_2(||\theta^*||_2)$,\\ $\mathcal{L}(\theta_T, \mathcal{D}_n) - \mathop{\min}\limits_{\theta \in \mathcal{B}}\mathcal{L}(\theta, \mathcal{D}_n)$, \\ w.h.p.\ for ACCFW} & 
        \makecell{\\$\frac{1}{(n\epsilon)^{2/3}}$, \\ \\ $T \asymp \left(n\epsilon\right)^{2/3}$, \\ in \cite{NOPL}} & 
        \makecell{\\$\frac{1}{n^{4/5}\epsilon}$, \\ \\ $T \asymp n^{2/5}\log(n)$, \\ in Theorem \ref{theorem:ERM_UP_C_Inc}} & 
        No\\
        \hline
        \makecell{GLM, $|y_i|, ||x_i||_2 \lesssim 1$, $p \asymp 1$, \\ Privacy, $||\theta_T - \theta^*||_2$, w.h.p.} & \makecell{\\$\frac{1}{n^{1/2}} + \frac{1}{(n\epsilon)^{1/3}}$, \\ \\ $T \asymp (n\epsilon)^{2/3}$, \\ in Proposition \ref{prop:Iter_Rate_Inc_C_Non}} & 
        \makecell{\\$\frac{1}{n^{1/2}} + \frac{1}{n^{2/5}\epsilon^{1/2}}$, \\ \\ $T \asymp n^{2/5}\log(n)$, \\ in Theorem \ref{theorem:Iter_Rate_Inc_C}} & 
        \makecell{Yes, \cite{Cai_Cost_GLM}, in \\ the dominant \\ (statistical error) \\ term}\\
        \hline
        \makecell{GLM, $|y_i|, ||x_i||_2 \lesssim 1$, \\ $||\theta^*||_2 - D, p \asymp 1$, Privacy, \\ $\mathbb{E}\left[\mathcal{L}(\theta_T, \mathcal{D}_n) - \mathop{\min}\limits_{\theta \in \mathcal{C}}\mathcal{L}(\theta, \mathcal{D}_n)\right]$} & 
        \makecell{\\$\frac{1}{(n\epsilon)^{2/3}}$, \\ \\ $T \asymp \left(n\epsilon\right)^{2/3}$, \\ in \cite{NOPL}} & 
        \makecell{\\$\frac{1}{n\epsilon}$, \\ \\ $T \asymp \log(n)$, \\in Theorem \ref{theorem:ERM_GLM_UP}} & 
        Unknown \\
        \hline
        \makecell{Linear Regression, $\lambda_{\min}(\Sigma) > 0$, \\ $p \asymp 1$, Heavy-Tailed Robustness, \\ $||\theta_T - \theta^*||_2$, w.h.p.} & \makecell{\\$\frac{1}{n^{1/6}}$, \\ \\ $T = n^{1/3}$, \\in Theorem \ref{theorem:HT_NonAccFW}} & \makecell{\\$\frac{1}{n^{1/5}}$, \\ \\ $T \asymp n^{1/5}\log(n)$, \\ in Theorem \ref{theorem:ACCFWROB}} & No\\
        \hline
        \makecell{Ridge Regression, $\lambda_{\min}(\Sigma) = 0$, \\ $p \asymp 1$, Heavy-Tailed Robustness, \\ $||\theta_T - \theta^*||_2$, w.h.p.} & \makecell{\\$\frac{1}{n^{1/9}} + c_{\mathcal{K}}$, \\ \\ $T = n^{1/3}$, \\ in Theorem \ref{theorem:NonACC_lambda=0}} & \makecell{\\$\frac{1}{c_{\mathcal{K}}^{1/4}n^{1/4}} + c_{\mathcal{K}}$ \\ $+ c_{\mathcal{K}}^{1/2}$, \\ \\ $T \asymp \log(n)/c_{\mathcal{K}}^2$, \\ in Theorem \ref{theorem:ACCFW_lambda=0}} & No \\ 
        \hline
    \end{tabular}
    \caption{\textbf{Frank-Wolfe vs.\ Accelerated Frank-Wolfe}, Constraint set $\mathcal{C} = \mathbb{B}_2(D)$, $\epsilon \lesssim 1$, $c_{\mathcal{K}} = \left\|[P^T\theta^*]_{[(m+1):p]}\right\|_2$}
    \label{tab1}
\end{table}

\begin{table}[!htbp]
    \centering
    \small  
    \setlength{\tabcolsep}{4pt}  
    \renewcommand{\arraystretch}{1.2} 
    \begin{tabular}{|c|c|c|c|}
        \hline
        \textbf{Setting and What We Are Bounding} & \textbf{GD} & \textbf{AGD} & \textbf{GD/AGD Optimal?} \\
        \hline
        \makecell{Convex, Smooth, Lipschitz Loss, $p \asymp 1$, \\ Privacy, $\mathcal{R}(\theta_T) - \mathop{\min}\limits_{\theta \in \mathbb{R}^p}\mathcal{R}(\theta)$} & 
        \makecell{\\$\frac{1}{n^{1/5}} + \frac{1}{n^{1/2}\epsilon}$, \\ \\ $T = n^{1/5}$, \\ in Theorem \ref{theorem:Priv_HT_GD_Smooth}} & 
        \makecell{\\$\frac{1}{n^{2/5}} + \frac{1}{n\epsilon^2}$, \\ \\ $T = n^{1/5}$, \\ in Theorem \ref{theorem:Priv_HT_AGD_Smooth}} &
        No/No \\
        \hline
        \makecell{Linear Regression, Squared Loss, \\ Optimization over $\mathbb{R}^p$, \\ Smooth and Strongly Convex Risk, \\ $||\theta_0 - \theta^*|| \lesssim \sqrt{p}$, Heavy-Tailed \\ Robustness, $||\theta_T - \theta^*||_2$, w.h.p.} & 
        \makecell{$\sqrt{\frac{p}{n}}$, \\ \\ $T \asymp \log(n)$, \\ in \cite{RE}} & 
        \makecell{\\$\sqrt{\frac{p}{n}}$, \\ \\ $T \asymp \log(n)$, \\ for $\frac{pT\log(T)}{n} \asymp 1$, \\ in Theorem \ref{theorem:htAGD}} & \makecell{Yes/Yes,\\  minimax rate for \\ linear regression \\ in \cite{Stats_Info_Duchi}} \\
        \hline
    \end{tabular}
    \caption{\textbf{Gradient Descent vs.\ Nesterov's AGD}, $\epsilon \lesssim 1$}
    \label{tab2}
\end{table}

Accelerating gradient methods is beneficial from the following perspective: Using an accelerated method leads to a smaller variance of noise required for privacy and a smaller number of iterations $T$, resulting in better statistical performance overall. We take this idea further in Section \ref{sec:Private Biased Parameter Estimation in GLMs: Upper and Lower Bounds for Excess Risk}, where we study parametric generalized linear models. Applying the general result in \cite{NOPL} over an $\ell_2$-ball that contains the true parameter $\theta^*$ yields a rate of $\widetilde{O}\left(\frac{1}{(n\epsilon)^{2/3}}\right)$, when $0 < \epsilon \lesssim 1$. In contrast, our accelerated Frank-Wolfe method applied to an $\ell_2$-ball that grows toward $\theta^*$ as $n \rightarrow \infty$ achieves a smaller error of $\widetilde{O}\left(\frac{1}{n^{4/5}\epsilon}\right)$, when $0 < \epsilon \lesssim 1$. 

Another aspect of regression based on perturbed gradients that has been studied extensively is robustness. This includes robustness to both outliers and heavy-tailed data. Instead of using a robust loss~\cite{Huber1, Huber2, Hampel}, one can use robust gradient estimators while optimizing non-robust loss functions. This idea has gained traction from Diakonikolas et al.~\cite{CompIntr} and Balakrishnan et al.~\cite{Balakrish}. Heavy tails present a challenge in estimation and regression, as explored by \cite{LUGOSI_MEAN_ESTIM_HT, PENSIA_ROB_HT, Minsker}. Prasad et al.~\cite{RE} extend the spectral algorithm of Lai et al.~\cite{Vpala} for Huber contamination and use the $G_{MOM}$ estimator for heavy-tailed robustness as gradient estimators in projected gradient descent. Their results yield a high-probability upper bound on the $\ell_2$-error of iterates. In our study of linear models, we assume the noise has a finite second moment and the covariates only have few finite moments.

In Section \ref{sec:Biased Parameter Estimation HT}, we demonstrate the benefits of the accelerated Frank–Wolfe method for heavy‐tailed linear models. Using the $G_{\mathrm{MOM}}$ estimator \cite{RE}, our accelerated scheme contracts gradient noise over $T$ iterations and yields tighter bounds on $||\theta_T-\theta^*||_2$.
When the covariance $\Sigma$ of the covariates is well‐conditioned ($\lambda_{\min}(\Sigma)>0$; Section \ref{sec:Well_Cond_Sec_FW}), our accelerated rate $\widetilde{O}\left((1 + \sigma_2)^{1/2}/n^{1/5}\right)$ improves on the one based on classical Frank–Wolfe, i.e., $\widetilde{O}\left((1 + \sigma_2)^{1/2}/n^{1/6}\right)$. In the ill‐conditioned case ($\lambda_{\min}(\Sigma)=0$; Section \ref{sec:Ill_Cond_Sec_FW}), acceleration achieves a rate of $\widetilde{O}\left((1 + \sigma_2)^{1/2}\frac{1}{c_{\mathcal{K}}^{1/4}n^{1/4}} + c_{\mathcal{K}} + c_{\mathcal{K}}^{1/2}\right)$, while the classical approach gives $\widetilde{O}\left((1 + \sigma_2)^{1/2}\frac{1}{n^{1/9}} + c_{\mathcal{K}}\right)$. Here, $c_{\mathcal K}$ vanishes as the problem becomes well‐conditioned, so acceleration trades off an extra vanishing bias for substantially faster convergence in $n$.

The idea of acceleration in optimization was popularized by
Nesterov’s accelerated method \cite{Nester}, which often outperforms projected gradient descent, and recent work has begun to explore its private analog. Xu el al.~\cite{Xu_ADMM_Nester_2021} studied accelerated updates in an ADMM setting for smooth convex losses with non-random input data, but gave guarantees only for standard projected gradient descent. Kuru et al.~\cite{Kuru_Nester_DP_2022} provided utility bounds for both vanilla and accelerated gradient methods on strongly convex losses. In Section \ref{sec:Smooth_Priv_HT_GD_AGD}, we analyze smooth risk functions optimized over $\mathbb{R}^p$, with random data, and show that differentially private Nesterov acceleration achieves excess risk $\widetilde{O}\left(\frac{1}{n^{2/5}} + \frac{1}{n\epsilon^2}\right)$, improving over the rate $\widetilde{O}\left(\frac{1}{n^{1/5}} + \frac{1}{n^{1/2}\epsilon}\right)$ achieved by projected gradient descent. Feldman et al.~\cite{SCO_Feldman_2020} obtain the optimal rate of $\frac{1}{n^{1/2}} + \frac{1}{n\epsilon}$, for optimization over $\mathbb{R}^p$, using a localization-based SGD approach.

With the effect of accelerated Frank-Wolfe on heavy-tailed robustness in mind, we conduct a similar analysis based on Nesterov's acceleration in Section \ref{sec:Strongly Convex Risks}. Building upon \cite{RE}, we establish a convergence result regarding Nesterov's momentum (cf.\ Theorem \ref{theorem:AGD}), for smooth and strongly convex risks. We then apply this result to linear regression with squared error loss using the $G_{MOM}$ estimator. The conclusion for strongly convex risks is that acceleration is less impactful, improving on the iteration count only up to constant factors. This stems from the fact that, for smooth and strongly convex functions, projected gradient descent and Nesterov's momentum both give exponential convergence rates in the iteration count. Additionally, \cite{RE} perform an analysis in the Huber $\epsilon$-contamination model; in Appendix \ref{sec:Huber epsilon-Contamination Robustness}, we analyze the performance of Theorem \ref{theorem:AGD} for Nesterov's method in the Huber model. 

We also mention other related work on gradient perturbation, notably private SGD, where there has been an extensive line of work concerning bounds on excess empirical risk \cite{SGD_DP_2013_Upd, Bassily_SGD_l2, Is_Inter_Nec_DP_ERM, DP_ERM_Faster, Priv_SGD_l1_Geom}, either with high probability or in expectation. Some authors focus on computational efficiency, as opposed to achieving tighter upper bounds on the excess empirical risk~\cite{Towards_Prac_DP_ERM, BOLT_ON_PRIV}. Other authors target the excess risk directly \cite{BasEtal19, SCO_Feldman_2020, BasEtal21}. For the excess empirical risk, Bassily et al.~\cite{Bassily_SGD_l2} consider private SGD for a convex, differentiable, $L_2$-Lipschitz loss, optimized over a convex, bounded set $\mathcal{C}$. For an iteration count of $T = n^2$, they obtain an upper bound of $\widetilde{O}\left(\frac{L_2||\mathcal{C}||_2\sqrt{p}}{n\epsilon}\right)$ on the expected excess empirical risk. Later, \cite{DP_ERM_Faster} improved the iteration count for convex, differentiable, smooth, $L_2$-Lipschitz regularized losses, optimized over $\mathbb{R}^p$. In Appendix~\ref{AppSGD}, we compare our results from Sections \ref{sec:{Private ERM for Distribution-Free Data: Upper and Lower Bounds}} and \ref{sec:Private Biased Parameter Estimation in GLMs: Upper and Lower Bounds for Excess Risk} to private SGD. We also mention the line of work~\cite{BasEtal19, BasEtal21}, which has a similar flavor to our paper, in that it analyzes private SGD under different $\ell_p$-geometries of the constraint set, and seeks to explore methods that can achieve more efficient convergence by leveraging geometry.

From a practical standpoint, our work is relevant to domains such as financial modeling, where heavy-tailed distributions better capture market shocks, or in medical imaging, where it is desirable to be robust to random artifacts and non-Gaussian distributions. It is worth noting that our approach differs from much of the existing differential privacy literature, which avoids parametric assumptions and does not aim to recover a true parameter $\theta^*$. Instead, we combine parametric modeling with privacy and robustness to heavy tails, addressing gaps in prior works such as~\cite{NOPL, BOLT_ON_PRIV, EFF_PRIV_SMOOTH}, while extending results such as \cite{RE}. Hence, our work enhances the study of heavy-tailed robust and differentially private regression by making use---on one hand---of accelerated gradient methods---and on the other hand---of parametric modeling perspectives, enabling more structured, targeted optimization procedures.


\section{Preliminaries}\label{sec:Notation and Preliminaries}

We introduce the required background material that will be central to our derivations in this paper. For a detailed presentation of notation, see Appendix \ref{sec:Notation}.

\subsection{Preliminaries on Optimization} \label{sec:Preliminaries on Optimization}

In this section, we introduce the fundamental aspects of our analysis. We start by presenting 
the general convex optimization settings before introducing differential privacy. For a differentiable function $F$, we denote its gradient by $\nabla F$ and its Hessian by $\nabla^2 F$. For more preliminary aspects related to smooth and strongly convex functions, see Appendix \ref{sec:Appendix_Optim_Prelim}. We shall make use of the notion of a strongly convex set, and because of its crucial importance in our work, we define it below:

\begin{definition}[Strongly Convex Set]
We say that a convex set $\mathcal{C} \subseteq \mathbb{R}^p$ is $\alpha_{\mathcal{C}}$-strongly convex if for any $x, y \in \mathcal{C}$, any $\gamma \in [0, 1]$, and any $z \in \mathbb{R}^p$ such
that $||z||_2 = 1$, we have
\begin{align*}
\gamma x + (1 - \gamma)y + \gamma(1 - \gamma)\frac{\alpha_{\mathcal{C}}}{2}||x - y||_2^2z \in \mathcal{C}.
\end{align*}
\end{definition}

Geometrically, the definition above says that $\mathcal{C}$ contains a ball of radius $\gamma(1 - \gamma)\frac{\alpha_{\mathcal{C}}}{2}||x - y||_2^2$ centered at $\gamma x + (1 - \gamma)y$. In particular, $\ell_2$-balls are strongly convex (cf.\ Lemma~\ref{lemma:STR_CONV_BALL}).


\subsubsection{Projected Gradient Descent} \label{sec:Projected Gradient Descent}

We can now introduce our main gradient optimization methods. For a convex set $\mathcal{C} \subseteq \mathbb{R}^p$ and a strictly convex, differentiable function $F : \mathbb{R}^p \rightarrow \mathbb{R}$, for an initial point $x_0 \in \mathcal{C}$ and stepsize $\eta$, consider the updates
\begin{align}
\label{eq:Proj_GD_Rule}
x_{t + 1} = \mathcal{P}_{\mathcal{C}}(x_t - \eta \nabla F(x_t)),
\end{align}
where $\mathcal{P}_{\mathcal{C}}$ is the projection operator in the $\ell_2$-norm onto our constraint set $\mathcal{C}$. Under strong convexity and smoothness, we can guarantee sub-exponential convergence in the iteration count $t$ for $||x_t - x_{*}||_2^2$, where $x_{*} \in \mathop{\arg\min}\limits_{x \in \mathcal{C}}F(x)$ (see Lemma \ref{lemma:Optimiz_Proj_GD_Conv} in Appendix \ref{sec:Appendix_GD}).


\subsubsection{Nesterov's Accelerated Gradient Descent (AGD)} \label{sec:Nesterov's Accelerated Gradient Descent (AGD)}

The next gradient method provides faster convergence rates than projected gradient descent. The idea is to take into account the previous two terms when moving to the $(t + 1)^{\text{th}}$ term in order to generate a type of momentum: For a strictly convex differentiable function $F : \mathbb{R}^p \rightarrow \mathbb{R}$, starting at some $(x_0, x_1)$ with $\eta, \lambda > 0$, and assuming that optimization occurs over $\mathcal{C} = \mathbb{R}^p$, define the iterates
\begin{align}
\label{eq:Nester_Rule}
x_{t + 1} = x_t - \eta\nabla F(x_t + \lambda(x_t - x_{t - 1})) + \lambda(x_t - x_{t - 1}).
\end{align}
Rates of convergence are provided in Lemma \ref{lemma:Optimiz_Nester_Conv} in Appendix \ref{sec:Appendix_AGD}. 


\subsubsection{The Frank-Wolfe Method} \label{sec:The Frank-Wolfe Method}

Finally, consider a convex, differentiable $F : \mathbb{R}^p \rightarrow \mathbb{R}$ that we wish to minimize over a compact, convex set $\mathcal{C}$. The algorithm runs the following for a learning rate $\eta 
 > 0$, starting at some $x_0 \in \mathcal{C}$:
\begin{align}
\label{eq:FW_Rule}
v_t = \mathop{\arg \min}\limits_{v \in \mathcal{C}} \nabla F(x_t)^Tv, \qquad x_{t + 1} = (1 - \eta)x_t + \eta v_t.
\end{align}
One can show a sub-linear convergence result under $\tau_u$-smoothness \cite{Recht, FWS}, with the learning rate varying with the number of iterations (cf.\ Lemma \ref{lemma:Optimiz_FW_Conv} in Appendix \ref{sec:Appendix_FW}). 

To allow for the noise introduced in the study of privacy, \cite{NOPL} considers a relaxed version of the classical Frank-Wolfe algorithm with varying learning rate $\frac{2}{t + 2}$ from \cite{Rev_FW_Prof_Free_Jaggi}. Instead of asking for $v_t$ to be precisely the minimizer of a the linear function $v^T\nabla F(x_t)$ over $\mathcal{C}$, they only ask for $v_t^T\nabla F(x_t)$ to be less than $\mathop{\min}\limits_{v \in \mathcal{C}}v^T\nabla F(x_t)$ plus some non-negative error term. The convergence rate is still linear in $t$ (cf.\ Lemma \ref{lemma:NonAccRel_FW}) in Appendix \ref{sec:Appendix_FW}).

If we optimize over a compact, strongly convex set $\mathcal{C}$ and the $\ell_2$-norm of the gradient is bounded below over $\mathcal{C}$, we can perform a similar relaxation and obtain approximate exponential convergence \cite{ACCFW}. Our accelerated, relaxed Frank-Wolfe algorithm is provided in Algorithm \ref{alg:ReAccFW}. We call it \emph{accelerated} since the convergence rate is exponential, and \emph{relaxed}, since we only require the linear objective $v^T\nabla F(x_t)$ at each step $t$, evaluated at $v_t$, to be close to $\mathop{\min}\limits_{v \in \mathcal{C}}v^T\nabla F(x_t)$.
The following convergence guaranteed is proved in Appendix \ref{sec:Appendix_FW}:

\begin{algorithm}
\caption{Relaxed and Accelerated Frank-Wolfe}
\label{alg:ReAccFW}
\begin{algorithmic}[1]
\Function{ReAccFW}{$r$, $\tau_u$, $\Delta$, $\alpha_{\mathcal{C}}$, $T$, compact and $\alpha_{\mathcal{C}}$-strongly convex set $\mathcal{C}$, $\eta = \min\left\{1, \frac{\alpha_{\mathcal{C}} r}{4\tau_u}\right\}$}
    \For{$t = 0$ to $T - 1$}
        \State Find $v_t \in \mathcal{C}$ s.t. $v_t^T\nabla F(x_t) \leq \mathop{\min}\limits_{v \in \mathcal{C}}v^T\nabla F(x_t) + \Delta$.
        \State $x_{t + 1} = (1 - \eta)x_t + \eta v_t$.
    \EndFor
    \State \textbf{return} $x_T$.
\EndFunction
\end{algorithmic}
\end{algorithm}

\begin{theorem}
\label{theorem:ACCFWV3}
Let $\mathcal{C} \subseteq \mathbb{R}^p$ be a compact, $\alpha_{\mathcal{C}}$-strongly convex set, and let $F: \mathbb{R}^p \rightarrow \mathbb{R}$ be a convex, differentiable, $\tau_u$-smooth function such that $0 < r \leq ||\nabla F(x)||_2$ for all $x \in \mathcal{C}$. Suppose $x_{*} \in \mathop{\arg\min}\limits_{x \in \mathcal{C}}F(x)$. Then Algorithm \ref{alg:ReAccFW} returns a sequence $x_t$ such that
\begin{align*}
F(x_t) - F(x_{*}) \leq c^t\left(F(x_0) - F(x_{*})\right) + \frac{3\Delta\eta}{2(1 - c)}, \quad \forall t \ge 0, \text{ with } c = \max\left\{\frac{1}{2}, 1 - \frac{\alpha_{\mathcal{C}} r}{8\tau_u}\right\}.
\end{align*}
\end{theorem}

Notice that the upper bound in Theorem~\ref{theorem:ACCFWV3} consists of a term that converges exponentially to $0$ with the number of iterations and an error term involving $\Delta$. The fact that the error is linear in $\Delta$ will be important later when we apply Algorithm \ref{alg:ReAccFW} in various statistical settings.
Further note that we have the crucial assumption that the norm of the gradient is bounded below by a positive quantity. Hence, any point that sets the gradient to 0 must lie outside $\mathcal{C}$. We mention this to preview our later results where we will optimize losses over sets that do not contain the true parameter of our model, such as Theorem \ref{theorem:ACCFWROB} in Section \ref{sec:Biased Parameter Estimation HT}.


\subsection{Preliminaries on Differential Privacy} 

In what follows, we will consider random estimators that we denote by $\hat{\theta}$. We assume $\hat{\theta} : \mathcal{E}^n \rightarrow \mathbb{R}^p$, where $\mathcal{E}$ is some input space that we will specify depending on the problem setting. Recall the classical notion of $(\epsilon, \delta)$-differential privacy, which will be denoted by $(\epsilon, \delta)$-DP:
\begin{definition}
\label{definition:DiffP}
A randomized algorithm/mechanism $\hat{\theta}$ satisfies $(\epsilon, \delta)$-differential privacy, for $\epsilon > 0$ and $\delta \geq 0$, if for all pairs of datasets $X$ and $X'$ differing in one element and for all $S$ in the range of $\hat{\theta}$, we have  $\mathbb{P}\left(\hat{\theta}(X) \in S \right) \leq e^\epsilon \mathbb{P}\left(\hat{\theta}(X') \in S \right) + \delta$.
\end{definition}
We present further technical preliminary aspects in Appendix \ref{sec:Appendix_DP_Prelim}.


\subsection{The General Statistical Settings and Models} \label{sec:The General Statistical Settings and Models}

In Sections \ref{sec:Private Biased Parameter Estimation in GLMs: Upper and Lower Bounds for Excess Risk}, \ref{sec:Biased Parameter Estimation HT}, and \ref{sec:Strongly Convex Risks}, we consider parametric models $\{P_{\theta} \ | \ \theta \in \Theta\}$, where data $\mathcal{D}_n = \{z_1, \dots, z_n \} \subseteq \mathcal{E}^n$ are drawn i.i.d.\ from a distribution $P_{\theta^*}$. We assume the existence of a true parameter $\theta^* \in \Theta =  \mathbb{R}^p$. To measure the error produced by some optimization procedure, we use a loss function $\mathcal{L}: \mathbb{R}^p \times \mathcal{E} \rightarrow \mathbb{R}$, which we will specify depending on the application. In some cases, we also look at the corresponding population-level risk $\mathcal{R(\theta)} = \mathbb{E}_{z \sim P_{\theta^*}}[\mathcal{L}(\theta, z)]$. Throughout the analysis, we take $\mathcal{L}$ and $\mathcal{R}$ to be convex and differentiable over $\mathbb{R}^p$. Crucially, we will optimize over some convex set $\mathcal{C} \subseteq \mathbb{R}^p$, and we let $\theta_{*} = \mathop{\arg\min}\limits_{\theta \in \mathcal{C}}\mathcal{R}(\theta)$. We will make it clear when $\theta^*$ is the minimizer over $\mathcal{C}$. As we will see, for all models and risks we use, the true parameter $\theta^*$ will be the global minimizer over $\mathbb{R}^p$ and $\nabla\mathcal{R}(\theta^*) = 0$, apart from the setting of linear regression with $\ell_2$-regularized squared error loss.

In this paper, we will use three metrics to measure the performance of our methods. We will start by using the \emph{excess empirical risk} $\mathcal{L}(\theta_T, \mathcal{D}_n) - \mathop{\min}\limits_{\theta \in \mathcal{C}}\mathcal{L}(\theta, \mathcal{D}_n)$, where $\mathcal{L}(\theta, \mathcal{D}_n) := \frac{1}{n}\sum_{i = 1}^n\mathcal{L}(\theta, z_i)$. This will be relevant in Section \ref{sec:{Private ERM for Distribution-Free Data: Upper and Lower Bounds}}, where we do not assume the data to be random. We will also carry this metric over to Section \ref{sec:Private Biased Parameter Estimation in GLMs: Upper and Lower Bounds for Excess Risk}, where we study the benefits of acceleration for the purpose of privacy, in the Frank-Wolfe method, for generalized linear models. For parametric models, we will also use the $\ell_2$-distance $||\theta_T - \theta^*||_2$ between our estimate and the true parameter. Lastly, in Section \ref{sec:Smooth_Priv_HT_GD_AGD}, where the data are random, but we have no parametric model, we will use the \emph{excess risk} $\mathcal{R}(\theta_T) - \mathop{\min}\limits_{\theta \in \mathcal{C}}\mathcal{R}(\theta)$.


\subsubsection{Linear Regression}
\label{sec:Linear Regression}

In this setting, we have $\{z_i\}_{i = 1}^n = \{(x_i, y_i)\}_{i = 1}^n$ \iid from a distribution $P_{\theta^*}$. Assume the sample $(x, y) \sim P_{\theta^*}$ follows the model $y = x^T\theta^* + w$, where $x \indep w$, $\mathbb{E}[x] = 0$, and $\mathbb{E}[w] = 0$. Let $\Sigma := \mathbb{E}[xx^T]$ and $\sigma_2^2 := \mathbb{E}[w^2]$. We also assume throughout that $\sigma_2^2 < \infty$, and $\lambda_{\min}(\Sigma)$ and $\lambda_{\max}(\Sigma)$ are absolute constants. Now we present the different loss and population risk functions that we shall use.

\begin{example} [Squared error loss]
\label{sec:LR_Squared_Loss}
Assume $\Sigma \succ 0$ and consider the squared error loss
\begin{align*}
\mathcal{L}(\theta, (x, y)) = \frac{1}{2}(y - x^T\theta)^2.
\end{align*}
Then $\nabla \mathcal{L}(\theta, (x, y)) = (x^T\theta - y)x$, and the risk is $\mathcal{R}(\theta) = \mathbb{E}_{(x, y) \sim P_{\theta^*}}[\mathcal{L}(\theta, (x, y))] = \frac{1}{2}(\theta^* - \theta)^T\Sigma(\theta^* - \theta) + \frac{\sigma_2^2}{2}$, so $\nabla \mathcal{R}(\theta) = \Sigma(\theta - \theta^*)$ and $\nabla^2 \mathcal{R}(\theta) = \Sigma$.
Note that if $\theta^* \in \mathcal{C}$, the population risk is minimized at $\theta^*$.
Clearly, we can take
$\tau_u = \lambda_{\max}(\Sigma)$ and $\tau_l = \lambda_{\min}(\Sigma)$.
\end{example}


\begin{example} [$\ell_2$-regularized squared error loss (ridge regression)]
\label{sec:LR_l2_Reg}

Suppose $x$ has bounded $4^{\text{th}}$ moments. Consider the SVD $\Sigma = PSP^T$, and suppose
$\Sigma$ has $m \in \{1, \dots, p\}$ non-zero eigenvalues. 
We wish to estimate $\theta^*$, and to guarantee strong convexity, we optimize the regularized risk $\mathcal{R}_{\gamma_{\mathcal{C}}}(\theta) = \mathbb{E}[(y - x^T\theta)^2] + \frac{\gamma_{\mathcal{C}}}{2}||\theta||_2^2$, for some penalty $\gamma_{\mathcal{C}} > 0$, i.e., ridge regression. Observe that optimizing the regularized risk over $\mathbb{R}^p$ is the same as optimizing it over the constraint set $\mathcal{C} = \mathbb{B}_2(D)$ when $D \geq ||(\Sigma + \gamma_{\mathcal{C}} I_p)^{-1}\Sigma\theta^*||_2$, so $\theta_* = (\Sigma + \gamma_{\mathcal{C}}I_p)^{-1}\Sigma\theta^* \in \mathcal{C}$. Accordingly, we consider the squared error loss
\begin{align*}
\mathcal{L}_{\gamma_{\mathcal{C}}}(\theta, (x, y)) = \frac{1}{2}(y - x^T\theta)^2 + \frac{\gamma_{\mathcal{C}}}{2}||\theta||_2^2.
\end{align*}
Note that $\nabla \mathcal{R}_{\gamma_{\mathcal{C}}}(\theta) = \Sigma(\theta - \theta^*) + \gamma_{\mathcal{C}}\theta$ and $
\nabla^2 \mathcal{R}_{\gamma_{\mathcal{C}}}(\theta) = \Sigma + \gamma_{\mathcal{C}}I_p$.
By examining the Hessian, it is clear that we can take $\tau_u = \lambda_{\max}(\Sigma) + \gamma_{\mathcal{C}}$ and $\tau_l = \lambda_{\min}(\Sigma) + \gamma_{\mathcal{C}}$. 
\end{example}

\begin{remark}
Let $\left\|[P^T\theta^*]_{[1:m]}\right\|_2 \in \mathbb{R}^m$ be the vector obtained from $P^T\theta^*$ with the first $m$ entries. As $\gamma_C \rightarrow 0$, we have $||\theta_*||_2 \rightarrow \left\|[P^T\theta^*]_{[1:m]}\right\|_2$. Also, as $\gamma_C \rightarrow \infty$, we have $||\theta_*||_2 \rightarrow 0$. Hence, minimizing the penalized objective for some $\gamma_{\mathcal{C}} > 0$ is equivalent to optimizing $\mathbb{E}[(y - x^T\theta)^2]$ over an $\ell_2$-ball $\mathcal{V}$ centered at $0$, such that $\mathcal{V} \subset \mathbb{B}_2\left(\left\|[P^T\theta^*]_{[1:m]}\right\|_2\right)$. Note that as $\gamma_{\mathcal{C}} \rightarrow 0$, the radius of $\mathcal{V}$ increases to $\left\|[P^T\theta^*]_{[1:m]}\right\|_2$. If $m = p$, we are in the well-conditioned setting, and the radius of $\mathcal{V}$ approaches $||\theta^*||_2$ as $\gamma_{\mathcal{C}} \rightarrow 0$.
\end{remark}


\subsubsection{Generalized Linear Models (GLMs)}\label{sec:Generalized Linear Models (GLM)}

We will treat the linear regression model separately from general GLMs, since we will assume that $w$ in the linear model has a heavy-tailed distribution, so it does not necessarily fit in the general GLM as part of an exponential family. We will assume we have enough regularity to swap gradients in $\theta$ and expectations.

As in the case of linear regression, we have $\{z_i\}_{i = 1}^n = \{(x_i, y_i)\}_{i = 1}^n$ with $(x_i, y_i) \in \mathbb{R}^p \times \mathbb{R}$ \iid from $P_{\theta^*}$, but now $P_{\theta^*}$ links $y$ and $x$ in a conditional way:
\begin{align*}
P_{\theta^*}(y|x) \propto \exp\left(\frac{yx^T\theta^* - \Phi(x^T\theta^*)}{c(\sigma)} \right),
\end{align*}
with $c(\sigma)$ a known scale parameter and $\Phi : \mathbb{R} \rightarrow \mathbb{R}$ a known link function such that:
\begin{align*}
&|\Phi'(t)| \leq K_{\Phi'}, |\Phi''(t)| \leq K_{\Phi''}, \Phi''(t) > 0, \Phi''(t) = \Phi''(-t), \ \forall t \in \mathbb{R},\\
& \Phi'' \text{ is non-increasing on } \ [0, \infty),
\end{align*}
for absolute constants $K_{\Phi'}$ and $K_{\Phi''}$. Assume $\mathbb{E}[x] = 0$, $\mathbb{E}[xx^T] = \Sigma \succ 0$, and $\lambda_{\min}(\Sigma)$ and $\lambda_{\max}(\Sigma)$ are absolute constants. Since we know the conditional distribution of $y$ given $x$, we will use the negative log-likelihood loss
$\mathcal{L}(\theta, (x, y)) = -yx^T\theta + \Phi(x^T\theta)$, so $\nabla \mathcal{L}(\theta, (x, y)) = \left(\Phi'(x^T\theta) - y\right)x$. We do not have a closed-form expression for the risk, but by classical GLM theory, we have $\mathbb{E}[y|x] = \Phi^{'}(x^T\theta^*)$, so $\mathbb{E}[yx] = \mathbb{E}\left[\mathbb{E}[y|x]x\right] = \mathbb{E}[\Phi^{'}(x^T\theta^*)x]$. Thus, we have
\begin{align*}
\mathcal{R}(\theta) = -\theta^T\mathbb{E}[\Phi^{'}(x^T\theta^*)x] + \mathbb{E}[\Phi(x^T\theta)] = \mathbb{E}_{x}[\Phi(x^T\theta) - \Phi^{'}(x^T\theta^*)x^T\theta].
\end{align*}
By swapping expectations and gradients, we have
\begin{align}
\label{eq:GLMder}
\nabla \mathcal{R}(\theta) = \mathbb{E}_{x}[(\Phi^{'}(x^T\theta) - \Phi^{'}(x^T\theta^*))x], \qquad
\nabla^2 \mathcal{R}(\theta) = \mathbb{E}_{x}[\Phi''(x^T\theta)xx^T].
\end{align}
The following lemma regarding smoothness and strong convexity is proved in Appendix \ref{sec:Proof_GLM_prelim_lemma}:
\begin{lemma}
\label{lemma:GLM_prelim_lem}
Let $K_B > 0$ and consider a GLM. Then $\mathcal{R}$ is $K_{\Phi''}\lambda_{\max}(\Sigma)$-smooth over $\mathbb{R}^p$. Moreover, if $||x||_2 \leq L_x$ and $||\theta||_2 \leq K_B$ for all $\theta \in \mathcal{C}$, then $\mathcal{R}$ is $\Phi''(L_xK_B)\lambda_{\min}(\Sigma)$-strongly convex over $\mathcal{C}$. Finally, if $\theta^* \in \mathcal{C}$, then $\mathcal{R}$ is minimized at $\theta^*$, with $\nabla\mathcal{R}(\theta^*) = 0$.
\end{lemma}

Observe that logistic regression is a particular case of a GLM with $\Phi(t) = \log(1 + e^t)$ for all $t \in \mathbb{R}$, $K_y = 1$, $K_{\Phi'} = 1$, and $K_{\Phi''} = \frac{1}{4}$.


\section{Accelerating Frank-Wolfe}
\label{sec:Acc_FW_Method_All}

We aim to demonstrate the benefits of accelerated methods for privacy, particularly in the context of the Frank-Wolfe method. In Section \ref{sec:{Private ERM for Distribution-Free Data: Upper and Lower Bounds}}, we study empirical risk minimization (ERM) with deterministic data. Under the constraint of privacy, our goal is to obtain faster convergence guarantees on the expected excess empirical risk, and smaller iteration counts. We will focus on the accelerated Frank-Wolfe method described in Algorithm \ref{alg:ReAccFW}, and we will take Algorithm \ref{alg:PrivNon-accERM} from \cite{NOPL} as a baseline for comparison.

In Section \ref{sec:Private Biased Parameter Estimation in GLMs: Upper and Lower Bounds for Excess Risk}, we consider a GLM. The baseline for comparison is Algorithm \ref{alg:PrivNon-accERM}. We will first analyze a strategy of increasing the constraint set toward the true parameter $\theta^*$ as $n \rightarrow \infty$, so our estimator is consistent. For completeness, we also consider a regime where the constraint set is fixed with $n$. The measure of our performance will  again be excess empirical risk. We will further derive a bound on the $\ell_2$-error of the estimated parameter.

Motivated by the positive effect of acceleration in the context of privacy, in Section \ref{sec:Biased Parameter Estimation HT}, we study applications of the accelerated Frank-Wolfe method to heavy-tailed robustness. We focus on a parametric linear model, both in a well conditioned ($\lambda_{\min}(\Sigma) > 0$) and ill-conditioned ($\lambda_{\min}(\Sigma) = 0$) setting, as introduced in Examples \ref{sec:LR_Squared_Loss} and \ref{sec:LR_l2_Reg}. Several authors \cite{RE, CompIntr, Balakrish} considered noisy gradient methods, which can be seen as applications of robust mean estimators, to obtain robust estimators for various learning problems, such as estimation in parametric models \cite{DP_ROB_HIGH_DIM, NO_ROB_LR}. We will first establish a setup from \cite{RE} based on gradient estimators, and then present our results for estimating $\theta^*$ using variants of Frank-Wolfe.


\subsection{Private ERM for Distribution-Free Data}
\label{sec:{Private ERM for Distribution-Free Data: Upper and Lower Bounds}}

Common approaches to private ERM include output perturbation \cite{EFF_PRIV_SCO, (NEAR), BOLT_ON_PRIV, EFF_PRIV_SMOOTH} and noisy gradient descent \cite{NO_ROB_LR, NOPL}. Our central motivation is the paper \cite{NOPL}, where noisy gradients are incorporated into the classical Frank-Wolfe algorithm to obtain bounds on the expected excess empirical risk, when optimization occurs over a polytope. They specialize this result for the Lasso problem, and provide a lower bound result to show near-optimality of their method. They then present a similar noisy Frank-Wolfe algorithm, i.e., Algorithm~\ref{alg:PrivNon-accERM}, for a general convex set $\mathcal{C}$ of finite diameter and for $L_2$-Lipschitz losses in the $\ell_2$-norm.

\begin{algorithm}
\caption{$\mathcal{A}_{\text{Noise-FW(Gen-convex)}}$: Differentially Private Frank-Wolfe Algorithm (General Convex
Case)}
\label{alg:PrivNon-accERM}
\begin{algorithmic}[1]
\Function{$\mathcal{A}_{\text{Noise-FW(Gen-convex)}}$}{$\mathcal{D}_n = \{z_1, \ldots, z_n\}$, loss function $\mathcal{L}(\theta, \mathcal{D}_n) = \frac{1}{n}\sum_{i = 1}^n\mathcal{L}(\theta, z_i)$, Lipschitz constant $L_2$, $\epsilon, \delta, T$, constraint set $\mathcal{C}$}
\State Choose $\theta_0 \in \mathcal{C} \subseteq \mathbb{R}^p$ arbitrary
    \For{$t = 0$ to $T - 1$}
    \State $v_t = \mathop{\arg \min}\limits_{v \in \mathcal{C}} (\nabla \mathcal{L}(\theta_t, \mathcal{D}_n) + \xi_t)^Tv$, \ with \ $\xi_t \overset{\iid}{\sim} N\left(0, \frac{32L_2^2T\log^2(n/\delta)}{n^2\epsilon^2}I_p\right)$.
        \State $\theta_{t + 1} = (1 - \eta_t)\theta_t + \eta_t v_t$, \ with $\eta_t = \frac{2}{t + 2}$.
    \EndFor
    \State \textbf{return} $\theta_T$.
\EndFunction
\end{algorithmic}
\end{algorithm}

An easy application of Corollary~\ref{cor:ADV_COMP_0.9} shows that Algorithm~\ref{alg:PrivNon-accERM} is $(\epsilon, \delta)$-DP for $\epsilon \in (0, 0.9]$ and $\delta \in (0, 1)$. For a convex, bounded set $\mathcal{C}$, \cite{NOPL} derive an upper bound on the expected excess empirical risk of $\widetilde{O}\left(\frac{\Gamma_{\mathcal{L}}^{1/3}(L_2G_{\mathcal{C}})^{2/3}}{(n\epsilon)^{2/3}}\right)$, where $G_{\mathcal{C}} = \mathbb{E}\left[\mathop{\sup}\limits_{\theta \in \mathcal{C}}\theta^T b\right]$, with $b \sim N(0, I_p)$, is the Gaussian width of $\mathcal{C}$, and $\Gamma_{\mathcal{L}}$ is the curvature constant of $\mathcal{L}(\theta, z_1)$ (cf.\ Lemma \ref{lemma:NOPLconvDPFW} in Appendix \ref{sec:Appendix_FW} for more details).
This result takes the geometry of the set $\mathcal{C}$ into account only through $\Gamma_\mathcal{L}$ and $G_{\mathcal{C}}$. However, the proof of the utility guarantee does not rely on any particularities of the geometry of $\mathcal{C}$. Hence, Lemma~\ref{lemma:NOPLconvDPFW} could be sub-optimal in situations where one deals with $\ell_2$-norms, i.e., when dealing with $\ell_2$-balls centered at $0$.
For $\mathcal{D}_n = \{(x_i, y_i)\}_{i = 1}^n$, we define the mean squared error loss $\mathcal{L}(\theta, \mathcal{D}_n) := \frac{1}{2n} \sum_{i=1}^n (y_i - x_i^T \theta)^2$. 

We address this sub-optimality via acceleration, in Algorithm \ref{alg:PrivFWERM} of Section \ref{sec:The Upper Bound}. This algorithm is similar to Algorithm \ref{alg:PrivNon-accERM}, but uses a learning rate derived from Algorithm \ref{alg:ReAccFW}.
We then set the number of iterations $T$ based on $n$, with the intuition that Algorithm \ref{alg:PrivFWERM} should outperform Algorithm \ref{alg:PrivNon-accERM} in terms of the rates with $n$, $p$, and $T$. 
The iteration count $T$ is crucial: In Section \ref{sec:The Lower Bound}, we show the optimality of our upper bound in Theorem \ref{theorem:UPRidge} via a lower bound with rate $\frac{1}{n^{2/3}}$, assuming $p \asymp m^2$, $n \asymp \frac{m^3}{\log(m)}$, and $\mathcal{C} = \mathbb{B}_2(D)$, with $D \asymp \frac{1}{\sqrt{p}}$, as $m \rightarrow \infty$ (the same assumptions on $n$ and $p$ are used in \cite{NOPL} to prove the lower bound result for the Lasso analysis). Algorithm \ref{alg:PrivFWERM} achieves its utility guarantee with $T \asymp \log(n)$, while Lemma \ref{lemma:NOPLconvDPFW} requires $T = \widetilde{\Theta}(n^{4/9})$, under the same scaling of $n, p$, and $D$. Thus, our method attains the optimal $\frac{1}{n^{2/3}}$ rate (up to logarithmic factors) with only logarithmically many iterations.

We now discuss our approach in detail. We start with the upper bound in Section \ref{sec:The Upper Bound} and then move to the lower bound in Section \ref{sec:The Lower Bound}. Before stating the upper bound, we introduce our accelerated noisy Frank-Wolfe algorithm along with its privacy guarantee. The upper bound result is established without assumptions on $n$, $p$, or the radius $D$ of the $\ell_2$-ball $\mathcal{C}$ centered at $0$, using the squared error loss, and assuming $|y_i|, ||x_i||_\infty \leq 1$, for all $i \in [n]$. It also requires a lower bound on the $\ell_2$-norm of the empirical risk gradient, consistent with the form of the data in the lower bound result in Section \ref{sec:The Lower Bound}. To strengthen the upper bound result, we will show that the assumptions on the data can be satisfied with high probability, under a specific model. We focus on the scaling with both $n$ and $p$.


\subsubsection{Upper Bound}
\label{sec:The Upper Bound}

We now state our accelerated noisy Frank-Wolfe algorithm, which differs from Algorithm \ref{alg:PrivNon-accERM} in the choice of learning rate. The following privacy guarantee is proved in Appendix \ref{AppLemACC}:

\begin{algorithm}
\caption{Private Frank-Wolfe for ERM}
\label{alg:PrivFWERM}
\begin{algorithmic}[1]
\Function{PrivFWERM}{$\mathcal{D}_n = \{z_i\}_{i = 1}^n$, loss function $\mathcal{L}(\theta, \mathcal{D}_n) = \frac{\sum_{i = 1}^n\mathcal{L}(\theta, z_i)}{n}$, Lipschitz constant $L_2$, $\beta_{\mathcal{L}}$, $\alpha_{\mathcal{C}}$, $r$, $T$, $\epsilon$, $\delta$, constraint set $\mathcal{C}$}
    \State Choose $\theta_0 \in \mathcal{C}$ arbitrary
    \For{$t = 0$ to $T - 1$}
    \State $v_t = \mathop{\arg \min}\limits_{v \in \mathcal{C}} (\nabla \mathcal{L}(\theta_t, \mathcal{D}_n) + \xi_t)^Tv$, \ with \ $\xi_t \overset{\iid}{\sim} N\left(0, \frac{64L_2^2 T \log\left(\frac{5T}{2\delta}\right)\log\left(\frac{2}{\delta}\right)}{n^2\epsilon^2}I_p\right)$.
        \State $\theta_{t + 1} = (1 - \eta)\theta_t + \eta v_t$, where $\eta = \min\left\{1, \frac{\alpha_{\mathcal{C}}r}{4\beta_{\mathcal{L}}}\right\}$
    \EndFor
    \State \textbf{return} $\theta_T$.
\EndFunction
\end{algorithmic}
\end{algorithm}

\begin{lemma}
\label{lemma:ACC_PRIV_FW_STEP}
Algorithm \ref{alg:PrivFWERM} is $\left(\frac{\epsilon}{2} + \frac{\sqrt{T}\epsilon}{2\sqrt{2\log(2/\delta)}}(e^{\epsilon/2\sqrt{2T\log(2/\delta)}} - 1), \delta\right)$-DP, for $\delta \in (0, 1)$, $\epsilon < 2\sqrt{2T\log(2/\delta)}$, and $\delta < 2T$. If in addition $\epsilon \leq 0.9$, then $\theta_T$ is $(\epsilon, \delta)$-DP.
\end{lemma}

\paragraph{Distribution-Free Result}

We will use Algorithm \ref{alg:PrivFWERM} to design a mechanism $\hat{\theta}$ that is differentially private and achieves the rate $\frac{(\sqrt{p} + p||\mathcal{C}||_2)||\mathcal{C}||_2\sqrt{p}}{n\epsilon}$, up to logarithmic factors. As mentioned, we will impose some conditions on the data and later explain how the data in the lower bound argument (Theorem \ref{theorem:LBRidge} in Section \ref{sec:The Lower Bound}) satisfy the conditions. The proof of the following result can be found in Appendix \ref{AppThmUPR}:

\begin{theorem}
\label{theorem:UPRidge}
Let $S_1 > 0$ be an absolute constant. Let $\mathcal{E} = \mathbb{B}_{\infty}(1) \times [-1, 1]$ and $\mathcal{C} = \mathbb{B}_2\left(D\right)$, with $D > 0$ and $\alpha_{\mathcal{C}} = \frac{1}{D}$. Let $\mathcal{L}$ be the mean squared error loss, and $\beta_{\mathcal{L}} = \frac{1}{n}\left\|\sum_{i = 1}^nx_ix_i^T\right\|_2$. Then for any dataset $\mathcal{D}_n = \{(x_i, y_i)\}_{i = 1}^n$ such that $|y_i| \leq 1$, $||x_i||_{\infty} \leq 1$, and $\mathop{\inf}\limits_{\theta \in \mathcal{C}}\frac{\alpha_{\mathcal{C}}||\nabla \mathcal{L}(\theta, \mathcal{D}_n)||_2}{\beta_{\mathcal{L}}} \geq S_1$, Algorithm \ref{alg:PrivFWERM} with $0 < \epsilon \le 0.9$, $\delta \in (0, 1)$, $L_2 \leq \sqrt{p} + pD$, $r = \frac{S_1\beta_{\mathcal{L}}}{\alpha_{\mathcal{C}}}$,
and $T \asymp \log n$
returns $\theta_T$ which is $(\epsilon, \delta)$-DP and satisfies
\begin{align*}
\mathbb{E}\left[\mathcal{L}(\theta_T, \mathcal{D}_n) - \mathop{\min}\limits_{\theta \in \mathcal{C}}\mathcal{L}(\theta, \mathcal{D}_n)\right] \lesssim \frac{(\sqrt{p} + p||\mathcal{C}||_2)||\mathcal{C}||_2\sqrt{p}\log^{3/2}(n)\log(\log(n)/\delta)}{n\epsilon}.
\end{align*}
\end{theorem}

\begin{remark}
\label{remark:UP_Bridge_DP_Dist_Free}


Theorem \ref{theorem:UPRidge} only assumes $S_1, D > 0$. In Section \ref{sec:The Lower Bound}, to show that the data constructed for the lower bound satisfy the conditions in Theorem \ref{theorem:UPRidge}, we will further constrain the parameters to $D = \frac{\alpha_1}{\sqrt{p}}$, $0 < \alpha_1 < \frac{\sqrt{1 - \tau}}{1 + \tau}$, $0 < S_1 \leq \frac{\sqrt{1 - \tau} - (1 + \tau)\alpha_1}{\alpha_1(1 + \tau)}$, and $\tau = 0.001$.
\end{remark}

\begin{remark}
\label{remark:compare_dist_free_ACCFW}
We can compare the results in Lemma \ref{lemma:NOPLconvDPFW} and Theorem \ref{theorem:UPRidge}: We may bound
\begin{align*}
L_2 \leq ||yx - xx^T\theta||_2 \leq ||x||_2 + ||xx^T||_2||\theta||_2 \le \sqrt{p}||x||_\infty + ||x||_2^2 \|\theta\|_2  \lesssim \sqrt{p} + p||\mathcal{C}||_2,
\end{align*}
for arbitrary $|y|, ||x||_\infty \leq 1$.
In the context of Lemma~\ref{lemma:NOPLconvDPFW}, we have
\begin{align*}
G_{\mathcal{C}} = D\mathbb{E}\left[||b||_2\right] \asymp ||\mathcal{C}||_2\sqrt{p}, \qquad
\Gamma_{\mathcal{L}} \lesssim \mathop{\sup}\limits_{\theta \in \mathcal{C}}||x_1^T\theta||_2^2 \asymp p||\mathcal{C}||_2^2.
\end{align*}
(This bound is tight, as implied by \cite{NOPL}.) Hence, the upper bound in Lemma~\ref{lemma:NOPLconvDPFW} is $\widetilde{O}\left(\left(\frac{(\sqrt{p} + p||\mathcal{C}||_2)||\mathcal{C}||_2^2p}{n\epsilon}\right)^{2/3}\right)$. If $p, ||\mathcal{C}||_2 \asymp 1$, the rate in Theorem \ref{theorem:UPRidge} is improved to $\widetilde{O}\left(\frac{1}{n\epsilon}\right)$. Also, for $\epsilon \asymp 1, p \asymp m^2, n \asymp \frac{m^3}{\log(m)}$, and $D \asymp \frac{1}{\sqrt{p}}$, we prove the optimality of Theorem~\ref{theorem:UPRidge} in Theorem~\ref{theorem:LBRidge}, which shows that the expected empirical risk is $\widetilde{\Omega}\left(\frac{1}{n^{2/3}}\right)$. Under these conditions, the bound in
\cite{NOPL} becomes $\widetilde{O}\left(\frac{1}{m^{4/3}}\right) = \widetilde{O}\left(\frac{1}{n^{4/9}}\right)$, which is sub-optimal.
\end{remark}


\paragraph{Probabilistic Data}

We finish by analyzing the conditions on the dataset in Theorem \ref{theorem:UPRidge}. We will impose a linear model to prove the lower bound $\mathop{\inf}\limits_{\theta \in \mathcal{C}}\frac{\alpha_{\mathcal{C}}||\nabla \mathcal{L}(\theta, \mathcal{D}_n)||_2}{\beta_{\mathcal{L}}} \ge S_1$ with high probability. The proof of the following result can be found in Appendix \ref{AppPropERM}:

\begin{prop}
\label{prop:achievab_UP_ERM}
Let $c_1 > 1$ and $c_2 > \frac{5}{4}$ be absolute constants, and consider a regime where $n, p \rightarrow \infty$. Let $0 < C_1 \leq C_2 \leq 1$ and $S_1 > 0$ be absolute constants, and let $\mathcal{L}$ be the mean squared error loss. Suppose data $\mathcal{D}_n = \{(x_i, y_i)\}_{i = 1}^n$ are drawn \iid from the model
\begin{align}
\label{eq:model_distr_free_data}
&y = x^T\theta^* + w^{(p)}, \; |y| \leq 1, \; ||x||_{\infty} \leq 1, \; x \indep w^{(p)}, \notag\\
&\mathbb{E}[x] = 0, \; \Sigma = \mathbb{E}[xx^T], \; C_1 \leq \lambda_{\min}(\Sigma) \leq \lambda_{\max}(\Sigma) \leq C_2, \notag\\
&\left|w^{(p)}\right| \leq 1 + \sqrt{p}K_1(p), \; \mathbb{E}\left[w^{(p)}\right] = 0, \; w^{(p)} \in \mathcal{G}\left(\sigma^2(p)\right), \notag\\
&(2S_1(2C_2/C_1 + 1) + 1)D(p) \leq ||\theta^*||_2 \leq K_1(p),
\end{align}
where $K_1(p), D(p) \rightarrow 0$ as $p \rightarrow \infty$, and $\sigma^2(p) > 0$ for all $p \in \mathbb{N}$. Let $\mathcal{C} = \mathbb{B}_2\left(D(p)\right)$, with $\alpha_{\mathcal{C}} = \frac{1}{D(p)}$. Let $\beta_{\mathcal{L}} = \frac{1}{n}\left\|\sum_{i = 1}^nx_ix_i^T\right\|_2$. Then, for $p \geq \left(\frac{\sqrt{2}}{S_1\sqrt{2C_2 + C_1}}\right)^8$ and $n = \widetilde{\Omega}\left(\max\left\{\frac{p^{c_2}\sigma^2(p)}{D^2(p)}, p^{c_1}\right\}\right)$, with probability at least $1 - 2pe^{\frac{-nC_1^2}{8p\left(C_2+ C_1/3\right)}} - 2pe^{-\frac{nD^2(p)}{2p^{5/4}\sigma^2(p)}}$, we have
$\mathop{\inf}\limits_{\theta \in \mathcal{C}}\frac{\alpha_{\mathcal{C}}||\nabla \mathcal{L}(\theta, \mathcal{D}_n)||_2}{\beta_{\mathcal{L}}} \geq S_1$. Moreover, the conditions \eqref{eq:model_distr_free_data} can be satisfied if $w^{(p)}$ follows a truncated $N(0, \sigma^2(p))$ in the interval $[-1 - \sqrt{p}K_1(p), 1 + \sqrt{p}K_1(p)]$.
\end{prop}

\subsubsection{Lower Bound}
\label{sec:The Lower Bound}

In this section, we treat $\epsilon$ as an absolute constant and again focus on the mean squared error loss optimized over some set $\mathcal{C}$, with data from $\mathcal{E} = \mathbb{B}_{\infty}(1)\times [-1, 1]$. We will assume $\mathcal{C} \supseteq \left\{-\frac{\alpha_2}{p}, \frac{\alpha_2}{p}\right\}^p$ and choose $\alpha_2$ appropriately. Our arguments will follow the fingerprinting method from \cite{NOPL}.
The following theorem is proved in Appendix \ref{AppThmLBRidge}. The proof is a modification of a result in \cite{NOPL}, the key difference being the introduction of the term $\alpha_2$. The dimensions of the construction are as follows: For a sufficiently large positive integer $m$, we take $p = 1000m^2$ and $n = w + 0.001wp$, where $w = \frac{m}{\log(m)}$.

\begin{theorem}
\label{theorem:LBRidge}
Let $\alpha_2 \in (0.993, 1)$, and let $p$ be sufficiently large and $n$ be chosen appropriately. Let $\mathcal{C} \subseteq \mathbb{R}^p$ be such that $\left\{-\frac{\alpha_2}{p}, \frac{\alpha_2}{p}\right\}^p \subseteq \mathcal{C}$, and let $\mathcal{L}$ be the mean squared error loss. For any $(\epsilon, \delta)$-DP algorithm $\hat{\theta}$, where $\epsilon = 0.1$ and $\delta = o(1/n^2)$, there exists $\mathcal{D}_n = \{(x_i, y_i)\}_{i = 1}^n$, with $||x_i||_{\infty} \leq 1$ and $|y_i| \leq 1$, such that 
\begin{align*}
\mathbb{E}\left[\mathcal{L}(\hat{\theta}(\mathcal{D}_n), \mathcal{D}_n) - \mathop{\min}\limits_{\theta \in \mathcal{C}}\mathcal{L}(\theta, \mathcal{D}_n)\right] = \widetilde{\Omega}\left(\frac{1}{n^{2/3}}\right).
\end{align*}
\end{theorem}

\subsubsection{Minimax Optimality}


Consider the statement of Theorem~\ref{theorem:LBRidge} with $\mathcal{C} = \mathbb{B}_2\left(\frac{\alpha_2}{\sqrt{p}}\right)$. In Remark \ref{remark:UP_Bridge_DP_Dist_Free}, we explained that we would impose further restrictions on $\mathcal{C}$ and $S_1$ in order to prove that the dataset in Theorem~\ref{theorem:LBRidge} satisfies the conditions in Theorem~\ref{theorem:UPRidge}, in order to reconcile the bounds.
The proof of the following result is in Appendix \ref{AppPropMatch}. For a matrix $X$, denote by $X_{(-i)}$ the matrix obtained by removing the $i^{\text{th}}$ row of $X$. Call a column of a matrix a \emph{consensus column} if all entries are the same.

\begin{prop}
\label{prop:MatchData}
Let $m \in \mathbb{N}$, $\tau = 0.001$, $p = 1000m^2$, $w = \frac{m}{\log(m)}$, $k = \tau w p$, and $n = w + k$. Let $X \in \{-1, 1\}^{(w + 1) \times p}$ be such that for each $i \in [1, w + 1]$, there are at least $(1 - \tau)p$ consensus columns in each $X_{(-i)}$. Let $Z \in \{-1, 1\}^{k \times p}$ be such that $Z^TZ = kI_p$. Denote the $j^{\text{th}}$ row of $Z$ by $z_j$.
Consider the dataset $\mathcal{D}_n = \{(x_j, y_j)\}_{j = 1}^n = \{(x^j_{(-i)}, 1)\}_{j = 1}^w \cup \{(z_j, 0)\}_{j = 1}^k$, where $x^j_{(-i)}$ is the $j^{\text{th}}$ row of $X_{(-i)}$. Let $\mathcal{L}$ be the mean squared error loss and let $\mathcal{C} = \mathbb{B}_2\left(\frac{\alpha_1}{\sqrt{p}}\right)$, with $0 < \alpha_1 < \frac{\sqrt{1 - \tau}}{1 + \tau}$ and $\alpha_{\mathcal{C}} = \frac{\sqrt{p}}{\alpha_1}$. Let $\beta_{\mathcal{L}} = \frac{1}{n}\left\|\sum_{j = 1}^nx_jx_j^T\right\|_2$. Let $S_1 \in \left(0, \frac{\sqrt{1 - \tau} - (1 + \tau)\alpha_1}{\alpha_1(1 + \tau)}\right]$. Then
\begin{align}
\label{eq:MatchEq}
|y_j|, ||x_j||_{\infty} \leq 1, \ \forall j \in [n], \qquad \mathop{\inf}\limits_{\theta \in \mathcal{C}}\frac{\alpha_{\mathcal{C}}||\nabla \mathcal{L}(\theta, \mathcal{D}_n)||_2}{\beta_{\mathcal{L}}} \geq S_1.
\end{align}
\end{prop}

To summarize, if we choose $\alpha \in \left(0.993, \frac{\sqrt{1 - \tau}}{1 + \tau}\right)$, and since $\frac{\sqrt{1 - \tau}}{1 + \tau} \approx 0.9985$, Algorithm \ref{alg:PrivFWERM} for the ridge regression problem with $\mathcal{C} = \mathbb{B}_2\left(\frac{\alpha}{\sqrt{p}}\right)$ is nearly optimal up to logarithmic factors. More specifically, for any $\alpha \in \left(0.993, \frac{\sqrt{1 - \tau}}{1 + \tau}\right)$ and $S_1 \in \left(0, \frac{\sqrt{1 - \tau} - (1 + \tau)\alpha}{\alpha(1 + \tau)}\right]$
and the choice of $(n,p)$ appearing in Proposition~\ref{prop:MatchData}, define
the class of datasets
$$\mathcal{S}_n^\alpha = \left\{\mathcal{D}_n = \{(x_i, y_i)\}_{i = 1}^n: 
|y_i|, ||x_i||_{\infty} \leq 1, \ \forall i \in [n]  \mbox{ and } \mathop{\inf}\limits_{\theta \in \mathcal{C}}\left\|\nabla\mathcal{L}(\theta, \mathcal{D}_n)\right\|_2 \geq \frac{\alpha\beta_{\mathcal{L}}S_1}{\sqrt{p}}\right\},$$
with $\beta_{\mathcal{L}} = \frac{1}{n}\left\|\sum_{i = 1}^nx_ix_i^T\right\|_2$.
Taking $\Theta_{\epsilon, \delta, \mathcal{C}}$ to be the collection of all $(\epsilon, \delta)$-DP mechanisms with output in $\mathcal{C}$, with $\epsilon = 0.1$, $\delta \asymp \frac{1}{n^{\omega_1}}$, and $\omega_1 > 2$ an absolute constant, we have
\begin{align*}
\mathop{\inf}\limits_{\hat{\theta} \in \Theta_{\epsilon, \delta, \mathcal{C}}} \mathop{\sup}_{\substack{\mathcal{D}_n \in \mathcal{S}_n^\alpha}} \mathbb{E}\left[\mathcal{L}(\hat{\theta}, \mathcal{D}_n) - \mathop{\min}\limits_{\theta \in \mathcal{C}}\mathcal{L}(\theta, \mathcal{D}_n)\right] = \widetilde{\Theta}\left(\frac{1}{n^{2/3}}\right).
\end{align*} 
This follows directly from Theorem~\ref{theorem:UPRidge}, Theorem~\ref{theorem:LBRidge}, and Proposition~\ref{prop:MatchData}. 

Thus, we saw that by a careful choice of learning rate in the noisy Frank-Wolfe algorithm, we obtained a utility guarantee that is nearly optimal in certain cases and requires significantly fewer iterations than Algorithm \ref{alg:PrivNon-accERM}. This was facilitated by leveraging the strong convexity of $\mathcal{C}$ and a lower bound on the $\ell_2$-norm of the gradient of the empirical risk. 


\subsection{Private Estimation in GLMs}
\label{sec:Private Biased Parameter Estimation in GLMs: Upper and Lower Bounds for Excess Risk}

Continuing the study of ERM, we aim to use the accelerated Frank-Wolfe method (Algorithm \ref{alg:PrivFWERM}) to estimate the true parameter $\theta^*$ in a GLM. This builds on the idea in Section \ref{sec:The Upper Bound} that allowed us to obtain a high-probability statement regarding the conditions on the data in Theorem~\ref{theorem:UPRidge}, under a parametric model. Throughout this section, we will assume the data are generated from a GLM. We once again take Algorithm \ref{alg:PrivNon-accERM} as a baseline for comparison. Our goal is to showcase the advantage of acceleration during iterative optimization. Our methods will again rely on bringing Algorithm \ref{alg:PrivFWERM} in a form where we can use Theorem \ref{theorem:ACCFWV3} with high probability. However, since that result requires a lower bound on the $\ell_2$-norm of the gradient of the empirical risk, we will need to optimize over an $\ell_2$-ball $\mathcal{C}$ such that $\theta^* \notin \mathcal{C}$. To make the estimator consistent, we will allow $\mathcal{C}$ to increase toward $\theta^*$ as $n \rightarrow \infty$ in Section \ref{sec:Increasing C}. In Section \ref{sec:Fixed C and the Lower Bound}, we derive a complementary upper bound on the excess empirical risk, under the assumption that $\mathcal{C}$ is fixed.

Throughout this section, we will work with bounded covariates and responses. The loss will be the negative log likelihood (cf.\ Section \ref{sec:Generalized Linear Models (GLM)}). We first state a general theorem based on the accelerated Frank-Wolfe method for the upper bound and then specialize it to different sets $\mathcal{C}$. 
This time, we will consider the scaling of our bounds with $n$ only. Hence, quantities involving $p$, $c(\sigma)$ (as in Section \ref{sec:Generalized Linear Models (GLM)}), and $||\theta^*||_2$ will be treated as absolute constants.

The main message is that acceleration is again beneficial in terms of the number of iterations $T$ and the upper bound on the excess empirical risk. As we will see in Theorem~\ref{theorem:ERM_UP_C_Inc}, we can set $T \asymp n^{2/5}\log(n)$ in Algorithm \ref{alg:PrivFWERM}. In contrast, Algorithm \ref{alg:PrivNon-accERM} requires $T \asymp n^{2/3}$ (cf.\ Lemma~\ref{lemma:NOPLconvDPFW}). Moreover, Algorithm \ref{alg:PrivFWERM} yields an upper bound of $\frac{1}{n^{4/5}}$ (up to logarithmic factors) on the excess empirical risk for GLMs (cf.\ Remark~\ref{remark:Best_q_ERM}), in contrast to the rate of $\frac{1}{n^{2/3}}$ for Algorithm~\ref{alg:PrivNon-accERM} (cf.\ Lemma~\ref{lemma:NOPLconvDPFW}). This stems from the fact that the variance of the Gaussian noise added scales with $T$, so the smaller number of iterations results in a smaller variance of the noisy gradients, in turn producing better statistical performance.


\subsubsection{General Upper Bound}

We begin by providing a general upper bound, proved in Appendix \ref{AppThmERM}. The parameter $q < \frac{1}{2}$ will be optimized in Section \ref{sec:Increasing C}.

\begin{theorem}
\label{theorem:ERM_GLM_UP_General}
Let $\mathcal{E} = \mathbb{B}_2(L_x)\times [-K_y, K_y]$, with $K_y, L_x \asymp 1$. Suppose $\mathcal{C} = \mathbb{B}_{2}\left(D\right)$, with $D > 0$. Assume $\epsilon > 0$ and $\delta \in (0, 1)$. Suppose $\theta^* \in \mathbb{R}^p \setminus \mathcal{C}$. Set $\alpha_{\mathcal{C}} = \frac{1}{D}$. Consider the GLM setting from Section \ref{sec:Generalized Linear Models (GLM)}, with $|y_i| \leq K_y$ and $||x_i||_2 \leq L_x$, for all $i \in [n]$. Let $\zeta \in (0, 1)$ and $\beta_{\mathcal{L}} = K_{\Phi''}L_x^2$. Let $L_2 = (K_{\Phi'} + K_y)L_x$ and $q < \frac{1}{2}$. Then there are absolute constants $C'_1$ and $C_1$ such that for $n \geq C'_1$ and
\begin{align*}
0 < r \leq \frac{\Phi''(L_x||\theta^*||_2)\lambda_{\min}(\Sigma)}{2}(||\theta^*||_2 - D) - \sqrt{\frac{C_1\log(2/\zeta)}{n}} - \frac{1}{n^q},
\end{align*}
Algorithm \ref{alg:PrivFWERM} with $T = \log_{1/c}(n)$, where $c = \max\left\{\frac{1}{2}, 1 - \frac{\alpha_{\mathcal{C}}r}{8K_{\Phi''}L_x^2}\right\}$, returns $\theta_T$ which is $(\epsilon, \delta)$-DP, and with probability at least $1 - \zeta$, we have
\begin{align*}
\mathcal{L}(\theta_T, \mathcal{D}_n) - \mathop{\min}\limits_{\theta \in \mathcal{C}}\mathcal{L}(\theta, \mathcal{D}_n) \lesssim \frac{1}{n} + \frac{\eta\log\left(\log_{1/c}(n)/\delta\right)\sqrt{\log_{1/c}(n)\log\left(\log_{1/c}(n)/\zeta\right)}}{(1 - c)n\epsilon}.
\end{align*} 
\end{theorem}




\subsubsection{Accelerated Frank-Wolfe with Increasing $\mathcal{C}$}\label{sec:Increasing C}

In this section, we increase the constraint set $\mathcal{C} = \mathbb{B}_2(D)$ in such a way that $\|\theta^*\|_2 - D \asymp \frac{1}{n^{2/5}}$.
The proof of the following result (see Appendix \ref{AppThmERM2}) relies on Theorem~\ref{theorem:ERM_GLM_UP_General}, with $q = \frac{2}{5}$:

\begin{theorem}
\label{theorem:ERM_UP_C_Inc}
Let $\mathcal{E} = \mathbb{B}_2(L_x)\times [-K_y, K_y]$, with $K_y, L_x \asymp 1$. Suppose $\mathcal{C} = \mathbb{B}_{2}\left(D\right)$, with $||\theta^*||_2 - D \lesssim \frac{1}{n^{2/5}}$. Set $\alpha_{\mathcal{C}} = \frac{1}{D}$ and let $0 < \epsilon \le 0.9$ and $\delta \in (0, 1)$. Consider the GLM setting from Section \ref{sec:Generalized Linear Models (GLM)}, with $|y_i| \leq K_y$ and $||x_i||_2 \leq L_x$, for all $i \in [n]$. Let $\zeta \in (0, 1/3)$, $L_2 = (K_{\Phi'} + K_y)L_x$ and $\beta_{\mathcal{L}} = K_{\Phi''}L_x^2$. Then there are absolute constants $C'_1, C_1, C_2, C_3$, $N_\zeta, T_\zeta > 0$ such that for $n > \max\left\{C_2\log^5(2/\zeta), N_\zeta, C'_1\right\}$, $D \leq ||\theta^*||_2 - \frac{C_3}{n^{2/5}}$, and $r = \frac{1}{n^{2/5}} - \sqrt{\frac{C_1\log(2/\zeta)}{n}}$, Algorithm \ref{alg:PrivFWERM} with
$T = \widetilde{\Theta}\left(n^{2/5}\right)$
returns $\theta_T$ which is $(\epsilon, \delta)$-DP, and with probability at least $1 - 3\zeta$, we have
\begin{align}
\label{eq:Upper_ERM_GLM_Increasing_C}
\mathcal{L}(\theta_T, \mathcal{D}_n) - \mathop{\min}\limits_{\theta \in \mathbb{B}_2(||\theta^*||_2)}\mathcal{L}(\theta, \mathcal{D}_n) \lesssim \frac{T_\zeta\log(n/\delta)\sqrt{\log(n)\log(n/\zeta)}}{n^{4/5}\epsilon}.
\end{align} 
\end{theorem}

\begin{remark}
\label{remark:Best_q_ERM}
Lemma~\ref{lemma:NOPLconvDPFW} provides a bound in expectation, whereas Theorem~\ref{theorem:ERM_UP_C_Inc} provides a high-probability bound. We cannot use Lemma \ref{lemma:HP_TO_EXP} because the lower bound on $n$ in Theorem~\ref{theorem:ERM_UP_C_Inc} depends on $\zeta$.
Ignoring the mismatch, we compare the convergence rates: The exponent of $\epsilon$ in Lemma~\ref{lemma:NOPLconvDPFW} is better, i.e., $\frac{2}{3}$, as opposed to $1$ in Theorem \ref{theorem:ERM_UP_C_Inc}. If we treat $\epsilon$ as an absolute constant and focus on the dependence of the rates on $n$, we indeed improve over the rate of $\frac{1}{n^{2/3}}$ obtained using Lemma~\ref{lemma:NOPLconvDPFW} from Talwar et al.~\cite{NOPL}.
On the other hand, note that in Lemma~\ref{lemma:NOPLconvDPFW}, there are no distributional assumptions on the data, whereas we assume a GLM in Theorem~\ref{theorem:ERM_UP_C_Inc}. Assuming such a model allows us to use an accelerated version of the Frank-Wolfe method. Additionally, we are able to leverage the strong convexity of the $\ell_2$-ball $\mathcal{C}$, while Lemma~\ref{lemma:NOPLconvDPFW} only assumes that the underlying set $\mathcal{C}$ is convex and bounded. 
\end{remark}


Moreover, we can further derive a bound on the parameter error from Theorem~\ref{theorem:ERM_UP_C_Inc}. This leads to the following result, proved in Appendix \ref{AppThmIterRate}:

\begin{theorem}
\label{theorem:Iter_Rate_Inc_C}
Consider the setup from Theorem \ref{theorem:ERM_UP_C_Inc} and suppose also that $\zeta \in (0, 1/4)$ and $n > C_4 \log(2p/\zeta)$.
With probability at least $1 - 4\zeta$, Algorithm \ref{alg:PrivFWERM} with $T = \widetilde{\Theta}\left(n^{2/5}\right)$ returns $\theta_T$ satisfying
\begin{align*}
||\theta_T - \theta^*||_2 \lesssim \frac{T_\zeta\log(n)}{\sqrt{n}} + \frac{T_\zeta^{1/2}\log^{1/2}(n/\delta)\log^{1/4}(n)\log^{1/4}(n/\zeta)}{n^{2/5}\epsilon^{1/2}}.
\end{align*}
\end{theorem}

\begin{remark}
\label{remark:Compare_Iter_Inc_C_to_Cost_Priv}
The rate for $||\theta_T - \theta^*||_2$ in Theorem \ref{theorem:Iter_Rate_Inc_C} is $\widetilde{O}\left(\frac{1}{\sqrt{n}} + \frac{1}{n^{2/5}\sqrt{\epsilon}}\right)$. In \cite{Cai_Cost_GLM}, the minimax rate in terms of $n$ and $0 < \epsilon \lesssim 1$ is suggested to be $\frac{1}{\sqrt{n}} + \frac{1}{n\epsilon}$, provided $||\theta_T - \theta^*||_2$ stays bounded for all $n$ large enough, and optimization occurs over the whole of $\mathbb{R}^p$. There is a small discrepancy in \cite{Cai_Cost_GLM} for the upper and lower bounds: Their upper bound holds with probability at least $1 - c_1e^{-c_2n} - c_1e^{-c_2p} - c_1e^{-c_2\log(n)}$, for absolute constants $c_1, c_2 > 0$, while the lower bound is for the expected error. Disregarding these differences and treating $\epsilon$ as a constant, this leads to a rate of $\frac{1}{\sqrt{n}}$, up to log factors, which is achieved by Theorem \ref{theorem:Iter_Rate_Inc_C}. 

Note that if we want to beat the cost of privacy term in \cite{Cai_Cost_GLM}, we need to pick $\epsilon < \widetilde{O}\left(\frac{1}{n^{6/5}}\right)$. However, the upper bounds on $||\theta_T - \theta^*||_2$ in Theorem \ref{theorem:Iter_Rate_Inc_C} and \cite{Cai_Cost_GLM} blow up to infinity as $n \rightarrow \infty$ and are therefore not useful. 
It remains an open question to write the upper bound for an expected value, or the lower bound on an event with high probability.

On the other hand, note that Theorem \ref{theorem:Iter_Rate_Inc_C} holds with probability at least $1-4\zeta$ for any $\zeta\in(0,1/4)$ fixed at the beginning, while the probabilistic guarantee in \cite{Cai_Cost_GLM} cannot be made arbitrarily close to 1, regardless of $n$, if $p$ is fixed. Moreover, \cite{Cai_Cost_GLM} requires the initialization $\theta_0$ to lie in an $\ell_2$-ball of radius 3 around the minimizer of $\mathcal{L}(\cdot, \mathcal{D}_n)$, whereas Theorem \ref{theorem:Iter_Rate_Inc_C} holds for any $\theta_0 \in \mathcal{C}$ (we may thus choose $\theta_0 = 0$).
\end{remark}

Similar to the result of Theorem \ref{theorem:Iter_Rate_Inc_C}, we can derive a bound on the iterates for the non-accelerated Frank-Wolfe method, using the version of Lemma \ref{lemma:NOPLconvDPFW} from \cite{NOPL}, with high probability instead of expectation. The proof of the following result is in Appendix \ref{AppPropIter}:

\begin{prop}
\label{prop:Iter_Rate_Inc_C_Non}
Let $\mathcal{E} = \mathbb{B}_2(L_x)\times [-K_y, K_y]$, with $K_y, L_x \asymp 1$. Let $0 < \epsilon \lesssim 1$ and $\delta \in (0, 1)$. Consider the GLM setting from Section \ref{sec:Generalized Linear Models (GLM)}, with $|y_i| \leq K_y$ and $||x_i||_2 \leq L_x$, for all $i \in [n]$. Let $\zeta \in (0, 1/3)$, $L_2 = (K_{\Phi'} + K_y)L_x$ and $\beta_{\mathcal{L}} = K_{\Phi''}L_x^2$. Then there are absolute constants $C_1, C_2$, $T_\zeta, N_\zeta > 0$ such that for $n > \max\left\{C_1\log(2p/\zeta), C_2, N_\zeta\right\}$, Algorithm \ref{alg:PrivNon-accERM} with $T \asymp (n\epsilon)^{2/3}$ returns $\theta_T$ which is $(\epsilon, \delta)$-DP, and with probability at least $1 - 3\zeta$, we have 
\begin{align*}
||\theta_T - \theta^*||_2 = \widetilde{O}\left(\frac{T_\zeta}{\sqrt{n}} + \frac{\log^{1/2}(n\epsilon/\zeta)}{(n\epsilon)^{1/3}}\right).
\end{align*}
\end{prop}

\begin{remark}
As in Remark \ref{remark:Compare_Iter_Inc_C_to_Cost_Priv}, the results of \cite{Cai_Cost_GLM} suggest that the statistical rate of $\frac{1}{\sqrt{n}}$ appearing in Proposition~\ref{prop:Iter_Rate_Inc_C_Non} is optimal, whereas the cost of privacy term $\frac{1}{(n\epsilon)^{1/3}}$ is not. Note that the benefit of acceleration can be observed in the iteration count ($T = \widetilde{\Theta}\left(n^{2/5}\right)$ in Theorem \ref{theorem:Iter_Rate_Inc_C} vs.\ $T \asymp (n\epsilon)^{2/3}$ in Proposition \ref{prop:Iter_Rate_Inc_C_Non}) and in the cost of privacy term ($\frac{1}{n^{2/5}\epsilon^{1/2}}$ in Theorem \ref{theorem:Iter_Rate_Inc_C} vs.\ $\frac{1}{(n\epsilon)^{1/3}}$ in Proposition \ref{prop:Iter_Rate_Inc_C_Non}). Thus, as before, acceleration is useful for the reduction of the iteration count and the lower variance of the noise required for privacy.
\end{remark}


\subsubsection{Accelerated Frank-Wolfe with Fixed $\mathcal{C}$}
\label{sec:Fixed C and the Lower Bound}

We now consider the setting where the radius of $\mathcal{C}$ is independent of $n$. Rather than targeting $\theta^*$, we will seek to bound the excess empirical risk. The proof of the following result is provided in Appendix \ref{ThmERMGLM}:

\begin{theorem}
\label{theorem:ERM_GLM_UP}
Let $\mathcal{E} = \mathbb{B}_2(L_x)\times [-K_y, K_y]$, with $K_y, L_x \asymp 1$. Suppose $\mathcal{C} = \mathbb{B}_{2}\left(D\right)$, with $||\theta^*||_2 - D \asymp 1$. Suppose $\theta^* \in \mathbb{R}^p \setminus \mathcal{C}$. Set $\alpha_{\mathcal{C}} = \frac{1}{D}$ so that $\mathcal{C}$ is $\alpha_{\mathcal{C}}$-strongly convex. Consider the GLM setting from Section \ref{sec:Generalized Linear Models (GLM)}, with $|y_i| \leq K_y$ and $||x_i||_2 \leq L_x$, for all $i \in [n]$. Let $0 < \epsilon \le 0.9$ and $\delta \in (0, 1)$. Let $L_2 = (K_{\Phi'} + K_y)L_x$ and $\beta_{\mathcal{L}} = K_{\Phi''}L_x^2$. Then there are absolute constants $C'_1, C_1$, $C_2 > 0$, such that, for $n > \max\left\{\left(\frac{\sqrt{C_1\log(2n)} + 1}{C_2}\right)^{4}, C'_1\right\}$ and $r \in\left(\frac{C_2}{2}, C_2\right]$, Algorithm \ref{alg:PrivFWERM} with $T \asymp \log n$
returns $\theta_T$ which is $(\epsilon, \delta$)-DP and satisfies
\begin{align*}
\mathbb{E}\left[\mathcal{L}(\theta_T, \mathcal{D}_n) - \mathop{\min}\limits_{\theta \in \mathcal{C}}\mathcal{L}(\theta, \mathcal{D}_n)\right] \lesssim \frac{\log\left(\log(n)/\delta\right)\sqrt{\log(n)\log\left(n\log(n)\right)}}{n\epsilon}.
\end{align*}
\end{theorem}

\begin{remark}
\label{remark:More_General_C_ERM}
Since Theorem \ref{theorem:ERM_GLM_UP_General} does not target the true parameter $\theta^*$, the corresponding bias will not decrease to $0$ as $n \rightarrow \infty$. However, if we knew $||\theta^*||_2$, we could choose $\mathcal{C} = \mathbb{B}_2(D)$ with $D$ arbitrarily close to $||\theta^*||_2$.
Additionally, we can compare this result to the non-accelerated Frank-Wolfe result (cf.\ Lemma \ref{lemma:NOPLconvDPFW}). Compared to their rate of $\widetilde{O}\left(\frac{1}{(n\epsilon)^{2/3}}\right)$, with iteration count $T = \widetilde{O}\left((n\epsilon)^{2/3}\right)$, our  result in Theorem \ref{theorem:ERM_GLM_UP_General} achieves a rate $\widetilde{O}\left(\frac{1}{n\epsilon}\right)$ with iteration count $T \asymp \log(n)$. Hence, the accelerated Frank-Wolfe approach produces both a better rate and better iteration complexity. Both the non-accelerated and accelerated methods use all $n$ gradients of the data at each iteration step of the Frank-Wolfe procedure, so acceleration provides a better gradient complexity, as well.
\end{remark}

\subsection{Heavy-Tailed Robust Estimation in Linear Models}\label{sec:Biased Parameter Estimation HT}

We now shift from privacy to robustness. We examine the linear model from Section \ref{sec:Linear Regression}. Our method is heavy-tailed robust because we only assume that $\mathbb{E}[w^2] < \infty$ and $x$ has bounded fourth moments. The strong convexity of the squared error risk is guaranteed if $\lambda_{\min}(\Sigma) > 0$. To improve performance in ill-conditioned settings, we incorporate acceleration and introduce a regularizer to ensure strong convexity. This is analogous to the method in Section \ref{sec:Private Biased Parameter Estimation in GLMs: Upper and Lower Bounds for Excess Risk} of optimizing over an expanding $\ell_2$-ball centered at $0$. Hence, we split the analysis in two sections: Section \ref{sec:Well_Cond_Sec_FW} for the well-conditioned case and Section \ref{sec:Ill_Cond_Sec_FW} for the ill-conditioned scenario.
For the purpose of this section, we treat $p$ and $||\theta^*||_2$ as absolute constants. 
In the ill-conditioned case, we will assume that $\Sigma$ has $m$ non-zero eigenvalues.
For completeness, we will also analyze the alternative approach of projected gradient descent in Appendix \ref{sec:Projected Gradient Descent for l=0}.


\subsubsection{Approximate Gradient Estimators}

Let us now discuss the setup, based on \cite{RE}. Robust gradient estimators naturally trade off with accurately estimating $\theta_{*}$; the iterates may not converge exactly to $\theta_{*}$ as iterations increase. We aim to control the deviation of the estimators from the true gradients, and the next definition makes this precise:

\begin{definition}[Adapted from \cite{RE}]
\label{def:ge}
For \iid samples $\mathcal{D}_n = \{z_i\}_{i = 1}^n$ and a differentiable risk $\mathcal{R}(\theta)$ minimized at $\theta_*$, a function $g(\theta, \mathcal{D}_n, \zeta)$ is a \emph{gradient estimator} if there are functions $\alpha$ and $\beta$ such that at any fixed $\theta \in \mathcal{C}$, with probability at least $1 - \zeta$, we have
\begin{align*}
||g(\theta, \mathcal{D}_n, \zeta) - \nabla\mathcal{R}(\theta)||_2 \leq \alpha(n, \zeta)||\theta - \theta_*||_2 + \beta(n, \zeta).
\end{align*}
If $\mathcal{R}(\theta)$ is $\tau_l$-strongly convex and $\alpha(n, \zeta) < \frac{\tau_l}{2}$, we call $g$ \emph{stable}.
\end{definition}




As in \cite{RE}, we consider a geometric median of means ($G_{MOM}$) gradient estimator, described in Algorithm~\ref{alg:HTGE}. We could, in principle, look at a noisy version of this, so that we can achieve privacy, as well:
We could use a Lipschitz loss (such as a Huber loss, cf.\ Appendix \ref{Appendix A}) and a noisy version of the $G_{MOM}$ estimator to simultaneously obtain privacy and robustness. However, we focus only on heavy-tailed robustness for simplicity. Since the approach from \cite{RE} will also be used in the later sections, we combine all the gradient methods in one algorithm (Algorithm \ref{alg:RobPGDNFW} in Appendix \ref{Appendix A}) in the cases where our optimization occurs over $\mathcal{C} = \mathbb{R}^p$ or over a compact, convex $\mathcal{C} \subseteq \mathbb{R}^p$.
Instead of the choice $b = 1 + \lfloor 3.5 \log(1/\zeta)\rfloor$ in \cite{RE}, we use the bucket choice from \cite{Minsker} in Algorithm~\ref{alg:HTGE}. A key detail missing in \cite{RE} is the condition on $\zeta$: to ensure $b \leq n/2$ in the heavy-tailed case, $\zeta$ must be chosen accordingly, giving a lower bound on $n$ in terms of $\log(1/\zeta)$. In line with Algorithm \ref{alg:HTGE} that outputs a geometric median, \cite{Lera} examine mean estimators that concentrate exponentially around the true mean for distributions with bounded $2^{\text{nd}}$ moments. 
A theoretical guarantee for Algorithm~\ref{alg:HTGE} is provided in Lemma \ref{lemma:htlemma} in Appendix \ref{AppLemHT}.


\begin{algorithm}
\caption{Heavy-Tailed Gradient Estimator}
\label{alg:HTGE}
\begin{algorithmic}[1]
\Function{HTGE}{$S = \{\nabla \mathcal{L}(\theta; z_i)\}_{i=1}^n$, $n$, $\zeta$, $\psi(x) = (1 - x)\log\left(\frac{1 - x}{0.9}\right) + x \log\left(\frac{x}{0.1}\right)$, for $x \in (0, 1)$}
    \State Define number of buckets $b = \left\lfloor \frac{\log(1/\zeta)}{\psi(7/18)} \right\rfloor + 1 \leq 1 + \left\lfloor 3.5 \log(1/\zeta) \right\rfloor$.
    \State Partition $S$ into $b$ blocks $B_1, \ldots, B_b$ each of size $\left\lfloor \frac{n}{b} \right\rfloor$.
    \For{$i = 1$ to $b$}
        \State $\widehat{\mu}_i = \frac{1}{|B_i|} \sum_{s \in B_i} s$.
    \EndFor
    \State Let $\widehat{\mu} = \mathop{\arg \min\limits_{\mu}} \sum_{i=1}^{b} \| \mu - \widehat{\mu}_i \|_2$.
    \State \textbf{return} $\widehat{\mu}$.
\EndFunction
\end{algorithmic}
\end{algorithm}


To simplify analysis, we split the data into $T$ chunks. This is because the high-probability concentration result of the $G_{MOM}$ estimator will assume a fixed $\theta \in \mathcal{C}$, and when applied in our noisy gradient algorithm, we use independence between the randomness in $\theta_t$ and that of the gradient estimator to analyze $\theta_{t + 1}$. Denote $\widetilde{n} := \left\lfloor n/T \right\rfloor$ and $\widetilde{\zeta} := \left\lfloor \zeta/T \right\rfloor$.
For the remainder of this section, we consider the linear model introduced in Section \ref{sec:Linear Regression}. We suppress the dependency on $p$, $\lambda_{\max}(\Sigma), \lambda_{\min}(\Sigma)$, and $||\theta^*||_2$.

\subsubsection{The Well-Conditioned Case}\label{sec:Well_Cond_Sec_FW}

In this section, we assume $\lambda_{\min}(\Sigma) > 0$. We will use Algorithm \ref{alg:HTGE} to construct robust gradient estimators. The corresponding functions $\alpha$ and $\beta$ will be identified using Lemma \ref{lemma:geglmht}, which is taken from \cite{RE} (see Appendix \ref{Appendix A} for the statement of the lemma).
We will consider both the non-accelerated and accelerated Frank-Wolfe methods. The idea for the non-accelerated version is to bring Algorithm \ref{alg:RobPGDNFW} for the Frank-Wolfe method into the relaxed version of Lemma \ref{lemma:NonAccRel_FW}. For the accelerated version, we bring Algorithm \ref{alg:RobPGDNFW} for the Frank-Wolfe method into the relaxed version of Algorithm \ref{alg:ReAccFW}.

\paragraph{Frank-Wolfe}

We begin by analyzing the non-accelerated version. The proof of the following result can be found in Appendix \ref{sec:ThmNonACCFWl>0}:

\begin{theorem}
\label{theorem:HT_NonAccFW}
Consider the linear regression with squared error loss model from Example~\ref{sec:LR_Squared_Loss}. Let $\mathcal{C} \subseteq \mathbb{R}^p$ be convex and compact, such that $\theta^* \in \mathcal{C}$ and $||\mathcal{C}||_2 \lesssim 1$. Let $\zeta \in (0, 1)$. Then Algorithm \ref{alg:RobPGDNFW} for the Frank-Wolfe method with variable learning rate $\eta_t = \frac{2}{2 + t}$, and using Algorithm \ref{alg:HTGE} as gradient estimator, returns iterates $\{\theta_t\}_{t = 1}^T$ such that with probability at least $1 - \zeta$, for $T = n^{1/3}$, we have
\begin{align}
\label{eq:NotAccFW_HT}
||\theta_T - \theta^*||_2 \lesssim \frac{(1 + \sigma_2)^{1/2}\log^{1/4}(n/\zeta)}{n^{1/6}},
\end{align}
where $\widetilde{n} \geq 2b$, with $b$ as in Algorithm \ref{alg:HTGE}.
\end{theorem}

\begin{remark}
We can comment on the choice of $T$ in Theorem \ref{theorem:HT_NonAccFW}. By looking at the proof, in order to minimize $\frac{1}{T} + (1 + \sigma_2)\sqrt{\frac{T\log(T/\zeta)}{n}}$ over $T > 0$, we take $T = n^{1/3}$.
\end{remark}

\paragraph{Accelerated Frank-Wolfe}

We now move on to the accelerated version, where we aim to use Algorithm \ref{alg:ReAccFW}. To do so, we need to make sure the $\ell_2$-norm of the gradient of the squared error risk is bounded away from $0$ and the constraint set is strongly convex. Hence, we use the same strategy that we employed in Section \ref{sec:Increasing C}: We optimize over $\mathbb{B}_2(D)$, which increases toward $\theta^*$ as $n \rightarrow \infty$. More specifically, we will have $||\theta^*||_2 - D \asymp \frac{1}{n^{1/5}}$. The proof of the next theorem is provided in Appendix \ref{sec:ThmACCFWl>0}:

\begin{theorem}
\label{theorem:ACCFWROB}
Let $C_1 > 0$ be an absolute constant. Let $\zeta \in (0, 1)$. Consider the linear regression with squared error loss model from Example~\ref{sec:LR_Squared_Loss}. Let $\mathcal{C} = \mathbb{B}_2(D)$, where $||\theta^*||_2 - D \lesssim\frac{1}{n^{1/5}}$ and $D \leq ||\theta^*||_2 - \frac{C_1}{n^{1/5}}$. Then Algorithm \ref{alg:RobPGDNFW} for the Frank-Wolfe method with $\theta_0 \in \mathcal{C}$, $\eta = \min\left\{1, \frac{\alpha_{\mathcal{C}} u}{4\lambda_{\max}(\Sigma)}\right\}$, $\alpha_{\mathcal{C}} = \frac{1}{D}$, $\frac{1}{n^{1/5}} \lesssim u \leq \frac{C_1\lambda_{\min}(\Sigma)}{n^{1/5}}$, and using Algorithm \ref{alg:HTGE} as gradient estimator returns iterates $\{\theta_t\}_{t = 1}^T$ such that with probability at least $1 - \zeta$, for $T = \log_{1/c}\left(n^{2/5}\right) \asymp n^{1/5}\log(n)$, we have
\begin{align}
\label{eq:ACCFWHTrate}
||\theta_T - \theta^*||_2 \lesssim \frac{(1 + \sigma_2)^{1/2}\log^{1/4}(n)\log^{1/4}\left(n\log(n)/\zeta\right)}{n^{1/5}},
\end{align}
with $c = \max\left\{\frac{1}{2}, 1 - \frac{\alpha_{\mathcal{C}} u}{8\lambda_{\max}(\Sigma)}\right\}$, where $\widetilde{n} \geq 2b$, with $b$ is as in Algorithm \ref{alg:HTGE}.
\end{theorem}

\paragraph{Comparisons}

We compare Theorems \ref{theorem:HT_NonAccFW} and \ref{theorem:ACCFWROB}, emphasizing the benefits of acceleration in the Frank-Wolfe method. We see that the accelerated Frank-Wolfe approach in \eqref{eq:ACCFWHTrate} is better, having a rate of $\frac{1}{n^{1/5}}$, compared to $\frac{1}{n^{1/6}}$. Everything here is up to logarithmic factors. Notice also that both upper bounds have the same dependency on the variance of the noise $w$, namely $(1 + \sigma_2)^{1/2}$. On the other hand, the non-accelerated version is more general regarding the constraint set $\mathcal{C}$ and does not necessarily ask for the boundary of $\mathcal{C}$ to be close to $\theta^*$. 

We can also compare iteration counts. The accelerated version is faster, since it requires $T \asymp n^{1/5}\log(n)$ iterations, as opposed to $T = n^{1/3}$ for the non-accelerated approach. Hence, we can see a similar conclusion to the one in the context of privacy: Using an accelerated method leads to a smaller iteration count, which in turn leads to better statistical performance. The parallel we can draw between the privacy and robustness analyses is the fact that, in both cases, we are optimizing using noisy versions of gradients of certain objectives. The noise in both cases is more volatile as the iteration count increases. Hence, the benefit of a smaller iteration count becomes apparent in both private and robust optimization.

\subsubsection{The Ill-Conditioned Case}\label{sec:Ill_Cond_Sec_FW}

In this section, we wish to learn the true parameter $\theta^*$ when $\lambda_{\min}(\Sigma) = 0$.
We will construct a gradient estimator based on Algorithm \ref{alg:HTGE}. We will also identify the functions $\alpha$ and $\beta$ used in Definition \ref{def:ge}. We will keep track of $\sigma_2$ and $\gamma_{\mathcal{C}}$. As discussed in Example~\ref{sec:LR_l2_Reg}, we will minimize the regularized squared error risk $\mathcal{R}_{\gamma_{\mathcal{C}}}$ over an $\ell_2$-ball $\mathcal{C} = \mathbb{B}_2(D)$, with $D \geq ||(\Sigma + \gamma_{\mathcal{C}} I_p)^{-1}\Sigma\theta^*||_2$, and for some $\gamma_{\mathcal{C}} > 0$.
We take $D = \frac{||\mathcal{C}||_2}{2}$ to be an absolute constant.
Let us first construct an appropriate gradient estimator. The following lemma is proved in Appendix \ref{AppLemGeri}:

\begin{lemma}
\label{lemma:geri}
Consider the linear regression with $\ell_2$-regularized squared error loss model defined in Section \ref{sec:Linear Regression}, with i.i.d.\ samples $\mathcal{D}_n = \{z_i\}_{i = 1}^n = \{(x_i, y_i)\}_{i = 1}^n$ from a heavy-tailed distribution. Then Algorithm \ref{alg:HTGE} returns, for $\theta \in \mathcal{C}$ fixed, a gradient estimator $g$ such that
\begin{align*}
||g(\theta; \mathcal{D}_n, \widetilde{\zeta}) - \nabla\mathcal{R}_{\gamma_{\mathcal{C}}}(\theta)||_2  \lesssim \sqrt{\frac{\log(1/\widetilde{\zeta})}{\widetilde{n}}}||\theta - \theta_*||_2 + \sqrt{\frac{\left(\sigma_2^2 + \frac{\gamma_{\mathcal{C}}^2}{(\lambda_{\min}(\Sigma) + \gamma_{\mathcal{C}})^2}\right)\log(1/\widetilde{\zeta})}{\widetilde{n}}},
\end{align*}
with probability at least $1 - \widetilde{\zeta}$`, and $b \leq \widetilde{n}/2$, with $b$ as in Algorithm \ref{alg:HTGE}.
\end{lemma}

We now discuss the two Frank-Wolfe variants. For the non-accelerated version, the goal will be to bring Algorithm \ref{alg:RobPGDNFW} for the Frank-Wolfe method into the relaxed version of Lemma \ref{lemma:NonAccRel_FW}. Similarly, for the accelerated version, we aim to bring Algorithm \ref{alg:RobPGDNFW} for the Frank-Wolfe method into the relaxed version of Algorithm \ref{alg:ReAccFW}.


\paragraph{Frank-Wolfe}

We begin with the non-accelerated Frank-Wolfe method, and we also keep track of the dependency on $\left\|[P^T\theta^*]_{[(m+1):p]}\right\|_2$, the only term that vanishes when $m = p$. The proof is in Appendix \ref{ThmNonACC0}.

\begin{theorem}
\label{theorem:NonACC_lambda=0}
Consider the linear regression with $\ell_2$-regularized squared error loss model from Section \ref{sec:Linear Regression}, with $\frac{1}{n^{1/9}} \lesssim \gamma_{\mathcal{C}} \rightarrow 0$ as $n \rightarrow \infty$, and suppose we optimize over $\mathcal{C} = \mathbb{B}_2\left(D\right)$, with $D \geq ||(\Sigma + \gamma_{\mathcal{C}}I_p)^{-1}\Sigma\theta^*||_2$. Assume that the top $m$ eigenvalues of $\Sigma$ are positive, with $0 < m < p$. Let $[P^T\theta^*]_{[(m + 1):p]}$ be the vector in $\mathbb{R}^{p-m}$ containing the bottom $p-m$ entries of $P^T\theta^*$. Let $\zeta \in (0, 1)$. Then Algorithm \ref{alg:RobPGDNFW} for the Frank-Wolfe method with learning rate $\eta_t = \frac{2}{2 + t}$, using Algorithm \ref{alg:HTGE} as gradient estimator, returns iterates $\{\theta_t\}_{t = 1}^T$ such that with probability at least $1 - \zeta$, for $T = n^{1/3}$, we have
\begin{align*}
||\theta_T - \theta^*||_2 \lesssim \frac{(1 + \sigma_2)^{1/2}\log^{1/4}(n/\zeta)}{n^{1/9}} + \left\|[P^T\theta^*]_{[(m+1):p]}\right\|_2,
\end{align*}
if $\widetilde{n} \geq 2b$, with $b$ as in Algorithm \ref{alg:HTGE}.
\end{theorem}

\begin{remark}
The choice $T = n^{1/3}$ is not arbitrary in Theorem \ref{theorem:NonACC_lambda=0}: As the proof reveals, this is the best $T$ we can choose in order to minimize $\frac{1}{T} + (1 + \sigma_2)\sqrt{\frac{T\log(T/\zeta)}{n}}$ over $T > 0$.

\end{remark}


\paragraph{Accelerated Frank-Wolfe}

Now we can move on to optimizing $\mathcal{R}_{\gamma_{\mathcal{C}}}$ via the accelerated Frank-Wolfe method. The difference compared to projected gradient descent (as one can see in Appendix \ref{sec:Projected Gradient Descent for l=0}) and the non-accelerated Frank-Wolfe method cases is that we optimize over a fixed $\ell_2$-ball $\mathcal{K} \subsetneq \mathbb{B}_{2}\left(||(\Sigma + \gamma_{\mathcal{C}}I_p)^{-1}\Sigma\theta^*||_2\right)$. By choosing $\mathcal{K}$ appropriately, we can achieve a better rate of $\frac{1}{n^{1/4}}$, plus an error term given by $\left\|[P^T\theta^*]_{[(m + 1):p]}\right\|_2 + \left\|[P^T\theta^*]_{[(m + 1):p]}\right\|_2^{1/2}$.
The proof of the following result is provided in Appendix \ref{AppThmACCFW}:

\begin{theorem}
\label{theorem:ACCFW_lambda=0}
Let $C_1 > 1$ be an absolute constant. Let $\zeta \in (0, 1)$. Consider the linear regression with $\ell_2$-regularized squared error loss model from Section \ref{sec:Linear Regression}. Assume that the top $m$ eigenvalues of $\Sigma$ are positive with $0 < m < p$. Let $[P^T\theta^*]_{[(m + 1):p]}$ be the vector in $\mathbb{R}^{p - m}$ containing the bottom $p - m$ entries of $P^T\theta^*$ and let $c_{\mathcal{K}} = \left\|[P^T\theta^*]_{[(m + 1):p]}\right\|_2$. We optimize over $\mathcal{K} = \mathbb{B}_2\left(K\right)$, with $\left\|\left(\Sigma + C_1c_{\mathcal{K}}I_p\right)^{-1}\Sigma\theta^*\right\|_2 \leq K \leq \left\|\left(\Sigma + c_{\mathcal{K}}I_p\right)^{-1}\Sigma\theta^*\right\|_2$. Also, assume $\widetilde{n} \geq 2b$, with $b$ as in Algorithm \ref{alg:HTGE} and $\gamma_{\mathcal{C}} \in\left[\frac{c_{\mathcal{K}}}{4}, \frac{c_{\mathcal{K}}}{2}\right]$. Then Algorithm \ref{alg:RobPGDNFW} for the Frank-Wolfe method, with $\theta_0 \in \mathcal{K}$, $\eta = \min\left\{1, \frac{\alpha_{\mathcal{K}} u}{4(\lambda_{\max}(\Sigma) + \gamma_{\mathcal{C}})}\right\}$, $\alpha_{\mathcal{K}} = \frac{1}{K}$, $\gamma_{\mathcal{C}}c_{\mathcal{K}} \lesssim u \leq \gamma_{\mathcal{C}}\frac{S_{mm}^2\left\|[P^T\theta^*]_{[1:m]}\right\|_2c_{\mathcal{K}}}{2(S_{mm} + c_{\mathcal{K}})^3}$, and using Algorithm \ref{alg:HTGE} as gradient estimator, returns iterates $\{\theta_t\}_{t = 1}^T$ such that with probability at least $1 - \zeta$, for $T \asymp \log(n)/c_{\mathcal{K}}^2$,
we have
\begin{align*}
||\theta_T - \theta^*||_2 \lesssim \frac{(1 + \sigma_2)^{1/2}\log^{1/4}(n)\log^{1/4}\left(n/\zeta\right)}{\left\|[P^T\theta^*]_{[(m+1):p]}\right\|_2^{1/4}n^{1/4}} + \left\|[P^T\theta^*]_{[(m+1):p]}\right\|_2 + \left\|[P^T\theta^*]_{[(m+1):p]}\right\|_2^{1/2}.
\end{align*}
\end{theorem}

\begin{remark}
The bound in Theorem \ref{theorem:ACCFW_lambda=0} holds provided $\gamma_{\mathcal{C}} \in \left[\frac{c_{\mathcal{K}}}{4}, \frac{c_{\mathcal{K}}}{2}\right]$. This constant scaling of $\gamma_{\mathcal{C}}$ with $n$ is the best we can do with our analysis: in inequality~\eqref{eq:HT_ACCFWl=0_best_gc} in the proof of Theorem \ref{theorem:ACCFW_lambda=0}, if $\gamma_{\mathcal{C}}$ goes to $0$ or $\infty$ as $n \rightarrow \infty$, the upper bound tends to infinity. Hence, in order to control the error term that does not depend on $n$ and obtain a bound as small as possible in terms of $c_{\mathcal{K}} = \left\|[P^T\theta^*]_{[(m+1):p]}\right\|_2$, we choose $\gamma_{\mathcal{C}} \gtrsim c_{\mathcal{K}}$ such that $\gamma_{\mathcal{C}} \leq \frac{c_{\mathcal{K}}}{2}$.

From a practical standpoint, we need to optimize over a ball $\mathcal{K} = \mathbb{B}_2(K)$, with $\left\|\left(\Sigma + C_1c_{\mathcal{K}}I_p\right)^{-1}\Sigma\theta^*\right\|_2 \leq K \leq \left\|\left(\Sigma + c_{\mathcal{K}}I_p\right)^{-1}\Sigma\theta^*\right\|_2$, so we do not need to know $||\theta^*||_2$ or $\Sigma$ precisely. Moreover, from the hypothesis of Theorem \ref{theorem:ACCFW_lambda=0}, we do not need to know these parameters explicitly in order to choose $\alpha_{\mathcal{K}} = \frac{1}{K}$ and $u$.

\end{remark}

\paragraph{Comparisons}

We now compare the results of Theorems \ref{theorem:NonACC_lambda=0} and \ref{theorem:ACCFW_lambda=0}. Regarding the bounds on $||\theta_T - \theta^*||_2$, for each of the two results, we use a regularized risk and pick the penalty $\gamma_{\mathcal{C}}$ so that the bound on $||\theta_T - \theta^*||_2$ is a tight as possible. For Theorem \ref{theorem:NonACC_lambda=0}, we picked $\frac{1}{n^{1/9}} \lesssim \gamma_{\mathcal{C}} \rightarrow 0$; for Theorem \ref{theorem:ACCFW_lambda=0}, we picked $\gamma_{\mathcal{C}} \in \left[\frac{c_{\mathcal{K}}}{4}, \frac{c_{\mathcal{K}}}{2}\right]$. Also, we compare the upper bounds on $||\theta_T - \theta^*||_2$ up to logarithmic factors.

All these bounds have an error term involving $\left\|[P^T\theta^*]_{[(m+1):p]}\right\|_2$. In each result, we suppressed the dependence on other constants, such as $m$ or $\left\|[P^T\theta^*]_{[1:m]}\right\|_2$. This is because the only constant that vanishes once $m = p$ is $\left\|[P^T\theta^*]_{[(m+1):p]}\right\|_2$. In Theorem \ref{theorem:NonACC_lambda=0}, the bound is of the form $\widetilde{O}\left(\frac{1}{n^{1/9}}\right) + \left\|[P^T\theta^*]_{[(m+1):p]}\right\|_2$, and in Theorem \ref{theorem:ACCFW_lambda=0}, the bound is of the form 
\begin{align*}
\widetilde{O}\left(\frac{1}{\left\|[P^T\theta^*]_{[(m+1):p]}\right\|_2^{1/4}n^{1/4}}\right) + \left\|[P^T\theta^*]_{[(m+1):p]}\right\|_2 + \left\|[P^T\theta^*]_{[(m+1):p]}\right\|_2^{1/2}.
\end{align*}
If $\left\|[P^T\theta^*]_{[(m+1):p]}\right\|_2 \geq 1$, the bound in Theorem \ref{theorem:ACCFW_lambda=0} becomes $\widetilde{O}\left(\frac{1}{n^{1/4}}\right) + \left\|[P^T\theta^*]_{[(m+1):p]}\right\|_2$. In this case, the result in Theorem \ref{theorem:ACCFW_lambda=0} is tighter in terms of the rate with $n$ and the constant $\left\|[P^T\theta^*]_{[(m+1):p]}\right\|_2$. One intuition why Theorem \ref{theorem:NonACC_lambda=0} performs worse is because the non-accelerated Frank-Wolfe method does not take into account the nature of the constraint set. Moreover, notice the strategy used in Theorem \ref{theorem:ACCFW_lambda=0}: We did not optimize over the $\ell_2$-ball $\mathcal{C}$ with radius $||(\Sigma + \gamma_{\mathcal{C}}I_p)^{-1}\Sigma\theta^*||_2$, but over a ball of constant radius. The reason is because, in the case when $\lambda_{\min}(\Sigma) = 0$, we produce a constant error in the upper bound anyway, so we decided to pick the constant radius of the ball over which we optimize such that the additional constant  error incurred scales roughly like $\left\|[P^T\theta^*]_{[(m+1):p]}\right\|_2$. 

Observe that if instead $\left\|[P^T\theta^*]_{[(m+1):p]}\right\|_2 < 1$, the bound in Theorem \ref{theorem:ACCFW_lambda=0} becomes $\widetilde{O}\left(\frac{1}{\left\|[P^T\theta^*]_{[(m+1):p]}\right\|_2^{1/4}n^{1/4}}\right) + \left\|[P^T\theta^*]_{[(m+1):p]}\right\|_2^{1/2}$. Regarding the quantities of interest, we obtain the best rate with $n$ again, but a slightly higher term of $c_{\mathcal{K}}^{1/2} = \left\|[P^T\theta^*]_{[(m+1):p]}\right\|_2^{1/2}$, compared to $c_{\mathcal{K}}$, in the bound based on Theorem \ref{theorem:NonACC_lambda=0}. Note that the error $\left\|[P^T\theta^*]_{[(m+1):p]}\right\|_2$ can indeed be small even if $m$ is not close to $p$: If $P = I_p$ and $||\theta^*_{[(m+1):p]}||_2 \lesssim \frac{1}{p}$, the error becomes small. This also matches intuition, since these are the parameters corresponding to covariates that are constant almost surely, and their signal is low.

Additionally, we can compare the iteration counts: In Theorem \ref{theorem:NonACC_lambda=0}, we have $T = n^{1/3}$, whereas Theorem \ref{theorem:ACCFW_lambda=0}, requires $T = \log(n)/c_{\mathcal{K}}^2$. Since both methods use the full batch of data at each iteration to compute gradients, accelerated Frank-Wolfe is much more computationally efficient, at the cost of an additional $c_{\mathcal{K}}^{1/2}$ error term in the bound on $||\theta_T - \theta^*||_2$.



\section{Accelerating Classical Gradient Descent}
\label{sec:Unbiased Parameter Estimation}

In this section, we complement our results by studying the benefits of acceleration in classical gradient descent, using Nesterov's accelerated gradient descent (AGD). In Section \ref{sec:Smooth_Priv_HT_GD_AGD}, we study risk functions coming from convex, smooth, Lipschitz losses, optimized over $\mathbb{R}^p$. We compare classical gradient descent with Nesterov's AGD, with modifications providing both heavy-tailed robustness and privacy. Our arguments will be based on inexact gradient analysis (cf.\ Appendix \ref{sec:Proofs of the Auxiliary Results_Nester}); in particular, our approach will be based on the setting of gradient estimators introduced in Section \ref{sec:Biased Parameter Estimation HT}.

In Section \ref{sec:Strongly Convex Risks}, we focus on heavy-tailed robustness only, and strongly convex risks. We derive bounds on $||\theta_T - \theta^*||_2$ directly and compare classical projected gradient descent with Nesterov's AGD. As we will see, for strongly convex risks, acceleration will have less significant effects in terms of the rates with $n$ and $p$. For simplicity, we consider the linear regression with squared error loss model from Section \ref{sec:Linear Regression}.

To place our results in context, recall from classical optimization theory that for smooth functions, projected gradient descent converges at rate $O\left(1/T\right)$, while Nesterov's AGD converges at rate $O\left(1/T^2\right)$ \cite{Nester}. This leads, as we present in Section \ref{sec:Smooth_Priv_HT_GD_AGD}, to an overall better performance, rate-wise with $n$, by Nesterov's AGD. On the other hand, for strongly convex functions, both classical projected gradient descent and Nesterov's AGD converge exponentially with $T$ (cf.\ Appendices \ref{sec:Appendix_GD} and \ref{sec:Appendix_AGD}). The improvement is in the base of the exponent, which is smaller for Nesterov's method. In the context of linear regression with squared error risk (see Example \ref{sec:LR_Squared_Loss}), however, this changes the iteration count $T$ up to an absolute constant only. This then leads to a similar performance of projected gradient decent and Nesterov's AGD, rate-wise with $n$, as one can see in Section \ref{sec:Comparisons in the Heavy-Tailed Setting}.


\subsection{Model-Free, Private Estimation for Smooth Risks}
\label{sec:Smooth_Priv_HT_GD_AGD}

Throughout this section, we only track the dependence on $n$, and we treat $p$ as a constant. We assume the data $\mathcal{D}_n = \{z_i\}_{i = 1}^n \subseteq \mathcal{E}^n$ are i.i.d.\ from an arbitrary distribution $P$. We work with a loss function $\mathcal{L} : \mathbb{R}^p \times \mathcal{E} \rightarrow \mathbb{R}$ that is convex and $L_2$-Lipschitz over $\mathbb{R}^p$, for all $z \in \mathcal{E}$. We assume that the population-level risk $\mathcal{R(\theta)} = \mathbb{E}_{z \sim P}[\mathcal{L}(\theta, z)]$ is $\tau_u$-smooth over $\mathbb{R}^p$. Additionally, we assume the existence of a minimizer $\theta_* \in \mathop{\arg\min}\limits_{\theta \in \mathbb{R}^p}\mathcal{R}(\theta)$. We treat $\tau_u$ and $L_2$ as absolute constants.

We will follow the approach based on gradient estimators as in Definition \ref{def:ge}. We view the gradient methods in the sense of Algorithm \ref{alg:RobPGDNFW} for Projected GD and Nesterov's AGD when $\mathcal{C} = \mathbb{R}^p$. For our gradient estimator, we use the sample average of the loss gradients plus Gaussian noise to ensure privacy. Such an approach requires no assumptions on the moments of the distribution, hence is robust to heavy tails. Due to the Lipschitz loss, there is no need to use a $G_{MOM}$ estimator, as we did in Section \ref{sec:Biased Parameter Estimation HT}, and it is enough to consider the gradient average in illustrating the benefit of Nesterov's acceleration in differential privacy.

In Lemma \ref{lemma:Avg_Grad_Priv_GD_vs_AGD} of Appendix \ref{sec:Proofs of the Auxiliary Results_Nester}, we establish high-probability concentration of the noisy gradient average around the true gradients $\nabla\mathcal{R}(\theta)$, for any fixed $\theta \in \mathbb{R}^p$. This allows us to cast both gradient descent and Nesterov’s AGD in inexact forms, so we can directly use the results from Appendix \ref{sec:Proofs of the Auxiliary Results_Nester}. Let us now present the convergence rates on $\mathcal{R}(\theta_T) - \mathcal{R}(\theta_*)$.
The proof of the following result for projected gradient descent is provided in Appendix \ref{AppThmPriv}:

\begin{theorem}
\label{theorem:Priv_HT_GD_Smooth}
Let $T = n^{1/5}$, $0 < \epsilon \leq 0.9$, and $\delta \in (0, 1)$. Consider i.i.d.\ data $\mathcal{D}_n = \{z_i\}_{i = 1}^n$ from some distribution $P$. Let $\mathcal{L} : \mathbb{R}^p \times \mathcal{E} \rightarrow \mathbb{R}$ be convex and $L_2$-Lipschitz in $\theta$, for all $z \in \mathcal{E}$. Consider the corresponding risk $\mathcal{R(\theta)} = \mathbb{E}_{z \sim P}[\mathcal{L}(\theta, z)]$, and let $\theta_* \in \mathop{\arg\min}\limits_{\theta \in \mathbb{R}^p}\mathcal{R}(\theta)$. Assume $\mathcal{R}$ is $\tau_u$-smooth over $\mathbb{R}^p$, and that $\tau_u, L_2 \asymp 1$. Let $\zeta \in (0, 1)$. Split the data into $T$ subsets $\{Z_t\}_{t=1}^T$ of size $\widetilde{n}$ and take $\left\{\xi^{(t)}\right\}_{t = 1}^T \overset{\iid}{\sim} N\left(0, \frac{64L_2^2T\log(5T/2
\delta)\log(2/\delta)}{\widetilde{n}^2 \epsilon^2}I_p\right)$. For $\widetilde{n} > 8\log(4/\widetilde{\zeta})$, Algorithm \ref{alg:RobPGDNFW} for projected gradient descent over $\mathbb{R}^p$ initialized at $\theta_0 \in \mathbb{R}^p$, with $\eta = \frac{1}{\tau_u}$ and using $g(\cdot; Z_t, \widetilde{\zeta}) = \frac{1}{\widetilde{n}}\sum_{i \in Z_t}\nabla\mathcal{L}(\cdot, z_i) + \xi^{(t)}$ as gradient estimator at step $t \in [T]$, returns $(\epsilon, \delta)$-DP iterates $\{\theta_t\}_{t = 1}^T$ such that with probability at least $1 - \zeta$, we have
\begin{align*}
\mathcal{R}(\theta_T) - \mathcal{R}(\theta_*) \lesssim \frac{\sqrt{\log(n/\zeta)}}{n^{1/5}} + \frac{\log(n/\delta)\sqrt{\log(n/\zeta)}}{n^{1/2}\epsilon}.
\end{align*}
\end{theorem}

\begin{remark}
\label{RemChooseT}
Note that the choice of $T$ in Theorem \ref{theorem:Priv_HT_GD_Smooth} is not arbitrary, and is obtained by minimizing the excess risk upper bound over $q$ for $T = n^q$ (cf.\ inequality~\eqref{eq:Smooth_GD_Best_T}).
\end{remark}

We now present a result based on Nesterov's acceleration. The proof is in Appendix \ref{AppThmPrivSmooth}:

\begin{theorem}
\label{theorem:Priv_HT_AGD_Smooth}
Consider the setup in Theorem \ref{theorem:Priv_HT_GD_Smooth}. For $\widetilde{n} > 8\log(4/\widetilde{\zeta})$, Algorithm \ref{alg:RobPGDNFW} for Nesterov's AGD initialized at $\theta_0, \theta_1 \in \mathbb{R}^p$, with $\eta = \frac{1}{\tau_u}$, varying learning rate at the $t^{\text{th}}$ step $\lambda = \frac{t - 1}{t + 2}$, and using $g(\cdot; Z_t, \widetilde{\zeta}) = \frac{1}{\widetilde{n}}\sum_{i \in Z_t}\nabla\mathcal{L}(\cdot, z_i) + \xi^{(t)}$ as gradient estimator at step $t \in [T]$, returns $(\epsilon, \delta)$-DP iterates $\{\theta_t\}_{t = 1}^T$ such that with probability at least $1 - \zeta$, we have
\begin{align*}
\mathcal{R}(\theta_T) - \mathcal{R}(\theta_*) \lesssim \frac{\log(n/\zeta)}{n^{2/5}} + \frac{\log(n/\zeta)\log^2(n/\delta)}{n\epsilon^2}.
\end{align*}
\end{theorem}

\begin{remark}
As in Remark~\ref{RemChooseT}, the choice of $T$ comes from optimizing the function $T = n^q$ over $q$ (cf.\ inequality~\eqref{eq:Smooth_AGD_Best_T}).
\end{remark}


\begin{remark}
\label{remark:compare_AGD_GD_Smooth}
We can see the benefits of acceleration in the context of classical gradient descent, when using convex and smooth losses. The rate in Theorem \ref{theorem:Priv_HT_GD_Smooth} is $\widetilde{O}\left(\frac{1}{n^{1/5}} + \frac{1}{n^{1/2}\epsilon}\right)$, while the one in Theorem \ref{theorem:Priv_HT_AGD_Smooth} is $\widetilde{O}\left(\frac{1}{n^{2/5}} + \frac{1}{n\epsilon^2}\right)$. Note also that both algorithms have the same gradient complexity since both use $T = n^{1/5}$ iterations and the same sample splitting procedure to compute the gradient estimators at each step. 

Both approaches are private and robust to heavy tails, since they do not make any moment assumptions on the data distribution $P$. The Lipschitz assumption plays a crucial role, as we can see in Lemma \ref{lemma:Avg_Grad_Priv_GD_vs_AGD} of Appendix \ref{sec:Proofs of the Auxiliary Results_Nester}. However, we do not have to assume any finite moments of the distribution $P$, since the Lipschitz property of the gradients takes care of this.
\end{remark}


\subsection{Strongly Convex Risks and Heavy-Tailed Robustness}\label{sec:Strongly Convex Risks}

In this section, we examine the case where the risk is also strongly convex, and we track the scaling with both $n$ and $p$. Specifically, we consider the linear regression with squared error loss model introduced in Section \ref{sec:Linear Regression}. To ensure strong convexity, we assume $\lambda_{\min}(\Sigma) > 0$. We consider heavy-tailed robustness only, since our method to achieve privacy would require Lipschitz gradients, which is not satisfied for the squared error loss with unbounded data. As we will see, the benefits in this case will not be as significant as in Section \ref{sec:Smooth_Priv_HT_GD_AGD}, since in the strongly convex case, the decay of both classical gradient descent (see Appendix \ref{sec:Appendix_GD}) and Nesterov's AGD (see Appendix \ref{sec:Appendix_AGD}) is exponential in the iteration count $T$. The improvement of Nesterov's method is a smaller constant under the exponent $T$ in the exponentially decaying term.


We now present the main result that allows us to obtain the desired approximate convergence rates for an arbitrary smooth, strongly convex risk $\mathcal{R}$.
The proof, which roughly follows the analysis in \cite{Recht}, can be found in Appendix \ref{SecThmAGD}.
In what follows, we define
\begin{align*}
f_1(x) := \frac{(x + 1)\sqrt{1 - 1/\sqrt{x}} - x + 1}{2}, \qquad
f_2(x) := \frac{1 - 2(x - 1)\frac{\sqrt{x} - 1}{\sqrt{x} + 1}}{2\left(1 + 2\frac{\sqrt{x} - 1}{\sqrt{x} + 1}\right)}.
\end{align*}

\begin{theorem}
\label{theorem:AGD}
Let $\mathcal{C} = \mathbb{R}^p$, so that $\theta^* \in \mathcal{C}$. Let $\zeta \in (0, 1)$. Suppose $1 < \frac{\tau_u}{\tau_l} < 1.76$.
Given a gradient estimator $g$ with $f_1\left(\frac{\tau_u}{\tau_l}\right) < \frac{\alpha}{\tau_l} < f_2\left(\frac{\tau_u}{\tau_l}\right)$, Algorithm \ref{alg:RobPGDNFW} for Nesterov's method initialized at $\theta_0, \theta_1 \in \mathcal{C}$, with $\eta = \frac{1}{\tau_u}$ and $\lambda = \frac{\sqrt{\tau_u} - \sqrt{\tau_l}}{\sqrt{\tau_u} + \sqrt{\tau_l}}$, returns iterates $\{\theta_t\}_{t = 1}^T$ such that with probability at least $1 - \zeta$, we have 
\begin{align*}
||\theta_t - \theta^*||_2 \leq \sqrt{\frac{2}{\tau_l}\left(\mathcal{R}(\theta_0) - \mathcal{R}(\theta^*)\right) + ||\theta_0 - \theta^*||_2^2}\left(1 - \sqrt{\frac{\tau_l}{\tau_u}}\right)^{t/2} + \left(\frac{\tau_u}{\tau_l}\right)^{1/4}\sqrt{\frac{R}{\tau_l}},
\end{align*}
with $R = O\left(\alpha(\widetilde{n}, \widetilde{\zeta})^2\right)$ if $\tau_u, \tau_l, \sigma \asymp 1$.
\end{theorem}

In other words, the bound on $||\theta_t - \theta^*||_2$ takes the form of a constant multiplied by an exponential term plus a constant error, i.e., independent of $t$.

\begin{remark}
\label{remark:Nester_Iter_LB}
Note that for $x \geq 1$ we have $f_2(x) \leq \frac{1}{2}$, so because $\tau_u > \tau_l$, the gradient estimator is stable.
Similar to the case of projected gradient descent, the first term in the upper bound in Theorem \ref{theorem:AGD} is decreasing in $t$ and the second is increasing, so for a fixed $n$ and probability $\zeta$, we run Nesterov's AGD to make the first term is smaller than the second, leading to the choice
$$T \ge \log_{\left(1 - \sqrt{\frac{\tau_l}{\tau_u}}\right)^{-1/2}}\left(\left(\frac{\tau_l}{\tau_u}\right)^{1/4}\sqrt{\frac{\tau_l}{R}}\sqrt{\frac{2}{\tau_l}\left(\mathcal{R}(\theta_0) - \mathcal{R}(\theta^*)\right) + ||\theta_0 - \theta^*||_2^2}\right).
$$
\end{remark}

\begin{remark}
\label{RemFaster}
A straightforward calculation shows that since $f_1\left(\frac{\tau_u}{\tau_l}\right) < \frac{\alpha}{\tau_l}$, the convergence rate of robust Nesterov's AGD is faster than the convergence rate of robust projected gradient descent in the sense that the base of the exponent is smaller.
\end{remark}



\subsubsection{Example: Linear Regression}

We now present applications of projected gradient descent and Nesterov's AGD to heavy-tailed linear regression. Note that as in \cite{RE}, we could also study general GLMs.
The proof of the following result is in Appendix \ref{AppThmHTAGD}:

\begin{theorem}
\label{theorem:htAGD}
Let $\mathcal{C} = \mathbb{R}^p$, so $\theta^* \in \mathcal{C}$. Let $\zeta \in (0, 1)$. Consider the linear regression with squared error loss model from Example~\ref{sec:LR_Squared_Loss} under the heavy-tailed setting. Suppose $1 < \frac{\tau_u}{\tau_l} < 1.76$. Then there is an absolute constant $C_1 > 0$, such that if 
\begin{align*}
\left(\frac{C_1}{\tau_l f_2\left(\frac{\tau_u}{\tau_l}\right)}\right)^2p \log(1/\widetilde{\zeta}) < \widetilde{n} < \left(\frac{C_1}{\tau_l f_1\left(\frac{\tau_u}{\tau_l}\right)}\right)^2p \log(1/\widetilde{\zeta}),
\end{align*}
Algorithm \ref{alg:RobPGDNFW} for Nesterov's AGD initialized at $\theta_0, \theta_1 \in \mathcal{C}$, with $\eta = \frac{2}{\tau_u}$ and $\lambda = \frac{\sqrt{\tau_u} - \sqrt{\tau_l}}{\sqrt{\tau_u} + \sqrt{\tau_l}}$ and using Algorithm \ref{alg:HTGE} as gradient estimator, with $\alpha(\widetilde{n}, \widetilde{\zeta}) = C_1\sqrt{\frac{p \log(1/\widetilde{\zeta})}{\widetilde{n}}}$, returns iterates $\{\theta_t\}_{t = 1}^T$ such that with probability at least $1 - \zeta$, with $\widetilde{\zeta}$ such that $b \leq \widetilde{n}/2$, we have
\begin{align}
\label{eq:ht_Nester}
||\theta_t - \theta^*||_2 \leq \sqrt{\frac{2}{\tau_l}\left(\mathcal{R}(\theta_0) - \mathcal{R}(\theta^*)\right) + ||\theta_0 - \theta^*||_2^2}\left(1 - \sqrt{\frac{\tau_l}{\tau_u}}\right)^{t/2} + \left(\frac{\tau_u}{\tau_l}\right)^{1/4}\sqrt{\frac{R}{\tau_l}}.
\end{align}
Here, $R = O\left(\alpha(\widetilde{n}, \widetilde{\zeta})^2\right)$ if $\sigma \asymp 1$.
\end{theorem}


\subsubsection{Comparisons}
\label{sec:Comparisons in the Heavy-Tailed Setting}

Assume $\sigma_2 \asymp 1$. In the heavy-tailed setting, the error for projected gradient descent (Lemma \ref{lemma:GLMhtGD} in Appendix~\ref{sec:Proofs of the Main Results_Nester}) scales as $O(\alpha(\widetilde{n}, \widetilde{\zeta}))$, since $\tau_u, \tau_l \asymp 1$. In particular, we need $\widetilde{n} \gtrsim p \log(1/\widetilde{\zeta})$. 
Nesterov’s method (Theorem \ref{theorem:htAGD}) converges faster in the exponentially decaying term (see Remark \ref{RemFaster}), but since $\tau_u, \tau_l \asymp 1$, the requirement $p \log(1/\widetilde{\zeta}) \lesssim \widetilde{n} \lesssim p \log(1/\widetilde{\zeta})$ forces the error term $\sqrt{R} \asymp \alpha(\widetilde{n}, \widetilde{\zeta})$ to remain bounded away from zero as $n, p \rightarrow \infty$. Hence, Nesterov's AGD yields faster rates with $t$, while keeping the error term asymptotically the same as with projected gradient descent.

We can choose $T$ to balance the exponentially decaying term and error. Since $k = \frac{\tau_u - \tau_l + 2\alpha(\widetilde{n}, \widetilde{\zeta})}{\tau_u + \tau_l} < \frac{\tau_u}{\tau_u + \tau_l} < 1$, setting $T = \log_{\frac{\tau_u + \tau_l}{\tau_u}}(\sqrt{n})$ in inequality \eqref{eq:GLMhtGD} yields  
\begin{align}
\label{eq:ht_gdd}
||\theta_T - \theta^*||_2 \lesssim \frac{||\theta_0 - \theta^*||_2}{\sqrt{n}} + \sqrt{\frac{p \log\left(n\right)\log\left(\log\left(n\right)/\zeta\right)}{n}}.
\end{align}
Note that due to stability, we can bound $k$ and $\frac{1}{1 - k}$ above by absolute constants, so $T$ can be chosen independently of $\alpha(\widetilde{n}, \widetilde{\zeta})$ to make $T \asymp \log(n)$.  

For Nesterov's AGD, by $\tau_u$-smoothness and $\nabla\mathcal{R}(\theta^*) = 0$, we have $\mathcal{R}(\theta_0) - \mathcal{R}(\theta^*) \lesssim ||\theta_0 - \theta^*||_2^2$, so taking $T = \frac{2\log(\sqrt{n})}{\log\left(\frac{1}{1 - \sqrt{\frac{\tau_l}{\tau_u}}}\right)}$ in inequality \eqref{eq:ht_Nester} results in 
\begin{align}
\label{eq:ht_Nesterr}
||\theta_t - \theta^*||_2 \lesssim \frac{||\theta_0 - \theta^*||_2}{\sqrt{n}} + \sqrt{\frac{p\log(n)\log(\log(n))}{n}}.
\end{align}
Therefore, when $p$ and $n$ grow together such that $\frac{p\log(n)\log(\log(n))}{n} \asymp 1$, both methods behave similarly in terms of statistical error, while Nesterov's AGD requires fewer iterations.

Finally, if $||\theta_0 - \theta^*||_2$ is an absolute constant in $n$ and $p$, the bounds \eqref{eq:ht_gdd} and \eqref{eq:ht_Nesterr} include terms decaying as $\frac{1}{\sqrt{n}}$, plus non-decaying error terms, whereas if $||\theta_0 - \theta^*||_2 \lesssim \sqrt{p}$, the overall rate becomes $\sqrt{\frac{p}{n}}$, which is minimax optimal for $w \sim N\left(0, \sigma_2^2\right)$ (see \cite{Stats_Info_Duchi}).


\section{Simulations}
\label{SecSims}

In this section, we report the results of simulations on synthetic data. Brief descriptions are provided in the figure captions, with more details in Appendix~\ref{AppSims}.

\begin{figure}[htbp]
  \centering
  \begin{minipage}[b]{0.45\linewidth}
    \centering
    \includegraphics[width=\linewidth]{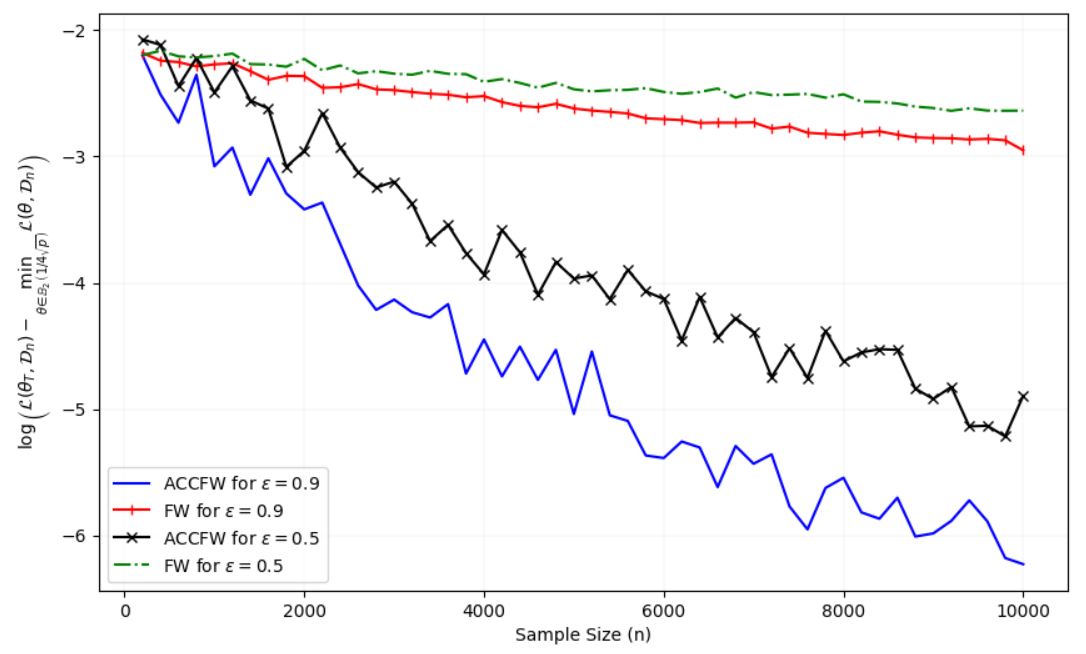}
    \caption{We compare Theorem \ref{theorem:UPRidge} with Lemma \ref{lemma:NOPLconvDPFW}, using Algorithms \ref{alg:PrivFWERM} and \ref{alg:PrivNon-accERM}. The plot shows the log excess mean squared error loss vs.\ $n$. In line with Remark \ref{remark:compare_dist_free_ACCFW}, Algorithm \ref{alg:PrivFWERM} outperforms Algorithm \ref{alg:PrivNon-accERM}, and larger $\epsilon$ leads to faster convergence.}
    \label{fig:plot1}
  \end{minipage}\hfill
  \begin{minipage}[b]{0.45\linewidth}
    \centering
    \includegraphics[width=\linewidth]{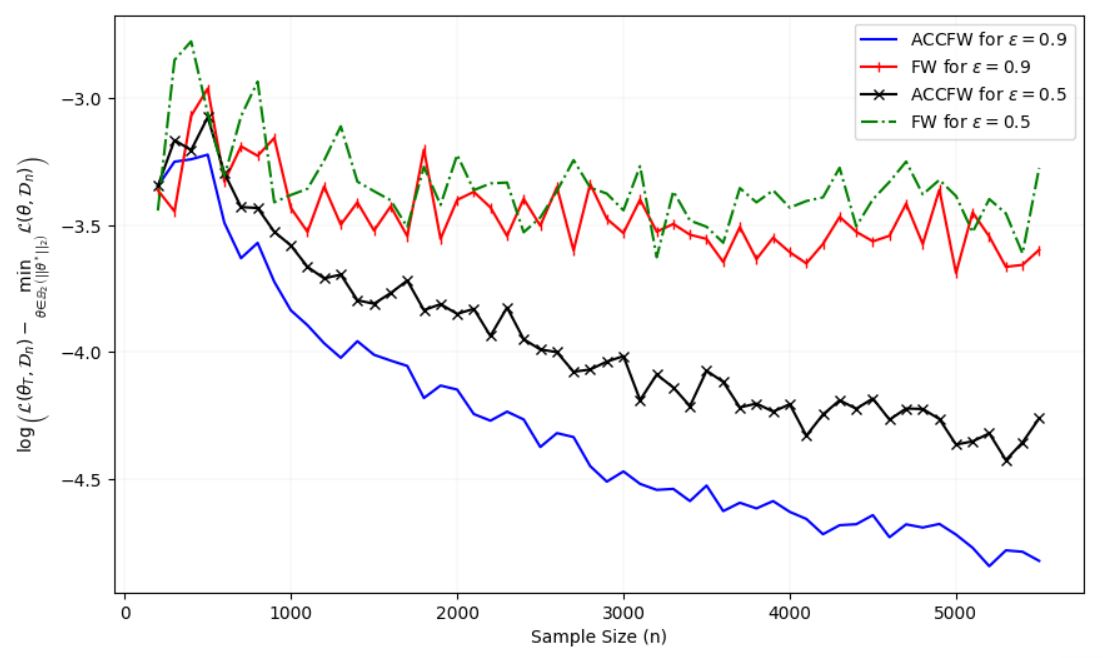}
    \caption{We compare Theorem \ref{theorem:ERM_UP_C_Inc} with Lemma \ref{lemma:NOPLconvDPFW}, using Algorithms \ref{alg:PrivFWERM} and \ref{alg:PrivNon-accERM}. The plot shows the log excess empirical risk vs.\ $n$. We can see that Algorithm \ref{alg:PrivFWERM} does better than Algorithm \ref{alg:PrivNon-accERM} (cf.\ Remark \ref{remark:Best_q_ERM}), and larger $\epsilon$ leads to faster convergence.}
    \label{fig:plot2}
  \end{minipage}\hfill
  
  \vspace{1em}
  
  
  \begin{minipage}[b]{0.45\linewidth}
    \centering
    \includegraphics[width=\linewidth]{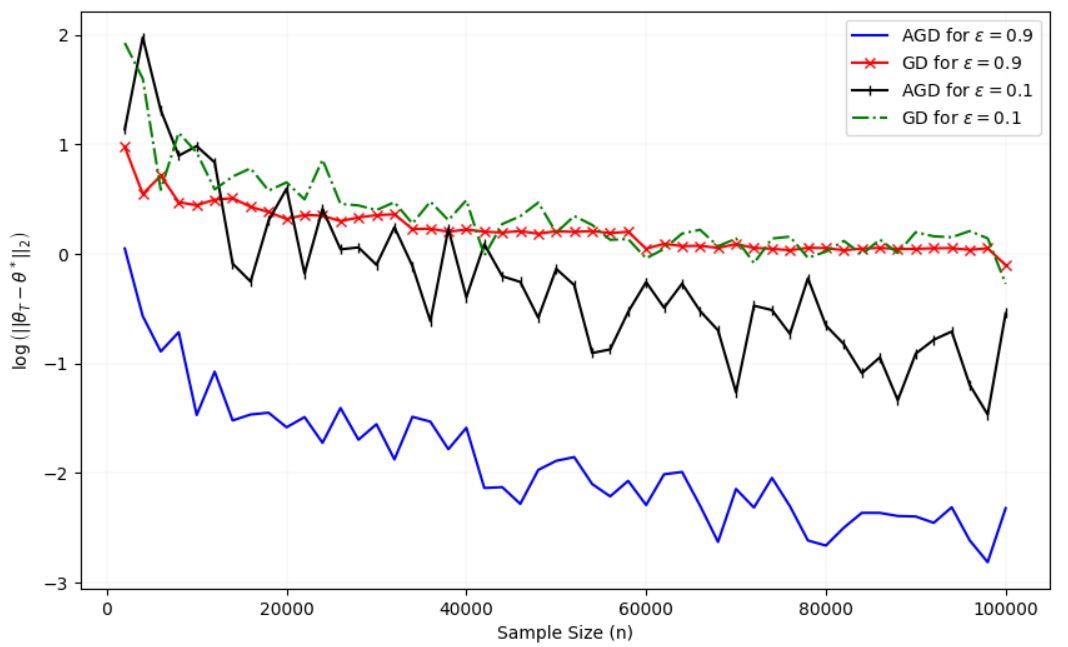}
    \caption{We compare Nesterov’s AGD (Theorem \ref{theorem:Priv_HT_AGD_Smooth}) with projected GD (Theorem \ref{theorem:Priv_HT_GD_Smooth}) using Algorithm \ref{alg:RobPGDNFW} and the pseudo-Huber loss (with $q = \frac{1}{5}$, see Appendix \ref{Appendix A}). The plot displays $\log(\|\theta_T-\theta^*\|_2)$ vs.\ $n$. Nesterov’s AGD outperforms projected GD (cf.\ Remark \ref{remark:compare_AGD_GD_Smooth}), and larger $\epsilon$ accelerates convergence. By the smoothness of the risk (cf.\ Lemma \ref{lemma:pHu_prelim}), we can further deduce a bound on $\mathcal{R}(\theta_T)-\mathcal{R}(\theta^*)$.}
    \label{fig:plot3}
  \end{minipage}\hfill
  \begin{minipage}[b]{0.45\linewidth}
    \centering
    \includegraphics[width=\linewidth]{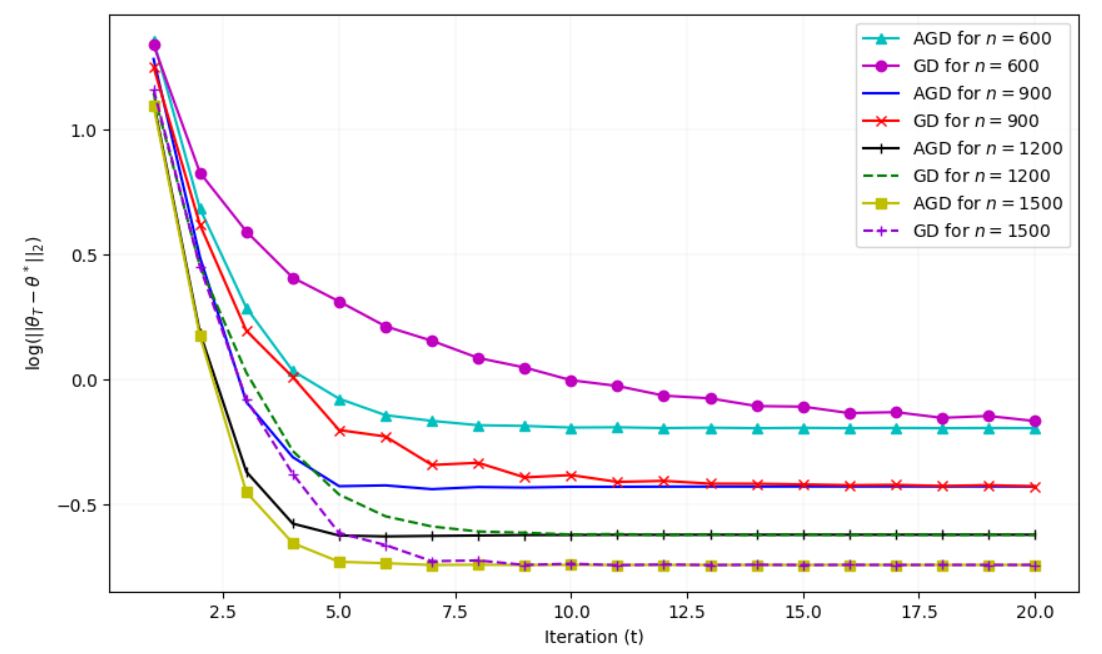}
    \caption{We compare Nesterov’s AGD (Theorem \ref{theorem:htAGD}) with projected GD (Lemma \ref{lemma:GLMhtGD}) using Algorithm \ref{alg:RobPGDNFW} and the squared error loss. The plot shows $\log(\|\theta_t-\theta^*\|_2)$ vs.\ $t$. We can see a faster convergence of Nesterov's AGD in the exponentially decaying term with $t$ (cf.\ Remark~\ref{RemFaster}), while a larger $n$ leads to a smaller error term, in line with the results of Theorem \ref{theorem:htAGD} and Lemma \ref{lemma:GLMhtGD}.}
    \label{fig:plot4}
  \end{minipage}\hfill 
  
\end{figure}
\begin{figure}[htbp]
  \centering
 
  \begin{minipage}[b]{0.45\linewidth}
    \centering
    \includegraphics[width=\linewidth]{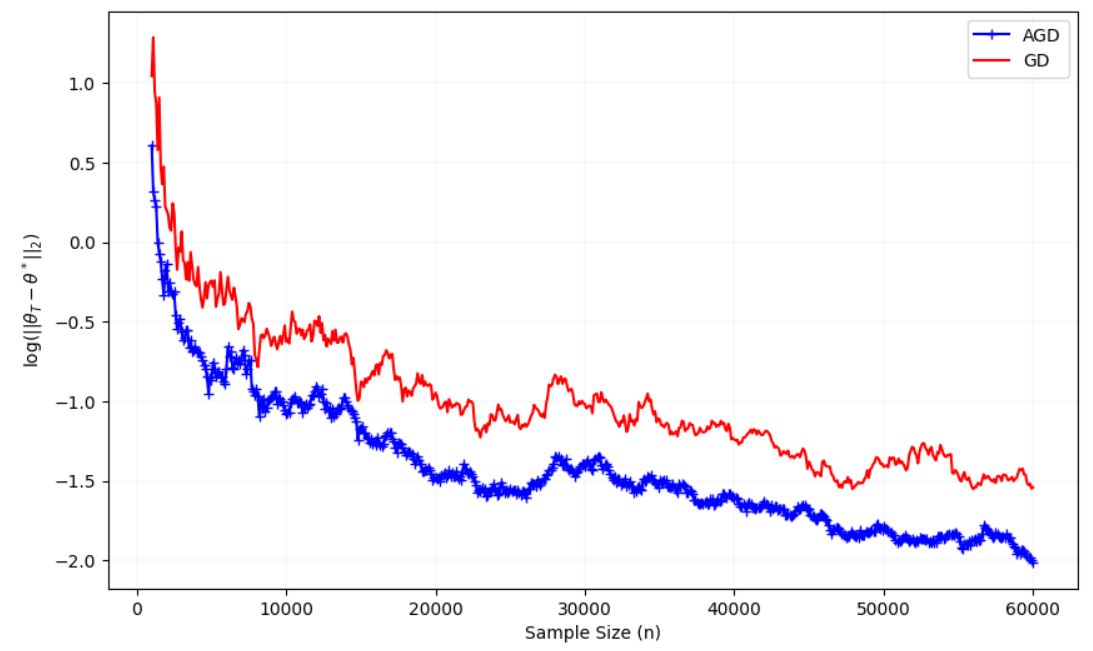}
    \caption{We compare Nesterov’s AGD (Theorem \ref{theorem:htAGD}) to projected GD (Lemma \ref{lemma:GLMhtGD}). We plot $\log||\theta_T-\theta^*||_2$ vs.\ $n$. The results show that Nesterov’s AGD yields a slight improvement, supporting the prediction that AGD’s advantage is up to an absolute constant.}
    \label{fig:plot5}
  \end{minipage}\hfill
  \begin{minipage}[b]{0.45\linewidth}
    \centering
    \includegraphics[width=\linewidth]{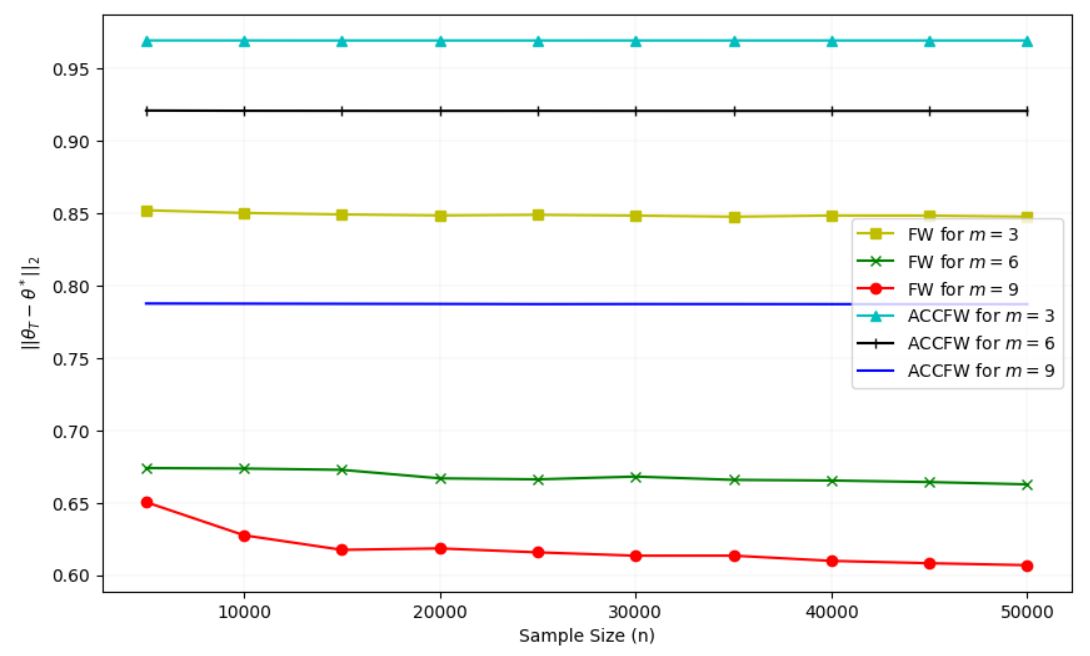}
    \caption{We compare Theorem \ref{theorem:ACCFW_lambda=0} with Theorem \ref{theorem:NonACC_lambda=0}, using Algorithms \ref{alg:PrivFWERM} and \ref{alg:PrivNon-accERM}. The plot shows $\|\theta_T-\theta^*\|_2$ vs.\ $n$. As predicted by Theorems \ref{theorem:ACCFW_lambda=0} and \ref{theorem:NonACC_lambda=0}, the error plateaus at non-zero levels due to the $c_{\mathcal{K}}$ term. The non-accelerated version converges more slowly but ultimately incurs less error, while the accelerated version reaches its plateau faster.}
    \label{fig:plot6}
  \end{minipage}
  \vspace{1em}  
  
  \begin{minipage}[b]{0.45\linewidth}
    \centering
    \includegraphics[width=\linewidth]{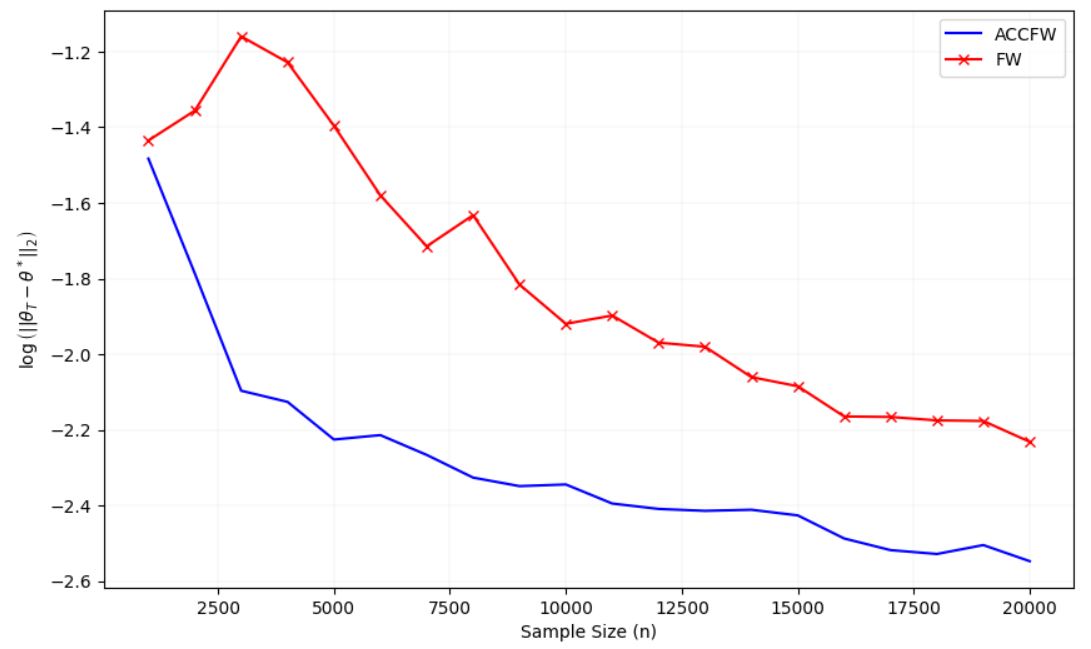}
    \caption{We compare Nesterov’s AGD (Theorem \ref{theorem:htAGD}) to projected GD (Lemma \ref{lemma:GLMhtGD}). The plot shows $\log(\|\theta_T-\theta^*\|_2)$ versus $n$. We observe that Algorithm \ref{alg:PrivFWERM} outperforms Algorithm \ref{alg:PrivNon-accERM}.}
    \label{fig:plot_l_bigger_0}
  \end{minipage}\hfill
\end{figure}

\section{Conclusion}

We have demonstrated that accelerating the Frank-Wolfe method and classical gradient descent can guarantee better statistical convergence rates, under differential privacy or heavy-tailed robustness. 
With appropriate assumptions and a careful choice of learning rate, we improved on the private Frank-Wolfe approach from Talwar et al.~\cite{NOPL} and proved minimax optimality for particular choices of $n, p$, and $\mathcal{C}$. We then analyzed our methods in the context of parameter estimation in GLMs. For heavy-tailed robustness, we considered the linear regression model, and showed that our accelerated method converges faster when the population covariance $\Sigma$ is well-conditioned. When $\lambda_{\min}(\Sigma)=0$, it trades a faster rate with $n$ for a small extra error term $c_\mathcal{K}^{1/2}$, which vanishes as conditioning improves.

On the other hand, our analysis of accelerated Frank-Wolfe crucially requires a lower bound on the $\ell_2$-norm of the gradient. It is an open question whether similar performance could be guaranteed without this assumption. \cite{ACCFW} considers strongly convex sets and strongly convex functions, but with a learning rate that depends on the input data: For the purpose of privacy, one would also need to add noise to the learning rate, making the analysis more complex. It would also be interesting to study the optimality of our accelerated algorithm for more general choices of $n$, $p$, and $\mathcal{C}$ than in Section \ref{sec:The Lower Bound}. Moreover, one could also try to derive a lower bound that explicitly includes the dependency on $||\mathcal{C}||_2$. 

Throughout Sections \ref{sec:Private Biased Parameter Estimation in GLMs: Upper and Lower Bounds for Excess Risk} and \ref{sec:Biased Parameter Estimation HT}, we focused on the scaling with $n$, but an analysis that tracks the presence of $p$ is encouraged. Likewise, handling GLMs with unbounded $y$ and $x$, as well as more general $\Phi$, remains open. In our framework, having $D \uparrow ||\theta^*||_2$ forces $T$ to scale polynomially in $n$. Finding an approach that allows $D \uparrow ||\theta^*||_2$ while keeping $T \asymp \log(n)$ could recover the $\frac{1}{n\epsilon}$ rate from Theorem \ref{theorem:ERM_GLM_UP} (and of the SGD method in Appendix \ref{sec:Comparison_SGD_GLM}). The study in Section  \ref{sec:Private Biased Parameter Estimation in GLMs: Upper and Lower Bounds for Excess Risk} relied on Lipschitz losses, but other methods, such as gradient clipping \cite{Deep_Learn_DP_Clip}, could also be studied. Section \ref{sec:Biased Parameter Estimation HT} focused on linear regression, and one could analyze other parametric models. The anticipated difficulty lies in the derivation of the $\alpha$ and $\beta$ functions, and explicit expressions for minimizers or $\ell_2$-regularized risks. Moreover, in Appendix \ref{sec:The Case when l>0}, we derive the minimax optimal rate using projected gradient descent, in the context of Section \ref{sec:Well_Cond_Sec_FW}; matching this rate using a Frank-Wolfe variant would be interesting.

Turning to Nesterov's AGD, we showed that for smooth risks and model-free random data, a faster convergence rate can be achieved through acceleration. This echoes the quadratic convergence of Nesterov's AGD in $T$, compared to the linear rate for projected gradient descent \cite{Nester}. Regarding heavy-tailed robustness and strongly-convex risks, we examined the linear regression model, where Nesterov's AGD was less impactful on the rate with $n$ and $p$.
Note that our study of Nesterov's AGD relies on optimization over $\mathbb{R}^p$. It would be interesting to study analogous constrained optimization methods, potentially using proximal methods \cite{Beck_Proximal_2009}. Regarding Section \ref{sec:Smooth_Priv_HT_GD_AGD}, one might carry out our derivations by keeping track of $p$, as well. The performance of a stochastic variant of Nesterov's AGD could also be compared to the optimal localized-based SGD approach from \cite{SCO_Feldman_2020}.
Furthermore, the approach from Section \ref{sec:Strongly Convex Risks} imposed some constraints on $\tau_u$, $\tau_l$ and $R$. The constraint on $R$ led to the requirement $\frac{p\log(n)\log(\log(n))}{n} \asymp 1$. Hence, an approach that avoids these constraints is encouraged.

A growing body of research simultaneously tackles private and heavy-tailed robust estimation  \cite{DP_ROB_HIGH_DIM, NO_ROB_LR, Kamath_Priv_HT, ROB_LR_OPT_POLY}. Given our current work, focusing on a linear regression model with $||x||_2 \lesssim 1$ and $\lambda_{\min}(\Sigma) \asymp \lambda_{\max}(\Sigma) \asymp \frac{1}{p}$, one could add Gaussian noise to a $G_{MOM}$ estimator using gradients of the pseudo-Huber loss (cf.\ Appendix \ref{Appendix A}). The $\alpha$ and $\beta$ functions can be computed as in Lemma \ref{lemma:htlemma}, accounting for a cost of privacy term. A private estimator $\theta_T$ can be obtained using Lemma \ref{lemma:GD}.
The resulting rate on $||\theta_T - \theta^*||_2$ would be $\widetilde{O}\left(\frac{p}{\sqrt{n}} + \frac{p\sqrt{p}}{n\epsilon}\right)$. Its minimax optimality could then be derived by bounding the statistical error using KL-divergence (cf. Appendix \ref{sec:Notation}) and an application of the local Fano's method \cite{Prob_Stats_Chen, Stats_Info_Duchi}, combined with score attack arguments from \cite{Cost_Priv_Lin_Reg}. Under strong convexity of the risk, Nesterov's AGD can similarly be seen to improve the performance rate up to absolute constants.

Note that our analyses regarding accelerated gradient methods relied heavily on $\ell_2$-norms. Hence, analogous derivations for $\ell_q$-norms, with $q \in [1, \infty]\setminus\{2\}$, in the spirit of \cite{BasEtal21}, are encouraged.
Finally, it is still an open question to us how one can carry out the privacy and robustness analyses using more modern gradient variants, and with provable guarantees. In particular, one could look into adapting methods such as AdaGrad \cite{AdaGrad_Duchi_2011}, RMSprop \cite{RMSprop_2012}, or Adam \cite{Adam_2104} to incorporate privacy or heavy-tailed robustness.


\appendix


\section{Preliminaries}\label{sec:Appendix_Prelim}

In this appendix, we present a more detailed version of the material introduced in Section \ref{sec:Notation and Preliminaries}. We start by defining several important terms that we will use in our analysis in Appendix \ref{sec:Notation}. In Appendix \ref{sec:Appendix_Optim_Prelim}, we introduce more background material on the theory of optimization, and we give precise theoretical guarantees for the optimization methods introduced in Section \ref{sec:Preliminaries on Optimization}.


\subsection{Notation}\label{sec:Notation}
Throughout the paper, the abbreviation ``w.h.p." stands for ``with high probability."

We define the ball centered at $0$ of radius $r > 0$ in $\mathbb{R}^p$, $p \in \mathbb{N}$, with respect to the norm $||\cdot||$ (e.g. $\ell_1$, $\ell_2$, $\ell_{\infty}$ etc.) as $\mathbb{B}_{||\cdot||}(r) = \left\{x \in \mathbb{R}^p | \  ||x|| \leq r\right\}$.

For a set $\mathcal{C} \subseteq \mathbb{R}^p$, for some $p \geq 1$, we denote its diameter by $||\mathcal{C}||_2 = \mathop{\sup}\limits_{x, y \in \mathcal{C}}||x - y||_2$. Note that we shall talk about the diameter of a set in the sense of the $\ell_2$-norm.

In our analysis, we will work with datasets of the form $\mathcal{D}_n = \{(x_i, y_i)\}_{i = 1}^n$, with $x_i \in \mathbb{R}^p$ and $y_i \in \mathbb{R}$ for all $i \in [n] = \left\{1, \dots, n\right\}$. We will care primarily about the dependency on $n$ and sometimes we will also care about the dependency on $p$. In every section, we specify what we care about, and everything else will be treated as an absolute constant. We have the following definition:
\begin{definition}
Let $f$ and $g$ be two functions taking as input $m = (m_1, \dots, m_k)^T \in \mathbb{N}^{k}$, with $k \in \mathbb{N}$, and taking values in $[0, \infty)$. We only care about the dependence on $m$ and assume that $k$ is an absolute constant. 
\begin{enumerate}[label=(\roman*)]
\item We say $f(m) \lesssim g(m)$ (equivalently, $f(m) = O(g(m))$ and $g(m) = \Omega(f(m))$) if there are absolute constants $K > 0$ and $M = (M_1, \dots, M_k)^T \in (0, \infty)^k$ such that $f(m) \leq Kg(m)$ for all $m$ such that $m_i > M_i$, for all $i \in [k]$. Similarly, we say $f(m) \asymp g(m)$ if there are absolute constants $K > 0$ and $M = (M_1, \dots, M_k)^T \in (0, \infty)^k$ such that $f(m) = Kg(m)$ for all $m$ such that $m_i > M_i$, for all $i \in [k]$.
\item We say $f(m) = \Theta(g(m))$ if $f(m) = O(g(m))$ and $f(m) = \Omega(g(m))$.
\item We say $f(m) = \widetilde{O}(g(m))$ if $f(m) = O(g(m))$ up to logarithmic factors. Similarly, we define $\widetilde{\Omega}$ and $\widetilde{\Theta}$.
\end{enumerate}
\end{definition}
Note that when we say $f(m) \asymp 1$, we mean that $f(m)$ is a positive absolute constant in $m$ for $m_i > M_i$ for all $i \in [k]$, for some absolute constants $\{M_i\}_{i = 1}^k$. Similarly, we interpret $f(m) \lesssim 1$ as $f(m) \lesssim g(m)$ and $g(m) \asymp 1$.

For two probability density functions $p$ and $q$ supported on some domain $\mathcal{D}$, the $\mathrm{KL}$-divergence between $p$ and $q$ is defined as
\begin{align*}
D(p||q) = \int_{\mathcal{D}}p(x)\log\left(\frac{p(x)}{q(x)}\right) \, dx.
\end{align*}

Let us now introduce some notation from linear algebra. For a matrix $A \in \mathbb{R}^{m \times m}$, with $m \in \mathbb{N}$, we denote its largest and smallest eigenvalues by $\lambda_{\max}(A)$ and $\lambda_{\min}(A)$, respectively. Additionally, for a matrix $B \in \mathbb{R}^{m \times k}$, with $k, m \in \mathbb{N}$, we denote its operator norm, i.e., its highest singular value, by $||B||_2$. If $m = k$ and $B$ is real, symmetric, and positive semi-definite, then $||B||_2 = \lambda_{\max}(B)$. Also, we denote the identity matrix of size $p$ by $I_p$.

Finally, we state the following definition:
\begin{definition}
\label{def:bd_moments}
Given a random vector $x \in \mathbb{R}^p$ with $\mathbb{E}[x] = \mu$, we say it has bounded $2k^{\text{th}}$ moments if there exists an absolute constant $\widetilde{C}_{2k}$ such that for any $||v||_2 = 1$, we have
\begin{align*}
\mathbb{E}\left[((x - \mu)^Tv)^{2k}\right] \leq \widetilde{C}_{2k}\left(\mathbb{E}\left[((x - \mu)^Tv)^{2}\right]\right)^k.
\end{align*}
\end{definition}
This assumption is a technical one that will allow us to establish bounds on the expectation of even powers by bounding expectations of a square that will usually reduce itself to a term involving the $2$-norm of a covariance matrix, i.e., its highest eigenvalue.

\subsection{Background on Optimization Theory}\label{sec:Appendix_Optim_Prelim}

\begin{definition}[Smoothness and Strong Convexity]
Let $\mathcal{C} \subseteq \mathbb{R}^p$ be convex and let $F : \mathbb{R}^p \rightarrow \mathbb{R}$ be a differentiable and convex function. We say $F$ is $\tau_u$-smooth over $\mathcal{C}$, for $\tau_u > 0$, if
\begin{align*}
F(x) - F(y) - \nabla F(y)^T(x - y) \leq \frac{\tau_u}{2}||x - y||_2^2, \quad \forall x, y \in \mathcal{C}.
\end{align*}
Additionally, we say that $F$ is $\tau_l$-strongly convex over $\mathcal{C}$, for $\tau_l > 0$, if
\begin{align*}
\frac{\tau_l}{2}||x - y||_2^2 \leq F(x) - F(y) - \nabla F(y)^T(x -y), \quad \forall x, y \in \mathcal{C}.
\end{align*}
\end{definition}

Note that if $F$ is twice continuously differentiable, then $F$ is $\tau_u$-smooth if and only if $\nabla^2 F(x) \preceq \tau_u I_p$ for all $x \in \mathcal{C}$, and it is $\tau_l$-strongly convex if and only if $\nabla^2F(x) \succeq \tau_l I_p$ for all $x \in \mathcal{C}$. Moreover, we have a useful lemma regarding smooth and strongly convex functions:
\begin{lemma}[Lemma $3.11$ in \cite{lemma311}]
\label{lemma:auxilGD}
Let $\mathcal{C} \subseteq \mathbb{R}^p$ be convex. For $F : \mathbb{R}^p \rightarrow \mathbb{R}$ a differentiable function that is $\tau_l$-strongly convex and $\tau_u$-smooth over $\mathcal{C}$, we have for all $x, y \in \mathcal{C}$ that
\begin{align*}
(\nabla F(x) - \nabla F(y))^T(x - y) \geq \frac{\tau_l\tau_u}{\tau_l + \tau_u}||x - y||_2^2 + \frac{1}{\tau_l + \tau_u}||\nabla F(x) - \nabla F(y)||_2^2.
\end{align*}
\end{lemma}

\begin{lemma}[Corollary $1$ in \cite{ACCFW}]
\label{lemma:STR_CONV_BALL}
The $\ell_2$-ball of radius $r$ centered at $0$ in $\mathbb{R}^p$, denoted by $\mathbb{B}_2(r)$, is $\frac{1}{r}$-strongly convex.
\end{lemma}


\subsubsection{Background on Projected Gradient Descent}\label{sec:Appendix_GD}

Let us present a convergence guarantee regarding projected gradient descent. Under the strong convexity and smoothness assumptions, we can guarantee the following result:

\begin{lemma}[\cite{Nester}]
\label{lemma:Optimiz_Proj_GD_Conv}
Let $\mathcal{C} \subseteq \mathbb{R}^p$. Let $F : \mathbb{R}^p \rightarrow \mathbb{R}$ be a differentiable function that is $\tau_l$-strongly convex and $\tau_u$-smooth over $\mathcal{C}$. If $\eta = \frac{2}{\tau_l +  \tau_u}$ and $x_{*} \in \mathop{\arg\min}\limits_{x \in \mathcal{C}}F(x)$ is such that $\nabla F(x_{*}) = 0$, the projected gradient descent method in \eqref{eq:Proj_GD_Rule} generates a sequence $\{x_t\}_{t \geq 1}$ such that 
\begin{align*}
||x_t - x_{*}||_2^2 \leq \left(\frac{\tau_u - \tau_l}{\tau_u + \tau_l}\right)^{2t}||x_0 - x_{*}||_2^2, \quad \forall t.
\end{align*}
\end{lemma}

This is the key lemma that \cite{RE} relies on for their proof regarding the convergence rates for their robust gradient estimator in Lemma \ref{lemma:GD}, and it is also a proof we take inspiration from for proving the convergence rates for the robust AGD method in Theorem \ref{theorem:AGD}.

\subsubsection{Background on Nesterov's AGD}\label{sec:Appendix_AGD}

We present a more detailed analysis of Nesterov's AGD. Under the strong convexity and smoothness assumptions, we have the following guarantee for Nesterov's AGD:

\begin{lemma}[\cite{Recht}]
\label{lemma:Optimiz_Nester_Conv}
Let $F : \mathbb{R}^p \rightarrow \mathbb{R}$ be a differentiable function that is $\tau_l$-strongly convex and $\tau_u$-smooth over $\mathbb{R}^p$.
If $x_{*} \in \mathop{\arg\min}\limits_{x \in \mathbb{R}^p}F(x)$ is such that $\nabla F(x_{*}) = 0$, with $\eta = \frac{1}{\tau_u}$ and $\lambda = \frac{\sqrt{\tau_u} - \sqrt{\tau_l}}{\sqrt{\tau_u} + \sqrt{\tau_l}}$, then Nesterov's accelerated gradient method in \eqref{eq:Nester_Rule} generates a sequence $\{x_t\}_{t \geq 2}$ such that
\begin{align*}
||x_t - x_{*}||_2^2 \leq \left(1 - \sqrt{\frac{\tau_l}{\tau_u}}\right)^t\frac{2}{\tau_l}\left(F(x_0) - F(x_{*}) + \frac{\tau_l}{2}||x_0 - x_{*}||_2^2\right), \quad \forall t.
\end{align*}
\end{lemma}
If $\frac{\tau_u}{\tau_l}$ is large enough, in this case larger than the second largest point $x'' \in (11, 12)$ (see Figure \ref{fig:my_label}) that solves $1 - \sqrt{\frac{\tau_l}{\tau_u}} = \left(\frac{\frac{\tau_u}{\tau_l} - 1}{\frac{\tau_u}{\tau_l} + 1}\right)^2$ as a function of $\frac{\tau_u}{\tau_l} \geq 1$, we achieve a faster convergence rate than in Lemma \ref{lemma:Optimiz_Proj_GD_Conv}. Now notice that when $\frac{\tau_u}{\tau_l} < x''$, the rate  in the bound on the error for projected gradient descent is faster and our intuition is that we should achieve a better convergence with Nesterov's AGD if the problem is better conditioned, i.e., if the condition number $\frac{\tau_u}{\tau_l}$ is close to $1$. Furthermore, it is also interesting to note that if we drop the strong convexity assumption, then Nesterov's method converges quadratically in $t$, while projected gradient descent converges linearly.

\begin{figure}
\centering\includegraphics[width=0.7\linewidth]{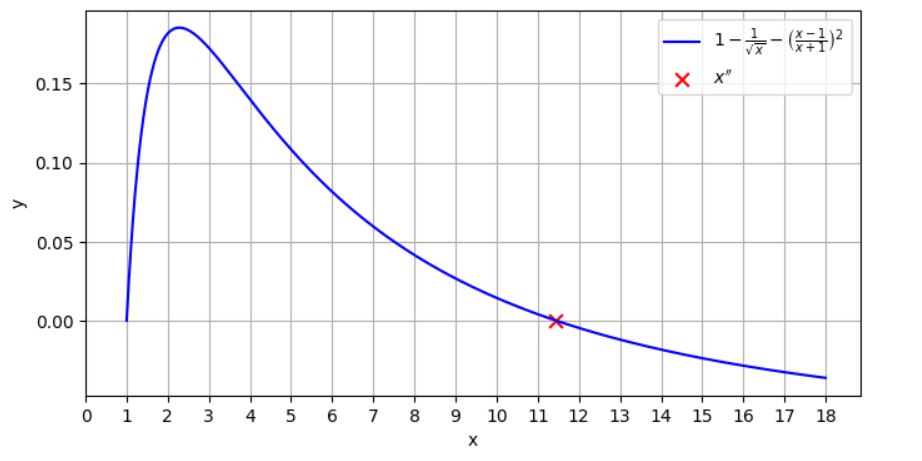}
\caption{$1 - \frac{1}{\sqrt{x}} - \left(\frac{x - 1}{x + 1}\right)^2$ for $x \geq 1$ and $x'' \in (11, 12)$}
\label{fig:my_label}
\end{figure}

\subsubsection{Background on the Frank-Wolfe Method}\label{sec:Appendix_FW}

Let us present a sub-linear convergence guarantee for the Frank-Wolfe method, when we deal with smooth functions:

\begin{lemma}[\cite{Recht}, \cite{FWS}]
\label{lemma:Optimiz_FW_Conv}
Let $\mathcal{C} \subseteq \mathbb{R}^p$ be compact and convex. Let $F : \mathbb{R}^p \rightarrow \mathbb{R}$ be a differentiable function that is $\tau_u$-smooth over $\mathcal{C}$. For $x_{*} \in \mathop{\arg\min}\limits_{x \in \mathcal{C}}F(x)$, with $\nabla F(x_{*}) = 0$, the iterates in the Frank-Wolfe algorithm in \eqref{eq:FW_Rule}, with varying learning rate $\eta_t = \frac{2}{2 + t}$, satisfy
\begin{align*}
F(x_t) - F(x_{*}) \leq \frac{2\tau_u ||\mathcal{C}||_2^2}{t + 2}, \quad \forall t.
\end{align*}
\end{lemma}

If we impose $\tau_l$-strong convexity, using the definition of $\tau_l$-strong convexity, we obtain
\begin{align*}
||x_t - x_{*}||_2^2 \leq \frac{2}{\tau_l} (F(x_t) - F(x_{*})) \leq \frac{4\tau_u||\mathcal{C}||_2^2}{\tau_l(t + 2)}, \quad \forall t.
\end{align*}
It is interesting to note the linear convergence rate here, which also matches how projected gradient descent converges if we do not ask for strong convexity, whereas Nesterov's method converges quadratically in the absence of strong convexity \cite{Nester}.

Let us now present the proof of Theorem \ref{theorem:ACCFWV3}. Note that the idea is inspired by Lemma \ref{lemma:NonAccRel_FW}, while the proof is inspired by \cite{ACCFW}.

\begin{proof}
Define $h_t := F(x_t) - F(x_{*})$ for all $t$. By the minimality of $v_t$, we have
\begin{align}
\label{EqnFBd}
(v_t - x_t)^T\nabla F(x_t) \leq (x_{*} - x_t)^T \nabla F(x_t) + \Delta \leq -h_t + \Delta,
\end{align}
where we used the convexity of $F$ in the second inequality. Now set $c_t = \frac{1}{2}(x_t + v_t)$ and $w_t \in \mathop{\arg\min}\limits_{||w||_2 \leq 1}w^T\nabla F(x_t)$. We have $w_t^T\nabla F(x_t) = -||\nabla F(x_t)||_2$. By the $\alpha_{\mathcal{C}}$-strong convexity of $\mathcal{C}$, we then have 
$$\widetilde{v}_t := c_t + \frac{\alpha_{\mathcal{C}}}{8}||v_t - x_t||_2^2w_t \in \mathcal{C}.$$
Again using the minimality of $v_t$, and applying inequality~\eqref{EqnFBd}, we then obtain
\begin{align}
\label{eq:dif1V3}
(v_t - x_t)^T \nabla F(x_t) &\leq (\widetilde{v}_t - x_t)^T \nabla F(x_t) + \Delta \notag\\
&= \frac{1}{2}(v_t - x_t)^T \nabla F(x_t) + \frac{\alpha_{\mathcal{C}}}{8}||v_t - x_t||_2^2w_t^T \nabla F(x_t) + \Delta \notag\\
&\leq -\frac{h_t}{2} + \frac{3}{2}\Delta - \frac{\alpha_{\mathcal{C}}||v_t - x_t||_2^2}{8}||\nabla F(x_t)||_2.
\end{align}

Using the $\tau_u$-smoothness of $F$ and the definition of $x_{t + 1}$, we also have
\begin{align*}
F(x_{t + 1}) \leq F(x_t) + \eta(v _t - x_t)^T \nabla F(x_t) + \frac{\tau_u}{2}\eta^2||v_t - x_t||_2^2,
\end{align*}
and by subtracting $F(x_{*})$ from both sides, we obtain
\begin{align*}
h_{t + 1} \leq h_t + \eta(v _t - x_t)^T \nabla F(x_t) + \frac{\tau_u}{2}\eta^2||v_t - x_t||_2^2.
\end{align*}
Combined with inequality~\eqref{eq:dif1V3}, we then obtain
\begin{align*}
h_{t + 1} &\leq h_t\left(1 - \frac{\eta}{2}\right) -\eta\frac{\alpha_{\mathcal{C}}||v_t - x_t||_2^2}{8}||\nabla F(x_t)||_2 + \frac{\tau_u}{2}\eta^2||v_t - x_t||_2^2 + \frac{3}{2}\Delta\eta\\
&= h_t\left(1 - \frac{\eta}{2}\right) + \frac{||v_t - x_t||_2^2}{2}\left(\eta^2\tau_u - \eta\frac{\alpha_{\mathcal{C}}||\nabla F(x_t)||_2}{4}\right) + \frac{3}{2}\Delta\eta\\
&\leq h_t\left(1 - \frac{\eta}{2}\right) + \frac{||v_t - x_t||_2^2}{2}\eta\tau_u\left(\eta - \frac{\alpha_{\mathcal{C}} r}{4\tau_u}\right) + \frac{3}{2}\Delta\eta.
\end{align*}
If $\frac{\alpha_{\mathcal{C}} r}{4} \geq \tau_u$, then $\eta = 1$; otherwise, we have $\eta = \frac{\alpha_{\mathcal{C}} r}{4\tau_u}$. Hence, we have
\begin{align*}
h_{t + 1} & \leq h_t\max\left\{\frac{1}{2}, 1 - \frac{\alpha_{\mathcal{C}} r}{8\tau_u}\right\} + \frac{3}{2}\Delta\min\left\{1, \frac{\alpha_{\mathcal{C}} r}{4\tau_u}\right\} \\
&\leq h_t\max\left\{\frac{1}{2}, 1 - \frac{\alpha_{\mathcal{C}} r}{8\tau_u}\right\} + \frac{3}{2}\Delta \eta\\
&= ch_t + \frac{3}{2}\Delta \eta.
\end{align*}
Since $c < 1$, we may iterate to obtain
\begin{align*}
h_t \leq c^th_0 + \frac{3\Delta\eta}{2(1 - c)},
\end{align*}
which completes the proof.
\end{proof}

Let us now present Lemma \ref{lemma:NonAccRel_FW}, for the relaxed version of the classical non-accelerated Frank-Wolfe method.
\begin{lemma}[\cite{NOPL}, \cite{Rev_FW_Prof_Free_Jaggi}]
\label{lemma:NonAccRel_FW}
Let $F : \mathbb{R}^p \rightarrow \mathbb{R}$ be convex and differentiable. Let $\mathcal{C} \subseteq \mathbb{R}^p$ be convex and compact. Let $\Delta > 0$ be fixed, let $x_1 \in \mathcal{C}$, and let $T > 0$. Suppose $\{v_t\}_{t = 1}^T$ is a sequence of vectors from $\mathcal{C}$, with $x_{t + 1} = (1 - \mu_t)x_t + \mu_t v_t$, such that for all $t \in [T]$, we have $v_t^T\nabla F(x_t) \leq \mathop{\min}\limits_{v \in \mathcal{C}}v^T\nabla F(x_t) + \frac{\Delta\mu_t\Gamma_{F}}{2}$. Here, $\mu_t = \frac{2}{t + 2}$ and
\begin{align*}
\Gamma_F = \mathop{\sup}\limits_{\substack{x, y \in \mathcal{C}, \gamma \in (0, 1], \\ z = x + \gamma(y - x)}}\frac{2}{\gamma^2}\left(F(z) - F(x) - (z - x)^T\nabla F(x)\right),
\end{align*}
i.e., $\Gamma_F$ is the curvature constant of $F$. Then
\begin{align*}
F(x_T) - \mathop{\min}\limits_{x \in \mathcal{C}}F(x) \leq \frac{2\Gamma_F}{T + 2}(1 + \Delta).
\end{align*}
\end{lemma}

We now state Lemma \ref{lemma:NOPLconvDPFW}.

\begin{lemma}[Theorem B.2 in Talwar et al.~\cite{NOPL}]
\label{lemma:NOPLconvDPFW}
Let $L_2$ be as in Algorithm \ref{alg:PrivNon-accERM}, and assume $0 < \epsilon \lesssim 1$. Let $G_{\mathcal{C}} = \mathbb{E}\left[\mathop{\sup}\limits_{\theta \in \mathcal{C}}\theta^T b\right]$, with $b \sim N(0, I_p)$, be the Gaussian width of $\mathcal{C}$, and let 
\begin{align}
\label{eq:curv_const_Talwar_DP}
\Gamma_{\mathcal{L}} = \mathop{\sup}\limits_{\substack{x, y \in \mathcal{C}, \gamma \in (0, 1], \\ a = x + \gamma(y - x)}}\frac{2}{\gamma^2}\left(\mathcal{L}(a, z_1) - \mathcal{L}(x, z_1) - (a - x)^T\nabla \mathcal{L}(x, z_1)\right)
\end{align}
be the curvature constant of $\mathcal{L}(\theta, z_1)$. Setting $T = \left(\frac{n\epsilon\Gamma_{\mathcal{L}}}{L_2G_{\mathcal{C}}}\right)^{2/3}$, Algorithm \ref{alg:PrivNon-accERM} returns $\theta_T$ such that
\begin{align*}
\mathbb{E}\left[\mathcal{L}(\theta_T, \mathcal{D}_n) - \mathop{\min}\limits_{\theta \in \mathcal{C}}\mathcal{L}(\theta, \mathcal{D}_n)\right] = O\left(\frac{\Gamma_{\mathcal{L}}^{1/3}(L_2G_{\mathcal{C}})^{2/3}\log^2(n/\delta)}{(n\epsilon)^{2/3}}\right).
\end{align*}
\end{lemma}

\subsection{Background on Differential Privacy}\label{sec:Appendix_DP_Prelim}

Let us present some technicalities regarding the notion of differential privacy. Firstly, one of the most common ways of making the output of a mechanism private is to add noise to the output. However, in order to do that, one generally requires the output of the mechanism on any dataset $X$ to not change too much if we change one of the $n$ elements of $X$. We call this notion bounded sensitivity, and we state it formally below:
\begin{definition}
A function $\mathcal{G} : \mathcal{E}^n \rightarrow \mathbb{R}^p$ has $\ell_2$-bounded sensitivity $sens(\mathcal{G})$ if $\mathop{\sup}\limits_{X \sim X'}||\mathcal{G}(X) - \mathcal{G}(X')||_2 = sens(\mathcal{G}) < \infty$. Here, $X \sim X'$ means that $X$ and $X'$ differ in one element.
\end{definition}

Note that generally, it is enough to work with an upper bound on the sensitivity. With this in mind, we present a way of making any vector-valued function differentially private by adding Gaussian noise:

\begin{lemma}[\cite{IMPROVE_DP}]
\label{lemma:mGDP}
Let $\epsilon, \delta \in (0, 1)$. Define the Gaussian mechanism that operates on a function $\mathcal{G} : \mathcal{E}^n \rightarrow \mathbb{R}^p$ with $\ell_2$-bounded sensitivity $sens(\mathcal{G}) = \mathop{\sup}\limits_{X \sim X'}||\mathcal{G}(X) - \mathcal{G}(x')||_2 < \infty$ as $\widehat{\theta}(X) = \mathcal{G}(X) + \xi$, where $\xi \sim N(0, \sigma^2I_p)$ and $\sigma^2 = \frac{2 sens(\mathcal{G})^2\log(1.25/\delta)}{\epsilon^2}$. Then $\widehat{\theta}$ is $(\epsilon, \delta)$-DP.
\end{lemma}

With a way of turning the output of any deterministic function with $\ell_2$-bounded sensitivity differentially private, it is natural to ask if an adaptive sequence of iterations of mechanisms that are themselves $(\epsilon, \delta)$-DP stays differentially private. The answer is affirmative. We present two results in this regard, namely the basic and advanced composition theorems. The basic composition is a pessimistic result that is tight, for example, if the sequence of algorithms consists of Gaussian mechanisms and the noise random variables are independent. The advanced composition is a tighter result when, for example, the noise does not add linearly. These results are useful when we make gradient methods private by noise addition and we have to ensure privacy of the whole iterative gradient algorithm. We state them below.

\begin{lemma}[Basic Composition \cite{Boost_DP}]
\label{lemma:Basic_Comp}
For every $\epsilon, \delta \geq 0$ and $T \in \mathbb{N}$, the family of $(\epsilon, \delta)$-DP mechanisms are $(T\epsilon, T\delta)$-DP under
$T$-fold adaptive composition.
\end{lemma}

\begin{lemma}[Advanced Composition \cite{Boost_DP}]
\label{lemma:compGDP}
For every $\epsilon > 0$, $\delta \in \left(0, 1\right)$ and $T \in \mathbb{N}$, the class of $\left(\frac{\epsilon}{2\sqrt{2T\log(2/\delta)}}, \frac{\delta}{2T}\right)$-DP mechanisms is $(\epsilon_{tot}, \delta_{tot})$-DP under $T$-fold adaptive composition, for
\begin{align*}
\epsilon_{tot} = \frac{\epsilon}{2} + \frac{\epsilon\sqrt{T}}{2\sqrt{2\log(2/\delta)}}(e^{\frac{\epsilon}{2\sqrt{2T\log(2/\delta)}}} - 1), \qquad \delta_{tot} = \delta.
\end{align*}
\end{lemma}
For $\epsilon \leq 0.9$, we obtain the following Corollary from Lemma \ref{lemma:compGDP}:
\begin{corollary}[\cite{Comp_Thm_Kairouz}]
\label{cor:ADV_COMP_0.9}
For every $\epsilon \in (0, 0.9]$, $\delta \in \left(0, 1\right)$ and $T \in \mathbb{N}$, the class of $\left(\frac{\epsilon}{2\sqrt{2T\log(2/\delta)}}, \frac{\delta}{2T}\right)$-DP mechanisms is $(\epsilon, \delta)$-DP under $T$-fold adaptive composition.
\end{corollary}

\subsection{Preliminaries on Concentration Inequalities}

Now we have a brief section on concentration inequalities. One crucial notion for our analysis and for the analysis of the performance of machine learning algorithms is sub-Gaussianity. We state it below in general for vectors in $\mathbb{R}^p$ for $p \geq 1$, with the understanding that for $p = 1$,
 we talk about one-dimensional variables.
 
\begin{definition}
A zero-mean random vector $X \in \mathbb{R}^p$ is sub-Gaussian with parameter $\sigma^2$ if
\begin{align*}
\mathbb{E}[e^{v^T X}] \leq e^\frac{||v||_2^2\sigma^2}{2}, \quad \forall v \in \mathbb{R}^p.
\end{align*}
We write equivalently that a zero-mean random vector $X$ is sub-Gaussian with parameter $\sigma^2$ if $X \in \mathcal{G}\left(\sigma^2\right)$.
\end{definition}
For sub-Gaussian random variables and vectors, we have the following concentration results:
\begin{lemma}[\cite{SG_DEF}]
\label{lemma:SG_conc_variable}
Let $X \in \mathbb{R}$ be zero-mean and $X \in \mathcal{G}(\sigma^2)$. Then, for all $t \geq 0$, we have
\begin{align*}
\max\left\{\mathbb{P}(X \geq t), \mathbb{P}(X \leq - t)\right\} \leq e^{\frac{-t^2}{2\sigma^2}}.
\end{align*}
\end{lemma}
\begin{lemma}[Lemma $1$ in \cite{CIGauss}]
\label{lemma:CIGV}
Let $X \in \mathbb{R}^p$, with $p \geq 2$, be zero-mean, such that $X \in \mathcal{G}(\sigma^2)$. Then for all $t > 0$, we have 
\begin{align*}
\mathbb{P}\left(||X||_2 > t\right) \leq 4^pe^{-\frac{t^2}{8\sigma^2}}.
\end{align*}
\end{lemma}
One important class of sub-Gaussian random variables ($p = 1$) is the one of bounded random variables. For this, we have Hoeffding's Lemma:

\begin{lemma}[Hoeffding's Lemma, \cite{Hoeff}]
Let $X \in [a, b]$ be zero-mean. Then $X \in \mathcal{G}\left(
\frac{(b - a)^2}{4}\right)$.
\end{lemma}

\subsection{Proof of Lemma \ref{lemma:GLM_prelim_lem}}\label{sec:Proof_GLM_prelim_lemma}

Note that since $\|\Phi''\|_\infty \leq K_{\Phi''}$, we have for all $||v||_2 = 1$ that
\begin{align*}
v^T\nabla^2\mathcal{R}(\theta)v \leq K_{\Phi''}\mathbb{E}_x[(v^Tx)^2]\leq K_{\Phi''}v^T\Sigma v \leq K_{\Phi''}\lambda_{\max}(\Sigma), \quad \forall \theta \in \mathcal{C}.
\end{align*}
Hence, $\mathcal{R}$ is $K_{\Phi''}\lambda_{\max}(\Sigma)$-smooth over $\mathbb{R}^p$. 

Assume now that $||x||_2 \leq L_x$ and $||\theta||_2 \leq K_B$ for all $\theta \in \mathcal{C}$. By Cauchy-Schwarz and the assumptions on $\Phi''$ stated at the start of Section \ref{sec:Generalized Linear Models (GLM)}, we have $\Phi''(x^T\theta) \geq \Phi''(L_xK_B)$ for all $\theta \in \mathcal{C}$. Thus, for all $||v||_2 = 1$, we have
\begin{align*}
v^T\nabla^2\mathcal{R}(\theta)v \geq \Phi''(L_xK_B)v^T\Sigma v \geq \Phi''(L_xK_B)\lambda_{\min}(\Sigma), \quad \forall \theta \in \mathcal{C}.
\end{align*}
Thus, since $\Sigma \succ 0$ and $\Phi''(L_xK_B) > 0$, we see that $\mathcal{R}$ is $\Phi''(L_xK_B)\lambda_{\min}(\Sigma)$-strongly convex over $\mathcal{C}$, as required.

Finally, assume $\theta^* \in \mathcal{C}$. Note that since $\Phi$ is convex and $-yx^T\theta$ and $x^T\theta$ are linear in $\theta$, the functions $\mathcal{L}$ and $\mathcal{R}$ are convex. By equation~\eqref{eq:GLMder}, we see that $\nabla\mathcal{R}(\theta^*) = 0$. Hence, since $\mathcal{R}$ is convex over $\mathcal{C}$, we conclude that $\mathcal{R}$ is minimized at $\theta^*$.


\section{Simulation Details}
\label{AppSims}

We provide more implementation details for the figures in Section~\ref{SecSims}. Figures \ref{fig:plot1}, \ref{fig:plot2}, \ref{fig:plot6}, and \ref{fig:plot_l_bigger_0} are based on the Frank-Wolfe method and acceleration, while Figures \ref{fig:plot3}, \ref{fig:plot4}, and \ref{fig:plot5} consider projected gradient descent and Nesterov's AGD. Unless specified otherwise, whenever we deal with a GLM or a linear regression model, we take the true parameter $\theta^* = (1, \dots, 1)^T$, and for linear regression, we simulate $x \indep w$. All the implementations were done based on NumPy in Python.

\textbf{Figure \ref{fig:plot1}:} We compare Theorem \ref{theorem:UPRidge} with Lemma \ref{lemma:NOPLconvDPFW}, using Algorithms \ref{alg:PrivFWERM} (ACCFW) and \ref{alg:PrivNon-accERM} (FW). We simulate $n = 10,000$ linearly separable data points, with $p = 10$, $||x_i||_\infty\leq 1$, $y_i = \operatorname{sgn}(x_i^T v^*)$, $|x_i^Tv^*|\geq \frac{\sqrt{p}}{2}$, and $v^*=\frac{(1,\dots,1)^T}{\sqrt{p}}$, $\forall\, i\in[n]$. We optimize over $\mathcal{C} = \mathbb{B}_2\left(\frac{1}{4\sqrt{p}}\right)$ (hence, $S_1 = 1$), with $\delta = \frac{1}{3}$, and we pick $\Gamma_{\mathcal{L}}$ and $G_{\mathcal{C}}$ as described in Remark \ref{remark:compare_dist_free_ACCFW}. We initialize $\theta_0 = 0$. The plot shows the logarithm of the excess mean squared error loss (we take $L_2 = \sqrt{p} + pD$) versus $n$, for $\epsilon\in\{0.5,\,0.9\}$. In line with Remark \ref{remark:compare_dist_free_ACCFW}, Algorithm \ref{alg:PrivFWERM} (rate $\sqrt{p}/(n\epsilon)$) outperforms Algorithm \ref{alg:PrivNon-accERM} (rate $(\sqrt{p}/(n\epsilon))^{2/3}$), and larger $\epsilon$ leads to faster convergence.

\textbf{Figure \ref{fig:plot2}:} We compare Theorem \ref{theorem:ERM_UP_C_Inc} with Lemma \ref{lemma:NOPLconvDPFW}, using Algorithms \ref{alg:PrivFWERM} (ACCFW) and \ref{alg:PrivNon-accERM} (FW). We simulate $n = 5,500$ independent data points from a logistic regression model (see Section \ref{sec:Generalized Linear Models (GLM)}), with $p = 3$, $L_x = 1$, $C_1 = 1$, $\zeta = \frac{1}{3}$, $\delta = \frac{1}{3}$, $\lambda_{\min}(\Sigma) = \frac{1}{3p}$, and $D = \frac{12p}{\Phi''(\sqrt{p})n^{2/5}}$. Each entry of $x_i$ is drawn independently from $Unif\left[-\frac{1}{\sqrt{p}}, \frac{1}{\sqrt{p}}\right]$. Also, $\theta_0 = 0$. The plot shows the logarithm of the excess empirical risk (we take $L_2 = (K_{\Phi'} + K_y)L_x = 2$) versus $n$, for $\epsilon\in\{0.5,\,0.9\}$. We can see that Algorithm \ref{alg:PrivFWERM} (rate $1/(n^{4/5}\epsilon)$) does better than Algorithm \ref{alg:PrivNon-accERM} (rate $(1/(n\epsilon))^{2/3}$), as discussed in Remark \ref{remark:Best_q_ERM}, and larger $\epsilon$ leads to faster convergence.

\textbf{Figure \ref{fig:plot3}:} We compare Nesterov’s AGD (Theorem \ref{theorem:Priv_HT_AGD_Smooth}) with projected GD (Theorem \ref{theorem:Priv_HT_GD_Smooth}) using Algorithm \ref{alg:RobPGDNFW} and the pseudo-Huber loss (with $q = \frac{1}{5}$, see Appendix \ref{Appendix A}). Gradient estimators and learning rates are as specified in Theorem \ref{theorem:Priv_HT_GD_Smooth}. We simulate $n = 100,000$ data points from the model $y = x^T\theta^* + w$, $w \sim ST(3)$, with $p = 10$ and each entry of $x$ drawn independently from $Unif\left[-\tfrac{1}{\sqrt{p}},\tfrac{1}{\sqrt{p}}\right]$ (so $L_x = 1$). We initialize $\theta_0 = 0$ (and $\theta_1 = (1.1, \dots, 1.1)^T$ for Nesterov's AGD). We take $\tau_u = \frac{1}{3p}$ (see Lemma \ref{lemma:pHu_prelim}), and $\delta = \frac{1}{3}$. The plot displays $\log(\|\theta_T-\theta^*\|_2)$ (with $L_2 = qL_x = \frac{1}{5}$) versus $n$, for $\epsilon \in \{0.1,\,0.9\}$. Nesterov’s AGD (rate $1/n^{2/5} + 1/(n\epsilon^2)$) outperforms projected GD (rate $1/n^{1/5} + 1/(n^{1/2}\epsilon)$, cf.\ Remark \ref{remark:compare_AGD_GD_Smooth}), and larger $\epsilon$ accelerates convergence. Moreover, since $\nabla\mathcal{R}(\theta^*) = 0$ and by the smoothness of the risk (see Lemma \ref{lemma:pHu_prelim}), a bound on $\|\theta_T-\theta^*\|_2^2$ implies a bound on $\mathcal{R}(\theta_T)-\mathcal{R}(\theta^*)$ (up to a constant), so this figure also reflects excess risk upper bounds. Gradient estimators and learning rates are as specified in Theorem \ref{theorem:htAGD}. 

\textbf{Figure \ref{fig:plot4}:} We compare Nesterov’s AGD (Theorem \ref{theorem:htAGD}) with projected GD (Lemma \ref{lemma:GLMhtGD}) using Algorithm \ref{alg:RobPGDNFW} and the squared error loss. We simulate $n = 1,500$ data points from the model $y = x^T\theta^* + w$, $w \sim ST(3)$, with $p = 100$, $x \sim N(0, \Sigma)$, and $\Sigma$ is a diagonal matrix with $\tau_l = \lambda_{\min}(\Sigma) = \frac{\lambda_{\max}(\Sigma)}{1.5} = \frac{\tau_u}{1.5} = \frac{2}{3}$. We initialize $\theta_0 = 0$ (and $\theta_1 = 0$ for Nesterov's AGD) and we take $\zeta = \frac{1}{10}$. The plot shows $\log(\|\theta_t-\theta^*\|_2)$ versus $t \in \{0, \dots T\}$, with $T = 20$, for $n \in \{600, 900, 1200, 1500\}$. We can see a faster convergence of Nesterov's AGD in the exponentially decaying term with $t$ (cf.\ Remark~\ref{RemFaster}), while a larger $n$ leads to a smaller error term (independent of $t$), in line with the results of Theorem \ref{theorem:htAGD} and Lemma \ref{lemma:GLMhtGD}.

\textbf{Figure \ref{fig:plot5}:} With the setup for Figure \ref{fig:plot4}, but with $n = 60,000$, we compare Nesterov’s AGD (Theorem \ref{theorem:htAGD}) to projected GD (Lemma \ref{lemma:GLMhtGD}), as described in Section \ref{sec:Comparisons in the Heavy-Tailed Setting}. We take $T = \log_{\frac{\tau_u + \tau_l}{\tau_u}}(\sqrt{n})$ for projected GD, and $T = \frac{\log(\sqrt{n})}{\frac{1}{2}\log\left(\frac{1}{1 - \sqrt{\frac{\tau_l}{\tau_u}}}\right)}$ for Nesterov's AGD. We plot $\log||\theta_T-\theta^*||_2$ versus $n$. The results show that Nesterov’s AGD yields a slight improvement (its curve is essentially a constant translation of that for projected GD), supporting the finding that AGD’s advantage is up to an absolute constant, and not an improved rate in $n$ and $p$.

\textbf{Figure \ref{fig:plot6}:} We compare Theorem \ref{theorem:ACCFW_lambda=0} with Theorem \ref{theorem:NonACC_lambda=0}, using Algorithms \ref{alg:PrivFWERM} (ACCFW) and \ref{alg:PrivNon-accERM} (FW). We simulate $n = 50,000$ samples from the model $y = x^T\theta^* + w$, $w \sim ST(3)$, with $p = 10$, and $x \sim N(0, \Sigma)$. The true parameter is $\theta^* = \frac{(1, \dots, 1)^T}{\sqrt{p}}$, and $\Sigma$ is diagonal with $\Sigma_{ii} = 1$ for $i \leq m$, and $\Sigma_{ii} = 0$ for $i > m$, where $m \in \left\{3, 6, 9\right\}$. All other parameters follow the settings in Theorems \ref{theorem:ACCFW_lambda=0} and \ref{theorem:NonACC_lambda=0}. We initialize $\theta_0 = 0$ and set $\zeta = \frac{1}{10}$. For each $m$, we simulate $50,000$ data points. The plot shows $\|\theta_T-\theta^*\|_2$ versus $n$. As expected from the bounds in Theorems \ref{theorem:ACCFW_lambda=0} and \ref{theorem:NonACC_lambda=0}, the error plateaus at non-zero levels due to the $c_{\mathcal{K}}$ term. Notably, the non-accelerated version converges more slowly but ultimately incurs less error, while the accelerated version reaches its plateau faster, reflected in the flatter curves.

\textbf{Figure \ref{fig:plot_l_bigger_0}:} We compare Theorem \ref{theorem:ACCFWROB} with Theorem \ref{theorem:HT_NonAccFW}, using Algorithms \ref{alg:PrivFWERM} (ACCFW) and \ref{alg:PrivNon-accERM} (FW). We simulate $n = 20,000$ samples from the model $y = x^T\theta^* + w$, $w \sim ST(3)$, with $p = 10$, and $x \sim N(0, \Sigma)$, with $\Sigma = I_p$. For FW, we take $\mathcal{C}$ to be an $\ell_2$-ball centered at $0$ that contains $\theta^*$. For ACCFW, we take $C_1 = 0.5$. We initialize $\theta_0 = 0$ and set $\zeta = \frac{1}{10}$. All the other parameters are as specified in Theorems \ref{theorem:HT_NonAccFW} and \ref{theorem:ACCFWROB}. The plot shows $\log(\|\theta_T-\theta^*\|_2)$ versus $n$. We can observe that Algorithm \ref{alg:PrivFWERM} (rate $1/n^{1/5}$) outperforms Algorithm \ref{alg:PrivNon-accERM} (rate $1/n^{1/6}$).


\section{Auxiliary Results}
\label{Appendix A}

We begin with two technical lemmas about sequences of real numbers:
\begin{lemma}
\label{lemma:9}
For a sequence of real numbers $(x_n)_{n\geq 0}$ with initial points $x_0$ and $x_1$, defined by
\begin{align*}
x_{n + 2} = ax_{n + 1} + bx_n + c,
\end{align*}
with $a, b, c \in \mathbb{R} \backslash \{0\}$, such that $a + b \neq 1$ and the solutions $\{s_1, s_2\}$ of $x^2 - ax - b = 0$ are real and distinct, we have constants $C_1$ and $C_2$ such that
\begin{align*}
x_{n} = C_1s_1^n + C_2s_2^n + \frac{c}{1 - a - b}.
\end{align*}
\end{lemma}
\begin{proof}
By letting $x_n = y_n + d$, we obtain
\begin{align*}
y_{n + 2} + d = ay_{n + 1} + ad + by_n + bd + c,
\end{align*}
and for $d = \frac{c}{1 - a - b}$, we obtain
\begin{align*}
y_{n + 2} = ay_{n + 1} + by_n.
\end{align*}
We impose $y_0 = C_1s_1 + C_2s_2$ and $y_1 = C_1s_1^2 + C_2s_2^2$,
so $(C_1, C_2)$ solve this system uniquely, since $s_1 \neq s_2$. Therefore, by induction, we have
\begin{align*}
y_n = C_1s_1^n + c_2s_2^n \Rightarrow x_n = C_1s_1^n + c_2s_2^n + \frac{c}{1 - a - b}.
\end{align*}
\end{proof}

\begin{remark}
\label{remark:Second_Ord_Ineq_Rmk}
In fact, for completeness, we have $C_1 = \frac{\frac{s_2}{s_1}(x_0 - d) - \frac{1}{s_1}(x_1 - d)}{s_2 - s_1}$ and $C_2 = \frac{\frac{-s_1}{s_2}(x_0 - d) + \frac{1}{s_2}(x_1 - d)}{s_2 - s_1}$. Note that if in addition, we have $a, b > 0$ and we require $x_{n + 2} \leq ax_{n + 1} + bx_n + c$, then we can show inductively that $x_n \leq C_1s_1^n + C_2s_2^n + \frac{c}{1 - a - b}$.
\end{remark}

\begin{lemma}[\cite{Schmidt_2011_Conv}]
\label{lemma:Smooth_GD_Helper}
Assume that the non-negative sequence $\{u_t\}_{t \geq 0}$ satisfies the following recursion for $t \geq 1$:
\begin{align*}
u_t^2 \leq S_t + \sum_{i = 1}^t\lambda_iu_i,
\end{align*}
with $\{S_t\}$ a non-decreasing sequence, $S_0^2 \geq u_0$, and $\lambda_i \geq 0$ for all $i \geq 0$. Then for all $t \geq 1$ and $a_t = \frac{1}{2}\sum_{i = 1}^t\lambda_i$, we have
\begin{align*}
u_t \leq a_t + \left(S_t + a_t^2\right)^{1/2}.
\end{align*}
\end{lemma}

Let us also recall the form of a t-distribution and some aspects related to it.
\begin{definition}
\label{def:t_dist}
A random variable $X$ follows a t-distribution with $\nu$ degrees of freedom, denoted by $ST(\nu)$, if its pdf takes the form
\begin{align*}
p(x) = \frac{\Gamma\left(\frac{\nu + 1}{2}\right)}{\sqrt{\pi\nu}\Gamma\left(\frac{\nu}{2}\right)}\left(1 + \frac{x^2}{\nu}\right)^{-\frac{\nu + 1}{2}}, \quad \forall x \in \mathbb{R}.
\end{align*}
\end{definition}
\begin{lemma}
\label{lemma:t_dist_lemma}
Let $X \sim ST(\nu)$. The second moment of $X$ exists if and only if $\nu > 2$, and is equal to $\frac{\nu}{\nu - 2}$. Additionally, if $X$ has $r$ finite moments, then if $\nu \in \mathbb{N}$, we have $r = \nu - 1$ and $r = \lfloor \nu \rfloor$ otherwise.
\end{lemma}

Now we have a useful lemma that allows us to pass from results with high probability to results in expectation:

\begin{lemma}
\label{lemma:HP_TO_EXP}
Let $Z \geq 0$ be a random variable. Suppose $Z \leq A + B\sqrt{\log\left(\frac{C}{\zeta}\right)}$ with probability at least $1 - \zeta$, for all $\zeta \in (0, 1)$, and $A, B, C > 0$ are constants independent of $\zeta$. Then
\begin{align*}
\mathbb{E}[Z] \leq A + \frac{\sqrt{\pi}}{2}BC.
\end{align*}
\end{lemma}

\begin{proof}
Since $Z \geq 0$, we have 
\begin{align*}
\mathbb{E}[Z] = \int_0^\infty\mathbb{P}(Z > s) \, ds = \int_0^A\mathbb{P}(Z > s) \, ds + \int_A^\infty \mathbb{P}(Z > s) \, ds \leq A + \int_A^\infty \mathbb{P}(Z > s) \, ds.
\end{align*}
Using the assumption, we then have
\begin{equation*}
\mathbb{E}[Z] \leq A + C\int_A^\infty e^{\left(\frac{s - A}{B}\right)^2}\, ds = A + C\int_0^\infty e^{\left(\frac{s}{B}\right)^2}\, ds
= A + \frac{\sqrt{\pi}}{2}BC,
\end{equation*}
as required.
\end{proof}

We also state a concentration result for random matrices of the form $vv^T$, for $v \in \mathbb{R}^p$ and $p \geq 1$:

\begin{lemma}[\cite{Prob_Stats_Chen}]
\label{lemma:Concxx^T}
Let $x_1, \dots, x_n$ be independent, zero-mean random vectors in $\mathbb{R}^p$. Suppose that for all $i \in [n]$, we have $\mathrm{Var}(x_i) = \Sigma$ and $||x_i||_2 \leq \sqrt{C_1}$, for some $C_1 > 0$. Then for all $t \geq 0$, we have
\begin{align*}
\mathbb{P}\left(||\Sigma_n - \Sigma||_2 \geq t\right) \le 2pe^{\frac{-nt^2}{2C_1(||\Sigma||_2 + 2t/3)}},
\end{align*}
with $\Sigma_n = \frac{1}{n}\sum_{i = 1}^n x_ix_i^T$.
\end{lemma}

Let us now recall the notion of a covering and covering number, and a corresponding result about $\ell_2$-balls:

\begin{definition}[Covering and covering number in the $\ell_2$-norm \cite{Stats_Info_Duchi}]
Let $\mathcal{D} \subseteq \mathbb{R}^p$, for $p \in \mathbb{N}$. An \emph{$\epsilon$-cover} of the set $\mathcal{D}$ with respect to the $\ell_2$-norm is a set $\{d_1, \dots, d_N\}$ such that for any point $d \in \mathcal{D}$, there exists some $v \in [N]$ such that $||d - d_v||_2 \leq \epsilon$. The \emph{$\epsilon$-covering number} of $\mathcal{D}$ is
\begin{align*}
N(\epsilon, \mathcal{D}, ||\cdot||_2) := \inf\left\{N \in \mathbb{N}: \text{ there exists an } \epsilon\text{-cover  } \{d_1, \dots, d_N\} \text{ of } \mathcal{D}\right\}.
\end{align*}
\end{definition}
\begin{lemma}[\cite{Stats_Info_Duchi}]
\label{lemma:Covering_Nr_Ball}
The $\epsilon$-covering number of $\mathbb{B}_2(r)$ in $\mathbb{R}^p$, for $r > 0$ and $p \in \mathbb{N}$, satisfies
\begin{align*}
\left(\frac{r}{\epsilon}\right)^p \leq N(\epsilon, \mathbb{B}_2(r), ||\cdot||_2) \leq \left(1 + \frac{2r}{\epsilon}\right)^p.
\end{align*}
\end{lemma}

We now recall a classical result about consistency of the maximum likelihood estimator:

\begin{lemma}[\cite{Lehmann_Point_Estim}]
\label{lemma:Consis_MLE}
Let $p \in \mathbb{N}$ and let $\mathcal{B} \subseteq \mathbb{R}^p$ be a compact parameter space. Let $\mathcal{P} = \left\{P_\theta: \theta\in \mathcal{B} \right\}$ be a parametric model and let $f(z, \theta)$ be the likelihood function for the data point $z$ at $\theta$. Let $\theta^* \in \mathcal{B}$ be the true parameter and $\widehat{\theta}_n$ be an MLE based on the random sample $\{z_i\}_{i = 1}^n \overset{\iid}{\sim} P_{\theta^*}$. Writing $\mathcal{Z} = \left\{z:  f(z, \theta^*) > 0\right\}$ for the support of $f(\cdot, \theta^*)$, suppose $\theta \mapsto f(z, \theta)$ is continuous, for all $z \in \mathcal{Z}$. Assume $\mathbb{E}_{z \sim P_{\theta^*}}\left[\mathop{\sup}\limits_{\theta \in \mathcal{B}}|\log(f(z, \theta))|\right] < \infty$. Then $\widehat{\theta}_n$ converges in probability to $\theta^*$.
\end{lemma}

Let us now state a result about the convergence of $M$-estimators:
\begin{lemma}[\cite{Weak_Conv_Van_Der}]
\label{lemma:M_Estim_Conv}
Let $\mathbb{M}_n$ be a real-valued stochastic processes indexed by a metric space $\left(\mathcal{B}, d\right)$. Let $\mathbb{M} : \mathcal{B} \rightarrow \mathbb{R}$ be a deterministic function. Assume $\theta^* = \arg\min\limits_{\theta \in \mathcal{B}}\mathbb{M(\theta)}$ and $\mathbb{M}(\theta) - \mathbb{M}(\theta^*) \gtrsim d(\theta, \theta^*)^2$, for every $\theta$ in a neighborhood of $\theta^*$. Let $\widehat{\theta}_n \in \arg\min\limits_{\theta \in \mathcal{B}}\mathbb{M}_n(\theta)$. Suppose that, for sufficiently large $n$ and sufficiently small $u > 0$, the centered process $\mathbb{U}_n = \mathbb{M}_n - \mathbb{M}$ satisfies
\begin{align*}
\mathbb{E}\left[\mathop{\sup}\limits_{d(\theta, \theta^*) \leq u}\left|\mathbb{U}_n(\theta) - \mathbb{U}_n(\theta^*)\right|\right] \lesssim \frac{\phi_n(u)}{\sqrt{n}},
\end{align*}
for functions $\phi_n$, such that $u \mapsto \frac{\phi_n(u)}{u^{\alpha}}$ is non-increasing for some $\alpha < 2$ (not depending on $n$). Let $r_n$ be such that 
\begin{align*}
r_n^2\phi_n\left(\frac{1}{r_n}\right) \leq \sqrt{n},
\end{align*}
for sufficiently large $n$. If $\widehat{\theta}_n$ converges in probability to $\theta^*$, then 
\begin{align*}
d(\widehat{\theta}_n, \theta^*) = O_{\mathbb{P}}(r_n^{-1}).
\end{align*}
\end{lemma}
Recall that a sequence $\left\{Z_n\right\}$ is $O_{\mathbb{P}}(x_n)$, where $x_n$ is a deterministic sequence of positive real numbers, if for every $\zeta \in (0, 1)$, there exists $T_\zeta$ and $N_\zeta > 0$, such that $\mathbb{P}\left(|Z_n| \leq T_\zeta x_n\right) \geq 1 - \zeta$, for all $n \geq N_\zeta$. Note that the version of Lemma \ref{lemma:M_Estim_Conv} in \cite{Weak_Conv_Van_Der} relies on the more general assumption that $\mathbb{M}_n(\widehat{\theta}_n) \geq \mathbb{M}_n(\theta) - O_{\mathbb{P}}\left(r_n^{-2}\right)$ (i.e., $\widehat{\theta}_n$ nearly minimizes $\mathbb{M}_n$), which is more general than the version we have stated. 

For bounded random vectors, we have the vector Bernstein inequality. The advantage of this result, compared to a vector concentration result such as Lemma \ref{lemma:CIGV}, is the lack of dependency on the dimension $p$ in the concentration bound.

\begin{lemma}[\cite{Vect_Bern}]
\label{lemma:Bern_Vec}
Let $\{x_i\}_{i = 1}^n$ be independent vectors in $\mathbb{R}^p$, for $p \geq 1$, and assume $\mathbb{E}[x_i] = 0$, $||x_i||_2 \leq \mu$, and $\mathbb{E}\left[||x_i||_2^2\right] \leq \sigma^2$, for all $i \in [n]$. Then for $0 < t < \frac{\sigma^2}{\mu}$, we have
\begin{align*}
\mathbb{P}\left(\left\|\frac{\sum_{i = 1}^nx_i}{n}\right\|_2 \geq t\right) \leq e^{-\frac{nt^2}{8\sigma^2} + \frac{1}{4}} < 2e^{-\frac{nt^2}{8\sigma^2}}.
\end{align*}
\end{lemma}

Next, we discuss some aspects of linear regression using the pseudo-Huber loss with parameter $q > 0$ \cite{PSHUBER_BARRON}:
\begin{align*}
\rho_q(t) = q^2\left(\sqrt{1 + \left(\frac{t}{q}\right)^2} - 1\right).
\end{align*}
The first and second derivatives are given by
\begin{align*}
&\psi_q(t) := \rho'_q(t) = \frac{t}{\sqrt{1 + \left(\frac{t}{q}\right)^2}}, 
&\psi'_q(t) := \rho''_q(t) = \frac{1}{\left(1 + \left(\frac{t}{q}\right)^2\right)^{3/2}}.
\end{align*}
We can derive the following lemma about the pseudo-Huber loss and the corresponding risk, under a parametric linear model:

\begin{lemma}
\label{lemma:pHu_prelim}
Let $L_x, C''_1, q > 0$ and $C'_2 \geq C'_1 > 0$. On the domain $||x||_2 \leq L_x \ and \ y \in \mathbb{R}$, define the loss
\begin{align}
\label{eq:ps_Hub_loss}
\mathcal{L}(\theta, (x, y)) = \rho_q(y - x^T\theta), \quad \forall \theta \in \mathbb{R}^p.
\end{align}
Then:
\begin{enumerate}
\item $\mathcal{L}$ is $qL_x$-Lipschitz in $\theta$.

\item Consider the linear regression model $y = x^T\theta^* + w$, with $\mathbb{E}[x] = 0$, $\mathbb{E}[w] = 0$, $\Sigma = \mathbb{E}[xx^T]$, and $x \indep w$. Assume $\lambda_{\max}(\Sigma) \leq \frac{C'_2}{p}$. Then the corresponding risk $\mathcal{R}$ to \eqref{eq:ps_Hub_loss} is $\frac{C'_2}{p}$-smooth over $\mathbb{R}^p$. 

\item Additionally, let $\mathcal{C}$ be a convex set such that $\theta^* \in \mathcal{C}$ and $||\mathcal{C}||_2 \leq C''_1\sqrt{p}$. Assume $\frac{C'_1}{p} \leq \lambda_{\min}(\Sigma)$ and $x$ has bounded $4^{\text{th}}$ moments, i.e., there exists $\widetilde{C}_4 > 0$ such that $\mathbb{E}\left[(x^Tv)^4\right] \leq \widetilde{C}_4\mathbb{E}\left[(x^Tv)^2\right]^2$, for any $||v||_2 = 1$. Then the risk is 
\begin{align*}
\frac{q^3(C'_1)^4}{4p\left((C'_1)^2q^2 + 8(C''_1)^2(C'_2)^3\widetilde{C}_4 + 2(C'_1)^2\sigma_2^2\right)^{3/2}}\mbox{-strongly convex}
\end{align*}
over $\mathcal{C}$. 

\item $\nabla\mathcal{R}(\theta^*) = 0$, so $\theta_{*} = \theta^*$ is the minimizer of $\mathcal{R}$ over $\mathcal{C}$.
\end{enumerate}
\end{lemma}

\begin{proof}
We first prove (1).
Note that
\begin{align*}
\nabla \mathcal{L}(\theta, (x, y)) = -\psi_q(y - x^T\theta)x, \ \forall \theta \in \mathbb{R}^p. 
\end{align*}
Hence, we clearly have $||\nabla \mathcal{L}(\theta, (x, y))||_2 \leq q L_x$ on the domain.

For (2), note that $\nabla\mathcal{R}(\theta) = -\mathbb{E}[\psi_q(y - x^T\theta)x]$, for all $\theta \in \mathbb{R}^p$. Since $\psi_q$ is bounded and differentiable, we can swap expectations and derivatives by the Dominated Convergence Theorem to obtain 
\begin{align*}
\nabla^2\mathcal{R}(\theta) = \mathbb{E}[\psi'_q(y - x^T\theta)xx^T].
\end{align*}
Take $\theta \in \mathbb{R}^p$. Note that $0 < \psi'_q(t) \leq 1$, for $t \in \mathbb{R}$. Hence, we have
\begin{align*}
\nabla^2\mathcal{R}(\theta) = \mathbb{E}[\psi'_q(y - x^T\theta)xx^T] \preceq \Sigma \preceq \lambda_{\max}(\Sigma)I_p \preceq \frac{C'_2}{p}I_p.
\end{align*}

For (3), let $a := \theta^* - \theta$. By Markov's inequality, since $x \indep w$, we have
\begin{align*}
\nabla^2\mathcal{R}(\theta) &= \mathbb{E}\left[\frac{1}{\left(1 + \left(\frac{x^Ta + w}{q}\right)^2\right)^{3/2}}xx^T\right] \\
& \succeq \mathbb{E}\left[\frac{1}{\left(1 + \left(\frac{|x^Ta| + |w|}{q}\right)^2\right)^{3/2}}xx^T\bigg\rvert|w| < 2\mathbb{E}[|w|]\right]\mathbb{P}(|w| < 2\mathbb{E}[|w|])\\
&\succeq \frac{1}{2}\mathbb{E}\left[\frac{1}{\left(1 + \left(\frac{|x^Ta| + \mathbb{E}[|w|]}{q}\right)^2\right)^{3/2}}xx^T\right]\\ 
&= \frac{1}{2}\mathbb{E}\left[\frac{q^3}{\left(q^2 + \left(|x^Ta| + \mathbb{E}[|w|]\right)^2\right)^{3/2}}xx^T\right].
\end{align*}
Let $C'_3 = \frac{2C'_2\sqrt{\widetilde{C}_4}}{C'_1}$ and $A = \left\{|x^Ta| < C'_3||\mathcal{C}||_2\sqrt{\lambda_{\max}(\Sigma)}\right\}$. Again using Markov's inequality, we obtain
\begin{align*}
\mathbb{P}(A^c) \leq \frac{\mathbb{E}\left[(x^Ta)^2\right]}{(C'_3)^2||\mathcal{C}||_2^2\lambda_{\max}(\Sigma)} = \frac{a^T\Sigma a}{(C'_3)^2||\mathcal{C}||_2^2\lambda_{\max}(\Sigma)} \leq \frac{||a||_2^2\lambda_{\max}(\Sigma)}{(C'_3)^2||\mathcal{C}||_2^2\lambda_{\max}(\Sigma)} \leq \frac{1}{(C'_3)^2},
\end{align*}
where $A^c$ denotes the complement of $A$. Hence, we have
\begin{align*}
\nabla^2\mathcal{R}(\theta) &\succeq \frac{1}{2}\mathbb{E}\left[\frac{q^3}{\left(q^2 + \left(|x^Ta| + \mathbb{E}[|w|]\right)^2\right)^{3/2}}xx^T\mathbbm{1}_A\right]\\
&\succeq \frac{q^3}{2\left(q^2 + \left(C'_3||\mathcal{C}||_2\sqrt{\lambda_{\max}(\Sigma)} + \mathbb{E}[|w|]\right)^2\right)^{3/2}}\left(\mathbb{E}[xx^T] - \mathbb{E}[xx^T\mathbbm{1}_{A^c}]\right).
\end{align*}
Take $||v||_2 = 1$ arbitrary. We have by Cauchy-Schwarz that
\begin{align*}
v^T\left(\mathbb{E}[xx^T] - \mathbb{E}[xx^T\mathbbm{1}_{A^c}]\right)v &\geq \lambda_{\min}(\Sigma) - \mathbb{E}\left[(x^Tv)^2\mathbbm{1}_{A^c}\right]\\
&\geq \lambda_{\min}(\Sigma) - \sqrt{\mathbb{E}\left[(x^Tv)^4\right]\mathbb{P}(A^c)}\\
&\geq \lambda_{\min}(\Sigma) - \frac{1}{C'_3}\sqrt{\mathbb{E}\left[(x^Tv)^4\right]}.
\end{align*}
Since $x$ has bounded $4^{\text{th}}$ moments, we have $\mathbb{E}\left[(x^Tv)^4\right] \leq \widetilde{C}_4\mathbb{E}\left[(x^Tv)^2\right]^2\leq \widetilde{C}_4\lambda_{\max}(\Sigma)^2$. Hence, we obtain
\begin{align*}
v^T\left(\mathbb{E}[xx^T] - \mathbb{E}[xx^T\mathbbm{1}_{A^c}]\right)v \geq \lambda_{\min}(\Sigma) - \frac{\lambda_{\max}(\Sigma)\sqrt{\widetilde{C}_4}}{C'_3} \geq \frac{C'_1}{p} - \frac{C'_2\sqrt{\widetilde{C}_4}}{C'_3p} = \frac{C'_1}{2p}. 
\end{align*}
Since $||v||_2 = 1$ was arbitrary, and using $||\mathcal{C}||_2 \leq C''_1\sqrt{p}$ and $C'_3 = \frac{2C'_2\sqrt{\widetilde{C}_4}}{C'_1}$, Jensen's inequality then implies
\begin{align*}
\nabla^2\mathcal{R}(\theta) &\succeq \frac{q^3C'_1}{4p\left(q^2 + \left(C'_3||\mathcal{C}||_2\sqrt{\lambda_{\max}(\Sigma)} + \mathbb{E}[|w|]\right)^2\right)^{3/2}}I_p \\
& \succeq \frac{q^3C'_1}{4p\left(q^2 + \left(C'_3C''_1\sqrt{C'_2} + \mathbb{E}[|w|]\right)^2\right)^{3/2}}I_p\\
&= \frac{q^3(C'_1)^4}{4p\left((C'_1)^2q^2 + \left(2C''_1C'_2\sqrt{C'_2\widetilde{C}_4} + C'_1\mathbb{E}[|w|]\right)^2\right)^{3/2}}I_p\\
&\succeq \frac{q^3(C'_1)^4}{4p\left((C'_1)^2q^2 + 8(C''_1)^2(C'_2)^3\widetilde{C}_4 + 2(C'_1)^2\mathbb{E}[|w|]^2\right)^{3/2}}I_p\\
&\succeq \frac{q^3(C'_1)^4}{4p\left((C'_1)^2q^2 + 8(C''_1)^2(C'_2)^3\widetilde{C}_4 + 2(C'_1)^2\sigma_2^2\right)^{3/2}}I_p,
\end{align*}
as wanted.

Finally, for (4), since $\theta^* \in \mathcal{C}$, $x \indep w$, and $\mathbb{E}[x] = 0$, we have 
\begin{align*}
\nabla \mathcal{R}(\theta^*) = \mathbb{E}\left[\psi_q(x^T(\theta^* - \theta^*) + w)x\right] = \mathbb{E}[\psi_q(w)x] = \mathbb{E}[\psi_q(w)]\mathbb{E}[x] = 0,
\end{align*}
as required.
\end{proof}

\begin{remark}
\label{remark:pHu_prelim_rmk}
In line with the notation introduced in Section \ref{sec:Preliminaries on Optimization}, we have in Lemma \ref{lemma:pHu_prelim} that $\mathcal{R}$ is $\tau_u$-smooth over $\mathbb{R}^p$ and $\tau_l$-strongly convex over $\mathcal{C}$, with $\tau_u = \frac{C'_2}{p}$ and 
\begin{align*}
\tau_l = \frac{q^3(C'_1)^4}{4p\left((C'_1)^2q^2 + 8(C''_1)^2(C'_2)^3\widetilde{C}_4 + 2(C'_1)^2\sigma_2^2\right)^{3/2}}.
\end{align*}
\end{remark}

\begin{remark}
We want to give a practical example of a distribution on $x = \left(x^{(1)}, \dots, x^{(p)}\right)$ that satisfies the stated conditions, namely $\mathbb{E}[xx^T] = \Sigma \succ 0$, $||x||_2 \leq L_x$, $\frac{C'_1}{p} \leq \lambda_{\min}(\Sigma) \leq \lambda_{\max}(\Sigma) \leq \frac{C'_2}{p}$, and $x$ has bounded $4^{\text{th}}$ moments as per Definition \ref{def:bd_moments}.

Take $\left\{x^{(i)}\right\}_{i = 1}^p$ to be \iid from a truncated $N(0, 1/p)$ in the interval $\left[-\frac{1}{\sqrt{p}}, \frac{1}{\sqrt{p}}\right]$. Then $\mathbb{E}[x] = 0$ and $||x||_2 \leq 1$. Also, $\Sigma = \mathrm{Var}\left(x^{(1)}\right)I_p \succ 0$. For our truncated Gaussian, we have
\begin{align*}
\lambda_{\min}(\Sigma) = \lambda_{\max}(\Sigma) = \mathrm{Var}\left(x^{(1)}\right) = \frac{1}{p}\left(1 - \frac{2\phi(1)}{\Phi_0(1) - \Phi_0(-1)}\right),
\end{align*}
where $\phi$ and $\Phi_0$ denote the standard Gaussian pdf and cdf, respectively. Hence, we can take $C'_1 = C'_2 = 1 - \frac{2\phi(1)}{\Phi_0(1) - \Phi_0(-1)}$. For the bounded $4^{\text{th}}$ moments, take $||v||_2 = 1$, with $v = (v_1, \dots, v_p)$, arbitrary. Then
\begin{align*}
\mathbb{E}\left[(x^Tv)^2\right]^2 = \left(v^T\Sigma v\right)^2 = \mathrm{Var}\left(x^{(1)}\right)^2 = \frac{(C'_1)^2}{p^2}.
\end{align*}
Also, by independence, the fact that the coordinates of $x$ have mean 0, and the truncation in $\left[-\frac{1}{\sqrt{p}}, \frac{1}{\sqrt{p}}\right]$, we have
\begin{align*}
\mathbb{E}\left[(x^Tv)^4\right] &= \mathbb{E}\left[\sum_{i, j, l, k = 1}^pv_iv_jv_lv_kx^{(i)}x^{(j)}x^{(l)}x^{(k)}\right]\\
&= \sum_{i = 1}^pv_i^4\mathbb{E}\left[\left(x^{(i)}\right)^4\right] + 3\sum_{i \neq j}v_i^2v_j^2\mathbb{E}\left[\left(x^{(i)}\right)^2\right]\mathbb{E}\left[\left(x^{(j)}\right)^2\right]\\
&\leq \frac{1}{p^2}\sum_{i = 1}^pv_i^4 + \frac{3}{p^2}\sum_{i \neq j}v_i^2v_j^2 \leq \frac{3||v||_2^4}{p^2} = \frac{3}{p^2}.
\end{align*}
Hence, we have $\mathbb{E}\left[(x^Tv)^4\right] \leq \widetilde{C}_4\mathbb{E}\left[(x^Tv)^2\right]^2$, for some absolute constant $\widetilde{C}_4 > 0$. So all the conditions are satisfied.
\end{remark}

Note also that Lemma~\ref{lemma:pHu_prelim} establishes the Lipschitz property \emph{globally} over $\mathbb{B}_2(L_x) \times \mathbb{R}$. This is because, when dealing with privacy, we need the Lipschitz property to hold not just for the data drawn from the proposed model.


\section{Proofs for Section \ref{sec:Acc_FW_Method_All}}

In this appendix, we provide the proofs for the results in Section \ref{sec:Acc_FW_Method_All}. In Appendix \ref{sec:Proofs of the Main Results_ACC_FW}, we present the proofs of the main results, while in Appendix \ref{sec:Proofs of the Auxiliary Results_ACC_FW}, we present the proofs of the supporting results.

We begin by providing the general statement of Algorithm \ref{alg:RobPGDNFW}.

\begin{algorithm}
\caption{Robust Gradient Descent}
\label{alg:RobPGDNFW}
\begin{algorithmic}[1]
\Function{RobPGDNFW}{$g(\cdot)$, $\{z_1, \ldots, z_n\}$, $\eta$, $\lambda$, $T$, $\zeta$}
    \State Split samples into $T$ subsets $\{Z_t\}_{t=1}^T$ of size $\widetilde{n}$.
    \For{$t = 0$ to $T - 1$}
    \If{$\mathcal{C} = \mathbb{R}^p$}
        \If{Projected GD}
            \State $\theta_{t+1} =  \theta_t - \eta g(\theta_t; Z_t, \widetilde{\zeta})$.
        \EndIf
        \If{Nesterov}
            \State $\theta_{t+1} =  \theta_t + \lambda(\theta_t - \theta_{t-1}) - \eta g(\theta_t + \lambda(\theta_t - \theta_{t-1}); Z_t, \widetilde{\zeta})$.
        \EndIf
        \EndIf
    \If{$\mathcal{C}$ is compact and convex in $\mathbb{R}^p$}
    \If{Projected GD}
            \State $\theta_{t+1} = \mathop{\arg \min\limits_{\theta \in \mathcal{C}}} \| \theta - \left( \theta_t - \eta g(\theta_t; Z_t, \widetilde{\zeta}) \right) \|_2^2$.
        \EndIf 
        \If{Frank-Wolfe}
        \State $v_t = \mathop{\arg \min}\limits_{v \in \mathcal{C}} g(\theta_t; Z_t, \widetilde{\zeta})^Tv$
        \State $\theta_{t + 1} = (1 - \eta)\theta_t + \eta v_t$
        \EndIf
    
    \EndIf
    
    \EndFor
\EndFunction
\end{algorithmic}
\end{algorithm}


\subsection{Proofs of Main Results from Section \ref{sec:Acc_FW_Method_All}}\label{sec:Proofs of the Main Results_ACC_FW}

Here, we present the proofs of the main results from Section \ref{sec:Acc_FW_Method_All}.


\subsubsection{Proof of Theorem \ref{theorem:UPRidge}}
\label{AppThmUPR}

The aim is to to apply Theorem \ref{theorem:ACCFWV3}. We want to bring Algorithm \ref{alg:PrivFWERM} in the form of Algorithm \ref{alg:ReAccFW}. For this, we need to verify the smoothness of the empirical loss, and we also need a lower bound on the $\ell_2$-norm of the gradient of the empirical risk. To ensure privacy, we need the Lipschitz property. Additionally, we need the strong convexity parameter of $\mathcal{C}$.

Note that the $\ell_2$-ball of radius $D$ is strongly convex with parameter $\frac{1}{D}$, by Lemma \ref{lemma:STR_CONV_BALL}, justifying our choice for $\alpha_{\mathcal{C}}$. For the Lipschitz property, we have for all $(x, y) \in \mathcal{E}$ and $\theta \in \mathcal{C}$ that
\begin{align*}
||yx - xx^T\theta||_2 \leq ||yx||_2 + ||xx^T||_2||\theta||_2 \leq \sqrt{p} + ||x||_2^2D\leq \sqrt{p} + pD,
\end{align*}
justifying our choice for $L_2 \leq \sqrt{p} + pD$.

Now consider a dataset $\mathcal{D}_n = \{(x_i, y_i)\}_{i = 1}^n$ as in the theorem hypothesis. The Hessian is $\frac{1}{n}\sum_{i = 1}^nx_ix_i^T$, justifying the choice of smoothness parameter $\beta_{\mathcal{L}} = \frac{1}{n}||\sum_{i = 1}^n x_ix_i^T||_2$. Regarding the lower bound on the $\ell_2$-norm of the gradient, the bound $\mathop{\inf}\limits_{\theta \in \mathcal{C}}\frac{\alpha_{\mathcal{C}}||\nabla \mathcal{L}(\theta, \mathcal{D}_n)||_2}{\beta_{\mathcal{L}}} \geq S_1$ immediately implies $||\nabla \mathcal{L}(\theta, \mathcal{D}_n)||_2 \geq \frac{S_1\beta_{\mathcal{L}}}{\alpha_{\mathcal{C}}} = r$, for all $\theta \in \mathcal{C}$.
Also note that by the assumption $\frac{\alpha_{\mathcal{C}} r}{\beta_{\mathcal{L}}} = S_1 \asymp 1$, we have
$\eta = \Theta\left(1\right)$. We have at step $t$ of Algorithm \ref{alg:PrivFWERM} that 
\begin{align*}
v_t^T(\nabla \mathcal{L}(\theta_t, \mathcal{D}_n) + \xi_t) \leq v^T(\nabla \mathcal{L}(\theta_t, \mathcal{D}_n) + \xi_t), \quad \forall v \in \mathcal{C},
\end{align*}
implying that
\begin{align*}
v_t^T\nabla \mathcal{L}(\theta_t, \mathcal{D}_n) \leq v^T\nabla \mathcal{L}(\theta_t, \mathcal{D}_n) + (v - v_t)^T\xi_t, \quad \forall v \in \mathcal{C},
\end{align*}
and
\begin{align*}
v_t^T\nabla \mathcal{L}(\theta_t, \mathcal{D}_n) \leq v^T\nabla \mathcal{L}(\theta_t, \mathcal{D}_n) + ||\mathcal{C}||_2||\xi_t||_2, \quad \forall v \in \mathcal{C}.
\end{align*}
Thus, by Lemma \ref{lemma:CIGV}, for $\zeta \in (0, 1)$ arbitrary and for the event
\begin{align*}
\Omega = \left\{||\xi_t||_2 \geq \sqrt{8\left(\frac{8L_2}{n}\right)^2\frac{T}{\epsilon^2} \log\left(\frac{5T}{2\delta}\right)\log\left(\frac{2}{\delta}\right)\log\left(\frac{4^pT}{\zeta}\right)}, \quad  \forall t \in [T] \right\},
\end{align*}
we have $\mathbb{P}(\Omega) \geq 1 - \zeta$. Note that we also took the variance of the Gaussian noise in Algorithm \ref{alg:PrivFWERM} into account.
Hence, on $\Omega$, we have
\begin{align*}
v_t^T\nabla \mathcal{L}(\theta_t, \mathcal{D}_n) \leq v^T\nabla \mathcal{L}(\theta_t, \mathcal{D}_n) + ||\mathcal{C}||_2\sqrt{8\left(\frac{8L_2}{n}\right)^2\frac{T}{\epsilon^2} \log\left(\frac{5T}{2\delta}\right)\log\left(\frac{2}{\delta}\right)\log(4^pT/\zeta)},
\end{align*}
for all $v \in \mathcal{C}$, implying that
\begin{align*}
v_t^T\nabla \mathcal{L}(\theta_t, \mathcal{D}_n)
& \le \mathop{\min}\limits_{v \in \mathcal{C}}v^T\nabla \mathcal{L}(\theta_t, \mathcal{D}_n) + O\left(\frac{L_2||\mathcal{C}||_2\log(T/\delta)\sqrt{T\log(4^pT/ \zeta)}}{n \epsilon}\right).
\end{align*}
Thus, on $\Omega$, we may apply Theorem \ref{theorem:ACCFWV3} with $\Delta = O\left(\frac{L_2||\mathcal{C}||_2\log(T/\delta)\sqrt{T \log(4^pT/ \zeta)}}{n \epsilon}\right)$. Note also, using the same notation as in the proof of Theorem \ref{theorem:ACCFWV3}, that
$$h_0 = \mathcal{L}(\theta_0, \mathcal{D}_n) - \mathop{\min}\limits_{\theta \in \mathcal{C}}\mathcal{L}(\theta, \mathcal{D}_n) \leq L_2||\mathcal{C}||_2.$$
Recall that $\eta = \Theta\left(1\right)$, and similarly, we have $c = \max\left\{\frac{1}{2}, 1 - \frac{\alpha_{\mathcal{C}} r}{8\beta_{\mathcal{L}}}\right\} = \Theta\left(1\right)$. Therefore, with probability at least $1 - \zeta$, noting that $T = \log_{1/c}\left(n\right) \asymp \log(n)$ and  $\log(4^pT/\zeta) \lesssim p\log(T/\zeta)$, Theorem \ref{theorem:ACCFWV3} implies that
\begin{align}
\label{eq:HP_ERM_UP_NO_DISTR_GENERAL}
\mathcal{L}(\theta_T, \mathcal{D}_n) - \mathop{\min}\limits_{\theta \in \mathcal{C}}\mathcal{L}(\theta, \mathcal{D}_n) &\leq h_0c^T + \frac{3\Delta\eta}{2(1 - c)} \notag\\ 
&\lesssim L_2||\mathcal{C}||_2c^T + \frac{L_2||\mathcal{C}||_2\log(T/\delta)\sqrt{pT\log(T/ \zeta)}}{n \epsilon} \notag\\
&\lesssim \frac{L_2||\mathcal{C}||_2}{n} + \frac{L_2||\mathcal{C}||_2\log(\log(n)/\delta)\sqrt{p\log(n)\log(\log(n)/\zeta)}}{n\epsilon}.
\end{align}
Since $0 < \epsilon \lesssim 1$ and inequality~\eqref{eq:HP_ERM_UP_NO_DISTR_GENERAL} holds for $n$ large enough independent of $\zeta$, applying Lemma \ref{lemma:HP_TO_EXP} implies that
\begin{align*}
\mathbb{E}\left[\mathcal{L}(\theta_T, \mathcal{D}_n) - \mathop{\min}\limits_{\theta \in \mathcal{C}}\mathcal{L}(\theta, \mathcal{D}_n)\right] &\lesssim \frac{L_2||\mathcal{C}||_2}{n} + \frac{L_2||\mathcal{C}||_2\log(\log(n)/\delta)\log(n)\sqrt{p\log(n)}}{n\epsilon}\\
&\lesssim \frac{L_2||\mathcal{C}||_2\sqrt{p}\log^{3/2}(n)\log(\log(n)/\delta)}{n\epsilon}\\
&\lesssim \frac{(\sqrt{p} + p||\mathcal{C}||_2)||\mathcal{C}||_2\sqrt{p}\log^{3/2}(n)\log(\log(n)/\delta)}{n\epsilon},
\end{align*}
as required.

Finally, note that since $c \asymp 1$, we have $T \asymp \log(n)$. Thus, the conditions of Lemma~\ref{lemma:ACC_PRIV_FW_STEP} are satisfied, so $\theta_T$ is $(\epsilon, \delta)$-DP.

\subsubsection{Proof of Theorem \ref{theorem:LBRidge}}
\label{AppThmLBRidge}

We use a modification of an argument by \cite{NOPL} based on fingerprinting codes (see also Chapter 5 of Vadhan~\cite{Vad17}). We begin by constructing a collection of datasets, at least one of which will lead to the desired lower bound.

First consider a matrix $Z \in \mathbb{R}^{k \times p}$ where the columns are mutually orthogonal vectors with entries in $\{-1, 1\}$, so that $Z^TZ = kI_p$ (note that this is possible because $k \gg p$). Denote the $i^{\text{th}}$ row of $Z$ by $z_i$.

We will also use the following construction and its corresponding DP guarantee:

\begin{lemma}[\cite{NOPL}]
\label{lemma:FingX}
Let $m$ be a sufficiently large integer, let $p = 1000m^2$, and let $w = \frac{m}{\log(m)}$. There exists a matrix $X \in \{-1, 1\}^{(w + 1) \times p}$ with the following property: For each $i \in [1, w + 1]$, there are at least $0.999p$ consensus columns $W_i$
in each $X_{(-i)}$. In addition, for algorithm $\hat{\theta}$ on input matrix $X_{(-i)}$ where $i \in [1, w + 1]$, if with probability at least $2/3$, $\hat{\theta}(X_{(-i)})$ produces a $p$-dimensional sign
vector which agrees with at least $3p/4$ columns in $W_i$, then $\hat{\theta}$ is not $(\epsilon, \delta)$-DP with respect to a single row change (to some other row in $X$). 
\end{lemma}

Next, we construct $w + 1$ datasets $D^{(i)}$ for $i \in [w + 1]$ as follows: Each dataset contains the rows of $Z$ with the corresponding response value being $0$, i.e., each dataset contains $(z_j, 0)$ for $j \in [k]$. Taking the matrix $X$ from Lemma~\ref{lemma:FingX}, further include the rows of $X_{(-i)}$ with response values equal to $1$, i.e., if $x_{(-i)}^{j}$ is the $j^{th}$ row of $X_{(-i)}$, take $D^{(i)}$ to contain $(x_{(-i)}^{j}, 1)$ for $j \in [w]$. Note that $n = w + k$.

For simplicity, suppose $\mathcal{L}$ is un-normalized by $2n$. This does not affect the analysis, and in the end, we will normalize back by dividing by $2n$. We now have for all $i \in [w + 1]$ and $\theta \in \mathcal{C}$ that
\begin{align*}
\mathcal{L}\left(\theta, D^{(i)}\right) = \sum_{j = 1}^w \left(1 - \theta^Tx_{(-i)}^{j}\right)^2 + \sum_{j = 1}^k (z_j^T\theta)^2 = \sum_{j = 1}^w \left(1 - \theta^Tx_{(-i)}^{j}\right)^2 + k||\theta||_2^2,
\end{align*}
since $Z^TZ = kI_p$ and all entries in $Z$ are in $\{-1, 1\}$. Now set $\theta' \in \left\{-\frac{\alpha_2}{p}, \frac{\alpha_2}{p}\right\}^p$ such that the signs of the coordinates of $\theta'$ match the signs for the
consensus columns of $X_{(-i)}$. Plugging this into $\mathcal{L}$, we see for all $i \in [w]$ that
\begin{align*}
\mathcal{L}\left(\theta', D^{(i)}\right) &= \sum_{j = 1}^w \left(1 - \theta'^Tx_{(-i)}^{j}\right)^2 + \alpha_2^2\frac{k}{p} \leq \sum_{j = 1}^w \left(1 - \frac{(1 - \tau) p \alpha_2}{p} + \frac{\tau p \alpha_2}{p}\right)^2 + \alpha_2^2 \tau w\\
&= \left((1 - \alpha_2 + 2\tau \alpha_2)^2 + \alpha_2^2 \tau \right)w,
\end{align*}
where $\tau = 0.001$, and in the inequality step, we used the fact that  the number of non-consensus columns is at most $\tau p$. Thus, we have
\begin{align*}
\mathop{\min}\limits_{\theta \in \mathcal{C}}\mathcal{L}\left(\theta, D^{(i)}\right) \leq \left((1 - \alpha_2 + 2\tau \alpha_2)^2 + \alpha_2^2 \tau \right)w.
\end{align*}

Now we state and prove a lemma that will allow us to conclude that for a $\theta \in \mathcal{C}$, its sign has to agree with the sign of most of the consensus columns of $X_{(-i)}$. Its proof is again essentially the same as in \cite{NOPL}, except for the introduction of the quantity $\alpha_2$.

\begin{lemma}[Adapted from \cite{NOPL}]
\label{lemma:ERM_Dist_Free_LB_Helper}
Let $\mathcal{L}$ the mean squared error loss. Fix $i \in [w]$ and $\theta \in \mathbb{R}^p$. Suppose $\mathcal{L}\left(\theta, D^{(i)}\right) < 1.1 \tau \alpha_2^2 w$. For $j \in W_i$, let $s_j$ be the consensus sign of column $j$. Then
\begin{align*}
\left|\left\{j \in W_i \ | \ sgn(\theta_j) = s_j\right\}\right| \geq \frac{3p}{4}.
\end{align*}
\end{lemma}
\begin{proof}
For notational purposes, for $S \subseteq [p]$, let $\theta |_{S}$ be the projection onto the coordinates in $S$. Now let
\begin{align*}
&S_1 = \left\{j \in W_i \ | \ sgn(\theta_j) = s_j\right\},\\
&S_2 = \left\{j \in W_i \ | \ sgn(\theta_j) \neq s_j\right\},\\ 
&S_3 = [p] \setminus W_i.
\end{align*}
Also, for $j \in [3]$, set $\theta^{(j)} := \theta |_{S_j}$. Suppose for the sake of contradiction that $|S_1| < \frac{3p}{4}$. Thus, since $|S_3| \leq \tau p$, we have by Cauchy-Schwarz that
\begin{align*}
||\theta^{(3)}||_2^2 \geq \frac{||\theta^{(3)}||_1^2}{|S_3|} \geq \frac{||\theta^{(3)}||_1^2}{\tau p}.
\end{align*}
Hence, $k||\theta^{(3)}||_2^2 \geq w||\theta^{(3)}||_1^2$. However, $k||\theta^{(3)}||_2^2 \leq k||\theta||_2^2 < 1.1 \tau \alpha_2^2 w$. This is because $\mathcal{L}\left(\theta, D^{(i)}\right) = \sum_{j = 1}^w \left(1 - \theta^Tx_{(-i)}^{j}\right)^2 + k||\theta||_2^2$ and $\mathcal{L}\left(\theta, D^{(i)}\right) < 1.1 \tau \alpha_2^2 w$. Thus, $||\theta^{(3)}||_1 \leq \alpha_2 \sqrt{1.1 \tau} \leq 0.04\alpha_2$. Also, since $|S_1| < \frac{3p}{4}$, we have
\begin{align*}
||\theta^{(1)}||_2^2 \geq \frac{||\theta^{(1)}||_1^2}{|S_1|} \geq \frac{4||\theta^{(1)}||_1^2}{3p}.
\end{align*}
But again, since $k||\theta||_2^2 < 1.1 \tau \alpha_2^2 w$, we have $||\theta^{(1)}||_1 \leq \alpha_2\sqrt{1.1\cdot 3/4} \leq 0.91\alpha_2$. We now have for $j \in [w]$ that $1 - \theta^Tx_{(-i)}^j = 1 -||\theta^{(1)}||_1 + ||\theta^{(2)}||_1 - \beta_j$, with $|\beta_j| \leq ||\theta^{(3)}||_1 \leq 0.04\alpha_2$. Since $0 < \alpha_2 < 1$, we obtain
\begin{align*}
\left|\theta^Tx_{(-i)}^j - 1\right| &= 1 - \theta^Tx_{(-i)}^j = 1 - ||\theta^{(1)}||_1 + ||\theta^{(2)}||_1 - \beta_j \geq 1 - ||\theta^{(1)}||_1 + ||\theta^{(2)}||_1 - |\beta_j|\\
&\geq 1 - ||\theta^{(1)}||_1 - ||\theta^{(3)}||_1 \geq 1 - \alpha_2(0.04 + 0.91) = 1 - 0.95\alpha_2.
\end{align*}
Since $\alpha_2 \in (0, 1)$, we have $(1 - 0.95\alpha_2)^2 \geq 1.1\alpha_2^2\tau$, so $\mathcal{L}\left(\theta, D^{(i)}\right) \geq (1 - 0.95\alpha_2)^2w \geq 1.1\tau\alpha_2^2w$. Therefore, we have a contradiction, implying that $|S_1| \geq \frac{3p}{4}$. This completes the proof of the lemma.
\end{proof}

Let us now continue with the proof of our theorem. We have that $\hat{\theta}$ is $(\epsilon, \delta)$-DP. Assume, for a constant $c$ small enough that will be determined later, that for all $i \in [w]$, we have
\begin{align*}
\mathbb{E}\left[\mathcal{L}\left(\hat{\theta}(D^{(i)}), D^{(i)}\right) - \mathop{\min}\limits_{\theta \in \mathcal{C}}\mathcal{L}\left(\theta, D^{(i)}\right)\right] \leq cw.
\end{align*}
By Markov's inequality we have with probability at least $2/3$ that
$$\mathcal{L}\left(\hat{\theta}(D^{(i)}), D^{(i)}\right) - \mathop{\min}\limits_{\theta \in \mathcal{C}}\mathcal{L}\left(\theta, D^{(i)}\right) \leq 3cw.
$$
But from before, we had $\mathop{\min}\limits_{\theta \in \mathcal{C}}\mathcal{L}\left(\theta, D^{(i)}\right) \leq \left((1 - \alpha_2 + 2\tau \alpha_2)^2 + \alpha_2^2 \tau \right)w$. Also, the function $1.1\tau x^2 - \left(\tau x^2 + (1 - x + 2\tau x)^2\right)$ is positive between the solution of the equation
\begin{align*}
1.1\tau x^2 = \tau x^2 + (1 - x + 2\tau x)^2
\end{align*}
in $x \in (0, 1)$, which is roughly $0.992063$, and $1$. Since $\alpha_2 \in (0.993, 1)$, the function $1.1\tau x^2 - \left(\tau x^2 + (1 - x + 2\tau x)^2\right)$ is positive at $x = \alpha_2$. Hence, for $c$ small enough, with probability at least $2/3$, we have
\begin{align*}
\mathcal{L}\left(\hat{\theta}(D^{(i)}), D^{(i)}\right) < \left((1 - \alpha_2 + 2\tau \alpha_2)^2 + \alpha_2^2 \tau + 3c \right)w \leq 1.1\tau \alpha_2^2w.
\end{align*}
Since $\hat{\theta}(D^{(i)}) \in \mathcal{C}$, we have by Lemma \ref{lemma:ERM_Dist_Free_LB_Helper} that $\hat{\theta}(D^{(i)})$ agrees with at least $\frac{3p}{4}$ consensus columns in $X_{(i)}$. This holds for all $i \in [w]$. But by Lemma \ref{lemma:FingX}, this contradicts the privacy of $\hat{\theta}$. Thus, there exists $i \in [w]$ such that 
\begin{align*}
\mathbb{E}\left[\mathcal{L}\left(\hat{\theta}(D^{(i)}), D^{(i)}\right) - \mathop{\min}\limits_{\theta \in \mathcal{C}}\mathcal{L}\left(\theta, D^{(i)}\right)\right] > cw.
\end{align*}
Hence, since $w = \frac{m}{\log(m)}$, $p = 1000m^2$, and $n = w + k \asymp \frac{m^3}{\log(m)}$, we obtain
\begin{align*}
\mathbb{E}\left[\mathcal{L}\left(\hat{\theta}(D^{(i)}), D^{(i)}\right) - \mathop{\min}\limits_{\theta \in \mathcal{C}}\mathcal{L}\left(\theta, D^{(i)}\right)\right] = \Omega\left(\frac{n^{1/3}}{\log^{2/3}(n)}\right).
\end{align*}
Normalizing back, i.e., dividing by $2n$, we obtain
\begin{align*}
\mathbb{E}\left[\mathcal{L}\left(\hat{\theta}(D^{(i)}), D^{(i)}\right) - \mathop{\min}\limits_{\theta \in \mathcal{C}}\mathcal{L}\left(\theta, D^{(i)}\right)\right] = \widetilde{\Omega}\left(\frac{1}{n^{2/3}}\right),
\end{align*}
as required.


\subsubsection{Proof of Theorem \ref{theorem:ERM_GLM_UP_General}}
\label{AppThmERM}

Let $z_i = (x_i, y_i)$ for all $i \in [n]$. The goal is to apply Theorem \ref{theorem:ACCFWV3}. We need to establish the Lipschitz condition, smoothness, and the lower bound on the $\ell_2$-norm of the gradient of $\mathcal{L}(\theta, \mathcal{D}_n)$. The Lipschitz and smoothness properties will be established on the whole of $\mathcal{E}$. For the lower bound on the gradient, we will first obtain a lower bound on $\|\mathbb{E}[\nabla \mathcal{L}(\theta, z_i)]\|_2$ and then use a concentration result of $\nabla \mathcal{L}(\theta, \mathcal{D}_n)$ around $\mathbb{E}[\nabla \mathcal{L}(\theta, z_i)]$.

For any pair $z = (x, y) \in \mathcal{E}$, not necessarily from the GLM, we have
\begin{align*}
\nabla \mathcal{L}(\theta, z) & = (\Phi'(x^T\theta) - y)x,\\
||\nabla \mathcal{L}(\theta, z)||_2 & \leq (K_{\Phi'} + K_y)L_x,
\end{align*}
so $\mathcal{L}(\theta, z)$ is $(K_{\Phi'} + K_y)L_x$-Lipschitz in $\theta$. Furthermore, we have
\begin{align*}
\nabla^2 \mathcal{L}(\theta, z) = \Phi''(x^T\theta)xx^T,
\end{align*}
so for any $h \in \mathbb{R}^p$, we obtain
\begin{align*}
h^T\nabla^2 \mathcal{L}(\theta, z)h = \Phi''(x^T\theta)(h^Tx)^2 \leq K_{\Phi''}L_x^2||h||_2^2.
\end{align*} 
Thus, for any $z \in \mathcal{E}$, the loss $\mathcal{L}(\theta, z)$ is $K_{\Phi''}L_x^2$-smooth over $\mathbb{R}^p$, implying that $\mathcal{L}(\theta, \mathcal{D}_n)$ is $K_{\Phi''}L_x^2$-smooth over $\mathbb{R}^p$, as well.

Let us now proceed to lower-bound $||\mathbb{E}[\nabla \mathcal{L}(\theta, z)]||_2$. For $\mathcal{R}(\theta) := \mathbb{E}[\mathcal{L}(\theta, z)]$, we have by classical GLM theory that
\begin{align*}
\mathbb{E}[y|x] = \Phi^{'}(x^T\theta^*),
\end{align*}
so $\mathbb{E}[yx] = \mathbb{E}\left[\mathbb{E}[y|x]x\right] = \mathbb{E}[\Phi^{'}(x^T\theta^*)x]$. Thus, we have
\begin{align*}
\mathcal{R}(\theta) = -\theta^T\mathbb{E}[\Phi^{'}(x^T\theta^*)x] + \mathbb{E}[\Phi(x^T\theta)] = \mathbb{E}_{x}[\Phi(x^T\theta) - \Phi^{'}(x^T\theta^*)x^T\theta].
\end{align*}
Since the quantities inside the expectation are bounded, using the Dominated Convergence Theorem, we can swap expectations and gradients. Therefore, we have
\begin{align*}
\nabla \mathcal{R}(\theta) = \mathbb{E}_{x}[(\Phi^{'}(x^T\theta) - \Phi^{'}(x^T\theta^*))x].
\end{align*}
Thus, for $h \in \mathbb{R}^p$, we have
\begin{align*}
h^T\nabla^2\mathcal{R}(\theta)h = \mathbb{E}_x[\Phi''(x^T\theta)(h^Tx)^2].
\end{align*}
Since $x^T\theta \leq L_x||\theta^*||_2$ for all $\theta \in \mathbb{B}_2\left(||\theta^*||_2\right)$, and since $\Phi''$ is even and non-decreasing on $(-\infty, 0]$ and non-increasing on $[0, \infty)$, we have $\Phi''(x^T\theta) \geq \Phi''\left(L_x||\theta^*||_2\right) > 0$, for all $\theta \in \mathbb{B}_2(||\theta^*||_2)$. Therefore, we have
\begin{align*}
h^T\nabla^2\mathcal{R}(\theta)h &\geq \Phi''\left(L_x||\theta^*||_2\right)h^T\mathbb{E}[xx^T]h = \Phi''(L_x||\theta^*||_2)h^T\Sigma h\\
&\geq \Phi''\left(L_x||\theta^*||_2\right)\lambda_{\min}(\Sigma)||h||_2^2 > 0.
\end{align*}
Hence, $\mathcal{R}(\theta)$ is $\Phi''\left(L_x||\theta^*||_2\right)\lambda_{\min}(\Sigma)$-strongly convex over $\mathbb{B}_2\left(||\theta^*||_2\right)$. Also, since $\Phi$ is convex over $\mathbb{R}$ and $\nabla \mathcal{R}(\theta^{*}) = 0$, the function $\mathcal{R}$ is minimized over $\mathbb{B}_2\left(||\theta^*||_2\right)$ at $\theta^{*}$. Hence, for all $\theta \in \mathbb{B}_2\left(||\theta^*||_2\right)$, and thus for all $\theta \in \mathcal{C}$ since $\mathcal{C} \subseteq \mathbb{B}_2\left(||\theta^*||_2\right)$, we have by strong convexity that
\begin{align*}
||\mathbb{E}[\nabla \mathcal{L}(\theta, z)]||_2 &= ||\nabla\mathcal{R}(\theta)||_2 = ||\nabla\mathcal{R}(\theta) - \nabla\mathcal{R}(\theta^*)||_2\\
&\geq \frac{\Phi''\left(L_x||\theta^*||_2\right)\lambda_{\min}(\Sigma)}{2}||\theta - \theta^{*}||_2\\ &\geq \frac{\Phi''\left(L_x||\theta^*||_2\right)\lambda_{\min}(\Sigma)}{2}\left(||\theta^*||_2 - D\right) > 0,
\end{align*}
since $\theta^* \in \mathbb{R}^p \setminus \mathcal{C}$, so there is a strict separation between $\theta^*$ and $\mathcal{C}$.

Now, for all $i \in [n]$ and $\theta \in \mathbb{R}^p$, recall that $\nabla \mathcal{L}(\theta, z_i) = (\Phi'(x_i^T\theta) - y_i)x_i$. Also, for $h \in \mathbb{R}^p$, we have $|(\Phi'(x_i^T\theta) - y_i)x_i^Th| \leq (K_{\Phi'} + K_y)L_x||h||_2$, so clearly,
\begin{align*}
& (\Phi'(x_i^T\theta) - y_i)x_i - \mathbb{E}[(\Phi'(x_i^T\theta) - y_i)x_i] \in \mathcal{G}\left((K_{\Phi'} + K_y)^2L_x^2\right), \quad \text{and} \\
& \frac{1}{n}\sum_{i = 1}^n\nabla \mathcal{L}(\theta, z_i) - \mathbb{E}[\nabla \mathcal{L}(\theta, z_1)] = \nabla \mathcal{L}(\theta, \mathcal{D}_n) - \mathbb{E}[\nabla \mathcal{L}(\theta, z_1)] \in \mathcal{G}\left(\frac{(K_{\Phi'} + K_y)^2L_x^2}{n}\right).
\end{align*}
Hence, by Lemma \ref{lemma:CIGV}, we have
\begin{align}
\label{eq:Conc_Grad_GLM_Eq}
\mathbb{P}\left(\left\|\frac{1}{n}\sum_{i = 1}^n\nabla \mathcal{L}(\theta, z_i) - \mathbb{E}[\nabla \mathcal{L}(\theta, z_1)]\right\|_2 \geq t \right) \leq 4^pe^{-\frac{t^2}{8s_n^2}}, \quad \forall \ t \geq 0 \ \mbox{and} \ \theta \in \mathbb{R}^p,
\end{align}
with $s_n^2 = \frac{(K_{\Phi'} + K_y)^2L_x^2}{n}$. Now take $\theta \in \mathcal{C}$ and let $Z_\theta = \left\|\frac{1}{n}\sum_{i = 1}^n\nabla \mathcal{L}(\theta, z_i) - \mathbb{E}[\nabla \mathcal{L}(\theta, z_1)]\right\|_2$. Note that, since $\mathcal{L}(\theta, z)$ is $K_{\Phi''}L_x^2$-smooth over $\mathbb{R}^p$ for all $z \in \mathcal{E}$, the function $Z_\theta$ is $2K_{\Phi''}L_x^2$-Lipschitz over $\mathbb{R}^p$. Now we use a covering argument to obtain a concentration result on $\mathop{\sup}\limits_{\theta \in \mathcal{C}}Z_\theta$. Let $t \geq 0$ and $\upsilon = \frac{t}{4K_{\Phi''}L_x^2}$. Take a $\upsilon$-cover $\{\theta^{1}, \dots, \theta^{N_\upsilon}\}$ of $\mathcal{C} = \mathbb{B}_2(D)$ with covering number $N_\upsilon := N(\upsilon, \mathcal{C}, ||\cdot||_2)$. Then, for $\theta \in \mathcal{C}$, there is some $k \in [N_\upsilon]$ such that $||\theta - \theta^k||_2 \leq \upsilon$, and by the Lipschitz property, we have  $|Z_\theta - Z_{\theta^k}| \leq 2K_{\Phi''}L_x^2\upsilon$. So, if $Z_\theta \geq t$, we have $Z_{\theta^k} \geq t - 2K_{\Phi''}L_x^2\upsilon = \frac{t}{2}$, since $Z_\theta, Z_{\theta^k} \geq 0$. Hence, by Lemma \ref{lemma:CIGV} and Lemma \ref{lemma:Covering_Nr_Ball}, we have
\begin{align*}
\mathbb{P}\left(\mathop{\sup}\limits_{\theta \in \mathcal{C}}Z_\theta \geq t\right) &\leq \mathbb{P}\left(\mathop{\sup}\limits_{k \in [N_\upsilon]}Z_{\theta^k} \geq \frac{t}{2}\right)\leq \sum_{k = 1}^{N_{\upsilon}}\mathbb{P}\left(Z_{\theta^k} \geq \frac{t}{2}\right) \leq \left(1 + \frac{2D}{\upsilon}\right)^p4^pe^{-\frac{t^2}{32s_n^2}}\\
&= \left(1 + \frac{8K_{\Phi''}L_x^2D}{t}\right)^p4^pe^{-\frac{t^2}{32s_n^2}} \leq \left(\frac{16K_{\Phi''}L_x^2D}{t}\right)^p4^pe^{-\frac{t^2}{32s_n^2}},
\end{align*}
for $t \leq 8K_{\Phi''}L_x^2D$. Since $D \leq ||\theta^*||_2 \asymp 1$, we have absolute constants $C_2$ and $C_3$ such that
\begin{align*}
\mathbb{P}\left(\mathop{\sup}\limits_{\theta \in \mathcal{C}}Z_\theta \geq t\right)\leq \frac{C_2}{t^p}4^pe^{-\frac{nt^2}{C_3}},
\end{align*}
and by rescaling $t$ with $4C_2^{1/p}t$, since $p$ is of constant order, we have 
\begin{align*}
\mathbb{P}\left(\mathop{\sup}\limits_{\theta \in \mathcal{C}}Z_\theta \geq t\right)\leq \frac{1}{t^p}e^{-\frac{nt^2}{C_4}},
\end{align*}
for $t \leq C_5$, with absolute constants $C_4$ and $C_5$. Fix $\zeta \in (0, 1)$. Thus, we want $t \leq C_5$ and $\frac{1}{t^p}e^{-\frac{nt^2}{C_4}} \leq \frac{\zeta}{2}$, or equivalently, $t^2 + \frac{pC_4}{n}\log(t) \geq \frac{C_4}{n}\log(2/\zeta)$. Pick $t = \sqrt{\frac{C_4\log(2/\zeta)}{n}} + \frac{1}{n^q}$. Then
\begin{align*}
\frac{1}{n^{2q}} + \frac{2}{n^q}\sqrt{\frac{C_4\log(2/\zeta)}{n}} + \frac{pC_4}{n}\log\left(\sqrt{\frac{C_4\log(2/\zeta)}{n}} + \frac{1}{n^q}\right) \geq 0,
\end{align*}
since if we pick $n$ greater than an absolute constant, the LHS scales like $\frac{1}{n^{2q}} + \frac{1}{n^{q + \frac{1}{2}}} - \frac{C''_1\log(n)}{n}$, which is greater than $0$, since $q < \frac{1}{2}$ and $C''_1$ is some absolute constant. Thus, there is an absolute constant $C'_1$ such that for $n \geq C'_1$,
the required conditions are satisfied and we have $\mathbb{P}(\Omega_1) \geq 1 - \frac{\zeta}{2}$, with
\begin{align*}
\Omega_1 = \left\{\left\|\frac{1}{n}\sum_{i = 1}^n\nabla \mathcal{L}(\theta, z_i) - \mathbb{E}[\nabla \mathcal{L}(\theta, z_1)]\right\|_2 \leq \sqrt{\frac{C_1\log(2/\zeta)}{n}} + \frac{1}{n^q}, \quad \forall \theta \in \mathcal{C}\right\}.
\end{align*}
Shifting our attention to Algorithm \ref{alg:PrivFWERM}, we have at step $t$ that 
\begin{align*}
v_t^T(\nabla \mathcal{L}(\theta_t, \mathcal{D}_n) + \xi_t) \leq v^T(\nabla \mathcal{L}(\theta_t, \mathcal{D}_n) + \xi_t), \quad \forall v \in \mathcal{C},
\end{align*}
so 
\begin{align*}
v_t^T\nabla \mathcal{L}(\theta_t, \mathcal{D}_n) &\leq v^T\nabla \mathcal{L}(\theta_t, \mathcal{D}_n) + (v - v_t)^T\xi_t\\
&\leq v^T\nabla \mathcal{L}(\theta_t, \mathcal{D}_n) + 2D||\xi_t||_2, \quad \forall v \in \mathcal{C}.
\end{align*}
By Lemma \ref{lemma:CIGV}, for
\begin{align*}
\Omega_2 = \left\{||\xi_t||_2 \leq \sqrt{8\left(\frac{8L_2}{n}\right)^2\frac{T}{\epsilon^2} \log\left(\frac{5T}{2\delta}\right)\log\left(\frac{2}{\delta}\right)\log(4^pT/\zeta)},  \quad \forall t \in [T] \right\},
\end{align*}
we have $\mathbb{P}(\Omega_2) \geq 1 - \frac{\zeta}{2}$. Let us work on $\Omega = \Omega_1 \cap \Omega_2$, with $\mathbb{P}(\Omega) \geq 1 - \zeta$. On $\Omega$, we have
\begin{align*}
||\nabla \mathcal{L}(\theta, \mathcal{D}_n)||_2 \geq \frac{\Phi''(L_x||\theta^*||_2)\lambda_{\min}(\Sigma)}{2}(||\theta^*||_2 - D) - \sqrt{\frac{C_1\log(2/\zeta)}{n}} - \frac{1}{n^q} \geq r > 0
\end{align*}
by the triangle inequality, so we have the desired lower bound on $||\nabla \mathcal{L}(\theta, \mathcal{D}_n)||_2$ for all $\theta \in \mathcal{C}$, with high probability. Next, note that 
\begin{align*}
v_t^T\nabla \mathcal{L}(\theta_t, \mathcal{D}_n) &\leq v^T\nabla \mathcal{L}(\theta_t, \mathcal{D}_n)\\
& \quad + 2D\sqrt{8\left(\frac{8L_2}{n}\right)^2\frac{T}{\epsilon^2} \log\left(\frac{5T}{2\delta}\right)\log\left(\frac{2}{\delta}\right)\log\left(\frac{4^pT}{\zeta}\right)}, \quad \forall v \in \mathcal{C},
\end{align*}
where we used the fact that $2D = \mathop{\sup}\limits_{x, y \in \mathcal{C}}||x - y||_2$. Thus, we have
\begin{align*}
v_t^T\nabla \mathcal{L}(\theta_t, \mathcal{D}_n) &\leq \mathop{\min}\limits_{v \in \mathcal{C}}v^T\nabla \mathcal{L}(\theta_t, \mathcal{D}_n)\\
& \quad + 2D\sqrt{8\left(\frac{8L_2}{n}\right)^2\frac{T}{\epsilon^2} \log\left(\frac{5T}{2\delta}\right)\log\left(\frac{2}{\delta}\right)\log\left(\frac{4^pT}{\zeta}\right)}.
\end{align*}
So on $\Omega$, we are in the context of Algorithm \ref{alg:ReAccFW} and Theorem \ref{theorem:ACCFWV3}, with
\begin{align*}
\Delta = 2D\sqrt{8\left(\frac{8L_2}{n}\right)^2\frac{T}{\epsilon^2} \log(5T/2\delta)\log(2/\delta)\log(4^pT/\zeta)}.
\end{align*}
Note also that $\mathcal{C}$ is compact and $\alpha_{\mathcal{C}}$-strongly convex by Lemma \ref{lemma:STR_CONV_BALL}, and that we established the smoothness condition and the lower bound on the $\ell_2$-norm of the gradient of the empirical risk. Therefore, with probability at least $1 - \zeta$, we have for $\eta = \min\left\{1, \frac{\alpha_{\mathcal{C}}r}{4K_{\Phi''}L_x^2}\right\}$ and $c = \max\left\{\frac{1}{2}, 1 - \frac{\alpha_{\mathcal{C}}r}{8K_{\Phi''}L_x^2}\right\}$ that
\begin{align*}
\mathcal{L}(\theta_T, \mathcal{D}_n) - \mathop{\min}\limits_{\theta \in \mathcal{C}}\mathcal{L}(\theta, \mathcal{D}_n) &\leq h_0c^T + \frac{3\Delta\eta}{2(1 - c)}\\
&= h_0c^T\\
& \quad + \frac{3\eta}{(1 - c)}D\sqrt{8\left(\frac{8L_2}{n}\right)^2\frac{T}{\epsilon^2} \log\left(\frac{5T}{2\delta}\right)\log\left(\frac{2}{\delta}\right)\log\left(\frac{4^pT}{\zeta}\right)}.
\end{align*}
Observe that $L_2 = (K_{\Phi'} + K_y)L_x \asymp 1$, $h_0 \leq 2L_2D \asymp 1$ by the Lipschitz property. Hence, we have
\begin{align*}
\mathcal{L}(\theta_T, \mathcal{D}_n) - \mathop{\min}\limits_{\theta \in \mathcal{C}}\mathcal{L}(\theta, \mathcal{D}_n) \lesssim \frac{1}{n} + \frac{\eta\log\left(\log_{1/c}(n)/\delta\right)\sqrt{\log_{1/c}(n)\log\left(\log_{1/c}(n)/\zeta\right)}}{(1 - c)n\epsilon},
\end{align*}
with probability at least $1 - \zeta$, as required.

Note that we needed the $L_2$-Lipschitz condition to hold for all datasets $\{(x_i, y_i)\}_{i = 1}^n$ in $\mathcal{E}$, not just for the data drawn \iid from the GLM. This is because we need $\theta_T$ to be private, and this is the case if the empirical risk is $L_2$-Lipschitz in $\theta$ for \emph{arbitrary} data.


\subsubsection{Proof of Theorem \ref{theorem:ERM_UP_C_Inc}}
\label{AppThmERM2}

We will prove a more general statement. Let $\zeta \in (0, 1/3)$ and $q < \frac{1}{2}$. Assuming $||\theta^*||_2 - D \lesssim \frac{1}{n^q}$, there are absolute constants $C'_1, C_1, C_2$, and $C_3$ such that for $n > \max\left\{C_2\log^{\frac{1}{1 - 2q}}(2/\zeta), C'_1\right\}$, $D \leq ||\theta^*||_2 - \frac{C_3}{n^q}$, and
\begin{align*}
r = \frac{1}{n^q} - \sqrt{\frac{C_1\log(2/\zeta)}{n}},
\end{align*}
we have with probability at least $1 - 3\zeta$ that Algorithm \ref{alg:PrivFWERM} with $T = \log_{1/c}(n)$ returns $\theta_T$ such that
\begin{align}
\label{eq:General_q_Inc_C}
\mathcal{L}(\theta_T, \mathcal{D}_n) - \mathop{\min}\limits_{\theta \in \mathbb{B}_2(||\theta^*||_2)}\mathcal{L}(\theta, \mathcal{D}_n) \lesssim \log(n/\delta)\sqrt{\log(n)\log(n/\zeta)}\left(\frac{1}{{n^{1 - q/2}\epsilon}} + \frac{1}{n^{q + \frac{1}{2}}} + \frac{1}{n^{2q}}\right).
\end{align}
All the other quantities are as in the theorem hypothesis. Once we prove this, we will optimize the upper bound on the excess empirical risk over $q < \frac{1}{2}$ to obtain the desired result.

Let $C_3 = \frac{4}{\Phi''(L_x||\theta^*||_2)\lambda_{\min}(\Sigma)}$. By assumption, we have $||\theta^*||_2 - D \geq \frac{4}{\Phi''(L_x||\theta^*||_2)\lambda_{\min}(\Sigma)n^q}$. By Theorem~\ref{theorem:ERM_GLM_UP_General}, there exist absolute constants $C'_1$ and $C_1$ such that for $n \geq C'_1$, $r = \frac{1}{n^q} - \sqrt{\frac{C_1\log(2/\zeta)}{n}}$, and $T = \log_{1/c}(n)$, Algorithm \ref{alg:PrivFWERM} returns $\theta_T$ such that with probability at least $1 - \zeta$, we have
\begin{align*}
\mathcal{L}(\theta_T, \mathcal{D}_n) - \mathop{\min}\limits_{\theta \in \mathcal{C}}\mathcal{L}(\theta, \mathcal{D}_n) \lesssim \frac{1}{n} + \frac{\eta\log\left(\log_{1/c}(n)/\delta\right)\sqrt{\log_{1/c}(n)\log\left(\log_{1/c}(n)/\zeta\right)}}{(1 - c)n\epsilon}.
\end{align*}
This is because, firstly,
\begin{align*}
r &= \frac{1}{n^q} - \sqrt{\frac{C_1\log(2/\zeta)}{n}} \leq \frac{||\theta^*||_2 - D}{C_3} - \sqrt{\frac{C_1\log(2/\zeta)}{n}}\\
&= \frac{\Phi''(L_x||\theta^*||_2)\lambda_{\min}(\Sigma)}{4}(||\theta^*||_2 - D) - \sqrt{\frac{C_1\log(2/\zeta)}{n}}\\
&\leq \frac{\Phi''(L_x||\theta^*||_2)\lambda_{\min}(\Sigma)}{2}(||\theta^*||_2 - D) - \sqrt{\frac{C_1\log(2/\zeta)}{n}} - \frac{1}{n^q},
\end{align*}
where in both inequalities, we used the fact that $||\theta^*||_2 - D \geq \frac{C_3}{n^q}$. Secondly, for $C_2 = (4C_1)^{\frac{1}{1 - 2q}}$, we have $n > C_2\log^{\frac{1}{1 - 2q}}(2/\zeta)$, so $r > \frac{1}{2n^q}$. Implicitly, $r > 0$, hence we can use Theorem \ref{theorem:ERM_GLM_UP_General} with $r$ as above. Also, in the proof of Theorem \ref{theorem:ERM_GLM_UP_General}, we showed that $\mathcal{L}(\theta, \mathcal{D}_n)$ is $L_2$-Lipschitz and $K_{\Phi''}L_x^2$-smooth, and on an event $\Omega$ which occurs with probability at least $1 - \zeta$, we have $||\nabla \mathcal{L}(\theta, \mathcal{D}_n)||_2 > r$ for all $\theta \in \mathcal{C}$. Here, $L_2 = (K_y + K_{\Phi'})L_x$ and $\eta = \min\left\{1, \frac{\alpha_{\mathcal{C}}r}{4K_{\Phi''}L_x^2}\right\}$. Moreover, $r \asymp \frac{1}{n^q}$, since $\frac{1}{2n^q} < r < \frac{1}{n^q}$, implying that $\eta, 1 - c \asymp \frac{1}{n^q}$. Therefore, since $0 < \epsilon \lesssim 1$, we have
\begin{align}
\label{eq:Inc_C_intermed}
\mathcal{L}(\theta_T, \mathcal{D}_n) - \mathop{\min}\limits_{\theta \in \mathcal{C}}\mathcal{L}(\theta, \mathcal{D}_n) &\lesssim \frac{1}{n} + \frac{n^{q/2}\log(n/\delta)}{n\epsilon}\sqrt{\frac{\log(n)\log(n/\zeta)}{n^q\log\left(\frac{1}{1 - \frac{1}{n^q}}\right)}} \notag\\
&\asymp \frac{\log(n/\delta)\sqrt{\log(n)\log(n/\zeta)}}{n^{1 - q/2}\epsilon},
\end{align}
where we used the facts that $\log(\log_{1/c}(n)/\delta) \lesssim \log(n/\delta)$ and $n^q\log\left(\frac{1}{1 - \frac{1}{n^q}}\right) \asymp 1$ in the above calculations. To reiterate for clarity, for $C'_1$ and $C_1$ as in Theorem \ref{theorem:ERM_GLM_UP_General}, $C_2 = (4C_1)^{\frac{1}{1 - 2q}}, C_3 = \frac{4}{\Phi''(L_x||\theta^*||_2)\lambda_{\min}(\Sigma)}$, $n > \max\left\{C_2\log^{\frac{1}{1 - 2q}}(2/\zeta), C'_1\right\}$, $||\theta^*||_2 - D \geq \frac{C_3}{n^q}$, $r = \frac{1}{n^q} - \sqrt{\frac{C_1\log(2/\zeta)}{n}}$, and $T = \log_{1/c}(n)$, Algorithm \ref{alg:PrivFWERM} returns $\theta_T$ such that on $\Omega$ we have inequality \eqref{eq:Inc_C_intermed}, with $\mathbb{P}(\Omega) \geq 1 - \zeta$. On $\Omega$, we then have
\begin{align}
\label{EqnMinLoss}
\mathcal{L}(\theta_T, \mathcal{D}_n) - \mathop{\min}\limits_{\theta \in \mathbb{B}_2(||\theta^*||_2)}\mathcal{L}(\theta, \mathcal{D}_n) &= \mathcal{L}(\theta_T, \mathcal{D}_n) - \mathop{\min}\limits_{\theta \in \mathcal{C}}\mathcal{L}(\theta, \mathcal{D}_n) \notag\\
& \quad + \mathop{\min}\limits_{\theta \in \mathcal{C}}\mathcal{L}(\theta, \mathcal{D}_n) - \mathop{\min}\limits_{\theta \in \mathbb{B}_2(||\theta^*||_2)}\mathcal{L}(\theta, \mathcal{D}_n) \notag \\
& \lesssim \frac{\log(n/\delta)\sqrt{\log(n)\log(n/\zeta)}}{n^{1 - q/2}\epsilon} \notag\\
& \quad + \mathop{\min}\limits_{\theta \in \mathcal{C}}\mathcal{L}(\theta, \mathcal{D}_n) - \mathop{\min}\limits_{\theta \in \mathbb{B}_2(||\theta^*||_2)}\mathcal{L}(\theta, \mathcal{D}_n) \notag \\
&\leq \frac{\log(n/\delta)\sqrt{\log(n)\log(n/\zeta)}}{n^{1 - q/2}\epsilon} \notag\\
& \quad + \mathop{\min}\limits_{\theta \in \mathcal{C}}\mathcal{L}(\theta, \mathcal{D}_n) - \mathcal{L}(\theta_{B, n}, \mathcal{D}_n),
\end{align}
where $\theta_{B, n}$ is any minimizer of $\mathcal{L}(\theta, \mathcal{D}_n)$ over $\mathbb{B}_2(||\theta^*||_2)$. Note that $\theta_{B, n}$ exists since $\mathcal{L}$ is continuous and $\mathbb{B}_2(||\theta^*||_2)$ is compact. Now define $\mathcal{A} : [0, \infty) \rightarrow \mathbb{R}$ as $\mathcal{A}(\lambda) = \mathcal{L}(\lambda\theta_{B, n}, \mathcal{D}_n)$. Note that $\mathcal{A}$ is continuous and $\mathcal{A}(1) = \mathop{\min}\limits_{\theta \in \mathbb{B}_2(||\theta^*||_2)}\mathcal{L}(\theta, \mathcal{D}_n)$. Also, $\mathcal{A}(0) = \mathcal{L}(0, \mathcal{D}_n) \geq \mathop{\min}\limits_{\theta \in \mathcal{C}}\mathcal{L}(\theta, \mathcal{D}_n)$, since $0 \in \mathcal{C}$. Moreover, we have
$$\mathcal{A}(0) \geq \mathop{\min}\limits_{\theta \in \mathcal{C}}\mathcal{L}(\theta, \mathcal{D}_n) \geq \mathop{\min}\limits_{\theta \in \mathbb{B}_2(||\theta^*||_2)}\mathcal{L}(\theta, \mathcal{D}_n) = \mathcal{A}(1).$$ Thus, by the Intermediate Value Theorem, there exists $\lambda_n \in [0, 1]$ such that $\mathcal{A}(\lambda_n) = \mathop{\min}\limits_{\theta \in \mathcal{C}}\mathcal{L}(\theta, \mathcal{D}_n)$. Hence, we have
\begin{align*}
\mathcal{L}(\theta_T, \mathcal{D}_n) - \mathop{\min}\limits_{\theta \in \mathbb{B}_2(||\theta^*||_2)}\mathcal{L}(\theta, \mathcal{D}_n) &\lesssim \frac{\log(n/\delta)\sqrt{\log(n)\log(n/\zeta)}}{n^{1 - q/2}\epsilon}\\
& \quad + \mathcal{L}(\lambda_n\theta_{B, n}, \mathcal{D}_n) - \mathcal{L}(\theta_{B, n}, \mathcal{D}_n).
\end{align*}
Now, we have a few cases:
\begin{enumerate}
\item Case $1$: $\theta_{B, n}$ is at the boundary of $\mathbb{B}_2(||\theta^*||_2)$. If $\lambda_n\theta_{B, n}$ is at the boundary of $\mathcal{C}$, then $$||\lambda_n\theta_{B, n} - \theta_{B, n}||_2 = ||\theta^*||_2 - D \asymp \frac{1}{n^q}.$$
Now suppose $\lambda_n\theta_{B, n}$ is in the interior of $\mathcal{C}$. Recall that $\mathcal{L}(\lambda_n\theta_{B, n}, \mathcal{D}_n) = \mathop{\min}\limits_{\theta \in \mathcal{C}}\mathcal{L}(\theta, \mathcal{D}_n)$. Since $\mathcal{L}(\theta, \mathcal{D}_n)$ is convex in $\theta$, we must then have $\nabla \mathcal{L}(\lambda_n\theta_{B, n}, \mathcal{D}_n) = 0$, so $\lambda_n\theta_{B, n}$ is a global minizer of $\mathcal{L}(\theta, \mathcal{D}_n)$. Hence, we have $\mathop{\min}\limits_{\theta \in \mathcal{C}}\mathcal{L}(\theta, \mathcal{D}_n) \leq \mathop{\min}\limits_{\theta \in \mathbb{B}_2(||\theta^*||_2)}\mathcal{L}(\theta, \mathcal{D}_n)$, so $\mathop{\min}\limits_{\theta \in \mathcal{C}}\mathcal{L}(\theta, \mathcal{D}_n) - \mathop{\min}\limits_{\theta \in \mathbb{B}_2(||\theta^*||_2)}\mathcal{L}(\theta, \mathcal{D}_n) = 0$. If $\lambda_n\theta_{B, n}$ is outside $\mathcal{C}$, then
$$||\lambda_n\theta_{B, n} - \theta_{B, n}||_2 \leq ||\theta^*||_2 - D \lesssim \frac{1}{n^q}.$$
\item Case $2$: $\theta_{B, n}$ is in the interior of $\mathbb{B}_2(||\theta^*||_2)$. If $\lambda_n\theta_{B, n}$ is at the boundary of $\mathcal{C}$, then $$||\lambda_n\theta_{B, n} - \theta_{B, n}||_2 \leq ||\theta^*||_2 - D \asymp \frac{1}{n^q}.$$
Suppose now that $\lambda_n\theta_{B, n}$ is in the interior of $\mathcal{C}$. Then, like in Case $1$, $\lambda_n\theta_{B, n}$ is a global minimum of $\mathcal{L}(\theta, \mathcal{D}_n)$, so $\mathop{\min}\limits_{\theta \in \mathcal{C}}\mathcal{L}(\theta, \mathcal{D}_n) - \mathop{\min}\limits_{\theta \in \mathbb{B}_2(||\theta^*||_2)}\mathcal{L}(\theta, \mathcal{D}_n) = 0$. If $\lambda_n\theta_{B, n}$ is outside $\mathcal{C}$, then $$||\lambda_n\theta_{B, n} - \theta_{B, n}||_2 = ||\theta_{B, n}||_2 - ||\lambda_n\theta_{B, n}||_2 \leq ||\theta^*||_2 - D \lesssim \frac{1}{n^q}.$$
\end{enumerate}
By looking at the two cases above, we see that $||\lambda_n\theta_{B, n} - \theta_{B, n}||_2 \lesssim \frac{1}{n^q}$ or $\mathop{\min}\limits_{\theta \in \mathcal{C}}\mathcal{L}(\theta, \mathcal{D}_n) - \mathop{\min}\limits_{\theta \in \mathbb{B}_2(||\theta^*||_2)}\mathcal{L}(\theta, \mathcal{D}_n) = 0$. But now, note that $\mathcal{L}(\theta, \mathcal{D}_n)$ is $K_{\Phi''}L_x^2$-smooth, and using Cauchy-Schwarz, we obtain
\begin{align*}
\mathcal{L}(\lambda_n\theta_{B, n}, \mathcal{D}_n) - \mathcal{L}(\theta_{B, n}, \mathcal{D}_n) &\leq ||\nabla \mathcal{L}(\theta_{B, n}, \mathcal{D}_n)||_2||\lambda_n\theta_{B, n} - \theta_{B, n}||_2\\
& \quad + \frac{K_{\Phi''}L_x^2}{2}||\lambda_n\theta_{B, n} - \theta_{B, n}||_2^2.
\end{align*}
Therefore, since $||\lambda_n\theta_{B, n} - \theta_{B, n}||_2 \lesssim \frac{1}{n^q}$ or $\mathop{\min}\limits_{\theta \in \mathcal{C}}\mathcal{L}(\theta, \mathcal{D}_n) - \mathop{\min}\limits_{\theta \in \mathbb{B}_2(||\theta^*||_2)}\mathcal{L}(\theta, \mathcal{D}_n) = 0$, and referring back to inequality~\eqref{EqnMinLoss}, we have in all cases on $\Omega$ that
\begin{align}
\label{EqnLq}
\mathcal{L}(\theta_T, \mathcal{D}_n) - \mathop{\min}\limits_{\theta \in \mathbb{B}_2(||\theta^*||_2)}\mathcal{L}(\theta, \mathcal{D}_n) &\lesssim \frac{\log(n/\delta)\sqrt{\log(n)\log(n/\zeta)}}{n^{1 - q/2}\epsilon}\notag\\ 
& \quad + ||\nabla \mathcal{L}(\theta_{B, n}, \mathcal{D}_n)||_2\frac{1}{n^q} + \frac{1}{n^{2q}}.
\end{align}
We need to control $||\nabla \mathcal{L}(\theta_{B, n}, \mathcal{D}_n)||_2$. For all $i \in [n]$, recall that $\nabla \mathcal{L}(\theta^*, z_i) = (\Phi'(x_i^T\theta^*) - y_i)x_i$. For $h \in \mathbb{R}^p$, we have $|(\Phi'(x_i^T\theta^*) - y_i)x_i^Th| \leq (K_{\Phi'} + K_y)L_x||h||_2$, so
\begin{align*}
& (\Phi'(x_i^T\theta^*) - y_i)x_i - \mathbb{E}[(\Phi'(x_i^T\theta^*) - y_i)x_i] \in \mathcal{G}\left((K_{\Phi'} + K_y)^2L_x^2\right), \\
& \frac{1}{n}\sum_{i = 1}^n\nabla \mathcal{L}(\theta^*, z_i) - \mathbb{E}[\nabla \mathcal{L}(\theta^*, z_1)] = \nabla \mathcal{L}(\theta^*, \mathcal{D}_n) \in \mathcal{G}\left(\frac{(K_{\Phi'} + K_y)^2L_x^2}{n}\right),
\end{align*}
and by Lemma \ref{lemma:CIGV}, we have $\mathbb{P}(\Omega_3) \geq 1 - \zeta$, where
\begin{align*}
\Omega_3 = \left\{\left\|\nabla \mathcal{L}(\theta^*, \mathcal{D}_n)\right\|_2  \leq \sqrt{\frac{8(K_{\Phi'} + K_y)^2L_x^2\log(4^p/\zeta)}{n}}\right\}.
\end{align*} 
Let $\Omega' = \Omega \cap \Omega_3$ and $\mathbb{P}(\Omega') \geq 1 - 2\zeta$. Now, using the $K_{\Phi''}L_x^2$-smoothness of $\mathcal{L}(\theta, \mathcal{D}_n)$ over $\mathbb{R}^p$, we have on $\Omega'$ that
\begin{align*}
||\nabla \mathcal{L}(\theta_{B, n}, \mathcal{D}_n)||_2 &\leq ||\nabla \mathcal{L}(\theta_{B, n}, \mathcal{D}_n) - \nabla \mathcal{L}(\theta^*, \mathcal{D}_n)||_2 + ||\nabla \mathcal{L}(\theta^*, \mathcal{D}_n)||_2\\
&\lesssim ||\theta_{B, n} - \theta^*||_2 + ||\nabla \mathcal{L}(\theta^*, \mathcal{D}_n)||_2 \lesssim ||\theta_{B, n} - \theta^*||_2 + \sqrt{\frac{\log(4/\zeta)}{n}}.
\end{align*}
Hence, we need to control $||\theta_{B, n} - \theta^*||_2$. To do that, we want  to use Lemma \ref{lemma:M_Estim_Conv}, with the metric space given by $(\mathbb{B}_2(||\theta^*||_2), ||\cdot||_2)$,  and we will check the conditions of that result. That is, we consider $\mathbb{B}_2(||\theta^*||_2)$ with the induced $\ell_2$-norm metric from $\mathbb{R}^p$. We have $\theta^* = \mathop{\arg\min}\limits_{\theta \in \mathbb{B}_2(||\theta^*||_2)}\mathcal{R}(\theta)$ and $\theta_{B, n} \in \mathop{\arg\min}\limits_{\theta \in \mathbb{B}_2(||\theta^*||_2)}\mathcal{L}(\theta, \mathcal{D}_n)$. Also, because of the strong convexity of $\mathcal{R}$ over a ball centered at $0$, as seen in Lemma \ref{lemma:GLM_prelim_lem}, and because $\theta^*$ is the minimizer of $\mathcal{R}$ over $\mathbb{R}^p$, as seen in Section \ref{sec:Generalized Linear Models (GLM)}, we have $\mathcal{R}(\theta) - \mathcal{R}(\theta^*) \gtrsim ||\theta - \theta^*||_2^2$, for all $\theta$ in a small enough neighborhood of $\theta^*$ in the metric space $(\mathbb{B}_2(||\theta^*||_2), ||\cdot||_2)$. Now, observe that $\theta_{B, n}$ is a maximum likelihood estimator (MLE) of $\mathcal{L}(\theta, \mathcal{D}_n)$ over $\mathbb{B}_2(||\theta^*||_2)$, since $\mathcal{L}$ is the negative log-likelihood loss. Note that we satisfy the conditions of Lemma \ref{lemma:Consis_MLE}, hence $\theta_{B, n}$ converges in probability to $\theta^*$. Let $\mathcal{K} = \mathbb{B}_2(||\theta^*||_2 + 1)$. As in the proof of Theorem \ref{theorem:ERM_GLM_UP_General}, using a covering argument and inequality~\eqref{eq:Conc_Grad_GLM_Eq}, we have for $Z_\theta = \left\|\frac{1}{n}\sum_{i = 1}^n\nabla \mathcal{L}(\theta, z_i) - \mathbb{E}[\nabla \mathcal{L}(\theta, z_1)]\right\|_2 = \left\|\nabla \mathcal{L}(\theta, \mathcal{D}_n) - \nabla \mathcal{R}(\theta)\right\|_2$ that
\begin{align*}
\mathbb{P}\left(\mathop{\sup}\limits_{\theta \in \mathcal{K}}Z_\theta \geq t\right) \leq \left(\frac{64K_{\Phi''}L_x^2(||\theta^*||_2  + 1)}{t}\right)^pe^{-\frac{t^2}{32s_n^2}}, \quad \forall \ t \leq 8K_{\Phi''}L_x^2(||\theta^*||_2  + 1).
\end{align*}
Since $||\theta^*||_2 + 1 \asymp 1$, and by rescaling $t$, there are absolute constants $C_4, C_5 > 0$ such that
\begin{align*}
\mathbb{P}\left(\mathop{\sup}\limits_{\theta \in \mathcal{K}}Z_\theta \geq t\right) \leq \frac{1}{t^p}e^{-\frac{nt^2}{C_4}}, \quad \forall \ t \leq C_5.
\end{align*}
We want $t \leq C_5$ and $\frac{1}{t^p}e^{-\frac{nt^2}{C_4}} \leq \frac{1}{n}$, or equivalently, $t^2 + \frac{pC_4}{n}\log(t) \geq \frac{C_4}{n}\log(n)$. Take $t = \sqrt{\frac{C_4\log(n)}{n}} + \frac{\log(n)}{\sqrt{n}}$. Hence, we have
\begin{align*}
t^2 + \frac{pC_4}{n}\log(t) &\geq \frac{\log^2(n)}{n} + \frac{2\log(n)\sqrt{C_4\log(n)}}{n} + \frac{C_4\log(n)}{n} + \frac{pC_4}{2n}\log\left(\frac{\log(n)}{n}\right)\\
&\geq \frac{C_4\log(n)}{n},
\end{align*}
for $n$ large enough. This is because, for $n$ large enough, we have
\begin{align*}
\frac{\log^2(n)}{n} \geq \frac{C_4}{2n}\log\left(\frac{n}{\log(n)}\right).
\end{align*}
Note also that, for $n$ large enough, we have $t \leq C_5$. Hence, there is an absolute constant $C_6 > 0$ such that for any $n \geq C_6$, we have $\mathbb{P}(\Omega_4) \geq 1 - \frac{1}{n}$, with
\begin{align*}
\Omega_4 = \left\{\left\|\nabla\mathcal{L}(\theta, \mathcal{D}_n) - \nabla\mathcal{R}(\theta)\right\|_2 \leq \sqrt{\frac{C_4\log(n)}{n}} + \frac{\log(n)}{\sqrt{n}}, \quad \forall \theta \in \mathcal{K}\right\}.
\end{align*}
Now take $u \leq 1$ and let $\mathbb{U}_n(\theta) = \mathcal{L}(\theta, \mathcal{D}_n) - \mathcal{R}(\theta)$. We have by the Mean Value Theorem that 
\begin{align*}
\mathop{\sup}_{\substack{||\theta - \theta^*||_2 \leq u\\ \theta \in \mathbb{B}_2(||\theta^*||_2)}}\left|\mathbb{U}_n(\theta) - \mathbb{U}_n(\theta^*)\right| \leq \mathop{\sup}\limits_{\theta \in \mathcal{K}}\left\|\nabla\mathcal{L}(\theta, \mathcal{D}_n) - \nabla\mathcal{R}(\theta)\right\|_2u,
\end{align*}
since the supremum only increases if we take it over $\mathcal{K} = \mathbb{B}_2(||\theta^*||_2 + 1)$. Therefore, we have
\begin{align*}
\mathbb{E}\left[\mathop{\sup}_{\substack{||\theta - \theta^*||_2 \leq u\\\theta \in \mathbb{B}_2(||\theta^*||_2)}}\left|\mathbb{U}_n(\theta) - \mathbb{U}_n(\theta^*)\right|\right] &\leq \mathbb{E}\left[\mathop{\sup}\limits_{\theta \in \mathcal{K}}\left\|\nabla\mathcal{L}(\theta, \mathcal{D}_n) - \nabla\mathcal{R}(\theta)\right\|_2u\right]\\
&= \mathbb{E}\left[\mathop{\sup}\limits_{\theta \in \mathcal{K}}\left\|\nabla\mathcal{L}(\theta, \mathcal{D}_n) - \nabla\mathcal{R}(\theta)\right\|_2u\mathbbm{1}_{\Omega_4}\right]\\
& \quad + \mathbb{E}\left[\mathop{\sup}\limits_{\theta \in \mathcal{K}}\left\|\nabla\mathcal{L}(\theta, \mathcal{D}_n) - \nabla\mathcal{R}(\theta)\right\|_2u\mathbbm{1}_{\Omega_4^c}\right],
\end{align*}
implying that
\begin{align*}
\mathbb{E}\left[\mathop{\sup}_{\substack{||\theta - \theta^*||_2 \leq u\\ \theta \in \mathbb{B}_2(||\theta^*||_2)}}\left|\mathbb{U}_n(\theta) - \mathbb{U}_n(\theta^*)\right|\right] \lesssim \frac{\log(n)u}{\sqrt{n}}\mathbb{P}(\Omega_4) + u\mathbb{P}(\Omega_4^c) \leq \frac{\log(n)u}{\sqrt{n}} + \frac{u}{n} \lesssim \frac{\log(n)u}{\sqrt{n}},
\end{align*}
for all $0 < u \leq 1$ and $n \geq C_6$, since $\mathcal{L}$ and $\mathcal{R}$ are $(K_{\Phi'} + K_y)L_x$-Lipschitz over $\mathbb{R}^p$ and $(K_{\Phi'} + K_y)L_x \asymp 1$, as seen in Theorem \ref{theorem:ERM_GLM_UP_General}. Take $\phi_n(u) = \log(n)u$ and $r_n = \frac{\sqrt{n}}{\log(n)}$. Note that $u \mapsto \frac{\phi_n(u)}{u} = \log(n)$ is non-increasing and $r_n^2\phi_n\left(\frac{1}{r_n}\right) = r_n\log(n) = \sqrt{n}$. Hence, all the conditions of Lemma \ref{lemma:M_Estim_Conv} are satisfied with $\alpha = 1 < 2$, so for $\zeta \in (0, 1/3)$, there are $T_\zeta, N_\zeta > 0$, such that $\mathbb{P}(\Omega_5) \geq 1 - \zeta$, for all $n \geq \max\left\{C_6, N_\zeta\right\}$, where 
\begin{align*}
\Omega_5 = \left\{||\theta_{B, n} - \theta^*||_2 \leq \frac{T_\zeta\log(n)}{\sqrt{n}}\right\}.
\end{align*}
Now we absorb $C'_1$ into $C_6$, i.e., relabel $\max\{C'_1, C_6\}$ by $C'_1$. Working on $\Omega'' = \Omega' \cap \Omega_5$, with $\mathbb{P}(\Omega'') \geq 1 - 3\zeta$, we have for $n > \max\left\{C_2\log^5(2/\zeta), N_\zeta, C'_1\right\}$ that 
\begin{align*}
\mathcal{L}(\theta_T, \mathcal{D}_n) - \mathop{\min}\limits_{\theta \in \mathbb{B}_2(||\theta^*||_2)}\mathcal{L}(\theta, \mathcal{D}_n) &\lesssim \frac{\log(n/\delta)\sqrt{\log(n)\log(n/\zeta)}}{n^{1 - q/2}\epsilon} + \frac{T_\zeta\log(n)}{n^{q + \frac{1}{2}}}\\
& \quad + \frac{\sqrt{\log(4/\zeta)}}{n^{q + \frac{1}{2}}} + \frac{1}{n^{2q}}\\
&\lesssim \log(n/\delta)T_\zeta\sqrt{\log(n)\log(n/\zeta)}\left(\frac{1}{{n^{1 - q/2}\epsilon}} + \frac{1}{n^{q + \frac{1}{2}}}\right)\\
& \quad + \frac{\log(n/\delta)T_\zeta\sqrt{\log(n)\log(n/\zeta)}}{n^{2q}},
\end{align*}
by plugging back into inequality~\eqref{EqnLq}. Now, for $q = \frac{2}{5}$, since $0 < \epsilon \lesssim 1$, we obtain the desired result.

Finally, using the assumption that $\epsilon \leq 0.9$, we have $\epsilon < 2\sqrt{2T\log(2/\delta)}$ and $\delta < 2T$, where $T \asymp n^q\log(n)$, which are needed in Lemma \ref{lemma:ACC_PRIV_FW_STEP} to ensure that the output of Algorithm \ref{alg:PrivFWERM} is $(\epsilon, \delta)$-DP.

\begin{remark}
We proved Theorem~\ref{theorem:ERM_UP_C_Inc} by deriving a more general statement with $q < \frac{1}{2}$: based on this approach, the best choice is $q = \frac{2}{5}$. Indeed, examining the RHS of inequality~\eqref{eq:General_q_Inc_C}, we can consider the lines $1 - \frac{q}{2}, q + \frac{1}{2}$ and $2q$. In order to obtain a rate better than $\frac{1}{n^{2/3}}$ up to logarithmic factors, we need $q > \frac{1}{3}$. Hence, to optimize the RHS of inequality~\eqref{eq:General_q_Inc_C} over $\frac{1}{3} < q < \frac{1}{2}$, we see that the best $q$ is at the intersection of $1 - \frac{q}{2}$ and $2q$, namely $q = \frac{2}{5}$.
\end{remark}

\subsubsection{Proof of Theorem \ref{theorem:Iter_Rate_Inc_C}}
\label{AppThmIterRate}

The conditions in the theorem hypothesis are part of the ones in Theorem~\ref{theorem:ERM_UP_C_Inc}. Hence, by  Theorem~\ref{theorem:ERM_UP_C_Inc}, we have with probability at least $1 - 3\zeta$, for $n > \max\left\{C_2\log^5(2/\zeta), N_\zeta, C'_1\right\}$, that
\begin{align*}
\mathcal{L}(\theta_T, \mathcal{D}_n) - \mathop{\min}\limits_{\theta \in \mathbb{B}_2(||\theta^*||_2)}\mathcal{L}(\theta, \mathcal{D}_n) \lesssim \frac{T_\zeta\log(n/\delta)\sqrt{\log(n)\log(n/\zeta)}}{n^{4/5}\epsilon}.
\end{align*}
Let $\Omega''$ be the event with probability at least $1 - 3\zeta$ and $\mathbb{X} \in \mathbb{R}^{p \times n}$ be the matrix with $x_i$ as the $i^{th}$ row, for $i \in [n]$. Let $v \in \mathbb{R}^p$ be such that $||v||_2 = 1$. Then
\begin{align*}
v^T\frac{\mathbb{X}^T\mathbb{X}}{n}v &= v^T\Sigma v - v^T\left(\Sigma - \frac{\mathbb{X}^T\mathbb{X}}{n}\right)v \geq \lambda_{\min}\left(\Sigma\right) - ||v||_2^2\left\|\frac{\mathbb{X}^T\mathbb{X}}{n} - \ \Sigma\right\|_2\\
&\geq \lambda_{\min}\left(\Sigma\right) - \left\|\frac{\mathbb{X}^T\mathbb{X}}{n} - \ \Sigma\right\|_2.
\end{align*}
Recall also that, in the context of the GLM defined in Section \ref{sec:Generalized Linear Models (GLM)}, $\lambda_{\min}(\Sigma)$ and $\lambda_{\max}(\Sigma)$ are positive absolute constants. Let $C_4 = \frac{8L_x^2(\lambda_{\max}(\Sigma) + \lambda_{\min}(\Sigma)/3)}{\lambda_{\min}(\Sigma)^2}$. Since $\{x_i\}_{i = 1}^n$ are i.i.d., $\mathbb{E}[x_1] = 0$, $||\Sigma||_2 = \lambda_{\max}\left(\Sigma\right)$, and $||x_1||_2 \leq \sqrt{L_x^2}$, by Lemma \ref{lemma:Concxx^T}, we have
\begin{align*}
\mathbb{P}\left(\left\|\frac{\mathbb{X}^T\mathbb{X}}{n} - \ \Sigma\right\|_2 > \frac{\lambda_{\min}\left(\Sigma\right)}{2}\right) &= \mathbb{P}\left(\left\|\frac{1}{n}\sum_{i = 1}^nx_ix_i^T - \ \Sigma\right\|_2 > \frac{\lambda_{\min}\left(\Sigma\right)}{2}\right)\\
&\leq 2pe^{\frac{-n\lambda_{\min}\left(\Sigma\right)^2}{8L_x^2\left(\lambda_{\max}\left(\Sigma\right)+ \lambda_{\min}\left(\Sigma\right)/3\right)}} \leq 2pe^{-\frac{n}{C_4}} \leq \zeta,
\end{align*}
since $n > C_4\log(2p/\zeta)$. Therefore, on $\Omega_6 = \left\{\left\|\frac{\mathbb{X}^T\mathbb{X}}{n} - \ \Sigma\right\|_2 \leq \frac{\lambda_{\min}\left(\Sigma\right)}{2}\right\}$, we have for $n > C_4\log(2p/\zeta)$ that $\lambda_{\min}\left(\frac{\mathbb{X}^T\mathbb{X}}{n}\right) \geq \frac{\lambda_{\min}(\Sigma)}{2}$. Recall now from Theorem \ref{theorem:ERM_GLM_UP_General} that $\nabla^2\mathcal{L}(\theta, \mathcal{D}_n) = \frac{1}{n}\sum_{i = 1}^n\Phi''(x_i^T\theta)x_ix_i^T$. Using the properties of $\Phi''$ outlined in Section \ref{sec:Generalized Linear Models (GLM)}, we have $\Phi''(x^T\theta) \geq \Phi''(L_x||\theta^*||_2)$, for all $\theta \in \mathbb{B}_2(||\theta^*||_2)$. Hence, on $\Omega_6$, we see that for all $\theta \in \mathbb{B}_2(||\theta^*||_2)$ and $n > C_4\log(2p/\zeta)$, we have
\begin{align*}
\nabla^2\mathcal{L}(\theta, \mathcal{D}_n) = \frac{1}{n}\sum_{i = 1}^n\Phi''(x_i^T\theta)x_ix_i^T \succeq \frac{\Phi''(L_x||\theta^*||_2)\mathbb{X}^T\mathbb{X}}{n} \succeq \frac{\Phi''(L_x||\theta^*||_2)\lambda_{\min}(\Sigma)}{2}I_p.
\end{align*}
Thus, on $\Omega_6$, the function $\mathcal{L}(\theta, \mathcal{D}_n)$ is $\frac{\Phi''(L_x||\theta^*||_2)\lambda_{\min}(\Sigma)}{2}$-strongly convex over $\mathbb{B}_2(||\theta^*||_2)$, for $n > C_4\log(2p/\zeta)$. Note that $\frac{\Phi''(L_x||\theta^*||_2)\lambda_{\min}(\Sigma)}{2} \asymp 1$. Let us now work on $\Omega''' = \Omega'' \cap \Omega_6$, so that $\mathbb{P}(\Omega''') \geq 1 - 4\zeta$. Take $n > \max\left\{C_2\log^5(2/\zeta), C_4\log(2p/\zeta), N_\zeta, C'_1\right\}$.  We had, using the notation from Theorem \ref{theorem:ERM_UP_C_Inc}, i.e., $\theta_{B, n} \in \mathop{\arg\min}\limits_{\theta \in \mathbb{B}_2(||\theta^*||_2)}\mathcal{L}(\theta, \mathcal{D}_n)$, that
\begin{align*}
\mathcal{L}(\theta_T, \mathcal{D}_n) - \mathcal{L}(\theta_{B, n}, \mathcal{D}_n) \lesssim \frac{T_\zeta\log(n/\delta)\sqrt{\log(n)\log(n/\zeta)}}{n^{4/5}\epsilon}.
\end{align*}
Because of the strong convexity of $\mathcal{L}(\theta, \mathcal{D}_n)$ over $\mathbb{B}_2(||\theta^*||_2)$, and because $\theta_{B, n}$ is a minimizer, we obtain
$$||\theta_T - \theta_{B, n}||_2^2 \lesssim \mathcal{L}(\theta_T, \mathcal{D}_n) - \mathcal{L}(\theta_{B, n}, \mathcal{D}_n).$$
Recall now from the proof of Theorem \ref{theorem:ERM_UP_C_Inc} that $\Omega''$ is an intersection of three events, each with probability at least $1 - \zeta$, and on one of those we had $||\theta_{B, n} - \theta^*||_2 \leq \frac{T_\zeta\log(n)}{\sqrt{n}}$. So, putting all this together, we obtain 
\begin{align*}
||\theta_T - \theta^*||_2 &\leq ||\theta_T - \theta_{B, n}||_2 + ||\theta_{B, n} - \theta^*||_2\\
&\lesssim \frac{T_\zeta\log(n)}{\sqrt{n}} + \frac{T_\zeta^{1/2}\log^{1/2}(n/\delta)\log^{1/4}(n)\log^{1/4}(n/\zeta)}{n^{2/5}\epsilon^{1/2}},
\end{align*}
as required.

\subsubsection{Proof of Theorem \ref{theorem:ERM_GLM_UP}}
\label{ThmERMGLM}

Let $\zeta \in (0, 1)$ be arbitrary. To start off, by Theorem \ref{theorem:ERM_GLM_UP_General} with $q = \frac{1}{4}$, there exist positive absolute constants $C'_1$ and $C_1$ such that for $C_2 = \frac{\Phi''(L_x||\theta^*||_2)\lambda_{\min}(\Sigma)(||\theta^*||_2 - D)}{4} > 0$, $n > \max\left\{\left(\frac{\sqrt{C_1\log(2/\zeta)} + 1}{C_2}\right)^{4}, C'_1\right\}$, $r \in\left(\frac{C_2}{2}, C_2\right]$, and $T = \log_{1/c}(n)$, Algorithm \ref{alg:PrivFWERM} returns $\theta_T$ such that with probability at least $1 - \zeta$, we have
\begin{align}
\label{eq:Fix_C_help_Main_Thm}
\mathcal{L}(\theta_T, \mathcal{D}_n) - \mathop{\min}\limits_{\theta \in \mathcal{C}}\mathcal{L}(\theta, \mathcal{D}_n) \lesssim \frac{1}{n} + \frac{\eta\log\left(\log_{1/c}(n)/\delta\right)\sqrt{\log_{1/c}(n)\log\left(\log_{1/c}(n)/\zeta\right)}}{(1 - c)n\epsilon}.
\end{align}
This is because if $n > \left(\frac{\sqrt{C_1\log(2/\zeta)} + 1}{C_2}\right)^{4}$, we have $$\frac{\Phi''(L_x||\theta^*||_2)\lambda_{\min}(\Sigma)}{4}(||\theta^*||_2 - D)n^{1/4} > \sqrt{C_1\log(2/\zeta)} + 1 \geq \frac{\sqrt{C_1\log(2/\zeta)}}{n^{1/4}} + 1.$$ Hence, we have
\begin{align*}
&\frac{\Phi''(L_x||\theta^*||_2)\lambda_{\min}(\Sigma)}{2}(||\theta^*||_2 - D) - \sqrt{\frac{C_1\log(2/\zeta)}{n}} - \frac{1}{n^{1/4}}\\
& \quad > \frac{\Phi''(L_x||\theta^*||_2)\lambda_{\min}(\Sigma)}{4}(||\theta^*||_2 - D) = C_2,
\end{align*}
and since $r \in\left(\frac{C_2}{2}, C_2\right]$, we have 
\begin{align*}
0 < r \leq \frac{\Phi''(L_x||\theta^*||_2)\lambda_{\min}(\Sigma)}{2}(||\theta^*||_2 - D) - \sqrt{\frac{C_1\log(2/\zeta)}{n}} - \frac{1}{n^{1/4}}.
\end{align*}
Moreover, $r > 0$, since $\theta^* \in \mathbb{R}^p \setminus \mathcal{C}$, so $||\theta^*||_2 - D > 0$. Thus, we can use Theorem \ref{theorem:ERM_GLM_UP_General} to conclude that inequality~\eqref{eq:Fix_C_help_Main_Thm} holds with probability at least $1 - \zeta$. Also, $\eta = \min\left\{1, \frac{\alpha_{\mathcal{C}}r}{4K_{\Phi''}L_x^2}\right\}$ and $c = \max\left\{\frac{1}{2}, 1 - \frac{\alpha_{\mathcal{C}}r}{8K_{\Phi''}L_x^2}\right\}$. Now, notice that
$$1 \asymp \frac{\Phi''(L_x||\theta^*||_2)\lambda_{\min}(\Sigma)(||\theta^*||_2 - D)}{8} < r \leq \frac{\Phi''(L_x||\theta^*||_2)\lambda_{\min}(\Sigma)(||\theta^*||_2 - D)}{4} \asymp 1.$$
Thus, $r = \Theta(1)$. Since $\alpha_{\mathcal{C}} \asymp 1$ as well, we have $\eta, c \asymp 1$. Hence, since $0 < \epsilon \lesssim 1$, with probability at least $1 - \zeta$, we have
\begin{align*}
\mathcal{L}(\theta_T, \mathcal{D}_n) - \mathop{\min}\limits_{\theta \in \mathcal{C}}\mathcal{L}(\theta, \mathcal{D}_n) &\lesssim \frac{1}{n} + \frac{\log\left(\log(n)/\delta\right)\sqrt{\log(n)\log\left(\log(n)/\zeta\right)}}{n\epsilon}\\
&\lesssim \frac{\log\left(\log(n)/\delta\right)\sqrt{\log(n)\log\left(\log(n)/\zeta\right)}}{n\epsilon}.
\end{align*}
Let $\Omega_\zeta$ denote the event where the preceding bound holds. Taking $\zeta = \frac{1}{n}$, we see that for $n > \max\left\{\left(\frac{\sqrt{C_1\log(2n)} + 1}{C_2}\right)^{4}, C'_1\right\}$, we have
\begin{align*}
\mathbb{E}\left[\mathcal{L}(\theta_T, \mathcal{D}_n) - \mathop{\min}\limits_{\theta \in \mathcal{C}}\mathcal{L}(\theta, \mathcal{D}_n)\right] &= \mathbb{E}\left[\left(\mathcal{L}(\theta_T, \mathcal{D}_n) - \mathop{\min}\limits_{\theta \in \mathcal{C}}\mathcal{L}(\theta, \mathcal{D}_n)\right)\mathbbm{1}_{\Omega_\zeta}\right]\\
& \quad + \mathbb{E}\left[\left(\mathcal{L}(\theta_T, \mathcal{D}_n) - \mathop{\min}\limits_{\theta \in \mathcal{C}}\mathcal{L}(\theta, \mathcal{D}_n)\right)\mathbbm{1}_{\Omega^c_\zeta}\right]\\
&\lesssim \frac{\log\left(\log(n)/\delta\right)\sqrt{\log(n)\log\left(n\log(n)\right)}}{n\epsilon}\\
& \quad + \frac{\mathbb{E}\left[\mathcal{L}(\theta_T, \mathcal{D}_n) - \mathop{\min}\limits_{\theta \in \mathcal{C}}\mathcal{L}(\theta, \mathcal{D}_n)\right]}{n}\\
&\lesssim \frac{\log\left(\log(n)/\delta\right)\sqrt{\log(n)\log\left(n\log(n)\right)}}{n\epsilon} + \frac{L_2||\mathcal{C}||_2}{n}\\
&\lesssim \frac{\log\left(\log(n)/\delta\right)\sqrt{\log(n)\log\left(n\log(n)\right)}}{n\epsilon},
\end{align*}
as required, where we used the $L_2$-Lipschitz property of the loss, together with the fact that $||\mathcal{C}||_2 \lesssim ||\theta^*||_2 \asymp 1$.

Note that $\epsilon < 2\sqrt{2T\log(2/\delta)}$ and $\delta < 2T$, since $T \asymp \log(n)$. Hence, by Lemma \ref{lemma:ACC_PRIV_FW_STEP}, $\theta_T$ is $(\epsilon, \delta)$-DP. 


\subsubsection{Proof of Theorem \ref{theorem:HT_NonAccFW}}\label{sec:ThmNonACCFWl>0}

In this context, we are working with \iid samples $\mathcal{D}_n = \{z_i\}_{i = 1}^n$ and the squared error risk $\mathcal{R}(\theta) = \frac{1}{2}(\theta - \theta^*)^T\Sigma(\theta - \theta^*) + \frac{\sigma_2^2}{2}$. Fix $\theta \in \mathcal{C}$. Since we are using Algorithm \ref{alg:HTGE} as gradient estimator, we have by Lemma \ref{lemma:geglmht} a $g$ such that
\begin{align*}
&\alpha(\widetilde{n}, \widetilde{\zeta}) \asymp \sqrt{\frac{\log(1/\widetilde{\zeta})}{\widetilde{n}}},
&\beta(\widetilde{n}, \widetilde{\zeta}) \asymp \sqrt{\frac{\sigma_2^2\log(1/\widetilde{\zeta})}{\widetilde{n}}}.
\end{align*}
Note that since $\theta^* \in \mathcal{C}$, we have  $\theta_* = \theta^*$. At any  $t \in \{1, \dots T\}$, with probability at least $1 - \widetilde{\zeta}$, we have
\begin{align*}
||g(\theta_t, \mathcal{D}_n, \widetilde{\zeta}) - \nabla\mathcal{R}(\theta_t)||_2 \leq \alpha(\widetilde{n}, \widetilde{\zeta})||\theta_t - \theta^*||_2 + \beta(\widetilde{n}, \widetilde{\zeta}).
\end{align*}
Hence, by a union bound, we have
\begin{align*}
 &\mathbb{P}(\exists t \ \mbox{s.t.} \ ||g(\theta_t, \mathcal{D}_n, \widetilde{\zeta}) - \nabla\mathcal{R}(\theta_t)||_2 > \alpha(\widetilde{n}, \widetilde{\zeta})||\theta_t - \theta^*||_2 + \beta(\widetilde{n}, \widetilde{\zeta})) \leq \sum_{t = 1}^T \widetilde{\zeta} \leq \zeta,
\end{align*}
implying that
\begin{align*}
 \mathbb{P}(\forall t, \ ||g(\theta_t, \mathcal{D}_n, \widetilde{\zeta}) - \nabla\mathcal{R}(\theta_t)||_2 \leq \alpha(\widetilde{n}, \widetilde{\zeta})||\theta_t - \theta^*||_2 + \beta(\widetilde{n}, \widetilde{\zeta})) \geq 1 - \zeta.
\end{align*}
On the latter event, using the notation $\alpha = \alpha(\widetilde{n}, \widetilde{\zeta}), \beta = \beta(\widetilde{n}, \widetilde{\zeta})$ and ignoring the dependency in $g$ on the samples and $\widetilde{\zeta}$, the gradient error $e_t := g(\theta_t) -\nabla\mathcal{R}(\theta_t)$ satisfies
\begin{align*}
||e_t||_2 \leq \alpha||\theta_t - \theta^*||_2 + \beta,
\end{align*}
implying that
\begin{align*}
v_t^T\nabla\mathcal{R}(\theta_t) &\leq v^T\nabla\mathcal{R}(\theta_t) + (v - v_t)^Te_t \leq v^T\nabla\mathcal{R}(\theta_t) + ||v - v_t||_2||e_t||_2\\
&\leq v^T\nabla\mathcal{R}(\theta_t) + ||\mathcal{C}||_2(\alpha ||\mathcal{C}||_2 + \beta), \quad \forall v \in \mathcal{C}.
\end{align*}
Let $\Gamma_{\mathcal{R}}$ be the curvature constant of $\mathcal{R}$. We then have
\begin{align*}
v_t^T\nabla\mathcal{R}(\theta_t) &\leq \mathop{\min}\limits_{v \in \mathcal{C}}v^T\nabla\mathcal{R}(\theta_t) + \frac{1}{2}\frac{2}{t + 2}\Gamma_{\mathcal{R}}\frac{||\mathcal{C}||_2(\alpha ||\mathcal{C}||_2 + \beta)}{\Gamma_{\mathcal{R}}}(t + 2)\\
&\leq \mathop{\min}\limits_{v \in \mathcal{C}}v^T\nabla\mathcal{R}(\theta_t) + \frac{1}{2}\frac{2}{t + 2}\Gamma_{\mathcal{R}}\frac{||\mathcal{C}||_2(\alpha ||\mathcal{C}||_2 + \beta)}{\Gamma_{\mathcal{R}}}(T + 2).
\end{align*}
Thus, on the event with probability at least $1 - \zeta$, since $\mathcal{C}$ is compact and convex, by Lemma \ref{lemma:NonAccRel_FW}, we obtain
\begin{align*}
\mathcal{R}(\theta_T) - \mathcal{R}(\theta^*) &\leq \frac{2\Gamma_{\mathcal{R}}}{T + 2}\left(1 + \frac{||\mathcal{C}||_2(\alpha ||\mathcal{C}||_2 + \beta)}{\Gamma_{\mathcal{R}}}(T + 2)\right)\\
&= \frac{2\Gamma_{\mathcal{R}}}{T + 2} + 2||\mathcal{C}||_2(\alpha ||\mathcal{C}||_2 + \beta).
\end{align*}
Now note that $\mathcal{R}(\theta)$ is a quadratic in $\theta$ with second-order term $\frac{1}{2}\theta^T\Sigma\theta = \theta^T\Sigma^{1/2}\Sigma^{1/2}\theta$.
By Remark $2$ in \cite{NOPL}, we have $\Gamma_{\mathcal{R}} \leq 4\mathop{\max}\limits_{\theta \in \mathcal{C}}\left\|\Sigma^{1/2}\theta\right\|_2^2 \lesssim 1$. Thus, since $||\mathcal{C}||_2 \lesssim 1$, we obtain
\begin{align*}
\mathcal{R}(\theta_T) - \mathcal{R}(\theta^*) \lesssim \frac{1}{T} + (1 + \sigma_2)\sqrt{\frac{T\log(T/\zeta)}{n}},
\end{align*}
and since $T = n^{1/3}$, this implies
\begin{align*}
\mathcal{R}(\theta_T) - \mathcal{R}(\theta^*) \lesssim \frac{(1 + \sigma_2)\sqrt{\log(n/\zeta)}}{n^{1/3}}.
\end{align*}
By $\lambda_{\min}(\Sigma)$-strong convexity of $\mathcal{R}$, because $\nabla\mathcal{R}(\theta^*) = 0$, $\lambda_{\min}(\Sigma) \asymp 1$, and $\lambda_{\min}(\Sigma) > 0$ we have
\begin{align*}
||\theta_T - \theta^*||_2 \lesssim \frac{(1 + \sigma_2)^{1/2}\log^{1/4}(n/\zeta)}{n^{1/6}},
\end{align*}
as required.


\subsubsection{Proof of Theorem \ref{theorem:ACCFWROB}}\label{sec:ThmACCFWl>0}

Recall the notation $\theta_* = \mathop{\arg\min}\limits_{\theta \in \mathcal{C}}\mathcal{R}(\theta)$. Following the same steps as in the proof of Theorem \ref{theorem:HT_NonAccFW}, we have with probability at least $1-\zeta$ at the $t^{\text{th}}$ step of Algorithm \ref{alg:RobPGDNFW} that
\begin{align*}
v_t^T\nabla\mathcal{R}(\theta_t) &\leq v^T\nabla\mathcal{R}(\theta_t) + (v - v_t)^Te_t \leq v^T\nabla\mathcal{R}(\theta_t) + ||v - v_t||_2||e_t||_2\\
&\leq v^T\nabla\mathcal{R}(\theta_t) + ||\mathcal{C}||_2(\alpha ||\theta_t - \theta_*||_2 + \beta)\\
&\leq v^T\nabla\mathcal{R}(\theta_t) + ||\mathcal{C}||_2(\alpha ||\mathcal{C}||_2 + \beta), \ \quad \forall v \in \mathcal{C}.
\end{align*}
Thus, we have
\begin{align*}
v_t^T\nabla\mathcal{R}(\theta_t) \leq \mathop{\min}\limits_{v \in \mathcal{C}}v^T\nabla\mathcal{R}(\theta_t) + ||\mathcal{C}||_2(\alpha ||\mathcal{C}||_2 + \beta).
\end{align*}
Now note that for $\theta \in \mathcal{C}$, we have
\begin{align*}
||\nabla\mathcal{R}(\theta)||_2 &= ||\Sigma(\theta^* - \theta)||_2 \geq \lambda_{\min}(\Sigma)\left(||\theta^*||_2 - ||\theta||_2\right) \geq \lambda_{\min}(\Sigma)\left(||\theta^*||_2 - D\right)\\ 
&\geq \frac{C_1\lambda_{\min}(\Sigma)}{n^{1/5}} \geq u \gtrsim \frac{1}{n^{1/5}}.
\end{align*}
Thus, with probability at least $1 - \zeta$, since $\mathcal{C}$ is compact and $\frac{1}{D}$-strongly convex by Lemma \ref{lemma:STR_CONV_BALL} and $\mathcal{R}$ is $\lambda_{\max}(\Sigma)$-smooth, we are in the context of Theorem \ref{theorem:ACCFWV3}. For the choice of $\eta$ in the theorem hypothesis, Theorem \ref{theorem:ACCFWV3} then implies 
\begin{align*}
\mathcal{R}(\theta_t) - \mathcal{R}(\theta_*) \leq \left(\mathcal{R}(\theta_0) - \mathcal{R}(\theta_*)\right)c^t + \frac{3\eta ||\mathcal{C}||_2(\alpha ||\mathcal{C}||_2 + \beta)}{2(1 - c)},
\end{align*}
with $c = \max\left\{\frac{1}{2}, 1 - \frac{\alpha_{\mathcal{C}} u}{8\lambda_{\max}(\Sigma)}\right\}$. Note that since $u \asymp \frac{1}{n^{1/5}}$, $\lambda_{\max}(\Sigma) \asymp 1$, and $D \asymp 1$, we have $\eta \asymp \frac{1}{n^{1/5}}$ and $c \asymp 1 - \frac{1}{n^{1/5}}$, so $\frac{1}{1 - c} \asymp n^{1/5}$. Also, $\mathcal{R}(\theta_0) - \mathcal{R}(\theta_*), ||\mathcal{C}||_2 \lesssim 1$. Thus, at iteration $T$, we obtain
\begin{align*}
\mathcal{R}(\theta_T) - \mathcal{R}(\theta_*) \lesssim c^T + (1 + \sigma_2)\sqrt{\frac{\log(1/\widetilde{\zeta})}{\widetilde{n}}}.
\end{align*}
Note that now $\log(1/c) \asymp \frac{1}{n^{1/5}}\log\left(\left(1 - \frac{1}{n^{1/5}}\right)^{-n^{1/5}}\right) \asymp \frac{1}{n^{1/5}}$. Since $T = \log_{1/c}\left(n^{2/5}\right) \asymp n^{1/5}\log(n)$, we have
\begin{align*}
\mathcal{R}(\theta_T) - \mathcal{R}(\theta_*) \lesssim \frac{1}{n^{2/5}} + (1 + \sigma_2)\sqrt{\frac{\log(n)\log\left(n\log(n)/\zeta\right)}{n^{4/5}}}.
\end{align*}
Now define $\mathcal{A} : [0, 1] \rightarrow \mathbb{R}$, as $\mathcal{A}(\lambda) = \mathcal{R}(\lambda\theta^*)$. Note that $\mathcal{A}(0) = \mathcal{R}(0) \geq \mathcal{R}(\theta_*) \geq \mathcal{R}(\theta^*) = \mathcal{A}(1)$. So, by the continuity of $\mathcal{A}$, the Intermediate Value Theorem implies that there exists $\lambda_* \in [0, 1]$ such that $\mathcal{A}(\lambda_*) = \mathcal{R}(\lambda_*\theta^*) = \mathcal{R}(\theta_*) = \mathop{\min}\limits_{\theta \in \mathcal{C}}\mathcal{R}(\theta)$. If $\lambda_*\theta^*$ is in the interior of $\mathcal{C}$, then $\nabla\mathcal{R}(\lambda_*\theta^*) = 0$, so $\lambda_*\theta^*$ is a global minimizer. This is a contradiction, since $\theta^*$ is the unique global minimizer of $\mathcal{R}$ and thus lies strictly outside $\mathcal{C}$. If $\lambda_*\theta^*$ is at the boundary or outside $\mathcal{C}$, then $||\lambda_*\theta^* - \theta^*||_2 \leq ||\theta^*||_2 - D \lesssim \frac{1}{n^{1/5}}$. Hence, by the $\lambda_{\max}(\Sigma)$-smoothness of $\mathcal{R}$, using the fact that $\nabla \mathcal{R}(\theta^*) = 0$ and that $\lambda_{\max}(\Sigma) \asymp 1$, we have
\begin{align*}
\mathcal{R}(\theta_*) - \mathcal{R}(\theta^*) &= \mathcal{R}(\lambda_*\theta^*) - \mathcal{R}(\theta^*) \lesssim ||\lambda_*\theta_{*} - \theta_*||_2^2 \lesssim \frac{1}{n^{2/5}}.
\end{align*}
Hence, we have
\begin{align}
\label{eq:best_gamma_AccFW_HT}
\mathcal{R}(\theta_T) - \mathcal{R}(\theta^*) \lesssim \frac{1}{n^{2/5}} + (1 + \sigma_2)\sqrt{\frac{\log(n)\log\left(n\log(n)/\zeta\right)}{n^{4/5}}} + \frac{1}{n^{2/5}},
\end{align}
and by the $\lambda_{\min}(\Sigma)$-strong convexity of $\mathcal{R}$ over $\mathbb{R}^p$, together with $\nabla\mathcal{R}(\theta^*) = 0$ and $\lambda_{\min}(\Sigma) \asymp 1$, we obtain
\begin{align*}
||\theta_T - \theta^*||_2 \lesssim \frac{(1 + \sigma_2)^{1/2}\log^{1/4}(n)\log^{1/4}\left(n\log(n)/\zeta\right)}{n^{1/5}},
\end{align*}
as required.

\begin{remark}
The choice of the exponent $\frac{1}{5}$ in $||\theta^*||_2 - D \lesssim\frac{1}{n^{1/5}}$, $D \leq ||\theta^*||_2 - \frac{C_1}{n^{1/5}}$, $T = \log_{1/c}(n^{2/5}) \asymp n^{1/5}\log(n)$, and $\frac{1}{n^{1/5}} \lesssim u \leq \frac{C_1\lambda_{\min}(\Sigma)}{n^{1/5}}$ is not arbitrary. Assume we started with $||\theta^*||_2 - D \lesssim \frac{1}{n^{q}}$, $D \leq ||\theta^*||_2 - \frac{C_1}{n^{q}}$, $T = \log_{1/c}(n^{2q}) \asymp n^q\log(n)$, and $\frac{1}{n^{q}} \lesssim u \leq \frac{C_1\lambda_{\min}(\Sigma)}{n^{q}}$, for some $q > 0$. Then inequality~\eqref{eq:best_gamma_AccFW_HT} becomes 
\begin{align*}
\mathcal{R}(\theta_T) - \mathcal{R}(\theta^*) \lesssim \frac{1}{n^{2q}} + (1 + \sigma_2)\frac{\sqrt{\log(n)\log\left(n\log(n)/\zeta\right)}}{n^{\frac{1 - q}{2}}} + \frac{1}{n^{2q}}.
\end{align*}
To minimize the RHS over $q > 0$, we need to look at the intersection of the lines $\frac{1 - q}{2}$ and $2q$. This leads to the optimal value $q = \frac{1}{5}$. 
\end{remark}


\subsubsection{Proof of Theorem \ref{theorem:NonACC_lambda=0}}
\label{ThmNonACC0}

Here, $\mathcal{R}_{\gamma_{\mathcal{C}}}(\theta) = \frac{1}{2}(\theta - \theta^*)^T\Sigma(\theta - \theta^*) + \frac{\sigma_2^2}{2} + \frac{\gamma_{\mathcal{C}}||\theta||_2^2}{2}$. Since the global minimum of $\mathcal{R}_{\gamma_{\mathcal{C}}}$ is $\theta_{*} = (\Sigma + \gamma_{\mathcal{C}}I_p)^{-1}\Sigma\theta^*$, minimizing $\mathcal{R}_{\gamma_{\mathcal{C}}}$ over $\mathbb{R}^p$ is equivalent to minimizing over $\mathcal{C} = \mathbb{B}_2\left(D\right)$, with $D \geq ||(\Sigma + \gamma_{\mathcal{C}}I_p)^{-1}\Sigma\theta^*||_2$. From Lemma \ref{lemma:geri}, we have a gradient estimator $g(\theta)$ with
\begin{align*}
&\alpha(\widetilde{n}, \widetilde{\zeta}) \asymp \sqrt{\frac{\log(1/\widetilde{\zeta})}{\widetilde{n}}},
&\beta(\widetilde{n}, \widetilde{\zeta}) \asymp \sqrt{\frac{(1 + \sigma_2^2)\log(1/\widetilde{\zeta})}{\widetilde{n}}},
\end{align*}
since $\lambda_{\min}(\Sigma) = 0$. Thus, by a union bound, we have 
\begin{align*}
 \mathbb{P}(\forall t, \ ||g(\theta_t, \mathcal{D}_n, \widetilde{\zeta}) - \nabla\mathcal{R}_{\gamma_{\mathcal{C}}}(\theta_t)||_2 \leq \alpha(\widetilde{n}, \widetilde{\zeta})||\theta_t - \theta_*||_2 + \beta(\widetilde{n}, \widetilde{\zeta})) \geq 1 - \zeta.
\end{align*}
On the latter event, using the notation $\alpha = \alpha(\widetilde{n}, \widetilde{\zeta})$ and $\beta = \beta(\widetilde{n}, \widetilde{\zeta})$ and ignoring the dependency in $g$ on the samples and $\widetilde{\zeta}$, we can bound the gradient $e_t := g(\theta_t)-\nabla\mathcal{R}_{\gamma_{\mathcal{C}}}(\theta_t)$ as
\begin{align*}
||g(\theta_t) - \nabla\mathcal{R}_{\gamma_{\mathcal{C}}}(\theta_t)||_2 = ||e_t||_2 \leq \alpha||\theta_t - \theta_*||_2 + \beta.
\end{align*}
Thus, we have
\begin{align*}
v_t^T\nabla\mathcal{R}_{\gamma_{\mathcal{C}}}(\theta_t) &\leq v^T\nabla\mathcal{R}_{\gamma_{\mathcal{C}}}(\theta_t) + (v - v_t)^Te_t \leq v^T\nabla\mathcal{R}_{\gamma_{\mathcal{C}}}(\theta_t) + ||v - v_t||_2||e_t||_2\\
&\leq v^T\nabla\mathcal{R}_{\gamma_{\mathcal{C}}}(\theta_t) + ||\mathcal{C}||_2(2\alpha D + \beta), \ \quad \forall v \in \mathcal{C}.
\end{align*}
Let $\Gamma_{\mathcal{R}_{\gamma_{\mathcal{C}}}}$ be the curvature constant of $\mathcal{R}_{\gamma_{\mathcal{C}}}$. We then have
\begin{align*}
v_t^T\nabla\mathcal{R}_{\gamma_{\mathcal{C}}}(\theta_t) &\leq \mathop{\min}\limits_{v \in \mathcal{C}}v^T\nabla\mathcal{R}_{\gamma_{\mathcal{C}}}(\theta_t) + \frac{1}{2}\frac{2}{t + 2}\Gamma_{\mathcal{R}_{\gamma_{\mathcal{C}}}}\frac{||\mathcal{C}||_2(2\alpha D + \beta)}{\Gamma_{\mathcal{R}_{\gamma_{\mathcal{C}}}}}(t + 2)\\
&\leq \mathop{\min}\limits_{v \in \mathcal{C}}v^T\nabla\mathcal{R}_{\gamma_{\mathcal{C}}}(\theta_t) + \frac{1}{2}\frac{2}{t + 2}\Gamma_{\mathcal{R}_{\gamma_{\mathcal{C}}}}\frac{||\mathcal{C}||_2(2\alpha D + \beta)}{\Gamma_{\mathcal{R}_{\gamma_{\mathcal{C}}}}}(T + 2).
\end{align*}
Thus, on the event with probability at least $1 - \zeta$, since $\mathcal{C}$ is compact and convex, by Lemma \ref{lemma:NonAccRel_FW}, we obtain
\begin{align*}
\mathcal{R}_{\gamma_{\mathcal{C}}}(\theta_T) - \mathcal{R}_{\gamma_{\mathcal{C}}}(\theta_{*}) &\leq \frac{2\Gamma_{\mathcal{R}_{\gamma_{\mathcal{C}}}}}{T + 2}\left(1 + \frac{||\mathcal{C}||_2(2\alpha D + \beta)}{\Gamma_{\mathcal{R}_{\gamma_{\mathcal{C}}}}}(T + 2)\right)\\
&= \frac{2\Gamma_{\mathcal{R}_{\gamma_{\mathcal{C}}}}}{T + 2} + 2||\mathcal{C}||_2(2\alpha D + \beta),
\end{align*}
with $\theta_{*} = (\Sigma + \gamma_{\mathcal{C}}I_p)^{-1}\Sigma\theta^*$. Now note that $\mathcal{R}_{\gamma_{\mathcal{C}}}(\theta)$ is a quadratic in $\theta$ with the second-order term given by $\frac{1}{2}\theta^T(\Sigma + \gamma_{\mathcal{C}}I_p)\theta = \frac{1}{2}\theta^T(\Sigma + \gamma_{\mathcal{C}}I_p)^{1/2}(\Sigma + \gamma_{\mathcal{C}}I_p)^{1/2}\theta$. By Remark $2$ in \cite{NOPL}, we have $\Gamma_{\mathcal{R}_{\gamma_{\mathcal{C}}}} \leq 4\mathop{\max}\limits_{\theta \in \mathcal{C}}\left\|(\Sigma + \gamma_{\mathcal{C}}I_p)^{1/2}\theta\right\|_2^2 \lesssim 1$. Thus, since $\gamma_{\mathcal{C}} \rightarrow 0$ and $2D = ||\mathcal{C}||_2, ||\theta^*||_2 \lesssim 1$, we obtain
\begin{align*}
\mathcal{R}_{\gamma_{\mathcal{C}}}(\theta_T) - \mathcal{R}_{\gamma_{\mathcal{C}}}(\theta_{*}) \lesssim \frac{1}{T} + (1 + \sigma_2)\sqrt{\frac{T\log(T/\zeta)}{n}},
\end{align*}
and since $T = n^{1/3}$, we have 
\begin{align*}
\mathcal{R}_{\gamma_{\mathcal{C}}}(\theta_T) - \mathcal{R}_{\gamma_{\mathcal{C}}}(\theta_{*}) \lesssim \frac{(1 + \sigma_2)\sqrt{\log(n/\zeta)}}{n^{1/3}}.
\end{align*}
Using the $\gamma_{\mathcal{C}}$-strong convexity of $\mathcal{R}_{\gamma_{\mathcal{C}}}$ and the fact that $\nabla\mathcal{R}_{\gamma_{\mathcal{C}}}(\theta_{*}) = 0$, we obtain
\begin{align*}
||\theta_T - \theta_{*}||_2 \lesssim \frac{(1 + \sigma_2)^{1/2}\log^{1/4}(n/\zeta)}{\gamma_{\mathcal{C}}^{1/2}n^{1/6}}.
\end{align*}
Now, note that since $\theta_{*} = (\Sigma + \gamma_{\mathcal{C}}I_p)^{-1}\Sigma\theta^*$, we obtain
\begin{align*}
||\theta_{*} - \theta^*||_2^2 = ||((S + \gamma_{\mathcal{C}}I_p)^{-1}S - I_p)P^T\theta^*||_2^2 \lesssim m\gamma_{\mathcal{C}}^2 + \left\|[P^T\theta^*]_{[(m+1):p]}\right\|_2^2,
\end{align*}
implying that
\begin{align}
\label{eq:NonACCFWl=0_bestgamma}
||\theta_T - \theta^*||_2 \lesssim \frac{(1 + \sigma_2)^{1/2}\log^{1/4}(n/\zeta)}{\gamma_{\mathcal{C}}^{1/2}n^{1/6}} + \sqrt{m}\gamma_{\mathcal{C}} + \left\|[P^T\theta^*]_{[(m+1):p]}\right\|_2.
\end{align}
For $\gamma_{\mathcal{C}} \gtrsim \frac{1}{n^{1/9}}$, obtain the desired bound.

\begin{remark}
Note that our choice for the value of $\gamma_{\mathcal{C}}$ in the bound on $||\theta_T - \theta^*||_2$ is based on the fact that the RHS quantity in inequality~\eqref{eq:NonACCFWl=0_bestgamma} is a decreasing function of $\gamma_{\mathcal{C}}$, for $\gamma_{\mathcal{C}}$ small enough, i.e., for $n$ large enough.

Additionally, we can comment on the choice of $\mathcal{C}$. We take $\mathcal{C}$ to be an $\ell_2$-ball with radius $D \geq ||(\Sigma + \gamma_{\mathcal{C}}I_p)^{-1}\Sigma\theta^*||_2$. In Theorem \ref{theorem:NonACC_lambda=0}, we take $\gamma_{\mathcal{C}} \gtrsim \frac{1}{n^{1/9}}$. In practice, if we pick $\gamma_{\mathcal{C}} = \frac{1}{n^{1/9}}$ and $D$ large enough, we can carry out the optimization from Theorem \ref{theorem:NonACC_lambda=0}.
\end{remark}


\subsubsection{Proof of Theorem \ref{theorem:ACCFW_lambda=0}}
\label{AppThmACCFW}

Note that since $\gamma_{\mathcal{C}} \leq \frac{c_{\mathcal{K}}}{2} < c_{\mathcal{K}}$, we have  $\mathcal{K} \subseteq \mathbb{B}_2(||\theta_*||_2)$. Recall that $\tau_u = \lambda_{\max}(\Sigma) + \gamma_{\mathcal{C}}$, $\tau_l = \gamma_{\mathcal{C}}$, and $\theta_* = (\Sigma + \gamma_{\mathcal{C}}I_p)^{-1}\Sigma\theta^*$. Following the same steps as in the proof of Theorem \ref{theorem:NonACC_lambda=0}, with probability at least $1 - \zeta$, we have at the $t^{\text{th}}$ step of Algorithm \ref{alg:RobPGDNFW} that
\begin{align*}
v_t^T\nabla\mathcal{R}_{\gamma_{\mathcal{C}}}(\theta_t) &\leq v^T\nabla\mathcal{R}_{\gamma_{\mathcal{C}}}(\theta_t) + (v - v_t)^Te_t \leq v^T\nabla\mathcal{R}_{\gamma_{\mathcal{C}}}(\theta_t) + ||v - v_t||_2||e_t||_2\\
&\leq v^T\nabla\mathcal{R}_{\gamma_{\mathcal{C}}}(\theta_t) + ||\mathcal{K}||_2(2\alpha ||\theta^*||_2 + \beta), \quad \forall v \in \mathcal{K}.
\end{align*}
Thus, we have
\begin{align*}
v_t^T\nabla\mathcal{R}_{\gamma_{\mathcal{C}}}(\theta_t) \leq \mathop{\min}\limits_{v \in \mathcal{K}}v^T\nabla\mathcal{R}_{\gamma_{\mathcal{C}}}(\theta_t) + ||\mathcal{K}||_2(2\alpha ||\theta^*||_2 + \beta).
\end{align*}
Now note that for $\theta \in \mathcal{K}$, and by the $\gamma_{\mathcal{C}}$-strong convexity of $\mathcal{R}_{\gamma_{\mathcal{C}}}$, we obtain
\begin{align}
\label{eq:HT_ACC_l=0LB}
||\nabla\mathcal{R}_{\gamma_{\mathcal{C}}}(\theta)||_2 &= ||\nabla\mathcal{R}_{\gamma_{\mathcal{C}}}(\theta) - \nabla\mathcal{R}_{\gamma_{\mathcal{C}}}(\theta_{*})||_2 \geq \gamma_{\mathcal{C}}||\theta - \theta_{*}||_2 \geq \gamma_{\mathcal{C}}(||\theta_{*}||_2 - ||\theta||_2) \notag\\
&\geq \gamma_{\mathcal{C}}\left(||(\Sigma + \gamma_{\mathcal{C}}I_p)^{-1}\Sigma\theta^*||_2 - ||(\Sigma + c_{\mathcal{K}}I_p)^{-1}\Sigma\theta^*||_2\right),
\end{align}
since $||\theta||_2 \leq ||(\Sigma + c_{\mathcal{K}}I_p)^{-1}\Sigma\theta^*||_2$ for all $\theta \in \mathcal{K}$. Also, the RHS of inequality~\eqref{eq:HT_ACC_l=0LB} is positive, since $\mathcal{K} \subsetneq \mathcal{C}$. Hence, using the decomposition of $\Sigma$, we obtain
\begin{align*}
||\nabla\mathcal{R}_{\gamma_{\mathcal{C}}}(\theta)||_2 \geq \gamma_{\mathcal{C}}\left(||(S + \gamma_{\mathcal{C}}I_p)^{-1}SP^T\theta^*||_2 - ||(S + c_{\mathcal{K}}I_p)^{-1}SP^T\theta^*||_2\right).
\end{align*}
Define $f : (0, c_{\mathcal{K}}] \rightarrow \mathbb{R}$ such that $f(z) = ||(S + zI_p)^{-1}SP^T\theta^*||_2$. We have, for $[P^T\theta^*]_j$ being the $j^{\text{th}}$ entry in $P^T\theta^*$, that
\begin{align}
\label{eq:l=0_f'}
&f(z) = \sqrt{\sum_{j = 1}^m\frac{S_{jj}^2[P^T\theta^*]_j^2}{(S_{jj} + z)^2}}, \notag\\
&|f'(z)| = \frac{\sum_{j = 1}^m\frac{S_{jj}^2[P^T\theta^*]_j^2}{(S_{jj} + z)^3}}{\sqrt{\sum_{j = 1}^m\frac{S_{jj}^2[P^T\theta^*]_j^2}{(S_{jj} + z)^2}}} \geq \frac{\frac{S_{mm}^2}{(S_{mm} + c_{\mathcal{K}})^3}\left\|[P^T\theta^*]_{[1:m]}\right\|_2^2}{\left\|[P^T\theta^*]_{[1:m]}\right\|_2}\notag\\ 
& \qquad \quad = \frac{S_{mm}^2\left\|[P^T\theta^*]_{[1:m]}\right\|_2}{(S_{mm} + c_{\mathcal{K}})^3}, \quad \forall z \in (0, c_{\mathcal{K}}].
\end{align}
Hence, by the Mean Value Theorem, using the lower bound on $|f'|$ and the fact that $\gamma_{\mathcal{C}} \leq \frac{c_{\mathcal{K}}}{2}$, we obtain
\begin{align*}
||\nabla\mathcal{R}_{\gamma_{\mathcal{C}}}(\theta)||_2 &\geq \gamma_{\mathcal{C}}\frac{S_{mm}^2\left\|[P^T\theta^*]_{[1:m]}\right\|_2}{(S_{mm} + c_{\mathcal{K}})^3}(c_{\mathcal{K}} - \gamma_{\mathcal{C}})\\
&\geq \gamma_{\mathcal{C}}\frac{S_{mm}^2\left\|[P^T\theta^*]_{[1:m]}\right\|_2c_{\mathcal{K}}}{2(S_{mm} + c_{\mathcal{K}})^3} \geq u, \quad \forall \theta \in \mathcal{K}.
\end{align*}
Thus, with probability at least $1 - \zeta$, since $\mathcal{K}$ is compact and $\alpha_{\mathcal{K}}$-strongly convex by Lemma \ref{lemma:STR_CONV_BALL}, and $\mathcal{R}_{\gamma_{\mathcal{C}}}$ is $(\lambda_{\max}(\Sigma) + \gamma_{\mathcal{C}})$-smooth, we are in the context of Theorem \ref{theorem:ACCFWV3}. Let $\theta_{*, \mathcal{K}}$ be the minimum of $\mathcal{R}_{\gamma_{\mathcal{C}}}$ in $\mathcal{K}$. Thus, for the choice of $\eta$ in the theorem hypothesis, Theorem \ref{theorem:ACCFWV3} implies that 
\begin{align*}
\mathcal{R}_{\gamma_{\mathcal{C}}}(\theta_t) - \mathcal{R}_{\gamma_{\mathcal{C}}}(\theta_{*, \mathcal{K}}) \leq \left(\mathcal{R}_{\gamma_{\mathcal{C}}}(\theta_0) - \mathcal{R}_{\gamma_{\mathcal{C}}}(\theta_{*, \mathcal{K}})\right)c^t + \frac{3\eta ||\mathcal{K}||_2(2\alpha ||\theta^*||_2 + \beta)}{2(1 - c)},
\end{align*}
with $c = \max\left\{\frac{1}{2}, 1 - \frac{\alpha_{\mathcal{K}} u}{8(\lambda_{\max}(\Sigma) + \gamma_{\mathcal{C}})}\right\}$. Note that since $\left\|\left(\Sigma + C_1c_{\mathcal{K}}I_p\right)^{-1}\Sigma\theta^*\right\|_2 \leq K \leq \left\|\left(\Sigma + c_{\mathcal{K}}I_p\right)^{-1}\Sigma\theta^*\right\|_2$, $\alpha_{\mathcal{K}} = \frac{1}{K}$, and $\gamma_{\mathcal{C}}c_{\mathcal{K}} \lesssim u \leq \gamma_{\mathcal{C}}\frac{S_{mm}^2\left\|[P^T\theta^*]_{[1:m]}\right\|_2c_{\mathcal{K}}}{2(S_{mm} + c_{\mathcal{K}})^3}$, we have $\alpha_{\mathcal{K}} u \asymp \gamma_{\mathcal{C}}c_{\mathcal{K}}$. Thus, $\eta \asymp \gamma_{\mathcal{C}}c_{\mathcal{K}}$ and $\frac{1}{1 - c} \asymp \frac{1}{\gamma_{\mathcal{C}}c_{\mathcal{K}}}$. By smoothness, because $\nabla\mathcal{R}_{\gamma_{\mathcal{C}}}(\theta_{*}) = 0$ and $\lambda_{\max}(\Sigma) + \gamma_{\mathcal{C}} \lesssim 1$, we then obtain 
\begin{align*}
\mathcal{R}_{\gamma_{\mathcal{C}}}(\theta_0) - \mathcal{R}_{\gamma_{\mathcal{C}}}(\theta_{*, \mathcal{K}}) &\lesssim ||\nabla\mathcal{R}_{\gamma_{\mathcal{C}}}(\theta_{*, \mathcal{K}}) - \nabla\mathcal{R}_{\gamma_{\mathcal{C}}}(\theta_*)||_2||\theta_0 - \theta_{*, \mathcal{K}}||_2\\
& \quad + (\lambda_{\max}(\Sigma) + \gamma_{\mathcal{C}})||\theta_0 - \theta_{*, \mathcal{K}}||_2^2\\
&\lesssim (||\theta_{*, \mathcal{K}}||_2 + ||\theta_*||_2)||\mathcal{K}||_2 + ||\mathcal{K}||_2^2\\
&\lesssim ||\theta^*||_2||\mathcal{K}||_2 + ||\mathcal{K}||_2^2 \lesssim 1.
\end{align*}
Thus, at iteration $T$, we have
\begin{align*}
\mathcal{R}_{\gamma_{\mathcal{C}}}(\theta_T) - \mathcal{R}_{\gamma_{\mathcal{C}}}(\theta_{*, \mathcal{K}}) \lesssim c^T + \sqrt{\frac{\log(1/\widetilde{\zeta})}{\widetilde{n}}} + \sqrt{\frac{(1 + \sigma_2)^2\log(1/\widetilde{\zeta})}{\widetilde{n}}}.
\end{align*}
Note that $\log(1/c) \asymp \gamma_{\mathcal{C}}c_{\mathcal{K}}\log\left((1 - \gamma_{\mathcal{C}}c_{\mathcal{K}})^{1/\gamma_{\mathcal{C}}c_{\mathcal{K}}}\right) \asymp\gamma_{\mathcal{C}}c_{\mathcal{K}}$. Hence, since $T = \log_{1/c}\left(n\right)$, we obtain
\begin{align*}
\mathcal{R}_{\gamma_{\mathcal{C}}}(\theta_T) - \mathcal{R}_{\gamma_{\mathcal{C}}}(\theta_{*, \mathcal{K}}) \lesssim \frac{1}{n} + (1 + \sigma_2)\sqrt{\frac{\log(n)\log\left(\log(n)/\gamma_{\mathcal{C}}\zeta\right)}{\gamma_{\mathcal{C}}c_{\mathcal{K}}n}}.
\end{align*}
Now define $\mathcal{A} : [0, 1] \rightarrow \mathbb{R}$ by $\mathcal{A}(\lambda) = \mathcal{R}_{\gamma_{\mathcal{C}}}(\lambda\theta_*)$. Note that
$$\mathcal{A}(0) = \mathcal{R}_{\gamma_{\mathcal{C}}}(0) \geq \mathcal{R}_{\gamma_{\mathcal{C}}}(\theta_{*, \mathcal{K}}) \geq \mathcal{R}_{\gamma_{\mathcal{C}}}(\theta_*) = \mathcal{A}(1).$$
Hence, by the continuity of $\mathcal{A}$, the Intermediate Value Theorem implies that there exists $\lambda_* \in [0, 1]$ such that $\mathcal{A}(\lambda_*) = \mathcal{R}_{\gamma_{\mathcal{C}}}(\lambda_*\theta_*) = \mathcal{R}_{\gamma_{\mathcal{C}}}(\theta_{*, \mathcal{K}})$. If $\lambda_*\theta_*$ is in the interior of $\mathcal{K}$, then $\nabla\mathcal{R}_{\gamma_{\mathcal{C}}}(\lambda_*\theta_*) = 0$, so $\lambda_*\theta_*$ is a global minimizer. This is a contradiction, since $\theta_*$ is the unique global minimizer of $\mathcal{R}_{\gamma_{\mathcal{C}}}$ and this lies strictly outside $\mathcal{K}$. If $\lambda_*\theta_*$ is at the boundary or outside $\mathcal{K}$, then $$||\lambda_*\theta_* - \theta_*||_2 \leq ||\theta_*||_2 - K \leq \left\|\left(\Sigma + \gamma_{\mathcal{C}}I_p\right)^{-1}\Sigma\theta^*\right\|_2 - \left\|\left(\Sigma + C_1c_{\mathcal{K}}I_p\right)^{-1}\Sigma\theta^*\right\|_2.$$ By inequality \eqref{eq:l=0_f'}, the Mean Value Theorem, and the fact that $C_1c_{\mathcal{K}} > c_{\mathcal{K}} > \frac{c_{\mathcal{K}}}{2} \geq \gamma_{\mathcal{C}} \geq \frac{c_{\mathcal{K}}}{4}$, there exists some $z_* \in \left[\gamma_{\mathcal{C}}, C_1c_{\mathcal{K}}\right]$ such that
\begin{align*}
||\lambda_*\theta_* - \theta_*||_2 &\leq \frac{\sum_{j = 1}^m\frac{S_{jj}^2[P^T\theta^*]_j^2}{(S_{jj} + z_*)^3}}{\sqrt{\sum_{j = 1}^m\frac{S_{jj}^2[P^T\theta^*]_j^2}{(S_{jj} + z_*)^2}}}(C_1c_{\mathcal{K}} - \gamma_{\mathcal{C}}) \leq \frac{\frac{\left\|[P^T\theta^*]_{[1:m]}\right\|_2^2}{S_{mm}}}{\sqrt{\frac{S^2_{mm}\left\|[P^T\theta^*]_{[1:m]}\right\|_2^2}{(S_{11} + C_1c_{\mathcal{K}})^2}}}\left(C_1 - \frac{1}{4}\right)c_{\mathcal{K}}\\
&= \frac{(S_{11} + C_1c_{\mathcal{K}})\left\|[P^T\theta^*]_{[1:m]}\right\|_2}{S^2_{mm}}\left(C_1 - \frac{1}{4}\right)c_{\mathcal{K}} \asymp c_{\mathcal{K}}.
\end{align*}
Additionally, by the $(\lambda_{\max}(\Sigma) + \gamma_{\mathcal{C}})$-smoothness of $\mathcal{R}_{\gamma_{\mathcal{C}}}$, and using the facts that $\nabla \mathcal{R}_{\gamma_{\mathcal{C}}}(\theta_{*}) = 0$ and $\lambda_{\max}(\Sigma) + \gamma_{\mathcal{C}} \lesssim 1$, we have
\begin{align*}
\mathcal{R}_{\gamma_{\mathcal{C}}}(\theta_{*, \mathcal{K}}) - \mathcal{R}_{\gamma_{\mathcal{C}}}(\theta_{*}) &= \mathcal{R}_{\gamma_{\mathcal{C}}}(\lambda_*\theta_*) - \mathcal{R}_{\gamma_{\mathcal{C}}}(\theta_{*}) \lesssim ||\lambda_*\theta_{*} - \theta_*||_2^2 \lesssim c_{\mathcal{K}}^2.
\end{align*}
Therefore, we have
\begin{align*}
\mathcal{R}_{\gamma_{\mathcal{C}}}(\theta_T) - \mathcal{R}_{\gamma_{\mathcal{C}}}(\theta_{*}) &\lesssim \frac{1}{n} + (1 + \sigma_2)\sqrt{\frac{\log(n)\log\left(\log(n)/\gamma_{\mathcal{C}}\zeta\right)}{\gamma_{\mathcal{C}}c_{\mathcal{K}}n}} + c_{\mathcal{K}}^2.
\end{align*}
Using the $\gamma_{\mathcal{C}}$-strong convexity of $\mathcal{R}_{\gamma_{\mathcal{C}}}$ and the fact that $\nabla\mathcal{R}_{\gamma_{\mathcal{C}}}(\theta_{*}) = 0$, we obtain
\begin{align*}
||\theta_T - \theta_{*}||_2 \lesssim \frac{1}{\gamma_{\mathcal{C}}^{1/2}n^{1/2}} + (1 + \sigma_2)^{1/2}\frac{\log^{1/4}(n)\log^{1/4}\left(\log(n)/\gamma_{\mathcal{C}}\zeta\right)}{c_{\mathcal{K}}^{1/4}\gamma_{\mathcal{C}}^{3/4}n^{1/4}} + \frac{c_{\mathcal{K}}}{\gamma_{\mathcal{C}}^{1/2}},
\end{align*}
and since $||\theta_{*} - \theta^*||_2^2 \lesssim m\gamma_{\mathcal{C}}^2 + \left\|[P^T\theta^*]_{[(m+1):p]}\right\|_2^2$, we then have
\begin{align}
\label{eq:HT_ACCFWl=0_best_gc}
||\theta_T - \theta^*||_2 &\lesssim \frac{1}{\gamma_{\mathcal{C}}^{1/2}n^{1/2}} + (1 + \sigma_2)^{1/2}\frac{\log^{1/4}(n)\log^{1/4}\left(\log(n)/\gamma_{\mathcal{C}}\zeta\right)}{c_{\mathcal{K}}^{1/4}\gamma_{\mathcal{C}}^{3/4}n^{1/4}} + \frac{c_{\mathcal{K}}}{\gamma_{\mathcal{C}}^{1/2}} + \sqrt{m}\gamma_{\mathcal{C}} \notag\\
& \quad + \left\|[P^T\theta^*]_{[(m+1):p]}\right\|_2.
\end{align}
Then, for $\gamma_{\mathcal{C}} \geq \frac{c_{\mathcal{K}}}{4}$, we obtain
\begin{align*}
||\theta_T - \theta^*||_2 &\lesssim (1 + \sigma_2)^{1/2}\frac{\log^{1/4}(n)\log^{1/4}\left(n/\zeta\right)}{c_{\mathcal{K}}^{1/4}n^{1/4}}\\
& \quad + \left\|[P^T\theta^*]_{[(m+1):p]}\right\|_2 + \left\|[P^T\theta^*]_{[(m+1):p]}\right\|_2^{1/2},
\end{align*}
as required.


\subsection{Proofs of the Auxiliary Results from Section \ref{sec:Acc_FW_Method_All}}\label{sec:Proofs of the Auxiliary Results_ACC_FW}

Here, we present the proofs of the auxiliary results from Section \ref{sec:Acc_FW_Method_All}.


\subsubsection{Proof of Lemma \ref{lemma:ACC_PRIV_FW_STEP}}
\label{AppLemACC}

Observe that for all $\theta$, $z_1, \dots, z_n$, $z'_1$, we have
\begin{align*}
\left\|\frac{1}{n}\sum_{j = 1}^{n - 1}\nabla \mathcal{L}(\theta, z_j) - \frac{1}{n}\sum_{j = 1}^{n - 1}\nabla \mathcal{L}(\theta, z_j) + \frac{1}{n}\nabla \mathcal{L}(z_1, \theta) - \frac{1}{n}\nabla \mathcal{L}(z'_1, \theta)\right\|_2 \leq \frac{2L_2}{n},
\end{align*}
since the loss is $L_2$-Lipschitz. Hence, the sensitivity is bounded above by $\frac{2L_2}{n}$, and by Lemma \ref{lemma:mGDP}, since $\epsilon < 2\sqrt{2T\log(2/\delta)}$ and $\delta < 2T$, each step of Algorithm \ref{alg:PrivFWERM} is $\left(\frac{\epsilon}{2\sqrt{2T\log(2/\delta)}}, \frac{\delta}{2T}\right)$-DP. Hence, using Lemma \ref{lemma:compGDP}, i.e., the advanced composition result, we obtain that $\theta_T$ is $\left(\frac{\epsilon}{2} + \frac{\sqrt{T}\epsilon}{2\sqrt{2\log(2/\delta)}}(e^{\epsilon/2\sqrt{2T\log(2/\delta)}} - 1), \delta\right)$-DP. Finally, for $\epsilon \leq 0.9$, using Corollary \ref{cor:ADV_COMP_0.9}, we conclude that $\theta_T$ is $(\epsilon, \delta)$-DP. 


\subsubsection{Proof of Proposition \ref{prop:achievab_UP_ERM}}
\label{AppPropERM}

Let $\mathbb{X} \in \mathbb{R}^{n \times p}$ be the matrix with $i^{\text{th}}$ row being $x_i$, for all $i \in [n]$. Let $\mathbb{Y} = (y_1, \dots, y_n)^T$ and $\mathbb{W}^{(p)} = \left(w_1^{(p)}, \dots, w_p^{(p)}\right)^T$. Let $v \in \mathbb{R}^p$ be such that $||v||_2 = 1$. Then we have
\begin{align*}
v^T\frac{\mathbb{X}^T\mathbb{X}}{n}v &= v^T\Sigma v + v^T\left(\frac{\mathbb{X}^T\mathbb{X}}{n} - \Sigma\right)v \leq \lambda_{\max}\left(\Sigma\right) + \left\|\frac{\mathbb{X}^T\mathbb{X}}{n} - \ \Sigma\right\|_2\\
&\leq C_2 + \left\|\frac{\mathbb{X}^T\mathbb{X}}{n} - \ \Sigma\right\|_2,
\end{align*}
and
\begin{align*}
v^T\frac{\mathbb{X}^T\mathbb{X}}{n}v &= v^T\Sigma v - v^T\left(\Sigma - \frac{\mathbb{X}^T\mathbb{X}}{n}\right)v \geq \lambda_{\min}\left(\Sigma\right) - \left\|\frac{\mathbb{X}^T\mathbb{X}}{n} - \ \Sigma\right\|_2\\
&\geq C_1 - \left\|\frac{\mathbb{X}^T\mathbb{X}}{n} - \ \Sigma\right\|_2.
\end{align*}
Note that since $||x_1||_{\infty} \leq 1$, we have $||x_1||_2 \leq \sqrt{p}$.
By Lemma \ref{lemma:Concxx^T}, we have
\begin{align*}
\mathbb{P}\left(\left\|\frac{\mathbb{X}^T\mathbb{X}}{n} - \ \Sigma\right\|_2 > \frac{C_1}{2}\right) = \mathbb{P}\left(\left\|\frac{1}{n}\sum_{i = 1}^nx_ix_i^T - \ \Sigma\right\|_2 > \frac{C_1}{2}\right) \leq 2pe^{\frac{-nC_1^2}{8p\left(C_2+ C_1/3\right)}} \rightarrow 0
\end{align*}
as $p \rightarrow \infty$, since $n = \widetilde{\Omega}(p^{c_1})$ and $c_1 > 1$. Hence,  with probability at least $1 - 2pe^{\frac{-nC_1^2}{8p\left(C_2+ C_1/3\right)}}$, we have $\left\|\frac{\mathbb{X}^T\mathbb{X}}{n} - \ \Sigma\right\|_2 \leq \frac{C_1}{2}$, so
$$\frac{C_1}{2} \leq \frac{\lambda_{\min}\left(\mathbb{X}^T\mathbb{X}\right)}{n} \leq \frac{\lambda_{\max}\left(\mathbb{X}^T\mathbb{X}\right)}{n} \leq \frac{2C_2 + C_1}{2},$$
since $||v||_2 = 1$ was arbitrary. Let $\Omega_1$ be this event which occurs with probability at least $1 - 2pe^{\frac{-nC_1^2}{8p\left(C_2+ C_1/3\right)}}$. Now recall that $\left\{w_i^{(p)}\right\}_{i = 1}^n$ are \iid and $w_i^{(p)} \in \mathcal{G}\left(\sigma^2(p)\right)$, for all $i \in [n]$. Hence, by a union bound, Lemma \ref{lemma:SG_conc_variable}, and the fact that $\mathbb{E}\left[\mathbb{W}^{(p)}\right] = 0$, we have
\begin{align*}
\mathbb{P}\left(\left\|\mathbb{W}^{(p)}\right\|_\infty > \frac{D(p)\sqrt{n}}{p^{5/8}}\right) &\leq \sum_{j = 1}^p\mathbb{P}\left(\left|w_j^{(p)}\right| > \frac{D(p)\sqrt{n}}{p^{5/8}}\right) \leq 2\sum_{j = 1}^pe^{-\frac{nD^2(p)}{2p^{5/4}\sigma^2(p)}}\\
&= 2pe^{-\frac{nD^2(p)}{2p^{5/4}\sigma^2(p)}} \rightarrow 0
\end{align*}
as $p \rightarrow \infty$, since $n = \widetilde{\Omega}\left(\frac{p^{c_2}\sigma^2(p)}{D^2(p)}\right)$ and $c_2 > \frac{5}{4}$. Hence, with probability at least $1 - 2pe^{-\frac{nD^2(p)}{2p^{5/4}\sigma^2(p)}}$, we have $\left\|\mathbb{W}^{(p)}\right\|_\infty \leq \frac{D(p)\sqrt{n}}{p^{5/8}}$. Let $\Omega_2$ be the event that the latter bound holds. Let $\Omega_3 = \Omega_1 \cap \Omega_2$, so $\mathbb{P}(\Omega_3) \geq 1 - 2pe^{\frac{-nC_1^2}{8p\left(C_2+ C_1/3\right)}} - 2pe^{-\frac{nD^2(p)}{2p^{5/4}\sigma^2(p)}}$.

Let us now work on $\Omega_3$. Note that $\beta_{\mathcal{L}} = \frac{\lambda_{\max}\left(\mathbb{X}^T\mathbb{X}\right)}{n} \leq \frac{2C_2 + C_1}{2}$. Fix $\theta \in \mathcal{C}$ arbitrary. Since $||\theta^*||_2 \geq (2S_1(2C_2/C_1 + 1) + 1)D(p)$, we have
\begin{align*}
\frac{\alpha_{\mathcal{C}}||\nabla \mathcal{L}(\theta, \mathcal{D}_n)||_2}{\beta_{\mathcal{L}}} &\geq \frac{2||\nabla \mathcal{L}(\theta, \mathcal{D}_n)||_2}{D(p)(2C_2 + C_1)} = \frac{2\left\|\mathbb{X}^T\mathbb{Y} - \mathbb{X}^T\mathbb{X}\theta\right\|_2}{D(p)(2C_2 + C_1)n}\\
&= \frac{2\left\|\mathbb{X}^T\mathbb{X}(\theta^* - \theta) - \mathbb{X}^T\mathbb{W}^{(p)}\right\|_2}{D(p)(2C_2 + C_1)n}\\
&\geq \frac{2\lambda_{\min}\left(\mathbb{X}^T\mathbb{X}\right)(||\theta^*||_2 - ||\theta||_2)}{D(p)(2C_2 + C_1)n} - \frac{2||\mathbb{X}/\sqrt{n}||_2\left\|\mathbb{W}^{(p)}\right\|_2}{D(p)(2C_2 + C_1)\sqrt{n}}\\
&\geq \frac{4S_1\lambda_{\min}\left(\mathbb{X}^T\mathbb{X}\right)}{C_1n} - \frac{2\sqrt{\beta_{\mathcal{L}}}\left\|\mathbb{W}^{(p)}\right\|_2}{D(p)(2C_2 + C_1)\sqrt{n}}\\
&\geq \frac{4S_1\lambda_{\min}\left(\mathbb{X}^T\mathbb{X}\right)}{C_1n} - \frac{\sqrt{2p}\left\|\mathbb{W}^{(p)}\right\|_\infty}{D(p)\sqrt{(2C_2 + C_1)n}},
\end{align*}
since the $\ell_2$-norm is less than $\sqrt{p}$ times the $\ell_{\infty}$-norm, and since $\beta_{\mathcal{L}} \leq \frac{2C_2 + C_1}{2}$. Again, since we are on $\Omega_3$, we obtain
\begin{align*}
\frac{\alpha_{\mathcal{C}}||\nabla \mathcal{L}(\theta, \mathcal{D}_n)||_2}{\beta_{\mathcal{L}}} &\geq 2S_1 - \frac{\sqrt{2}}{\sqrt{2C_2 + C_1}}\frac{\sqrt{p}}{D(p)\sqrt{n}}\frac{D(p)\sqrt{n}}{p^{5/8}} = 2S_1 - \frac{\sqrt{2}}{\sqrt{2C_2 + C_1}}\frac{1}{p^{1/8}} \geq S_1,
\end{align*}
as required, since $p \geq \left(\frac{\sqrt{2}}{S_1\sqrt{2C_2 + C_1}}\right)^8$.

Finally, let us prove that the conditions \eqref{eq:model_distr_free_data} can be satisfied if $w^{(p)}$ follows a $N(0, \sigma^2(p))$ distribution truncated in the interval $[-1 - \sqrt{p}K_1(p), 1 + \sqrt{p}K_1(p)]$. We then have $\mathbb{E}\left[w^{(p)}\right] = 0$ and $\left|w^{(p)}\right| \leq 1 + \sqrt{p}K_1(p)$, with $w^{(p)}$ having full support on $[-1 - \sqrt{p}K_1(p), 1 + \sqrt{p}K_1(p)]$. By Theorem $2.1$ in \cite{SUB_GAUSS_TRUNC}, we know that $w^{(p)}$ is sub-Gaussian with parameter
\begin{align*}
\sigma^2(p)\left(1 - \frac{2(1 + \sqrt{p}K_1(p))}{\sigma(p)}\frac{\phi\left(\frac{1 + \sqrt{p}K_1(p)}{\sigma(p)}\right)}{2\Phi_0\left(\frac{1 + \sqrt{p}K_1(p)}{\sigma(p)}\right) - 1}\right),
\end{align*}
which is less than $\sigma^2(p)$. Here, $\phi$ and $\Phi_0$ are the standard normal pdf and cdf, respectively. Hence, $w^{(p)} \in \mathcal{G}\left(\sigma^2(p)\right)$.

\begin{remark}\label{remark:DIST_FREE_HP_CONST}
In Proposition \ref{prop:achievab_UP_ERM}, we assumed that $C_1 \leq \lambda_{\min}(\Sigma) \leq \lambda_{\max}(\Sigma) \leq C_2 \leq 1$. Observe that since $||x||_{\infty} \leq 1$, the variance of each entry in $x$ is at most $\left(\frac{1 + 1}{2}\right)^2 = 1$, so the choice of $0 < C_1 \leq C_2 \leq 1$ ensures that the variance of each entry of $x$ stays below $1$.
%
\end{remark}

\begin{remark}
In Proposition \ref{prop:achievab_UP_ERM}, we asked for
$$(2S_1(2C_2/C_2 + 1) + 1)D(p) \leq ||\theta^*||_2 \leq K_1(p),$$
while in Theorem \ref{theorem:UPRidge}, we optimize over $\mathcal{C} = \mathbb{B}_2\left(D(p)\right)$. Thus, the lower bound on $||\theta^*||_2$ scales as $D(p)$, even though the constants place $\theta^*$ slightly outside $\mathcal{C}$. However, since $K_1(p), D(p) \rightarrow 0$, we have $||\theta_T - \theta^*||_2 \lesssim D(p) + K_1(p) = O\left(\max\left\{D(p), K_1(p)\right\}\right) \rightarrow 0$ as $p \rightarrow \infty$. 
\end{remark}


\subsubsection{Proof of Proposition \ref{prop:MatchData}}
\label{AppPropMatch}

For all $j \in [n]$, we have $y_j \in \{0, 1\}$ and $x_j \in \left\{-1, 1\right\}^p$, so clearly, $|y_j| \leq 1$ and $||x_j||_{\infty} \leq 1$. We now show that $\left\|\sum_{j = 1}^n x_jx_j^T\right\|_2 \leq wp(1 + \tau)$. The matrix $X_{(-i)}$ with the $x^j_{(-i)}$'s as rows has at least $(1 - \tau)p$ consensus columns, implying that
\begin{align*}
\left\|\sum_{j = 1}^n x_jx_j^T\right\|_2 &\leq \left\|\sum_{j = 1}^w x^j_{(-i)} (x^j_{(-i)})^T\right\|_2 + \left\|\sum_{j = 1}^k z_jz_j^T\right\|_2\\
&\leq \sum_{j = 1}^w\left\|x^j_{(-i)}\right\|_2^2 + k||I_p||_2 = \sum_{j = 1}^w p + k = (1 + \tau)wp,
\end{align*}
as needed, where we used the facts that $x^j_{(-i)}$ has all entries equal to either $-1$ or $1$, and $Z^TZ = kI_p$. We now prove that $\left\|\sum_{j = 1}^n y_jx_j\right\|_2 \geq w\sqrt{(1 - \tau)p}$. Since the target variables of the $z_j$'s are $0$ and those of the $x^j_{(-i)}$'s are $1$, we have
\begin{align*}
\left\|\sum_{j = 1}^n y_jx_j\right\|_2 = \left\|\sum_{j = 1}^w x^j_{(-i)}\right\|_2.
\end{align*}
In the sum $\sum_{j = 1}^w x^j_{(-i)}$, we have either $-w$ or $w$ in the positions of the consensus columns. Since $X_{(-i)}$ has at least $(1 - \tau)$ consensus columns, we have
\begin{align*}
\left\|\sum_{j = 1}^n y_jx_j\right\|_2 \geq w\sqrt{\sum_{j = 1}^{(1 - \tau)p} 1} = w\sqrt{(1 - \tau)p}. 
\end{align*}
Now fix $\theta \in \mathcal{C}$. We have 
\begin{align*}
\frac{\alpha_{\mathcal{C}}||\nabla \mathcal{L}(\theta, \mathcal{D}_n)||_2}{\beta_{\mathcal{L}}} &= \frac{\sqrt{p}\left\|\sum_{j = 1}^ny_jx_j - x_jx_j^T\theta\right\|_2}{\alpha_1\left\|\sum_{j = 1}^nx_jx_j^T\right\|_2} \geq \frac{\sqrt{p}\left(w\sqrt{(1 - \tau)p} - (1 + \tau)wp||\theta||_2\right)}{\alpha_1(1 + \tau)wp}\\
&\geq \frac{\sqrt{p}\left(w\sqrt{(1 - \tau)p} - (1 + \tau)\alpha_1w\sqrt{p}\right)}{\alpha_1(1 + \tau)wp} = \frac{\sqrt{1 - \tau} - (1 + \tau)\alpha_1}{\alpha_1(1 + \tau)} \geq S_1.
\end{align*}
Since $\theta \in \mathcal{C}$ was arbitrary, we can take an infimum over all $\theta \in \mathcal{C}$ to obtain $\mathop{\inf}\limits_{\theta \in \mathcal{C}}\frac{\alpha_{\mathcal{C}}||\nabla \mathcal{L}(\theta, \mathcal{D}_n)||_2}{\beta_{\mathcal{L}}} \geq S_1$. Thus, the dataset in the hypothesis satisfies the inequalities \eqref{eq:MatchEq}, as required.


\subsubsection{Proof of Proposition \ref{prop:Iter_Rate_Inc_C_Non}}
\label{AppPropIter}

Looking at the proof of Lemma \ref{lemma:NOPLconvDPFW} in \cite{NOPL}, they first obtain a result with high probability before passing to a result in expectation. Since in our setting, $p$ and $||\theta^*||_2$ are absolute constants, the curvature constant $\Gamma_{\mathcal{L}}$ of $\mathcal{L}$ and the Gaussian width $G_{\mathbb{B}_2(||\theta^*||_2)}$ of $\mathbb{B}_2(||\theta^*||_2)$ are also absolute constants. Hence, for $\zeta \in (0, 1)$, the arguments in \cite{NOPL} imply that with probability at least $1 - \zeta$, we have
\begin{align*}
\mathcal{L}(\theta_T, \mathcal{D}_n) - \mathcal{L}(\theta_{B, n}, \mathcal{D}_n) = \widetilde{O}\left(\frac{\log(T/\zeta)}{(n\epsilon)^{2/3}}\right) = \widetilde{O}\left(\frac{\log(n\epsilon/\zeta)}{(n\epsilon)^{2/3}}\right),
\end{align*}
where $\theta_{B, n} \in \mathop{\arg\min}\limits_{\theta \in \mathbb{B}_2(||\theta^*||_2)}\mathcal{L}(\theta_{B, n}, \mathcal{D}_n)$ and $\theta_T$ is the output of Algorithm \ref{alg:PrivNon-accERM}. Denote this high-probability event by $\Omega_7$. In the proof of Theorem \ref{theorem:Iter_Rate_Inc_C}, we showed the existence of an absolute constant $C_1 > 0$ and an event $\Omega_6$, such that $\mathbb{P}(\Omega_6) \geq 1 - \zeta$ and $\mathcal{L}(\theta, \mathcal{D}_n)$ is $\frac{\Phi''(L_x||\theta^*||_2)\lambda_{\min}(\Sigma)}{2}$-strongly convex over $\mathbb{B}_2(||\theta^*||_2)$, for $n > C_1\log(2p/\zeta)$. Moreover, in the proof of Theorem \ref{theorem:ERM_UP_C_Inc}, we showed the existence of $C_2, T_\zeta, N_\zeta > 0$ and an event
\begin{align*}
\Omega_5 = \left\{||\theta_{B, n} - \theta^*||_2 \leq \frac{T_\zeta\log(n)}{\sqrt{n}}\right\}
\end{align*}
such that $\mathbb{P}(\Omega_5) \geq 1 - \zeta$, for $n \geq \max\left\{C_2, N_\zeta\right\}$. Let $\Omega_8 = \Omega_5 \cap \Omega_6 \cap \Omega_7$, so $\mathbb{P}(\Omega_8) \geq 1 - 3\zeta$. On the event $\Omega_8$, for $n > \max\left\{C_1\log(2p/\zeta), C_2, N_\zeta\right\}$, we see that since $\frac{\Phi''(L_x||\theta^*||_2)\lambda_{\min}(\Sigma)}{2} \asymp 1$ and $\theta_{B, n}$ is a minimizer over $\mathbb{B}_2(||\theta^*||_2)$, strong convexity implies that
\begin{align*}
||\theta_T - \theta_{B, n}||_2 = \widetilde{O}\left(\frac{\log^{1/2}(n\epsilon/\zeta)}{(n\epsilon)^{1/3}}\right).
\end{align*}
Using the triangle inequality, we then have
\begin{align*}
||\theta_T - \theta^*||_2 \leq ||\theta_T - \theta_{B, n}||_2 + ||\theta_{B, n} - \theta^*||_2 = \widetilde{O}\left(\frac{T_\zeta}{\sqrt{n}} + \frac{\log^{1/2}(n\epsilon/\zeta)}{(n\epsilon)^{1/3}}\right),
\end{align*}
as required.


\subsubsection{Gradient bound for heavy-tailed data}
\label{AppLemHT}

We now state and prove the main result about gradient estimators used in Algorithm \ref{alg:HTGE}. We provide a proof since we aim to correct the aspect related to the choice of $b$ in \cite{RE}, as discussed in Section \ref{sec:Biased Parameter Estimation HT}.

\begin{lemma}
\label{lemma:htlemma}
Let $\mathcal{L}$ be a generic loss. Suppose $\mathcal{D}_n = \{z_i\}_{i = 1}^n$ are \iid samples from a heavy-tailed distribution. Then Algorithm \ref{alg:HTGE}, with $S = \{\nabla \mathcal{L}(\theta; z_i)\}_{i=1}^n$ and $\zeta \in (0, 1)$ such that $b \leq n/2$, returns for a fixed $\theta \in \mathbb{R}^p$ an estimate $\widehat{\mu}$ such that with probability at least $1 - \zeta$, we have
\begin{align*}
||\widehat{\mu} - \nabla\mathcal{R}(\theta)||_2 \leq 11\sqrt{\frac{Tr(\mathrm{\mathrm{Cov}}(\nabla \mathcal{L}(\theta, z)))\log(1.4/ \zeta)}{n}}.
\end{align*}
\end{lemma}

\begin{proof}

We will use the following geometric lemma:

\begin{lemma}[\cite{Minsker}]
\label{lemma:6}
Let $\{\mu_i\}_{i = 1}^b$ be points in $\mathbb{R}^p$ and let $\widehat{\mu} = \mathop{\arg \min\limits_{\mu}} \sum_{i=1}^{b} \| \mu - \mu_i \|_2$ be the geometric median of the points. For $\gamma_1 \in \left(0, \frac{1}{2}\right)$ and $r > 0$, if $||\widehat{\mu} - z||_2 > r(1 - \gamma_1)\sqrt{\frac{1}{1 - 2\gamma_1}}$, then there exists $J \subseteq \{1, \dots, b\}$ with $|J| > \gamma_1 b$ such that for all $j \in J$, we have $||\mu_j - z||_2 > r$.
\end{lemma}

In the context of Lemma \ref{lemma:6}, set $\gamma_1 = \frac{7}{18}$. For all $1 \leq b \leq n/2$ and $\theta \in \Theta$, we have
\begin{align*}
\mathbb{E}\left[||\widehat{\mu}_j - \nabla\mathcal{R}(\theta)||_2^2\right] \leq \frac{\mathbb{E}\left[||\nabla \mathcal{L}(\theta, z_i) - \nabla\mathcal{R}(\theta)||_2^2\right]}{|B_j|} \leq \frac{2b}{n}\mbox{tr}(\mathrm{\mathrm{Cov}}(\nabla \mathcal{L}(\theta, z))),
\end{align*}
so by Chebyshev's inequality, with $\phi >0$ such that $\phi^2 \geq \frac{2b}{0.1n}\mbox{tr}(\mathrm{\mathrm{Cov}}(\nabla \mathcal{L}(\theta, z)))$, we have
\begin{align*}
\mathbb{P}\left(||\widehat{\mu}_j - \nabla\mathcal{R}(\theta)||_2 \geq \phi \right) \leq \frac{2b}{n\phi^2}\mbox{tr}(\mathrm{\mathrm{Cov}}(\nabla \mathcal{L}(\theta, z))) \leq 0.1.
\end{align*}
Take $\phi^2 = \frac{2b}{0.1n}Tr(\mathrm{\mathrm{Cov}}(\nabla \mathcal{L}(\theta, z)))$ and suppose we are on the event
\begin{align*}
\Omega = \left\{||\widehat{\mu} - \nabla\mathcal{R}(\theta)||_2 > \phi (1 - \gamma_1)\sqrt{\frac{1}{1 - 2\gamma_1}}\right\}.
\end{align*}
By Lemma \ref{lemma:6}, we have $J \subseteq \{1, \dots, b\}$ such that $|J| > \gamma_1 b$ and $||\widehat{\mu}_j - \nabla\mathcal{R}(\theta)||_2 > \phi$ for all $j \in J$. Hence, we have
\begin{align*}
\mathbb{P}(\Omega) \leq \mathbb{P}\left(\sum_{j = 1}^b \mathbbm{1}_{\{||\widehat{\mu}_j - \nabla\mathcal{R}(\theta)||_2 > \phi\}} > \gamma_1 b\right).
\end{align*}
Using the fact that the $\widehat{\mu}_j's$ are i.i.d., we see that (cf.\ \cite{Minsker} and Lemma $23$ in \cite{Lera})
\begin{align*}
\mathbb{P}\left(\sum_{j = 1}^b \mathbbm{1}_{\{||\widehat{\mu}_j - \nabla\mathcal{R}(\theta)||_2 > \phi\}} > \gamma_1 b\right) \leq \mathbb{P}(Bin(b, 0.1) > \gamma_1 b) \leq e^{-b\psi(\gamma_1)},
\end{align*}
where the last inequality follows from a Chernoff bound. Thus, for all $\theta \in \Theta$, we have
\begin{align*}
\mathbb{P}\left(||\widehat{\mu} - \nabla\mathcal{R}(\theta)||_2 \leq \phi(1 - \gamma_1)\sqrt{\frac{1}{1 - 2\gamma_1}}\right) \geq 1 - e^{-b\psi(\gamma_1)}.
\end{align*}
Some calculations show that $(1 - \gamma_1)\sqrt{\frac{1}{1 - 2\gamma_1}}\sqrt{\frac{2}{0.1\psi(\gamma_1)}} \leq 11$ and $\log(\frac{1}{\zeta}) + \psi(\gamma_1) \leq \log(\frac{1.4}{\zeta})$. Thus, by noting that $b = 1 + \left\lfloor \frac{\log(1/\zeta)}{\psi(\gamma_1)} \right\rfloor$, which implies $b\psi(\gamma_1) \geq \log(1/\zeta)$ and $b\psi(\gamma_1) \leq \log(1.4/\zeta)$, we obtain
\begin{align*}
\mathbb{P}\left(||\widehat{\mu} - \nabla\mathcal{R}(\theta)||_2 \leq 11\sqrt{\frac{b\psi(\gamma_1)Tr(\mathrm{\mathrm{Cov}}(\nabla \mathcal{L}(\theta, z)))}{n}}\right) \geq 1 - e^{-b\psi(\gamma_1)},
\end{align*}
implying that
\begin{align*}
\mathbb{P}\left(||\widehat{\mu} - \nabla\mathcal{R}(\theta)||_2 \leq 11\sqrt{\frac{\log(1.4/\zeta)Tr(\mathrm{\mathrm{Cov}}(\nabla \mathcal{L}(\theta, z)))}{n}}\right) \geq 1 - \zeta,
\end{align*}
as required.
\end{proof}


\subsubsection{Proof of Lemma \ref{lemma:geri}}
\label{AppLemGeri}

Applying Lemma \ref{lemma:htlemma}, we see that Algorithm \ref{alg:HTGE} returns a gradient estimate such that for all $\theta \in \mathcal{C}$, we have with probability at least $1 - \widetilde{\zeta}$ that
\begin{align}
\label{EqnGradErr}
||g(\theta) - \nabla\mathcal{R}_{\gamma_{\mathcal{C}}}(\theta)||_2 \lesssim \sqrt{\frac{p||\mathrm{\mathrm{Cov}}(\nabla \mathcal{L}_{\gamma_{\mathcal{C}}}(\theta, z))||_2\log(1/ \widetilde{\zeta})}{\widetilde{n}}},
\end{align}
where we also bounded the trace above by $p$ times the largest eigenvalue. We have suppressed the dependency on the data and $\widetilde{\zeta}$ in $g$, for simplicity.

We also use the following result:

\begin{lemma}[Adapted from \cite{RE}]
\label{lemma:cov_ri}
Consider the linear regression with $\ell_2$-regularized squared error loss model defined in Example~\ref{sec:LR_l2_Reg} with $z = (x, y)$. For $\theta \in \mathcal{C}$, we have
\begin{align*}
||\mathrm{\mathrm{Cov}}(\nabla \mathcal{L}_{\gamma_{\mathcal{C}}}(\theta, z))||_2 \lesssim \sigma_2^2 + ||\Delta||_2^2 + \frac{\gamma_{\mathcal{C}}^2}{(\lambda_{\min}(\Sigma) + \gamma_{\mathcal{C}})^2},
\end{align*}
with $\Delta = \theta - \theta_{*}$.
\end{lemma}

\begin{proof}
For a fixed $\theta \in \mathcal{C}$, denote $\Delta^{'} = \theta - \theta^{*}$. In the linear regression with $\ell_2$-regularized squared error loss model, as stated when we introduced it in Section \ref{sec:Linear Regression}, we have $\nabla \mathcal{L}_{\gamma_{\mathcal{C}}}(\theta, (x, y)) = xx^T\Delta^{'} - wx + \gamma_{\mathcal{C}}\theta$ and $\nabla\mathcal{R}_{\gamma_{\mathcal{C}}}(\theta) = \Sigma\Delta^{'} + \gamma_{\mathcal{C}}\theta$, because $x \indep w$ and $\mathbb{E}[w] = 0$. Then, for any $\theta \in \mathcal{C}$, we have
\begin{align*}
\mathrm{\mathrm{Cov}}(\nabla \mathcal{L}_{\gamma_{\mathcal{C}}}(\theta, z)) &= \mathbb{E}\left[(\nabla \mathcal{L}_{\gamma_{\mathcal{C}}}(\theta, z) - \nabla\mathcal{R}_{\gamma_{\mathcal{C}}}(\theta))(\nabla \mathcal{L}_{\gamma_{\mathcal{C}}}(\theta, z) - \nabla\mathcal{R}_{\gamma_{\mathcal{C}}}(\theta))^T\right]\\
&= \mathbb{E}[((xx^T - \Sigma)\Delta^{'} - wx)((xx^T - \Sigma)\Delta^{'} - wx)^T] \\
& = \mathbb{E}[(xx^T - \Sigma)\Delta^{'}(\Delta^{'})^T(xx^T - \Sigma)] + \sigma_2^2\Sigma,
\end{align*}
again since $x \indep w$ and $\mathbb{E}[w] = 0$. Using the fact that $\lambda_{\max}$ is subadditive, we obtain
\begin{align*}
||\mathrm{\mathrm{Cov}}(\nabla \mathcal{L}_{\gamma_{\mathcal{C}}}(\theta, z))||_2 &= \lambda_{\max}(\mathrm{\mathrm{Cov}}(\nabla \mathcal{L}_{\gamma_{\mathcal{C}}}(\theta, z))) \\
& \leq \sigma_2^2\lambda_{\max}(\Sigma) + \lambda_{\max}\left(\mathbb{E}[(xx^T - \Sigma)\Delta^{'}(\Delta^{'})^T(xx^T - \Sigma)]\right)\\
&= \sigma_2^2\lambda_{\max}(\Sigma) + \mathop{\sup}\limits_{||\xi||_2 = 1}\xi^T\mathbb{E}[(xx^T - \Sigma)\Delta^{'}(\Delta^{'})^T(xx^T - \Sigma)]\xi\\
&\leq \sigma_2^2\lambda_{\max}(\Sigma) + \mathop{\sup}\limits_{||\xi||_2, ||\omega||_2 = 1}\xi^T\mathbb{E}[(xx^T - \Sigma)\Delta^{'}(\Delta^{'})^T(xx^T - \Sigma)]\omega\\
&\leq \sigma_2^2\lambda_{\max}(\Sigma) + ||\Delta^{'}||_2^2\mathop{\sup}\limits_{||\xi||_2, ||\omega||_2 = 1}\mathbb{E}[(\xi^T(xx^T - \Sigma)\omega)^2]\\
&\leq \sigma_2^2\lambda_{\max}(\Sigma) + ||\Delta^{'}||_2^2\mathop{\sup}\limits_{||\xi||_2, ||\omega||_2 = 1}\mathbb{E}[2(\xi^Tx)^2(x^T\omega)^2 + 2(\xi^T\Sigma\omega)^2]\\
&\leq \sigma_2^2\lambda_{\max}(\Sigma) + 2||\Delta^{'}||_2^2\mathop{\sup}\limits_{||\xi||_2, ||\omega||_2 = 1}\left(\mathbb{E}[(\xi^Tx)^2(x^T\omega)^2] + \lambda_{\max}(\Sigma)^2\right)\\
&\leq \sigma_2^2\lambda_{\max}(\Sigma)\\
& \quad + 2||\Delta^{'}||_2^2\mathop{\sup}\limits_{||\xi||_2, ||\omega||_2 = 1}\left(\sqrt{\mathbb{E}[(\xi^Tx)^4]}\sqrt{\mathbb{E}[(\omega^Tx)^4]} + \lambda_{\max}(\Sigma)^2\right)\\
&\leq \sigma_2^2\lambda_{\max}(\Sigma) + 2||\Delta^{'}||_2^2\left(\widetilde{C}_4\lambda_{\max}(\Sigma)^2 + \lambda_{\max}(\Sigma)^2\right) \\
& = \sigma_2^2\lambda_{\max}(\Sigma) + C_1||\Delta^{'}||_2^2\lambda_{\max}(\Sigma)^2,
\end{align*}
for some absolute constant $C_1 > 0$, where we used the inequality $(a + b)^2 \leq 2(a^2 + b^2)$ in the fourth inequality, the Cauchy-Schwarz inequality in the penultimate inequality, and the bounded $4^{\text{th}}$ moments assumption in the last inequality.

Now recall that the minimizer of $\mathcal{R}_{\gamma_{\mathcal{C}}}$ is $\theta_{*} = (\Sigma + \gamma_{\mathcal{C}} I_p)^{-1}\Sigma\theta^*$, so $\Delta = \Delta^{'} + (I_p - (\Sigma + \gamma_{\mathcal{C}} I_p)^{-1}\Sigma)\theta^*$. Therefore, we have $$||\Delta^{'}||_2 \leq ||\Delta||_2 + ||I_p - (\Sigma + \gamma_{\mathcal{C}} I_p)^{-1}\Sigma||_2||\theta^*||_2 \leq ||\Delta||_2 + \frac{\gamma_{\mathcal{C}}}{\lambda_{\min}(\Sigma) + \gamma_{\mathcal{C}}}||\theta^*||_2,$$
since the largest eigenvalue of $I_p - (\Sigma + \gamma_{\mathcal{C}} I_p)^{-1}\Sigma$ is $\frac{\gamma_{\mathcal{C}}}{\lambda_{\min}(\Sigma) + \gamma_{\mathcal{C}}}$. Also note that $||\theta^*||_2$ depends on $p$ only, which we assumed to be constant. Thus, again using the inequality $(a + b)^2 \leq 2(a^2 + b^2)$, we obtain
\begin{align*}
||\mathrm{\mathrm{Cov}}(\nabla \mathcal{L}_{\gamma_{\mathcal{C}}}(\theta, z))||_2 \lesssim \sigma_2^2 + ||\Delta||_2^2 + \frac{\gamma_{\mathcal{C}}^2}{(\lambda_{\min}(\Sigma) + \gamma_{\mathcal{C}})^2},
\end{align*}
as required.
\end{proof}

Plugging Lemma \ref{lemma:cov_ri} into the bound~\eqref{EqnGradErr}, we then obtain
\begin{align*}
||g(\theta) - \nabla\mathcal{R}_{\gamma_{\mathcal{C}}}(\theta)||_2 \leq \sqrt{\frac{\log(1/\widetilde{\zeta})}{\widetilde{n}}}||\theta - \theta_*||_2 + \sqrt{\frac{\sigma_2^2\log(1/\widetilde{\zeta}) + \frac{\gamma_{\mathcal{C}}^2}{(\lambda_{\min}(\Sigma) + \gamma_{\mathcal{C}})^2}\log(1/\widetilde{\zeta})}{\widetilde{n}}},
\end{align*}
as required, implying that $g$ is a gradient estimator with
\begin{align*}
&\alpha(\widetilde{n}, \widetilde{\zeta}) \asymp \sqrt{\frac{\log(1/\widetilde{\zeta})}{\widetilde{n}}},
&\beta(\widetilde{n}, \widetilde{\zeta}) \asymp \sqrt{\frac{\sigma_2^2\log(1/\widetilde{\zeta}) + \frac{\gamma_{\mathcal{C}}^2}{(\lambda_{\min}(\Sigma) + \gamma_{\mathcal{C}})^2}\log(1/\widetilde{\zeta})}{\widetilde{n}}}.
\end{align*}


\section{Supplementary Results for Section~\ref{sec:Biased Parameter Estimation HT}}
\label{Appendix Extra}

In this appendix, we complement the analysis in Sections \ref{sec:Well_Cond_Sec_FW} and \ref{sec:Ill_Cond_Sec_FW} by analyzing projected gradient descent. In Appendix \ref{sec:The Case when l>0}, we examine the case when $\lambda_{\min}(\Sigma) > 0$; in Appendix \ref{sec:Projected Gradient Descent for l=0}, we consider the ill-conditioned setting.

We will use the following result about projected gradient descent from \cite{RE}, which furnishes an approximate convergence bound on $||\theta_t - \theta_*||_2$, where $\theta_*$ is the minimizer of a generic risk in some constraint set $\mathcal{C} \subseteq \mathbb{R}^p$.
We state it for a generic risk $\mathcal{R}$ and convex set $\mathcal{C}$ such that $\nabla\mathcal{R}(\theta_{*}) = 0$. \cite{RE} uses this with $\theta_* = \theta^*$.

\begin{lemma}[\cite{RE}]
\label{lemma:GD}
Suppose $\theta_* \in \mathcal{C}$. Given a stable gradient estimator $g$, Algorithm \ref{alg:RobPGDNFW} for projected gradient descent initialized at $\theta_0 \in \mathcal{C}$, with $\eta = \frac{2}{\tau_l + \tau_u}$, returns iterates $\{\theta_t\}_{t = 1}^T$ such that with probability at least $1 - \zeta$, we have 
\begin{align*}
||\theta_t - \theta_{*}||_2 \leq ||\theta_0 - \theta_{*}||_2k^t + \frac{\eta\beta(\widetilde{n}, \widetilde{\zeta})}{1 - k},
\end{align*}
with $k = \frac{\tau_u - \tau_l + 2\alpha(\widetilde{n}, \widetilde{\zeta})}{\tau_u + \tau_l}$.
\end{lemma}

\begin{remark}
\label{remark:GD_Iter_LB}
The fact that the gradient estimator is stable implies $k < 1$, so the first term in the bound in Lemma \ref{lemma:GD} is decreasing in $T$, while the second is increasing. Hence, for a fixed $n$ and $\zeta$, we wish to run projected gradient descent until the first term is smaller than the second one, i.e., $T \geq \log_{1/k}\left((1 - k)||\theta_0 - \theta^*||_2/\beta(\widetilde{n}, \widetilde{\zeta})\right)$.

Note that since our gradient estimator is stable, we have $\alpha < \tau_l/2$, so $k < \frac{\tau_u - \tau_l + \tau_l}{\tau_u + \tau_l} = \frac{\tau_u}{\tau_u + \tau_l} < 1$, so indeed, we obtain a bound involving a term converging exponentially to $0$ and an error term. Additionally, note that $1 - k > \frac{\tau_l}{\tau_u + \tau_l} \neq 0$. This allows us to bound $\frac{1}{1 - k}$ above by an absolute constant if $\tau_u$ and $\tau_l$ are regarded as absolute constants themselves. 
\end{remark}

We now derive a general bound on $||\theta_T - \theta_*||_2$, where $\theta_* = (\Sigma + \gamma_{\mathcal{C}}I_p)^{-1}\Sigma\theta^*$, based on ridge regression (to accommodate for the ill-conditioned case). Later, we will choose $\gamma_{\mathcal{C}}$ appropriately to obtain a bound on $||\theta_T - \theta^*||_2$.

\begin{prop}
\label{prop:ridgeGD}
Consider the linear regression with $\ell_2$-regularized squared error loss model from Example~\ref{sec:LR_l2_Reg} under the heavy-tailed setting. Let $\zeta \in (0, 1)$. There exists an absolute constant $C_1 > 0$ such that, if $\widetilde{n} > \frac{4C_1^2\log(1/\widetilde{\zeta})}{\tau_l^2}$, Algorithm \ref{alg:RobPGDNFW} for projected gradient descent, initialized at $\theta_0 \in \mathcal{C}$ with $\eta = \frac{2}{\tau_u + \tau_l}$, and using Algorithm \ref{alg:HTGE} as gradient estimator with $\alpha(\widetilde{n}, \widetilde{\zeta}) = C_1\sqrt{\frac{\log(1/\widetilde{\zeta})}{\widetilde{n}}}$, returns iterates $\{\theta_t\}_{t = 1}^T$ such that with probability at least $1 - \zeta$, with $\widetilde{\zeta}$ such that $b \leq \widetilde{n}/2$ and with $T = \log_\frac{\tau_u + \tau_l}{\tau_u}(\sqrt{n})$, we have
\begin{align*}
||\theta_T - \theta_{*}||_2 &\lesssim \frac{1}{\sqrt{n}} + \left(\frac{\lambda_{\max}(\Sigma) + \lambda_{\min}(\Sigma) + 2\gamma_{\mathcal{C}}}{\lambda_{\min}(\Sigma) + \gamma_{\mathcal{C}}}\right)\\
& \qquad \cdot \sqrt{\frac{\left(\sigma_2^2 + \frac{\gamma_{\mathcal{C}}^2}{(\lambda_{\min}(\Sigma) + \gamma_{\mathcal{C}})^2}\right)\log(n)\log\left(\frac{\log(n)}{\zeta\log\left(\frac{\lambda_{\max}(\Sigma) + \lambda_{\min}(\Sigma) + 2\gamma_{\mathcal{C}}}{\lambda_{\max}(\Sigma) + \gamma_{\mathcal{C}}}\right)}\right)}{n\log\left(\frac{\lambda_{\max}(\Sigma) + \lambda_{\min}(\Sigma) + 2\gamma_{\mathcal{C}}}{\lambda_{\max}(\Sigma) + \gamma_{\mathcal{C}}}\right)}}.
\end{align*}
\end{prop}

\begin{proof}
From Lemma \ref{lemma:geri}, we obtain a gradient estimator $g(\theta)$ with corresponding functions $\alpha(\widetilde{n}, \widetilde{\zeta})$ and $\beta(\widetilde{n}, \widetilde{\zeta})$. The assumption on $n$ implies by inverting the expression that $\alpha(\widetilde{n}, \widetilde{\zeta}) < \tau_l/2$, i.e., that the gradient estimator is stable. Then, for $k = \frac{\tau_u - \tau_l + 2\alpha(\widetilde{n}, \widetilde{\zeta})}{\tau_u + \tau_l} < \frac{\tau_u}{\tau_u + \tau_l} < 1$ and by Lemma \ref{lemma:GD}, optimizing $\mathcal{R}_{\gamma_{\mathcal{C}}}$ over $\mathcal{C}$ using projected gradient descent yields iterates $\{\theta_t\}_{t = 1}^T$ such that with probability at least $1 - \zeta$, we have
\begin{align*}
||\theta_t - \theta_{*}||_2 &\leq ||\theta_0 - \theta_{*}||_2k^t + \frac{\eta\beta(\widetilde{n}, \widetilde{\zeta})}{1 - k} \lesssim k^t + \frac{\beta(\widetilde{n}, \widetilde{\zeta})}{1 - k}\\
&\leq \left(\frac{\tau_u}{\tau_u + \tau_l}\right)^t + \frac{\tau_u + \tau_l}{\tau_l}\beta(\widetilde{n}, \widetilde{\zeta}),
\end{align*}
since $k < \frac{\tau_u}{\tau_u + \tau_l}$. We now plug in the expression for $\beta(\widetilde{n}, \widetilde{\zeta})$ from Lemma \ref{lemma:geri} and at step $T$ to obtain
\begin{align*}
||\theta_T - \theta_{*}||_2 &\lesssim \frac{1}{\sqrt{n}} + \frac{\tau_u + \tau_l}{\tau_l}\sqrt{\frac{\left(\sigma_2^2 + \frac{\gamma_{\mathcal{C}}^2}{(\lambda_{\min}(\Sigma) + \gamma_{\mathcal{C}})^2}\right)\log_{\frac{\tau_u + \tau_l}{\tau_u}}(n)\log\left(\log_{\frac{\tau_u + \tau_l}{\tau_u}}(n)/\zeta\right)}{n}}\\
&\leq \frac{1}{\sqrt{n}} + \left(\frac{\lambda_{\max}(\Sigma) + \lambda_{\min}(\Sigma) + 2\gamma_{\mathcal{C}}}{\lambda_{\min}(\Sigma) + \gamma_{\mathcal{C}}}\right)\\
& \qquad \cdot \sqrt{\frac{\left(\sigma_2^2 + \frac{\gamma_{\mathcal{C}}^2}{(\lambda_{\min}(\Sigma) + \gamma_{\mathcal{C}})^2}\right)\log(n)\log\left(\frac{\log(n)}{\zeta\log\left(\frac{\lambda_{\max}(\Sigma) + \lambda_{\min}(\Sigma) + 2\gamma_{\mathcal{C}}}{\lambda_{\max}(\Sigma) + \gamma_{\mathcal{C}}}\right)}\right)}{n\log\left(\frac{\lambda_{\max}(\Sigma) + \lambda_{\min}(\Sigma) + 2\gamma_{\mathcal{C}}}{\lambda_{\max}(\Sigma) + \gamma_{\mathcal{C}}}\right)}},
\end{align*}
as required.
\end{proof}


\subsection{Projected Gradient Descent for $\lambda_{\min}(\Sigma) > 0$}
\label{sec:The Case when l>0}

Our aim is to apply Proposition \ref{prop:ridgeGD}. Recall that $\mathcal{C} = \mathbb{B}_2(D)$, with $D \geq ||(\Sigma + \gamma_{\mathcal{C}}I_p)^{-1}\Sigma\theta^*||_2$. In this case, when $\lambda_{\min}(\Sigma) > 0$, we have $\left\|[P^T\theta^*]_{[1:m]}\right\|_2 = ||\theta^*||_2$, since $m = p$. 
\begin{corollary}
\label{corollary:HT_Ridge_min>0}
Consider the linear regression with $\ell_2$-regularized squared error loss model from Example~\ref{sec:LR_l2_Reg} under the heavy-tailed setting. Let $\zeta \in (0, 1)$. Assume $\lambda_{\min}(\Sigma) > 0$ and $\gamma_{\mathcal{C}} = \frac{1}{\sqrt{n}}$. There exists an absolute constant $C_1 > 0$ such that if $\widetilde{n} > \frac{4C_1^2\log(1/\widetilde{\zeta})}{\tau_l^2}$, Algorithm \ref{alg:RobPGDNFW} for projected gradient descent, initialized at $\theta_0 \in \mathcal{C}$ with $\eta = \frac{2}{\tau_u + \tau_l}$, and using Algorithm \ref{alg:HTGE} as gradient estimator with $\alpha(\widetilde{n}, \widetilde{\zeta}) = C_1\sqrt{\frac{\log(1/\widetilde{\zeta})}{\widetilde{n}}}$, returns iterates $\{\theta_t\}_{t = 1}^T$ such that with probability at least $1 - \zeta$, with $\widetilde{\zeta}$ such that $b \leq \widetilde{n}/2$ and with $T = \log_{\frac{\tau_u + \tau_l}{\tau_u}}(\sqrt{n})$, we have
\begin{align}
\label{eq:RI_HT}
||\theta_T - \theta^*||_2 \lesssim (1 + \sigma_2)\sqrt{\frac{\log(n)\log(\log(n)/\zeta)}{n}}.
\end{align}
\end{corollary}

\begin{proof}
By Proposition \ref{prop:ridgeGD}, we see that with probability at least $1 - \zeta$, we have
\begin{align*}
||\theta_T - \theta_{*}||_2 &\lesssim \frac{1}{\sqrt{n}} + \left(\frac{\lambda_{\max}(\Sigma) + \lambda_{\min}(\Sigma) + 2\gamma_{\mathcal{C}}}{\lambda_{\min}(\Sigma) + \gamma_{\mathcal{C}}}\right)\\
& \qquad \cdot \sqrt{\frac{\left(\sigma_2^2 + \frac{\gamma_{\mathcal{C}}^2}{(\lambda_{\min}(\Sigma) + \gamma_{\mathcal{C}})^2}\right)\log(n)\log\left(\frac{\log(n)}{\zeta\log\left(\frac{\lambda_{\max}(\Sigma) + \lambda_{\min}(\Sigma) + 2\gamma_{\mathcal{C}}}{\lambda_{\max}(\Sigma) + \gamma_{\mathcal{C}}}\right)}\right)}{n\log\left(\frac{\lambda_{\max}(\Sigma) + \lambda_{\min}(\Sigma) + 2\gamma_{\mathcal{C}}}{\lambda_{\max}(\Sigma) + \gamma_{\mathcal{C}}}\right)}}.
\end{align*}
Note that $\gamma_{\mathcal{C}} \rightarrow 0$ as $n \rightarrow \infty$, so $\frac{\lambda_{\max}(\Sigma) + \gamma_{\mathcal{C}}}{\lambda_{\max}(\Sigma) + \lambda_{\min}(\Sigma) + 2\gamma_{\mathcal{C}}} < \frac{\lambda_{\max}(\Sigma)}{\lambda_{\max}(\Sigma) + \lambda_{\min}(\Sigma)} < 1$ for $n$ greater than an absolute constant and $\frac{\lambda_{\max}(\Sigma) + \lambda_{\min}(\Sigma) + 2\gamma_{\mathcal{C}}}{\lambda_{\min}(\Sigma) + \gamma_{\mathcal{C}}} \lesssim 1$. Furthermore, we have
\begin{align*}
||\theta_{*} - \theta^*||_2 &\lesssim \left\|\left((\Sigma + \gamma_{\mathcal{C}}I_p)^{-1}\Sigma - I_p\right)\theta^*\right\|_2 \leq \left\|(\Sigma + \gamma_{\mathcal{C}}I_p)^{-1}\Sigma - I_p\right\|_2||\theta^*||_2\\
&\leq \frac{\gamma_{\mathcal{C}}}{\lambda_{\min}(\Sigma) + \gamma_{\mathcal{C}}}||\theta^*||_2 \lesssim \gamma_{\mathcal{C}},
\end{align*}
since $\lambda_{\min}(\Sigma) > 0$ and $||\theta^*||_2 \asymp 1$. Therefore, we have
\begin{align}
\label{eq:HT_RI_MIN>0}
||\theta_T - \theta^*||_2 \lesssim \frac{1}{\sqrt{n}} + \sqrt{\frac{(1 + \sigma_2^2)\log(n)\log(\log(n)/\zeta)}{n}} + \gamma_{\mathcal{C}}.
\end{align}
Since $\gamma_{\mathcal{C}} = \frac{1}{\sqrt{n}}$, we obtain
\begin{align*}
||\theta_T - \theta^*||_2 \lesssim (1 + \sigma_2)\sqrt{\frac{\log(n)\log(\log(n)/\zeta)}{n}},
\end{align*}
as required.
\end{proof}

\begin{remark}
Recall that $T = \log_{\frac{\tau_u + \tau_l}{\tau_u}}(\sqrt{n})$ and $\frac{\tau_u}{\tau_u + \tau_l} < \frac{\lambda_{\max}(\Sigma)}{\lambda_{\max}(\Sigma) + \lambda_{\min}(\Sigma)}$, the latter of which is an absolute constant. Hence, the number of iterations required is sublogarithmic in $n$.

The upper bound \eqref{eq:HT_RI_MIN>0} is polynomial in $\gamma_{\mathcal{C}}$, so we could have chosen $\gamma_{\mathcal{C}}$ much smaller than $\frac{1}{\sqrt{n}}$. However, the result would not have changed because of the presence of the rate of $\frac{1}{\sqrt{n}}$ in inequality \eqref{eq:HT_RI_MIN>0} already. We chose $\frac{1}{\sqrt{n}}$ so that the last term $\gamma_{\mathcal{C}}$ in inequality \eqref{eq:HT_RI_MIN>0} scales like $\frac{1}{\sqrt{n}}$. Regardless of the choice of $\gamma_{\mathcal{C}}$, the best rate  we can hope for in this case is $\frac{1}{\sqrt{n}}$. Note also that if we take $\gamma_{\mathcal{C}} = 0$, i.e., we are in the case when $\theta_* = \theta^*$, we are back in the linear regression with squared error loss model and we minimize over $\mathcal{C}$. We obtain a rate of $\frac{1}{\sqrt{n}}$ for $||\theta_T - \theta^*||_2$, up to logarithmic factors. This is consistent with what we have in Section \ref{sec:Strongly Convex Risks}, because when $\gamma_{\mathcal{C}} = 0$, we are in the context of Lemma \ref{lemma:GLMhtGD}, where we have a rate of $\frac{1}{\sqrt{n}}$ for $||\theta_T - \theta^*||_2$, up to logarithmic factors.
\end{remark}


We now compare the results of Corollary~\ref{corollary:HT_Ridge_min>0}, Theorem~\ref{theorem:HT_NonAccFW}, and Theorem~\ref{theorem:ACCFWROB}. Up to logarithmic factors, we see that the projected gradient descent approach is the best at rate $\frac{1}{\sqrt{n}}$ (cf.\ inequality \eqref{eq:RI_HT}), followed by the accelerated Frank-Wolfe approach at rate $\frac{1}{n^{1/5}}$ (cf.\ inequality \eqref{eq:ACCFWHTrate}). The worst rate of the three is the non-accelerated Frank-Wolfe approach at rate $\frac{1}{n^{1/6}}$ (cf.\ inequality \eqref{eq:NotAccFW_HT}). The $\frac{1}{\sqrt{n}}$ rate is minimax optimal for $w \sim N\left(0, \sigma_2^2\right)$ (see \cite{Stats_Info_Duchi}). Hence, the ridge regression approach in Corollary~\ref{corollary:HT_Ridge_min>0} is minimax optimal and robust to heavy-tails in the noise and covariates. Moreover, projected gradient descent outperforms the Frank-Wolfe methods in terms of iteration count: The iteration count in Corollary~\ref{corollary:HT_Ridge_min>0} is logarithmic in $n$, while the iteration counts in Theorem~\ref{theorem:HT_NonAccFW} and Theorem~\ref{theorem:ACCFWROB} are polynomial in $n$ ($n^{1/3}$ and $\widetilde{\Theta}\left(n^{1/5}\right)$, respectively).

However, there is a potential downside to using projected gradient descent rather than the Frank-Wolfe methods, in terms of robustness to heavy tails in the noise $w$. Suppose $w \sim ST(\nu)$ with $\nu > 2$ so that $\sigma_2^2 = \mathbb{E}[w^2] < \infty$. In this case, $\sigma_2 = \sqrt{\frac{\nu}{\nu - 2}} > 1$. The term $1 + \sigma_2$ appears in the upper bound on $||\theta_T - \theta^*||_2$ in inequality \eqref{eq:RI_HT}, whereas in the bounds \eqref{eq:NotAccFW_HT} and \eqref{eq:ACCFWHTrate}, we have an improved dependency of $\sigma_2$ in the form of $(1 + \sigma_2)^{1/2}$. Note that as $\nu$ increases, i.e., as the number of finite moments of $w$ increases, $\sigma_2$ decreases, so all the bounds become tighter. This makes intuitive sense, because as we gather more information about $w$, we can obtain a more precise bound.


\subsection{Projected Gradient Descent for $\lambda_{\min}(\Sigma) = 0$}
\label{sec:Projected Gradient Descent for l=0}

As in Section \ref{sec:Ill_Cond_Sec_FW}, we now assume that the top $m$ eigenvalues of $\Sigma$ are positive, with $0 < m < p$. In the following corollary, we keep track of the dependency on $\left\|[P^T\theta^*]_{[(m+1):p]}\right\|_2$, the only term that vanishes when in the well-conditioned case ($m = p$). 

\begin{corollary}
\label{corollary:HT_Ridge_min=0}
Consider the linear regression with $\ell_2$-regularized squared error loss model from Example~\ref{sec:LR_l2_Reg} under the heavy-tailed setting. Let $\zeta \in (0, 1)$. Assume that the top $m$ eigenvalues of $\Sigma$ are positive, with $0 < m < p$. Let $[P^T\theta^*]_{[(m + 1):p]}$ be the vector in $\mathbb{R}^{p - m}$ containing the bottom $p - m$ entries of $P^T\theta^*$. Assume $\frac{1}{n^{1/5}} \lesssim \gamma_{\mathcal{C}} \rightarrow 0$ as $n \rightarrow \infty$. There exists an absolute constant $C_1 > 0$ such that, if $\widetilde{n} > \frac{4C_1^2\log(1/\widetilde{\zeta})}{\tau_l^2}$, Algorithm \ref{alg:RobPGDNFW} for projected gradient descent, initialized at $\theta_0 \in \mathcal{C}$ with $\eta = \frac{2}{\tau_u + \tau_l}$ and using Algorithm \ref{alg:HTGE} as gradient estimator with $\alpha(\widetilde{n}, \widetilde{\zeta}) = C_1\sqrt{\frac{\log(1/\widetilde{\zeta})}{\widetilde{n}}}$, returns iterates $\{\theta_t\}_{t = 1}^T$ such that with probability at least $1 - \zeta$, with $\widetilde{\zeta}$ such that $b \leq \widetilde{n}/2$ and with $T = \log_{\frac{\tau_u + \tau_l}{\tau_u}}(\sqrt{n}) = \widetilde{O}\left(n^{1/5}\right)$, we have
\begin{align*}
||\theta_T - \theta^*||_2 \lesssim (1 + \sigma_2)\frac{\sqrt{\log(n)\log\left(n/\zeta\right)}}{n^{1/5}} + \left\|[P^T\theta^*]_{[(m+1):p]}\right\|_2.
\end{align*}
\end{corollary}

\begin{proof}
We have $\theta_{*} = (\Sigma + \gamma_{\mathcal{C}}I_p)^{-1}\Sigma\theta^*$, and by Proposition \ref{prop:ridgeGD}, with probability at least $1 - \zeta$, we have
\begin{align*}
||\theta_T - \theta_{*}||_2 \lesssim \frac{1}{\sqrt{n}} + \frac{\lambda_{\max}(\Sigma) + 2\gamma_{\mathcal{C}}}{\gamma_{\mathcal{C}}}\sqrt{\frac{(1 + \sigma_2^2)\log(n)\log\left(\frac{\log(n)}{\zeta\log\left(\frac{\lambda_{\max}(\Sigma) + 2\gamma_{\mathcal{C}}}{\lambda_{\max}(\Sigma) + \gamma_{\mathcal{C}}}\right)}\right)}{n\log\left(\frac{\lambda_{\max}(\Sigma) + 2\gamma_{\mathcal{C}}}{\lambda_{\max}(\Sigma) + \gamma_{\mathcal{C}}}\right)}}.
\end{align*}
Since $\gamma_{\mathcal{C}} \rightarrow 0$ as $n \rightarrow \infty$, we have $\frac{1}{\log\left(\frac{\lambda_{\max}(\Sigma) + 2\gamma_{\mathcal{C}}}{\lambda_{\max}(\Sigma) + \gamma_{\mathcal{C}}}\right)} = \frac{1}{\gamma_{\mathcal{C}}\log\left(\left(1 + \frac{\gamma_{\mathcal{C}}}{\lambda_{\max}(\Sigma) + \gamma_{\mathcal{C}}}\right)^{1/\gamma_{\mathcal{C}}}\right)} \asymp \frac{1}{\gamma_{\mathcal{C}}}$. Thus, we have
\begin{align*}
||\theta_T - \theta_{*}||_2 \lesssim \frac{1}{\sqrt{n}} + \frac{\lambda_{\max}(\Sigma) + 2\gamma_{\mathcal{C}}}{\gamma_{\mathcal{C}}}\sqrt{\frac{(1 + \sigma_2^2)\log(n)\log\left(\frac{\log(n)}{\zeta\gamma_{\mathcal{C}}}\right)}{n\gamma_{\mathcal{C}}}}.
\end{align*}
Since $\theta_{*} = (\Sigma + \gamma_{\mathcal{C}}I_p)^{-1}\Sigma\theta^*$, we have
\begin{align*}
||\theta_{*} - \theta^*||_2^2 = ||((S + \gamma_{\mathcal{C}}I_p)^{-1}S - I_p)P^T\theta^*||_2^2 \lesssim m\gamma_{\mathcal{C}}^2 + \left\|[P^T\theta^*]_{[(m+1):p]}\right\|_2^2.
\end{align*}
Hence, we obtain
\begin{align}
\label{eq:HT_RI_MIN=0}
||\theta_T - \theta^*||_2 &\lesssim \frac{1}{\sqrt{n}} + \frac{\sqrt{(1 + \sigma_2^2)\log(n)\log\left(\frac{\log(n)}{\zeta\gamma_{\mathcal{C}}}\right)}}{\gamma_{\mathcal{C}}^{3/2}n^{1/2}} + \frac{\sqrt{(1 + \sigma_2^2)\log(n)\log\left(\frac{\log(n)}{\zeta\gamma_{\mathcal{C}}}\right)}}{\gamma_{\mathcal{C}}^{1/2}n^{1/2}} \notag\\
& \quad + \sqrt{m}\gamma_{\mathcal{C}} + \left\|[P^T\theta^*]_{[(m+1):p]}\right\|_2,
\end{align}
so for $\gamma_{\mathcal{C}} \gtrsim \frac{1}{n^{1/5}}$, we obtain
\begin{align*}
||\theta_T - \theta^*||_2 \lesssim (1 + \sigma_2)\frac{\sqrt{\log(n)\log\left(n/\zeta\right)}}{n^{1/5}} + \left\|[P^T\theta^*]_{[(m+1):p]}\right\|_2,
\end{align*}
as required.

Furthermore, note that $T \asymp \log_{\frac{\tau_u + \tau_l}{\tau_u}}(n)$ and $\log\left(\frac{\tau_u + \tau_l}{\tau_u}\right) \lesssim \frac{1}{\gamma_{\mathcal{C}}} \lesssim n^{1/5}$, implying that $T \lesssim n^{1/5}\log(n) = \widetilde{O}(n^{1/5})$.
\end{proof}

\begin{remark}
Observe that the upper bound for $||\theta_T - \theta^*||_2$ in Corollary \ref{corollary:HT_Ridge_min=0} is of the form $\widetilde{O}\left(\frac{1}{n^{1/5}}\right) + \left\|[P^T\theta^*]_{[(m + 1):p]}\right\|_2$. In other words, we have one term that vanishes with $n$, and one term that decreases with $m$.

Moreover, note that the choice of $\gamma_{\mathcal{C}} \gtrsim \frac{1}{n^{1/5}}$ is not arbitrary and the rate of $\frac{1}{n^{1/5}}$ is the best possible using our analysis: In inequality~\eqref{eq:HT_RI_MIN=0}, the best rate we can hope for is polynomial in $n$, and if we take $\gamma_{\mathcal{C}} = \frac{1}{n^q}$, the best rate is obtained by taking the intersection between the lines $\frac{1 - 3q}{2}, \frac{1 - q}{2}$, and $q$. Also, we choose $\gamma_{\mathcal{C}} \gtrsim \frac{1}{n^{1/5}}$, since the bound \eqref{eq:HT_RI_MIN=0} is decreasing for $\gamma_{\mathcal{C}}$ small enough, i.e., for $n$ large enough.

Additionally, to interpret the result of Corollary \ref{corollary:HT_Ridge_min=0} based on our introduction of the $\ell_2$-regularization in Example~\ref{sec:LR_l2_Reg}, note that the method is equivalent to optimizing the squared error risk $\mathcal{R}$ over an $\ell_2$-ball $\mathcal{V}$ centered at $0$ that increases with $n$ towards $\mathbb{B}_2\left(\left\|[P^T\theta^*]_{[1:m]}\right\|_2\right)$.
Then we can learn $\theta^*$ at rate $\frac{1}{n^{1/5}}$ and up to an error that vanishes if $m = p$.
\end{remark}



We now compare the result of Corollary \ref{corollary:HT_Ridge_min=0} with that of Theorem \ref{theorem:ACCFW_lambda=0}. In Corollary \ref{corollary:HT_Ridge_min=0}, we have a bound of the form $\widetilde{O}\left(\frac{1}{n^{1/5}}\right) + \left\|[P^T\theta^*]_{[(m+1):p]}\right\|_2$; in Theorem \ref{theorem:ACCFW_lambda=0}, the bound is of the form 
\begin{align*}
\widetilde{O}\left(\frac{1}{\left\|[P^T\theta^*]_{[(m+1):p]}\right\|_2^{1/4}n^{1/4}}\right) + \left\|[P^T\theta^*]_{[(m+1):p]}\right\|_2 + \left\|[P^T\theta^*]_{[(m+1):p]}\right\|_2^{1/2}.
\end{align*}
If $\left\|[P^T\theta^*]_{[(m+1):p]}\right\|_2 \geq 1$, the bound in Theorem \ref{theorem:ACCFW_lambda=0} is $\widetilde{O}\left(\frac{1}{n^{1/4}}\right) + \left\|[P^T\theta^*]_{[(m+1):p]}\right\|_2$. In this case, the result in Theorem \ref{theorem:ACCFW_lambda=0} is tighter in terms of the rate with $n$ and the constant $\left\|[P^T\theta^*]_{[(m+1):p]}\right\|_2$. We do wish to point out that the other suppressed constants multiplying $\left\|[P^T\theta^*]_{[(m+1):p]}\right\|_2$ in Theorem \ref{theorem:ACCFW_lambda=0} can be much larger compared to Corollary \ref{corollary:HT_Ridge_min=0} due to the nature of our derivations. Hence, if $m$ is not close to $p$, the result of Corollary \ref{corollary:HT_Ridge_min=0} could be better because its constant error could be much smaller.

Additionally, observe that if $\left\|[P^T\theta^*]_{[(m+1):p]}\right\|_2 < 1$, the upper bound in Theorem \ref{theorem:ACCFW_lambda=0} scales like \\ $\widetilde{O}\left(\frac{1}{\left\|[P^T\theta^*]_{[(m+1):p]}\right\|_2^{1/4}n^{1/4}}\right) + \left\|[P^T\theta^*]_{[(m+1):p]}\right\|_2^{1/2}$. We obtain the best rate with $n$ again, but with a slightly higher term of $c_{\mathcal{K}}^{1/2} = \left\|[P^T\theta^*]_{[(m+1):p]}\right\|_2^{1/2}$, compared to $c_{\mathcal{K}}$, in the bound based on Corollary \ref{corollary:HT_Ridge_min=0}. However, as we explained in Section \ref{sec:Biased Parameter Estimation HT}, the term $c_{\mathcal{K}}^{1/2}$ can indeed be small in practice. In general, Corollary \ref{corollary:HT_Ridge_min=0} is the most practical, since for the Frank-Wolfe methods, we have to impose further restrictions on some parameters, such as lower or upper bounds involving $||\theta^*||_2$ or $\Sigma$.

Also, we remark that the method in Corollary \ref{corollary:HT_Ridge_min=0} targets $||\theta_T - \theta^*||_2$ directly, and going to the excess regularized risk is at the cost of a constant factor due to the fact that the smoothness parameter is a constant factor. The Frank-Wolfe methods target the excess regularized risk, and going to $||\theta_T - \theta^*||_2$ is at the cost of a $\gamma_{\mathcal{C}}$ term due to strong convexity. This influences the convergence rate with $n$ for the non-accelerated version, while the rate in the accelerated version is not affected by this multiplication with $\gamma_{\mathcal{C}}$, since $\gamma_{\mathcal{C}} \in \left[\frac{c_{\mathcal{K}}}{4}, \frac{c_{\mathcal{K}}}{2}\right]$. Moreover, the approach in Corollary \ref{corollary:HT_Ridge_min=0} takes into account the strong convexity of the risk in the proof of the convergence rate for projected gradient descent, as we can see in Lemma \ref{lemma:GD}. The proofs of convergence of the Frank-Wolfe methods (Lemma \ref{lemma:NonAccRel_FW} and Theorem \ref{theorem:ACCFWV3}) do not take strong convexity of the risk into account. Hence, a more fair comparison could be between the performance of the Frank-Wolfe methods and a projected gradient descent approach that only takes the smoothness of the risk into account. Then the performance of the projected gradient descent approach would guarantee a worse rate than the one in Lemma \ref{lemma:GD}.

Finally, in terms of the second moment of the noise, assume $w \sim ST(\nu)$, with $\nu > 2$. Then, $\sigma_2^2 = \frac{\nu}{\nu - 2} > 1$. The bound in Corollary \ref{corollary:HT_Ridge_min=0} has a $1 + \sigma_2$ factor, while Theorem \ref{theorem:NonACC_lambda=0} and Theorem \ref{theorem:ACCFW_lambda=0} have a $(1 + \sigma_2)^{1/2}$ factor. Thus, the Frank-Wolfe methods have tighter bounds in terms of $\sigma_2$, for $||\theta_T - \theta^*||_2$.


\section{Proofs for Section \ref{sec:Unbiased Parameter Estimation}}

In this appendix, we present the proofs of the results in Section \ref{sec:Unbiased Parameter Estimation}. In Appendix \ref{sec:Proofs of the Main Results_Nester}, we provide the proofs of the main results in Section \ref{sec:Unbiased Parameter Estimation}; in Appendix \ref{sec:Proofs of the Auxiliary Results_Nester}, we present the proofs of the auxiliary statements.


\subsection{Proofs of the Main Results from Section \ref{sec:Unbiased Parameter Estimation}}\label{sec:Proofs of the Main Results_Nester}

Here, we present the proofs of the main theorems from Section \ref{sec:Unbiased Parameter Estimation}.
For reference, we also include a statement regarding the convergence of robust projected gradient descent:

\begin{lemma}[\cite{RE}]
\label{lemma:GLMhtGD}
Let $\mathcal{C} \subseteq \mathbb{R}^p$ and $\zeta \in (0, 1)$. Consider the linear regression with squared error loss model from Example~\ref{sec:LR_Squared_Loss} under the heavy-tailed setting. Assume $\theta^* \in \mathcal{C}$. Then there is an absolute constant $C_1 > 0$ such that, if $\widetilde{n} > \frac{4C_1^2p \log(1/\widetilde{\zeta})}{\tau_l^2}$, Algorithm \ref{alg:RobPGDNFW} for projected gradient descent, initialized at $\theta_0 \in \mathcal{C}$ with $\eta = \frac{2}{\tau_u + \tau_l}$ and using Algorithm \ref{alg:HTGE} as gradient estimator, with $\alpha(\widetilde{n}, \widetilde{\zeta}) = C_1\sqrt{\frac{p \log(1/\widetilde{\zeta})}{\widetilde{n}}}$, returns iterates $\{\theta_t\}_{t = 1}^T$ such that with probability at least $1 - \zeta$, with $\widetilde{\zeta}$ such that $b \leq \widetilde{n}/2$, we have for some $k < 1$ that
\begin{align}
\label{eq:GLMhtGD}
||\theta_t - \theta^*||_2 \lesssim ||\theta_0 - \theta^*||_2k^t + \frac{\sigma_2}{1 - k}\sqrt{\frac{p \log (1/\widetilde{\zeta})}{\widetilde{n}}}.
\end{align}
\end{lemma}


\subsubsection{Proof of Theorem \ref{theorem:Priv_HT_GD_Smooth}}
\label{AppThmPriv}

By Lemma \ref{lemma:Avg_Grad_Priv_GD_vs_AGD}, we know that $g$ is a gradient estimator with $\alpha(\widetilde{n}, \widetilde{\zeta}) = 0$ and
\begin{align*}
\beta(\widetilde{n}, \widetilde{\zeta}) \asymp \sqrt{\frac{T\log(T/\zeta)}{n}} + \frac{T\sqrt{T\log(T/\zeta)\log^2(T/\delta)}}{n\epsilon},
\end{align*}
implying that
\begin{align*}
 \mathbb{P}(\forall t, \ ||g(\theta_t, \mathcal{D}_n, \widetilde{\zeta}) - \nabla\mathcal{R}(\theta_t)||_2 \leq \beta(\widetilde{n}, \widetilde{\zeta})) \geq 1 - \zeta.
\end{align*}
On this high-probability event, using the notation $\beta = \beta(\widetilde{n}, \widetilde{\zeta})$ and ignoring the dependency in $g$ on the samples and $\widetilde{\zeta}$, the error term $e_t := g(\theta_t) - \nabla\mathcal{R}(\theta_t)$ is bounded as $\|e_t\|_2 \le \beta$. Consider the $t^{\text{th}}$ step in Algorithm \ref{alg:RobPGDNFW}. Recall that $\mathcal{R}$ is $\tau_u$-smooth over $\mathbb{R}^p$. Since $\eta = \frac{1}{\tau_u}$, we have for $t \in \left\{0, \dots, T - 1\right\}$ that 
\begin{align*}
\theta_{t + 1} = \theta_t - \frac{1}{\tau_u}(\nabla \mathcal{R}(\theta_t) + e_t).
\end{align*}
Thus, by Lemma \ref{lemma:Inexact_GD_lemma}, for $a_T = \sum_{i = 1}^T\frac{||e_{i - 1}||_2}{\tau_u} \leq \frac{T\beta}{\tau_u} \asymp T\beta$, we have 
\begin{align}
\label{eq:Smooth_GD_Best_T}
\mathcal{R}(\theta_T) - \mathcal{R}(\theta_*) &\leq \frac{\frac{\tau_u}{2}||\theta_0 - \theta_*||_2^2 + (2a_T + ||\theta_0 - \theta_*||_2)\left(\tau_ua_T + 2\sum_{i = 2}^T(i - 1)||e_{i - 1}||_2\right)}{T} \notag\\
&\lesssim \frac{||\theta_0 - \theta_*||_2^2}{T} + (T\beta + ||\theta_0 - \theta_*||_2)\left(1 + T\right)\beta \notag\\
&\lesssim \frac{1}{T} + T^2\beta^2 + T\beta \notag\\
&\lesssim \frac{1}{T} + \frac{T^3\log(T/\zeta)}{n} + \frac{T^5\log(T/\zeta)\log^2(T/\delta)}{n^2\epsilon^2} \notag\\
& \quad + \frac{T\sqrt{T}\sqrt{\log(T/\zeta)}}{\sqrt{n}} + \frac{T^2\log(T/\delta)\sqrt{T\log(T/\zeta)}}{n\epsilon}.
\end{align} 
Since $T = n^{1/5}$, we obtain
\begin{align*}
\mathcal{R}(\theta_T) - \mathcal{R}(\theta_*) &\lesssim \frac{1}{n^{1/5}} + \frac{\sqrt{\log(n/\zeta)}}{n^{1/5}} + \frac{\log(n/\delta)\sqrt{\log(n/\zeta)}}{n^{1/2}\epsilon}\\
&\lesssim \frac{\sqrt{\log(n/\zeta)}}{n^{1/5}} + \frac{\log(n/\delta)\sqrt{\log(n/\zeta)}}{n^{1/2}\epsilon},
\end{align*}
as required.

Finally, using the assumption that $\epsilon \leq 0.9$, we have $\epsilon < 2\sqrt{2T\log(2/\delta)}$ and $\delta < 2T$, where $T = n^{1/5}$. Since each step of the gradient descent algorithm is $\left(\frac{\epsilon}{2\sqrt{2T\log(2/\delta)}}, \frac{\delta}{2T}\right)$-DP by Lemma \ref{lemma:Avg_Grad_Priv_GD_vs_AGD}, we have by Lemma \ref{lemma:compGDP} that $\theta_T$ is $(\epsilon, \delta)$-DP.


\subsubsection{Proof of Theorem \ref{theorem:Priv_HT_AGD_Smooth}}
\label{AppThmPrivSmooth}

By Lemma \ref{lemma:Avg_Grad_Priv_GD_vs_AGD}, we know that $g$ is a gradient estimator with $\alpha(\widetilde{n}, \widetilde{\zeta}) = 0$ and
\begin{align*}
\beta(\widetilde{n}, \widetilde{\zeta}) \asymp \sqrt{\frac{T\log(T/\zeta)}{n}} + \frac{T\sqrt{T\log(T/\zeta)\log^2(T/\delta)}}{n\epsilon},
\end{align*}
implying that
\begin{align*}
 \mathbb{P}(\forall t, \ ||g(\theta_t, \mathcal{D}_n, \widetilde{\zeta}) - \nabla\mathcal{R}(\theta_t)||_2 \leq \beta(\widetilde{n}, \widetilde{\zeta})) \geq 1 - \zeta.
\end{align*}
On this high-probability event, using the notation $\beta = \beta(\widetilde{n}, \widetilde{\zeta})$ and ignoring the dependency in $g$ on the samples and $\widetilde{\zeta}$, the error $e_t := g(\theta_t) - \nabla\mathcal{R}(\theta_t)$ satisfies $\|e_t\|_2 \le \beta$. Consider the $t^{\text{th}}$ step in Algorithm \ref{alg:RobPGDNFW}. Recall that $\mathcal{R}$ is $\tau_u$-smooth over $\mathbb{R}^p$. Since $\eta = \frac{1}{\tau_u}$ and $\lambda = \frac{t - 1}{t + 2}$, we have for $t \in \left\{1, \dots, T - 1\right\}$ that 
\begin{align*}
y_t & = \theta_t + \frac{t - 1}{t + 2}(\theta_t - \theta_{t - 1}), \\
\theta_{t + 1} & = y_t - \frac{1}{\tau_u}(\nabla \mathcal{R}(y_t) + e_t).
\end{align*}
Thus, by Lemma \ref{lemma:Inexact_AGD_lemma}, we have at iteration $T$ that 
\begin{align}
\label{eq:Smooth_AGD_Best_T}
\mathcal{R}(\theta_T) - \mathcal{R}(\theta_*) &\leq \frac{2\tau_u}{(T + 1)^2}\left(||\theta_0 - \theta_*||_2 + 2\sum_{i = 1}^Ti\frac{||e_{i - 1}||_2}{\tau_u}\right)^2 \notag\\
&\lesssim \frac{1}{T^2} + \frac{T^4\beta^2}{T^2} = \frac{1}{T^2} + T^2\beta^2 \notag\\
&\asymp \frac{1}{T^2} + \frac{T^3\log(T/\zeta)}{n} + \frac{T^5\log(T/\zeta)\log^2(T/\delta)}{n^2\epsilon^2}.
\end{align}
Since $T = n^{1/5}$, we obtain
\begin{align*}
\mathcal{R}(\theta_T) - \mathcal{R}(\theta_*) &\lesssim \frac{1}{n^{2/5}} + \frac{\log(n/\zeta)}{n^{2/5}} + \frac{\log(n/\zeta)\log^2(n/\delta)}{n\epsilon^2}\\
&\lesssim \frac{\log(n/\zeta)}{n^{2/5}} + \frac{\log(n/\zeta)\log^2(n/\delta)}{n\epsilon^2},
\end{align*}
as required. 

Finally, using the assumption that $\epsilon \leq 0.9$, we have $\epsilon < 2\sqrt{2T\log(2/\delta)}$ and $\delta < 2T$, where $T = n^{1/5}$. Since each step of the gradient descent algorithm is $\left(\frac{\epsilon}{2\sqrt{2T\log(2/\delta)}}, \frac{\delta}{2T}\right)$-DP by Lemma \ref{lemma:Avg_Grad_Priv_GD_vs_AGD}, we have by Lemma \ref{lemma:compGDP} that $\theta_T$ is $(\epsilon, \delta)$-DP.


\subsubsection{Proof of Theorem \ref{theorem:AGD}}
\label{SecThmAGD}

Here, we have $\mathcal{C} = \mathbb{R}^p$, so $\theta_{*} = \theta^*$ and $\nabla\mathcal{R}(\theta^*) = 0$. We have \iid samples $\mathcal{D}_n = \{z_i\}_{i = 1}^n$ satisfying
\begin{align*}
&\mathbb{P}\Bigg(\forall t, \ ||g(\theta_t + \lambda(\theta_t - \theta_{t - 1}), \mathcal{D}_n, \widetilde{\zeta}) - \nabla\mathcal{R}(\theta_t + \lambda(\theta_t - \theta_{t - 1}))||_2\\
&\qquad \qquad \leq \alpha(\widetilde{n}, \widetilde{\zeta})||\theta_t + \lambda(\theta_t - \theta_{t - 1}) - \theta^*||_2 + \beta(\widetilde{n}, \widetilde{\zeta})\Bigg) \\
& \qquad \geq 1 - \zeta.
\end{align*}
Working on this event of probability at least $1 - \zeta$, we have, using the notation $\alpha = \alpha(\widetilde{n}, \widetilde{\zeta})$ and $\beta = \beta(\widetilde{n}, \widetilde{\zeta})$, and ignoring the dependency in $g$ on the samples and $\widetilde{\zeta}$, that $g(\theta_t + \lambda(\theta_t - \theta_{t - 1})) = \nabla\mathcal{R}(\theta_t + \lambda(\theta_t - \theta_{t - 1})) + e_t$, where
\begin{align*}
||e_t||_2 \leq \alpha||\theta_t + \lambda(\theta_t - \theta_{t - 1}) - \theta^*||_2 + \beta.
\end{align*}
Since $\nabla\mathcal{R}(\theta^*) = 0$, and by letting $y_t := \theta_t + \lambda(\theta_t - \theta_{t - 1 })$, we obtain
\begin{align*}
||\theta_{t + 1} - \theta^*||_2 &= ||y_t - \eta g(y_t)- \theta^* - \eta \nabla\mathcal{R}(\theta^*)||_2\\
&= ||y_t - \theta^* -\eta(\nabla\mathcal{R}(y_t) - \nabla\mathcal{R}(\theta^*)) - \eta e_t||_2\\
&\leq ||y_t - \theta^* - \eta(\nabla\mathcal{R}(y_t) - \nabla\mathcal{R}(\theta^*))||_2 + \eta||e_t||_2.
\end{align*}
Note that $L_u = 2\tau_u - \tau_l > \tau_u$, so $\mathcal{R}$ is $L_u$-smooth. Hence, since $\eta = \frac{1}{\tau_u} = \frac{2}{L_u + \tau_l}$, using Lemma \ref{lemma:auxilGD}, we obtain
\begin{align*}
&||y_t - \theta^* - \eta(\nabla\mathcal{R}(y_t) - \nabla\mathcal{R}(\theta^*))||_2^2 = ||y_t - \theta^*||_2^2 + \eta^2||\nabla\mathcal{R}(y_t) - \nabla\mathcal{R}(\theta^*)||_2^2\\
& \qquad \qquad \qquad \qquad \qquad \qquad \qquad \qquad \qquad - 2\eta(\nabla\mathcal{R}(y_t) - \nabla\mathcal{R}(\theta^*))^T(y_t - \theta^*)\\
& \qquad \quad\leq ||y_t - \theta^*||_2^2 + \eta^2||\nabla\mathcal{R}(y_t) - \nabla\mathcal{R}(\theta^*)||_2^2\\
& \qquad \qquad - 2\eta\left(\frac{\tau_lL_u}{\tau_l + L_u}||y_t - \theta^*||^2 + \frac{1}{\tau_l + L_u}||\nabla\mathcal{R}(y_t) - \nabla\mathcal{R}(\theta^*)||^2\right)\\
& \qquad \quad = \left(1 - \frac{2\eta\tau_lL_u}{\tau_l + L_u}\right)||y_t - \theta^*||_2^2 + \eta\left(\eta - \frac{2}{\tau_l + L_u}\right)||\nabla\mathcal{R}(y_t) - \nabla\mathcal{R}(\theta^*)||_2^2\\
& \qquad \quad = \left(1 - \frac{2\eta\tau_lL_u}{\tau_l + L_u}\right)||y_t - \theta^*||_2^2 = \left(\frac{L_u - \tau_l}{\tau_l + L_u}\right)^2||y_t - \theta^*||_2^2.
\end{align*}
Thus, using the bound on $||e_t||_2$, we obtain
\begin{align*}
||\theta_{t + 1} - \theta^*||_2 & \leq \frac{L_u - \tau_l}{\tau_l + L_u}||y_t - \theta^*||_2 + \eta||e_t||_2 \\
& \leq \frac{L_u - \tau_l + 2\alpha}{\tau_l + L_u}||y_t - \theta^*||_2 + \eta\beta \\
& = k||y_t - \theta^*||_2 + \eta\beta \\
& = k||(1 + \lambda)(\theta_t - \theta^*) - \lambda(\theta_{t - 1} - \theta^*)||_2 + \eta\beta\\
&\leq (1 + \lambda)k||\theta_t - \theta^*||_2 + \lambda k||\theta_{t - 1} - \theta^*||_2 + \eta\beta,
\end{align*}
with $k = \frac{L_u - \tau_l + 2\alpha}{L_u + \tau_l}$. Also, $\lambda k > 0$ and $(1 + 2\lambda)k \neq 1$, and the solutions of the equation $x^2 - (1 + \lambda)kx - \lambda k = 0$ are $\frac{(1 + \lambda)k + \sqrt{(1 + \lambda)^2k^2 + 4\lambda k}}{2}$ and $\frac{(1 + \lambda)k - \sqrt{(1 + \lambda)^2k^2 + 4\lambda k}}{2}$, which are distinct.

Since $\frac{\alpha}{\tau_l} < f_2\left(\frac{\tau_u}{\tau_l}\right)$, we have
\begin{align*}
\frac{(1 + \lambda)k + \sqrt{(1 + \lambda)^2k^2 + 4\lambda k}}{2} < 1 &\iff (1 + \lambda)^2k^2 + 4\lambda k < 4 - 4(1 + \lambda)k + (1 + \lambda)^2k^2\\
&\iff \lambda < \frac{1 - k}{2k} \iff \frac{\sqrt{\tau_u} - \sqrt{\tau_l}}{\sqrt{\tau_u} + \sqrt{\tau_l}} < \frac{\tau_l - \alpha}{L_u - \tau_l + 2\alpha}\\
&\iff \frac{\alpha}{\tau_l} < \frac{1 - 2\lambda(\tau_u/\tau_l - 1)}{2\lambda - 1} \iff \frac{\alpha}{\tau_l} < 2f_2\left(\frac{\tau_u}{\tau_l}\right),
\end{align*}
which is true since $\frac{\alpha}{\tau_l} < f_2\left(\frac{\tau_u}{\tau_l}\right)$, and we also have 
\begin{align*}
-1 < \frac{(1 + \lambda)k - \sqrt{(1 + \lambda)^2k^2 + 4\lambda k}}{2} &\iff (1 + \lambda)^2k^2 + 4\lambda k\\
& \qquad \quad < 4 + 4(1 + \lambda)k + (1 + \lambda)^2k^2\\
&\iff \lambda k < 1 + k + \lambda k,
\end{align*}
which is true, as well. By Lemma \ref{lemma:9} and Remark \ref{remark:Second_Ord_Ineq_Rmk}, we have constants $C_1$ and $C_2$ such that for all $t \in \{1, \dots T\}$, we have
\begin{equation}
\label{eq:initb}
\begin{aligned}
||\theta_t - \theta^*||_2 \leq &C_1\left(\frac{(1 + \lambda)k + \sqrt{(1 + \lambda)^2k^2 + 4\lambda k}}{2}\right)^t\\
&+ C_2\left(\frac{(1 + \lambda)k - \sqrt{(1 + \lambda)^2k^2 + 4\lambda k}}{2}\right)^t + \frac{\eta\beta}{1 - (1 + \lambda)k - \lambda k}.
\end{aligned}
\end{equation}
Now that we have established an initial bound for $||\theta_t - \theta^*||_2$, we can move on to improve it. For $\rho^2 = 1 - \sqrt{\frac{\tau_l}{\tau_u}}$, $\widetilde{\theta_t} := \theta_t - \theta^*$, $\widetilde{y_t} := y_t - \theta^*$, $u_t:= \frac{1}{\tau_u}\nabla\mathcal{R}(y_t)$, and $\widetilde{u_t} := u_t - \nabla\mathcal{R}(\theta^*) = u_t$, consider the following quantities:
\begin{align*}
&V_0 := \mathcal{R}(\theta_0) - \mathcal{R}(\theta^*) + \frac{\tau_u}{2}||\widetilde{\theta}_0 - \rho^2\widetilde{\theta}_0||_2^2 = \mathcal{R}(\theta_0) - \mathcal{R}(\theta^*) + \frac{\tau_l}{2}||\theta_0 -\theta^*||_2^2, \\
&V_t := \mathcal{R}(\theta_t) - \mathcal{R}(\theta^*) + \frac{\tau_u}{2}||\widetilde{\theta}_t - \rho^2\widetilde{\theta}_{t - 1}||_2^2, \quad \forall 1 \le t \le T.
\end{align*}
Using $\tau_u$-smoothness, $\eta = \frac{1}{\tau_u}$, and the iterative step in Algorithm \ref{alg:RobPGDNFW} for Nesterov's method, we obtain
\begin{align*}
V_{t + 1} &= \mathcal{R}(\theta_{t + 1}) - \mathcal{R}(\theta^*) + \frac{\tau_u}{2}||\widetilde{\theta}_{t + 1} - \rho^2\widetilde{\theta}_t||_2^2\\
&\leq \mathcal{R}(y_t) - \mathcal{R}(\theta^*) + \frac{\tau_u}{2}||\widetilde{\theta}_{t + 1} - \rho^2\widetilde{\theta}_t||_2^2 + \nabla\mathcal{R}(y_t)^T(\theta_{t + 1} - y_t) + \frac{\tau_u}{2}||\theta_{t + 1} - y_t||_2^2\\
&= \mathcal{R}(y_t) - \mathcal{R}(\theta^*) + \frac{\tau_u}{2}||\widetilde{\theta}_{t + 1} - \rho^2\widetilde{\theta}_t||_2^2 - \tau_u||\widetilde{u}_t||_2^2 + \frac{\tau_u}{2}\left\|\frac{1}{\tau_u}g(y_t)\right\||_2^2 - \frac{1}{\tau_u}\nabla{R}(y_t)^Te_t\\
&= \mathcal{R}(y_t) - \mathcal{R}(\theta^*) + \frac{\tau_u}{2}||\widetilde{\theta}_{t + 1} - \rho^2\widetilde{\theta}_t||_2^2 - \tau_u||\widetilde{u}_t||_2^2 + \frac{\tau_u}{2}||\widetilde{u}_t||_2^2 + \frac{1}{\tau_u}\nabla\mathcal{R}(y_t)^Te_t\\
& \quad + \frac{\tau_u}{2}\left\|\frac{1}{\tau_u}e_t\right\|_2^2 - \frac{1}{\tau_u}\nabla{R}(y_t)^Te_t\\
&= \mathcal{R}(y_t) - \mathcal{R}(\theta^*) + \frac{\tau_u}{2}||\widetilde{\theta}_{t + 1} - \rho^2\widetilde{\theta}_t||_2^2 - \frac{\tau_u}{2}||\widetilde{u}_t||_2^2 + \frac{1}{2\tau_u}||e_t||_2^2\\
&= \rho^2(\mathcal{R}(y_t) - \mathcal{R}(\theta^*) + \tau_u\widetilde{u}_t^T(\widetilde{\theta}_t - \widetilde{y}_t)) - \rho^2\tau_u\widetilde{u}_t^T(\widetilde{\theta_t} - \widetilde{y}_t)\\
& \quad + (1 - \rho^2)(\mathcal{R}(y_t) - \mathcal{R}(\theta^*) - \tau_u\widetilde{u}_t^T\widetilde{y}_t)\\
& \quad + (1 - \rho^2)\tau_u\widetilde{u}_t^T\widetilde{y}_t -\frac{\tau_u}{2}||\widetilde{u}_t||_2^2 + \frac{\tau_u}{2}||\widetilde{\theta}_{t + 1} - \rho^2\widetilde{\theta}_t||_2^2 + \frac{1}{2\tau_u}||e_t||_2^2,
\end{align*}
where in the last equality, we added and subtracted the same terms multiple times. Using the definition of $\tau_l$-strong convexity, we obtain
\begin{align*}
\mathcal{R}(y_t) &\leq \mathcal{R}(\theta_t) - \nabla\mathcal{R}(y_t)^T(\theta_t - y_t) - \frac{\tau_l}{2}||\theta_t - y_t||_2^2\\
&= \mathcal{R}(\theta_t) - \tau_u\widetilde{u}_t^T(\widetilde{\theta_t} - \widetilde{y}_t) - \frac{\tau_l}{2}||\widetilde{\theta_t} - \widetilde{y}_t||_2^2
\end{align*}
and
\begin{align*}
\mathcal{R}(\theta^*) \geq \mathcal{R}(y_t) - \tau_u\widetilde{u}_t^T\widetilde{y}_t + \frac{\tau_l}{2}||\widetilde{y}_t||_2^2 \Rightarrow \mathcal{R}(y_t) - \mathcal{R}(\theta^*) \leq \tau_u\widetilde{u}_t^T\widetilde{y}_t - \frac{\tau_l}{2}||\widetilde{y}_t||_2^2.
\end{align*}
Plugging these two bounds into the inequality involving $V_{t + 1}$, we then obtain
\begin{align*}
V_{t + 1} &\leq \rho^2\left(\mathcal{R}(\theta_t) - \mathcal{R}(\theta^*) - \frac{\tau_l}{2}||\widetilde{\theta}_t - \widetilde{y}_t||_2^2\right) - \frac{\tau_l(1 - \rho^2)}{2}||\widetilde{y}_t||_2^2 - \rho^2\tau_u\widetilde{u}_t^T(\widetilde{\theta}_t - \widetilde{y}_t)\\
& \quad + (1 - \rho^2)\tau_u\widetilde{u}_t^T\widetilde{y}_t -\frac{\tau_u}{2}||\widetilde{u}_t||_2^2 + \frac{\tau_u}{2}||\widetilde{\theta}_{t + 1} - \rho^2\widetilde{\theta}_t||_2^2 + \frac{1}{2\tau_u}||e_t||_2^2\\
&=\rho^2V_t + R_t,
\end{align*}
with
\begin{align}
\label{EqnRt}
R_t &:= -\frac{\tau_l\rho^2}{2}||\widetilde{\theta}_t - \widetilde{y}_t||_2^2 - \frac{\tau_l(1 - \rho^2)}{2}||\widetilde{y}_t||_2^2 + \tau_u\widetilde{u}_t^T(\widetilde{y}_t - \rho^2\widetilde{\theta}_t) - \frac{\tau_u}{2}||\widetilde{u}_t||_2^2 \notag\\
& \qquad + \frac{\tau_u}{2}||\widetilde{\theta}_{t + 1} - \rho^2\widetilde{\theta}_t||_2^2 - \frac{\rho^2\tau_u}{2}||\widetilde{\theta}_t - \rho^2\widetilde{\theta}_{t - 1}||_2^2 + \frac{1}{2\tau_u}||e_t||_2^2.
\end{align}
Let us now examine $R_t$ more closely and bound it above by a finite quantity so that we will be able to iterate the recursive inequality involving $V_t$. We shall use the inequality that we have derived on $||\theta_t - \theta^*||_2$ in inequality \eqref{eq:initb}. First, we have
\begin{align*} 
\frac{\tau_u}{2}||\widetilde{\theta}_{t + 1} - \rho^2\widetilde{\theta}_t||_2^2 &= \frac{\tau_u}{2}||\widetilde{y}_t - \eta\nabla\mathcal{R}(y_t) -\rho^2\widetilde{\theta}_t - \eta e_t||_2^2\\
&= \frac{\tau_u}{2}||\widetilde{y_t} - \eta\nabla\mathcal{R}(y_t) -\rho^2\widetilde{\theta}_t||_2^2 + \frac{\tau_u}{2}\eta^2||e_t||_2^2\\
& \quad - 2\frac{\tau_u}{2}(\widetilde{y}_t - \eta\nabla\mathcal{R}(y_t) -\rho^2\widetilde{\theta}_t)^Te_t\\
&= \frac{\tau_u}{2}||\widetilde{y}_t - \widetilde{u}_t - \rho^2\widetilde{\theta}_t||_2^2 + \frac{1}{2\tau_u}||e_t||_2^2 - \tau_u(\widetilde{y}_t - \widetilde{u}_t - \rho^2\widetilde{\theta}_t)^Te_t.
\end{align*}
Putting this into equation~\eqref{EqnRt}, we obtain an expression which does not involve $e_t$ and one that does. Let us look at the one that does not involve $e_t$ and recall that $\rho^2 = 1 - \sqrt{\frac{\tau_l}{\tau_u}}$:
\begin{align*}
&-\frac{\tau_l\rho^2}{2}||\widetilde{\theta}_t - \widetilde{y}_t||_2^2 - \frac{\tau_l(1 - \rho^2)}{2}||\widetilde{y}_t||_2^2 + \tau_u\widetilde{u}_t^T(\widetilde{y}_t - \rho^2\widetilde{\theta}_t) - \frac{\tau_u}{2}||\widetilde{u}_t||_2^2 + \frac{\tau_u}{2}||\widetilde{y}_t - \widetilde{u}_t - \rho^2\widetilde{\theta}_t||_2^2\\
&- \frac{\rho^2\tau_u}{2}||\widetilde{\theta}_t - \rho^2\widetilde{\theta}_{t - 1}||_2^2\\
& \quad = -\frac{\tau_l\rho^2}{2}||\widetilde{\theta}_t - \widetilde{y}_t||_2^2 - \frac{\tau_l(1 - \rho^2)}{2}||\widetilde{y}_t||_2^2 + \tau_u\widetilde{u}_t^T(\widetilde{y}_t - \rho^2\widetilde{\theta}_t) - \frac{\tau_u}{2}||\widetilde{u}_t||_2^2\\
& \qquad + \frac{\tau_u}{2}||\widetilde{y}_t - \rho^2\widetilde{\theta}_t||_2^2 + \frac{\tau_u}{2}||\widetilde{u}_t||_2^2 - \tau_u\widetilde{u}_t^T(\widetilde{y}_t - \rho^2\widetilde{\theta}_t) - \frac{\rho^2\tau_u}{2}||\widetilde{\theta}_t - \rho^2\widetilde{\theta}_{t - 1}||_2^2\\
& \quad = -\frac{\tau_l\rho^2}{2}||\widetilde{\theta}_t - \widetilde{y}_t||_2^2 - \frac{\tau_l(1 - \rho^2)}{2}||\widetilde{y}_t||_2^2 + \frac{\tau_u}{2}||\widetilde{y}_t - \rho^2\widetilde{\theta}_t||_2^2 - \frac{\rho^2\tau_u}{2}||\widetilde{\theta}_t - \rho^2\widetilde{\theta}_{t - 1}||_2^2.
\end{align*}
By adding and subtracting $\rho^2\widetilde{y}_t$ and expanding the square, we then obtain
\begin{align*}
\frac{\tau_u}{2}||\widetilde{y}_t - \rho^2\widetilde{\theta}_t||_2^2 = \frac{\tau_u\rho^4}{2}||\widetilde{y}_t - \widetilde{\theta}_t||_2^2 + \frac{\tau_u(1 - \rho^2)^2}{2}||\widetilde{y}_t||_2^2 + \tau_u\rho^2(1 - \rho^2)(\widetilde{y}_t - \widetilde{\theta}_t)^T\widetilde{y}_t.
\end{align*}
For the term $- \frac{\rho^2\tau_u}{2}||\widetilde{\theta}_t - \rho^2\widetilde{\theta}_{t - 1}||_2^2$, using the definitions of $\lambda$ and $\rho$, we have $\frac{2\lambda}{1 + \lambda} = \rho^2$ and $\lambda = \frac{\rho^2}{2 - \rho^2}$. Thus, using $\widetilde{\theta}_{t - 1} = \frac{(1 + \lambda)\widetilde{\theta}_t - \widetilde{y}_t}{\lambda}$ we obtain
\begin{align*}
- \frac{\rho^2\tau_u}{2}||\widetilde{\theta}_t - \rho^2\widetilde{\theta}_{t - 1}||_2^2 &= - \frac{\rho^2\tau_u}{2}||\widetilde{\theta}_t - (2 - \rho^2)\widetilde{y}_t||_2^2 = - \frac{\rho^2\tau_u}{2}||\widetilde{\theta}_t - \widetilde{y}_t - (1 - \rho^2)\widetilde{y}_t||_2^2\\
&= - \frac{\rho^2\tau_u}{2}||\widetilde{\theta}_t - \widetilde{y}_t||_2^2 - \frac{\rho^2\tau_u(1 - \rho^2)^2}{2}||\widetilde{y}_t||_2^2\\
& \quad + \tau_u\rho^2(1 - \rho^2)(\widetilde{\theta}_t - \widetilde{y}_t)^T\widetilde{y}_t.
\end{align*}
Putting these together, the term in the expression for $R_t$ not involving $e_t$ becomes
\begin{align*}
&-\frac{\tau_l\rho^2}{2}||\widetilde{\theta}_t - \widetilde{y}_t||_2^2 - \frac{\tau_l(1 - \rho^2)}{2}||\widetilde{y}_t||_2^2 + \frac{\tau_u}{2}||\widetilde{y}_t - \rho^2\widetilde{\theta}_t||_2^2 - \frac{\rho^2\tau_u}{2}||\widetilde{\theta}_t - \rho^2\widetilde{\theta}_{t - 1}||_2^2 \\
& \quad = \left(\frac{-\tau_l\rho^2}{2} + \frac{\tau_u\rho^4}{2} - \frac{\tau_u\rho^2}{2}\right)||\widetilde{\theta}_t - \widetilde{y}_t||_2^2\\
& \qquad +\left(\frac{\tau_u(1 - \rho^2)^2}{2} - \frac{\rho^2\tau_u(1 - \rho^2)^2}{2} - \frac{\tau_l(1 - \rho^2)}{2}\right)||\widetilde{y}_t||_2^2 \\
& \quad = -\frac{1}{2}\tau_u\rho^2\left(\frac{\tau_l}{\tau_u} + \sqrt{\frac{\tau_l}{\tau_u}}\right)||\widetilde{\theta}_t - \widetilde{y}_t||_2^2,
\end{align*}
since the term multiplying $||\widetilde{y}_t||_2^2$ is $0$, and we again used the definition of $\rho$. Importantly, this expression is negative. Thus, we can write $R_t$ in a more compact form, and applying Cauchy-Schwarz and the triangle inequality repeatedly, we obtain
\begin{align*}
R_t &= -\frac{1}{2}\tau_u\rho^2\left(\frac{\tau_l}{\tau_u} + \sqrt{\frac{\tau_l}{\tau_u}}\right)||\widetilde{\theta}_t - \widetilde{y}_t||_2^2 + \frac{1}{2\tau_u}||e_t||_2^2 - \tau_u(\widetilde{y}_t - \eta\nabla\mathcal{R}(y_t) -\rho^2\widetilde{\theta}_t)^Te_t\\
& \quad + \frac{1}{2\tau_u}||e_t||_2^2\\
&\leq \frac{1}{\tau_u}||e_t||_2^2 + \tau_u||\widetilde{y}_t - \eta\nabla\mathcal{R}(y_t) -\rho^2\widetilde{\theta}_t||_2||e_t||_2 \leq \eta||e_t||_2^2 + \tau_u(||y_t - \theta^*||_2\\
& \quad + \eta||\nabla\mathcal{R}(y_t)||_2 + \rho^2||\theta_t - \theta^*||_2)||e_t||_2\\
&\leq \eta||e_t||_2^2 + \tau_u((1 + \lambda)||\widetilde{\theta}_t||_2 + \lambda||\widetilde{\theta}_{t - 1}||_2 + \eta||\widetilde{y}_t||_2 + \rho^2||\widetilde{\theta}_t||_2)||e_t||_2\\
&\leq \eta||e_t||_2^2\\
& \quad + \tau_u((1 + \lambda)||\widetilde{\theta}_t||_2 + \lambda||\widetilde{\theta}_{t - 1}||_2 + \eta((1 + \lambda)||\widetilde{\theta}_t||_2 + \lambda||\widetilde{\theta}_{t - 1}||_2) + \rho^2||\widetilde{\theta}_t||_2)||e_t||_2\\
&= \eta||e_t||_2^2 + \tau_u\left[(1 + \eta)((1 + \lambda)||\widetilde{\theta}_t||_2 + \lambda||\widetilde{\theta}_{t - 1}||_2) + \rho^2||\widetilde{\theta}_t||_2\right]||e_t||_2\\
&\leq \eta\left(\alpha||\widetilde{\theta}_t||_2 + \beta\right)^2 + \tau_u\left[(1 + \eta)((1 + \lambda)||\widetilde{\theta}_t||_2 + \lambda||\widetilde{\theta}_{t - 1}||_2) + \rho^2||\widetilde{\theta}_t||_2\right](\alpha||\widetilde{\theta}_t||_2 + \beta).
\end{align*}
Define $x^* \approx 1.76759$ to be the solution to the equation $f_1(x) = f_2(x)$ for $x \geq 1$. Since $f_1\left(\frac{\tau_u}{\tau_l}\right) < \frac{\alpha}{\tau_l}$ and $\frac{\tau_u}{\tau_l} < x^*$, we have $\tau_u < \frac{x^*\alpha}{f_1\left(\frac{\tau_u}{\tau_l}\right)}$ and $f_1\left(\frac{\tau_u}{\tau_l}\right) \neq 0$, because $\tau_u \neq \tau_l$. Thus, there exists a constant $C_3^{'}$ depending on $\tau_u$ and $\tau_l$ such that $k = \frac{L_u - \tau_l + 2\alpha}{L_u + \tau_l} < C_3^{'}\alpha$. Therefore, there is a constant $C_3^{''}$ depending on $\tau_u$ and $\tau_l$ such for any $t$, since $\left|\frac{(1 + \lambda)k + \sqrt{(1 + \lambda)^2k^2 + 4\lambda k}}{2}\right| < 1$, we have
\begin{align*}
\left|\frac{(1 + \lambda)k + \sqrt{(1 + \lambda)^2k^2 + 4\lambda k}}{2}\right|^t < C_3^{''}\alpha,
\end{align*}
since under the square root, we take out a $k^2$ and bound below $k \ge \frac{L_u - \tau_l}{L_u + \tau_l}$. Thus, there is a constant $C_3$ depending on $\tau_u$ and $\tau_l$ such that for all $t$, we have$||\widetilde{\theta}_t||_2 \leq C_3\alpha + \frac{\eta\beta}{1 - (1 + \lambda)k - \lambda k}$, using inequality \eqref{eq:initb}. Thus, using the last bound on $R_t$ and the bound on $||\widetilde{\theta_t}||_2$ involving $C_3$, we obtain
\begin{align*}
R_t &\leq \left(\alpha^2 C_3 + \frac{\eta\alpha\beta}{1 - (1 + \lambda)k - \lambda k} + \beta\right)\\
& \qquad \cdot \left[\eta\beta + \left(\eta\alpha + \tau_u(1 + \eta)(1 + 2\lambda + \rho^2)\right)\left(C_3\alpha + \frac{\eta\beta}{1 - (1 + \lambda)k - \lambda k}\right)\right]\\
&= \left(\alpha^2 C_3 + \frac{\eta\alpha\beta}{1 - (1 + \lambda)k - \lambda k} + \beta\right)\\
& \qquad \cdot \left[\eta\beta + \left(\eta\alpha + \tau_u(1 + \eta)\left(2\lambda + \sqrt{\frac{\tau_l}{\tau_u}}\right)\right)\left(C_3\alpha + \frac{\eta\beta}{1 - (1 + \lambda)k - \lambda k}\right)\right].
\end{align*}
Call this RHS term $\frac{R}{2}$. We will add this quantity at the end so that the calculations are not too messy. Thus, we have
\begin{align*}
V_{t + 1} \leq \rho^2V_t + R \Rightarrow V_t \leq \rho^{2t}V_0 + \frac{R}{2(1 - \rho^2)},
\end{align*}
implying that
\begin{align*}
\mathcal{R}(\theta_t) - \mathcal{R}(\theta^*) \leq V_t \leq \rho^{2t}V_0 + \frac{R}{2(1 - \rho^2)}.
\end{align*}
Using the $\tau_l$-strong convexity of $\mathcal{R}$, we then obtain
\begin{align*}
||\theta_t - \theta^*||_2^2 \leq \frac{2}{\tau_l}V_0\rho^{2t} + \frac{R}{\tau_l(1 - \rho^2)}.
\end{align*}
Using the fact that for $x, y \geq 0$, we have $\sqrt{x + y} \leq \sqrt{x} + \sqrt{y}$, we obtain
\begin{align*}
||\theta_t - \theta^*||_2 &\leq \sqrt{\frac{2}{\tau_l}V_0}\left(1 - \sqrt{\frac{\tau_l}{\tau_u}}\right)^{t/2} + \sqrt{\frac{R}{\tau_l(1 - \rho^2)}}\\ 
&= \sqrt{\frac{2}{\tau_l}\left(\mathcal{R}(\theta_0) - \mathcal{R}(\theta^*)\right) + ||\theta_0 - \theta^*||_2^2}\left(1 - \sqrt{\frac{\tau_l}{\tau_u}}\right)^{t/2} + \left(\frac{\tau_u}{\tau_l}\right)^{1/4}\sqrt{\frac{R}{\tau_l}},
\end{align*}
as required. Finally, note that if $\tau_u, \tau_l, \sigma \asymp 1$, then $R = O\left(\alpha(\widetilde{n}, \widetilde{\zeta})^2\right)$.

\begin{remark}
One can also carry out calculations to see that the initial bound on $||\theta_t - \theta^*||_2$ that we derived in inequality \eqref{eq:initb} having two exponential terms is worse than the one for the projected gradient descent, since $\frac{(1 + \lambda)k + \sqrt{(1 + \lambda)^2k^2 + 4\lambda k}}{2}$ with $k = \frac{L_u - \tau_l + 2\alpha}{L_u + \tau_l}$ is greater than $\frac{\tau_u - \tau_l + 2\alpha}{\tau_u + \tau_l}$. 

Note also that the inequality involving $f_1$ and the assumption that $\tau_u \neq \tau_l$ could be dropped. These assumptions are used only to bound $k = \frac{\tau_u - \tau_l + \alpha}{\tau_u} = \frac{L_u - \tau_l + 2\alpha}{L_u + \tau_l}$ above by a constant multiple of $\alpha$ and to obtain faster rates than the projected gradient descent method. If we do not ask for these assumptions, we cannot guarantee faster rates, and also, we could only hope to bound $||\theta_t - \theta^*||_2$ by a constant plus $\frac{\eta\beta}{1 - (1 + \lambda)k - \lambda k}$. This is important because, as one can see in Lemma \ref{lemma:Hu_eps_PGD} in Appendix \ref{sec:Huber epsilon-Contamination Robustness}, the error term in the Huber $\epsilon$-contamination setting for projected gradient descent applied to linear regression is asymptotically $O\left(\sqrt{\epsilon \log(p)}\right)$. If we only have that $||\theta_t - \theta^*||_2$ is less than a constant plus $\frac{\eta\beta}{1 - (1 + \lambda)k - \lambda k}$, the error term in Nesterov's AGD is potentially worse.
\end{remark}


\subsubsection{Proof of Theorem \ref{theorem:htAGD}}
\label{AppThmHTAGD}

We use the following result:
\begin{lemma}[\cite{RE}]
\label{lemma:geglmht}
Consider the linear regression with squared error loss model from Section \ref{sec:Linear Regression} with i.i.d.\ samples $\mathcal{D}_n = \{z_i\}_{i = 1}^n = \{(x_i, y_i)\}_{i = 1}^n$ from a heavy-tailed distribution. Then Algorithm \ref{alg:HTGE} returns, for a fixed $\theta \in \mathbb{R}^p$, a gradient estimator $g$ such that
\begin{align*}
||g(\theta; \mathcal{D}_n, \widetilde{\zeta}) - \nabla\mathcal{R}(\theta)||_2 \lesssim \sqrt{\frac{\ p \log(1/\widetilde{\zeta})}{\widetilde{n}}}||\theta - \theta^*||_2 + \sqrt{\frac{\sigma_2^2 p \log(1/\widetilde{\zeta})}{\widetilde{n}}},
\end{align*}
with probability at least $1 - \widetilde{\zeta}$, and for $\widetilde{\zeta}$ such that $b \leq \widetilde{n}/2$ with $b$ as in Algorithm \ref{alg:HTGE}. Hence, $g$ is a gradient estimator with
\begin{align*}
&\alpha(\widetilde{n}, \widetilde{\zeta}) \asymp \sqrt{\frac{p \log(1/\widetilde{\zeta})}{\widetilde{n}}},
&\beta(\widetilde{n}, \widetilde{\zeta}) \asymp \sqrt{\frac{p\sigma_2^2\log(1/\widetilde{\zeta})}{\widetilde{n}}}.
\end{align*}
\end{lemma}

From Lemma \ref{lemma:geglmht}, we obtain $g(\theta)$ with the corresponding $\alpha(\widetilde{n}, \widetilde{\zeta})$ and $\beta(\widetilde{n}, \widetilde{\zeta})$. The assumption on $n$ ensures that we have $f_1\left(\frac{\tau_u}{\tau_l}\right) < \frac{\alpha(\widetilde{n}, \widetilde{\zeta})}{\tau_l} < f_2\left(\frac{\tau_u}{\tau_l}\right)$, so stability is achieved. Using Theorem \ref{theorem:AGD}, we obtain the desired result, with $R = O\left(\alpha(\widetilde{n}, \widetilde{\zeta})^2\right)$, if $\tau_u, \tau_l, \sigma \asymp 1$.


\subsection{Auxiliary Results from Section \ref{sec:Unbiased Parameter Estimation}}
\label{sec:Proofs of the Auxiliary Results_Nester}

Here, we present statements and proofs of auxiliary results used in Section \ref{sec:Unbiased Parameter Estimation}.



\begin{lemma}
\label{lemma:Avg_Grad_Priv_GD_vs_AGD}
Let $T, \epsilon > 0$ and $\delta \in (0, 1)$ be such that $\epsilon < 2\sqrt{2T\log(2/\delta)}$ and $\delta < 2T$. Consider a data space $\mathcal{E}$ and a dataset $\mathcal{D}_n = \{z_i\}_{i = 1}^n \subseteq \mathcal{E}^n$ drawn \iid from some distribution $P$. Let $\mathcal{L} : \mathbb{R}^p \times \mathcal{E} \rightarrow \mathbb{R}$ be a loss that is convex in $\theta$ over the whole of $\mathbb{R}^p$. Moreover, assume that $\mathcal{L}$ is $L_2$-Lipschitz over $\mathbb{R}^p$, for all $z \in \mathcal{E}$. Consider the corresponding risk $\mathcal{R(\theta)} = \mathbb{E}_{z \sim P}[\mathcal{L}(\theta, z)]$. For $\theta \in \mathbb{R}^p$ fixed, $\zeta \in (0, 1)$, and $n > 8\log(4/\zeta)$, we have with probability at least $1 - \zeta$ that
\begin{align*}
||\widehat{\mu} - \nabla\mathcal{R}(\theta)||_2 \leq \sqrt{\frac{32L_2^2\log(4/\zeta)}{n}} + \frac{8L_2\sqrt{8pT\log(8/\zeta)\log(5T/2\delta)\log(2/\delta)}}{n\epsilon},
\end{align*}
where $\widehat{\mu} = \frac{1}{n}\sum_{i = 1}^n\nabla \mathcal{L}(\theta, z_i) + \xi$ and $\xi \sim N\left(0, \frac{64L_2^2T\log(5T/2
\delta)\log(2/\delta)}{n^2 \epsilon^2}I_p\right)$. Moreover, $\widehat{\mu}$ is $\left(\frac{\epsilon}{2\sqrt{2T\log(2/\delta)}}, \frac{\delta}{2T}\right)$-DP.
\end{lemma}

\begin{proof}
Let $\widehat{w} = \frac{1}{n}\sum_{i = 1}^n\nabla \mathcal{L}(\theta, z_i)$, so $\widehat{\mu} = \widehat{w} + \xi$. We have by Lemma \ref{lemma:CIGV} that $\mathbb{P}(\Omega_1) \geq 1 - \zeta/2$, where
\begin{align*}
\Omega_1 = \left\{\left\|\xi\right\|_2 \leq \frac{8L_2\sqrt{8pT\log(8/\zeta)\log(5T/2\delta)\log(2/\delta)}}{n\epsilon} \right\}.
\end{align*}
Now observe that $\mathbb{E}[\nabla \mathcal{L}(\theta, z_i) - \nabla \mathcal{R}(\theta)] = 0$ and $||\nabla \mathcal{L}(\theta, z_i) - \nabla \mathcal{R}(\theta)||_2 \leq 2L_2$, for all $i \in [n]$, and the data are independent. Also note that $\mathbb{E}\left[||\nabla \mathcal{L}(\theta, z_i) - \nabla \mathcal{R}(\theta)||_2^2\right] \leq 4L_2^2$. Since $n > 8\log(4/\zeta)$, we have  $\sqrt{\frac{32L_2^2\log(4/\zeta)}{n}} < \frac{4L_2^2}{2L_2}$. Hence, by Lemma \ref{lemma:Bern_Vec}, we have $\mathbb{P}(\Omega_2) \geq 1 - \zeta/2$, where
\begin{align*}
\Omega_2 = \left\{||\widehat{w} - \nabla\mathcal{R}(\theta)||_2 \leq \sqrt{\frac{32L_2^2\log(4/\zeta)}{n}}\right\}.
\end{align*}
Thus, for $n > 8\log(4/\zeta)$, with probability at least $1 - \zeta$, we have
\begin{align*}
||\widehat{\mu} - \nabla\mathcal{R}(\theta)||_2 \leq \sqrt{\frac{32L_2^2\log(4/\zeta)}{n}} + \frac{8L_2\sqrt{8pT\log(8/\zeta)\log(5T/2\delta)\log(2/\delta)}}{n\epsilon},
\end{align*}
as required.

Regarding privacy, the sensitivity of the gradients is bounded above by $\frac{2L_2}{n}$. Since $\epsilon < 2\sqrt{2T\log(2/\delta)}$ and $\delta < 2T$, and by the choice of the variance of the noise $\xi$, we have by Lemma \ref{lemma:mGDP} that $\widehat{\mu}$ is $\left(\frac{\epsilon}{2\sqrt{2T\log(2/\delta)}}, \frac{\delta}{2T}\right)$-DP.
\end{proof}



\begin{lemma}[Adapted from \cite{Schmidt_2011_Conv}]
\label{lemma:Inexact_GD_lemma}
Let $p \in \mathbb{N}$. Assume $F : \mathbb{R}^p \rightarrow \mathbb{R}$ is convex and $\beta_F$-smooth over $\mathbb{R}^p$, with $x_* \in \mathop{\arg\min}\limits_{x \in \mathbb{R}^p}F(x)$. Consider the gradient descent procedure initialized at $x_0$, such that
\begin{align*}
x_{t + 1} = x_t - \frac{1}{\beta_F}(\nabla F(x_t) + e_t), \quad \forall t \ge 0,
\end{align*}
with the sequence of errors $\{e_{t}\}_{t \geq 1}$ being arbitrary. For all $t \geq 1$ and $a_t = \sum_{i = 1}^t\frac{||e_{i - 1}||_2}{\beta_F}$, we have 
\begin{align*}
F(x_t) - F(x_*) \leq \frac{\frac{\beta_F}{2}||x_0 - x_*||_2^2 + (2a_t + ||x_0 - x_*||_2)\left(\beta_Fa_t + 2\sum_{i = 2}^t(i - 1)||e_{i - 1}||_2\right)}{t}.
\end{align*}
\end{lemma}

\begin{proof}
By the convexity and $\beta_F$-smoothness of $F$, we have for $i \leq t$ that
\begin{align*}
F(x_i) &\leq F(x_{i - 1}) + \nabla F(x_{i - 1})^T(x_i - x_{i - 1}) + \frac{\beta_F}{2}||x_i - x_{i - 1}||_2^2\\
&\leq F(x_*) + \nabla F(x_{i - 1})^T(x_{i - 1} - x_*) + \nabla F(x_{i - 1})^T(x_i - x_{i - 1}) + \frac{\beta_F}{2}||x_i - x_{i - 1}||_2^2\\
&= F(x_*) + \nabla F(x_{i - 1})^T(x_i - x_*) + \frac{\beta_F}{2}||x_i - x_{i - 1}||_2^2.
\end{align*}
Since $\nabla F(x_{i - 1}) = \beta_F(x_{i - 1} - x_i) - e_{i - 1}$, we obtain
\begin{align*}
F(x_i) &\leq F(x_*) + \frac{\beta_F}{2}||x_i - x_{i - 1}||_2^2 + \beta_F(x_{i - 1} - x_i)^T(x_i - x_*) - e_{i - 1}^T(x_i - x_*)\\
& \leq F(x_*) + \frac{\beta_F}{2}(x_i - x_{i - 1})^T(x_i - x_{i - 1} - 2x_i + 2x_*) + ||e_{i - 1}||_2||x_i - x_*||_2\\
&= F(x_*) - \frac{\beta_F}{2}||x_i - x_*||_2^2 + \frac{\beta_F}{2}||x_{i - 1} - x_*||_2^2 + ||e_{i - 1}||_2||x_i - x_*||_2.
\end{align*}
Hence, we have
\begin{align}
\label{eq:Smooth_GD_Sum}
\sum_{i = 1}^t(F(x_i) - F(x_*)) + \frac{\beta_F}{2}||x_t - x_*||_2^2 \leq \frac{\beta_F}{2}||x_0 - x_*||_2^2 + \sum_{i = 1}^t||e_{i - 1}||_2||x_i - x_*||_2.
\end{align}
Since $F(x_i) \leq F(x_{i - 1}) + \nabla F(x_{i - 1})^T(x_i - x_{i - 1}) + \frac{\beta_F}{2}||x_i - x_{i - 1}||_2^2$ and $\nabla F(x_{i - 1}) = \beta_F(x_{i - 1} - x_i) - e_{i - 1}$, we have for all $i \geq 1$ that
\begin{align*}
F(x_i) \leq F(x_{i - 1}) - \frac{\beta_F}{2}||x_i - x_{i - 1}||_2^2 -e_{i - 1}^T(x_i - x_{i - 1}) \leq F(x_{i - 1}) + ||e_{i - 1}||_2||x_i - x_{i - 1}||_2.
\end{align*}
Thus, using this in the RHS of inequality \eqref{eq:Smooth_GD_Sum}, we obtain for $i \leq t$ that
\begin{align}
\label{eq:Smooth_GD_Sum2}
& t(F(x_t) - F(x_*)) + \frac{\beta_F}{2}||x_t - x_*||_2^2 \notag\\
& \qquad \leq \frac{\beta_F}{2}||x_0 - x_*||_2^2 + \sum_{i = 1}^t||e_{i - 1}||_2||x_i - x_*||_2 + \sum_{i = 2}^t(i - 1)||e_{i - 1}||_2||x_i - x_{i - 1}||_2 \notag\\
& \qquad \leq \frac{\beta_F}{2}||x_0 - x_*||_2^2 + \sum_{i = 1}^t||e_{i - 1}||_2||x_i - x_*||_2 \notag\\
& \qquad \quad + \sum_{i = 2}^t(i - 1)||e_{i - 1}||_2(||x_i - x_*||_2 + ||x_{i - 1} - x_*||_2).
\end{align}
Hence, we need to control $||x_i - x_*||_2$ for $i \leq t$. By inequality \eqref{eq:Smooth_GD_Sum}, since $x_*$ is a minimizer, we have for all $t \geq 1$ that
\begin{align*}
||x_t - x_*||_2^2 \leq ||x_0 - x_*||_2^2 + \frac{2}{\beta_F}\sum_{i = 1}^t||e_{i - 1}||_2||x_i - x_*||_2. 
\end{align*}
Using Lemma \ref{lemma:Smooth_GD_Helper} with $S_t = ||x_0 - x_*||_2^2$,  $\lambda_i = \frac{2||e_{i - 1}||_2}{\beta_F}$, and $a_t = \sum_{i = 1}^t\frac{||e_{i - 1}||_2}{\beta_F}$, we obtain
\begin{align*}
||x_t - x_*||_2 \leq a_t + \left(||x_0 - x_*||_2^2 + a_t^2\right)^{1/2}.
\end{align*}
Since the sequence $\{a_i\}$ is increasing in $i$, we have for all $i \leq t$ that
\begin{align*}
||x_i - x_*||_2 &\leq a_i + \left(||x_0 - x_*||_2^2 + a_i^2\right)^{1/2} \leq a_t + \left(||x_0 - x_*||_2^2 + a_t^2\right)^{1/2}\\
&\leq 2a_t + ||x_0 - x_*||_2.
\end{align*}
Plugging this into inequality \eqref{eq:Smooth_GD_Sum2} and dropping the $\frac{\beta_F}{2}||x_t - x_*||_2^2$ term on the RHS, we obtain
\begin{align*}
& t(F(x_t) - F(x_*)) \\
& \qquad \leq \frac{\beta_F}{2}||x_0 - x_*||_2^2 + \sum_{i = 1}^t||e_{i - 1}||_2||x_i - x_*||_2\\
& \qquad \quad + \sum_{i = 2}^t(i - 1)||e_{i - 1}||_2(||x_i - x_*||_2 + ||x_{i - 1} - x_*||_2)\\
& \qquad \leq \frac{\beta_F}{2}||x_0 - x_*||_2^2 + (2a_t + ||x_0 - x_*||_2)\left(\beta_Fa_t + 2\sum_{i = 2}^t(i - 1)||e_{i - 1}||_2\right).
\end{align*}
Dividing by $t$, we obtain the desired result.
\end{proof}

\begin{lemma}[\cite{Schmidt_2011_Conv}]
\label{lemma:Inexact_AGD_lemma}
Let $p \in \mathbb{N}$. Assume $F : \mathbb{R}^p \rightarrow \mathbb{R}$ is convex and $\beta_F$-smooth over $\mathbb{R}^p$, with $x_* \in \mathop{\arg\min}\limits_{x \in \mathbb{R}^p}F(x)$. Consider Nesterov's accelerated gradient method initialized at $x_0$ and $x_1$, such that for $t \geq 1$:
\begin{align*}
&y_t = x_t + \frac{t - 1}{t + 2}(x_t - x_{t - 1}),\\
&x_{t + 1} = y_t - \frac{1}{\beta_F}(\nabla F(y_t) + e_t),
\end{align*}
with the sequence of errors $\{e_{t}\}_{t \geq 1}$ being arbitrary. For all $t \geq 1$, we have
\begin{align*}
F(x_t) - F(x_*) \leq \frac{2\beta_F}{(t + 1)^2}\left(||x_0 - x_*||_2 + 2\sum_{i = 1}^ti\frac{||e_{i - 1}||_2}{\beta_L}\right)^2.
\end{align*}
\end{lemma}


\section{Supplementary Results for Section~\ref{sec:Strongly Convex Risks}}
\label{sec:Huber epsilon-Contamination Robustness}

\subsection{Huber Contamination Robustness} 

We now discuss the notion of robustness in the Huber $\epsilon$-contamination setting, when the risk is strongly convex. The analysis will follow the logic used in Section \ref{sec:Strongly Convex Risks}. In the setting of Huber’s $\epsilon$-contamination model, instead of having observations directly from a distribution $F$, we observe data from a contaminated distribution with a proportion of expected outliers equal to $\epsilon$:
\begin{align*}
P = (1-\epsilon)F + \epsilon Q,
\end{align*}
for an arbitrary distribution $Q$ that allows us to model the outliers themselves. Several authors \cite{RE, CompIntr, Balakrish} considered noisy gradient methods, which can be seen as applications of robust mean estimators, to obtain robust estimators for various learning problems, such as estimation in parametric models \cite{DP_ROB_HIGH_DIM, NO_ROB_LR}.

Let us now discuss our approach in detail. Similar to the $G_{MOM}$ estimator in the heavy-tailed setting from Section \ref{sec:Strongly Convex Risks}, we have the $HuberGradientEstimator$ algorithm (Algorithm \ref{alg:HubGradTrunc}) from \cite{RE}. This comes together with another algorithm, namely the $HuberOutlierGraidientTruncation$ algorithm (Algorithm \ref{alg:HuOutlier}).
\begin{algorithm}
\caption{Huber Gradient Estimator}\label{alg:HubGradTrunc}
\begin{algorithmic}[1]
\Function{HuberGradientEstimator}{$\text{Sample Gradients } S = \{\nabla \mathcal{L}(\theta; z_i)\}_{i=1}^n$, Corruption Level $\epsilon$, Dimension $p$, $\delta$}
    \State $\widetilde{S} = \text{HuberOutlierGradientTruncation}(S, \epsilon, p, \delta)$.
    \If{$p=1$}
        \State \Return $\text{mean}(\widetilde{S})$
    \Else
        \State Compute $\Sigma_{\widetilde{S}}$, the covariance matrix of $\widetilde{S}$.
        \State Let $V$ be the span of the top $p/2$ principal components of $\Sigma_{\widetilde{S}}$ and $W$ be its complement.
        \State Compute $S_1 := P_V(\widetilde{S})$ where $P_V$ is the projection operation onto $V$.
        \State Let $\widehat{\mu}_{V} := \text{HuberGradientEstimator}(S_1, \epsilon, p/2, \delta)$.
        \State Set $\widehat{\mu}_{W} := \text{mean}(P_W(\widetilde{S}))$.
        \State Let $\widehat{\mu} \in \mathbb{R}^p$ be such that $P_V(\widehat{\mu}) = \widehat{\mu}_V$, and $P_W(\widehat{\mu}) = \widehat{\mu}_W$.
        \State \Return $\widehat{\mu}$.
    \EndIf
\EndFunction
\end{algorithmic}
\end{algorithm}

\begin{algorithm}
\caption{Huber Outlier Gradients Truncation}\label{alg:HuOutlier}
\begin{algorithmic}[1]
\Function{HuberOutlierGradientTruncation}{$S$, $\epsilon$, $p$, $\delta$}
    \If{$p=1$}
        \State Let $[a, b]$ be the smallest interval containing $1 - \epsilon - C\sqrt{\frac{\log(|S|/\delta)}{|S|}}(1 - \epsilon)$ fraction of points.
        \State $\widetilde{S} \gets S \cap [a, b]$.
        \State \Return $\widetilde{S}$
    \Else
    \State Let $[S]_i$ be the samples with $i^{\text{th}}$ coordinates only, $[S]_i = \{x^Te_i | x \in S\}$.
        \For{$i = 1$ to $p$}
            \State $a[i] = \text{HuberGradientEstimator}([S]_i, \epsilon, 1, \delta/p)$.
        \EndFor
        \State Let $B(r, a)$ be the ball of smallest radius centered at $a$ containing a $\left(1 - \epsilon - C_p\left(\sqrt{\frac{p}{|S|}\log\left(\frac{|S|}{p\delta}\right)}\right)\right)(1 - \epsilon)$ fraction of points in $S$.
        \State $\widetilde{S} \gets S \cap B(r, a)$.
        \State \Return $\widetilde{S}$
    \EndIf
\EndFunction
\end{algorithmic}
\end{algorithm}

For Algorithm \ref{alg:HubGradTrunc}, we have the following theoretical guarantee from \cite{RE}, which crucially makes a bounded $4^{\text{th}}$ moments assumption as per Definition \ref{def:bd_moments}:

\begin{lemma}[\cite{RE}]
\label{lemma:5}
For $\mathcal{D}_n = \{z_i\}_{i = 1}^n$ \iid samples from the Huber $\epsilon$-contaminated distribution, with the distribution of the true gradients $\nabla \mathcal{L}(\theta, z)$ having bounded $4^{\text{th}}$ moments, Algorithm \ref{alg:HubGradTrunc} returns, for any fixed $\theta \in \mathbb{R}^p$, an estimate $\widehat{\mu}$ such that with probability at least $1 - \zeta$, we have
\begin{align*}
||\widehat{\mu} - \nabla\mathcal{R}(\theta)||_2 \lesssim (\sqrt{\epsilon} + \gamma(n, p, \zeta, \epsilon))\sqrt{||\mathrm{\mathrm{Cov}}(\nabla \mathcal{L}(\theta, z))||_2\log(p)},
\end{align*}
where 
\begin{align*}
\gamma(n, p, \zeta, \epsilon) = \left(\frac{p\log(p)\log(n/(p\zeta))}{n}\right)^{3/8} + \left(\frac{\epsilon p^2 \log(p)\log(p \log(p)/\zeta)}{n}\right)^{1/4}.
\end{align*}
\end{lemma}
This tells us that under mild assumptions on the risk, we can hope to achieve $O\left(\sqrt{\epsilon \log(p)}\right)$ accuracy if $n \rightarrow \infty$, since $\gamma(n, p, \zeta, \epsilon) \rightarrow 0$. 

In order to apply Lemma \ref{lemma:5} to gradients, we need bounded $4^{\text{th}}$ moments for the gradients. Unfortunately, the applications in \cite{RE} in the Huber $\epsilon$-contamination setting are not entirely correct, since they do not check the bounded $4^{\text{th}}$ moments condition. We fix this problem in the linear regression setting by making some mild assumptions on the moments of $x$. In the context of linear regression with squared error loss, assume additionally that for the vector of covariates $x = \left(x^{(1)}, \dots, x^{(p)}\right)^T \in \mathbb{R}^p$, we have for all $i, j, k, l \in [p]$:
\begin{align}
\label{eq:add_x}
&\mathrm{Var}\left(x^{(i)}x^{(j)}\right) > C_1, \qquad C_2 \leq \sigma_2^2, \notag\\
& \mathrm{Cov}\left(x^{(k)}x^{(i)}, x^{(l)}x^{(j)}\right) = \begin{cases} 
      0 & \text{if any two indexes from $\{i, j, k, l\}$ are distinct}
      \\
      0 & \text{if $k = l$ and $i \neq j$}
      \\
      \mathrm{Var}\left(x^{(k)}x^{(i)}\right) & \text{if $k = l$ and $i = j$},
   \end{cases}
\end{align}
for absolute constants $C_1, C_2 > 0$.

For example, if $x \sim N\left(0, I_p\right)$ and $\sigma_2^2$ is an absolute constant, the conditions \eqref{eq:add_x} are satisfied. We now have a covariance bound lemma:

\begin{lemma}[Corrected from \cite{RE}]
\label{lemma:covglm}
Consider the linear regression with squared error loss model from Example~\ref{sec:LR_Squared_Loss}, with $z = (x, y)$. Assume additionally the conditions \eqref{eq:add_x}.
Then
\begin{align*}
||\mathrm{\mathrm{Cov}}(\nabla \mathcal{L}(\theta, z))||_2 \lesssim ||\Delta||_2^2 + \sigma_2^2,
\end{align*}
with $\Delta = \theta - \theta^*$, and we have bounded $4^{\text{th}}$ moments for the gradient distribution, i.e., for all $||v||_2 = 1$ and $\theta \in \mathbb{R}^p$, we have
\begin{align*}
\mathbb{E}\left[\left((\nabla \mathcal{L}(\theta, z) - \nabla\mathcal{R}(\theta))^Tv\right)^4\right] \leq  \widetilde{C}_4(\mathrm{Var}(\nabla \mathcal{L}(\theta, z)^Tv))^2.
\end{align*}
\end{lemma}

\begin{proof}
Recall that for the linear regression with squared error loss model, we have $\tau_l = \lambda_{\min}(\Sigma)$ and $\tau_u = \lambda_{\max}(\Sigma)$, both assumed to be absolute constants in  Section \ref{sec:Linear Regression}, unless stated otherwise. The bound on $||\mathrm{\mathrm{Cov}}(\nabla \mathcal{L}(\theta, z))||_2$ follows from Lemma $4$ in \cite{RE}. We prove the bounded $4^{\text{th}}$ moments statement. For any $||v||_2 = 1$ and $\theta \in \mathbb{R}^p$, we have
\begin{align*}
\mathrm{Var}\left(v^T\nabla \mathcal{L}(\theta, z)\right) &= \mathbb{E}\left[\left(v^T(xx^T - \Sigma)(\theta - \theta^*) - wv^Tx\right)^2\right]\\
&= \mathbb{E}\left[\left(v^T(xx^T - \Sigma)(\theta - \theta^*)\right)^2\right] + \sigma_2^2v^T\Sigma v\\
&\geq \mathbb{E}\left[\left(v^TA\Delta\right)^2\right] + \sigma_2^2\tau_l = v^T\mathbb{E}[A\Delta\Delta^TA]v + \sigma_2^2\tau_l\\
&= v^T\mathrm{Var}(A\Delta)v + \sigma_2^2\tau_l,
\end{align*}
where $A = xx^T - \Sigma$, and we used the fact that $x \indep w$ and $\mathbb{E}[A] = 0$. Write $A$ in row form, i.e., $A = [A_1, \dots, A_p]^T$, with $A_i$ being the $i^{\text{th}}$ row of $A$, and $i \in [p]$. Then $\mathrm{Var}(A\Delta) = \left(\mathrm{Cov}\left(A_i^T\Delta, A_j^T\Delta\right)\right)_{i, j = 1}^p$. Thus, for $v = (v_1, \dots, v_p)^T$, we obtain
\begin{align*}
v^T\mathrm{Var}(A\Delta)v &= \sum_{i, j = 1}^pv_iv_j\mathrm{Cov}\left(A_i^T\Delta, A_j^T\Delta\right) = \sum_{i, j = 1}^p v_iv_j\Delta^T\mathrm{Cov}(A_i, A_j)\Delta\\
&= \Delta^T \mathrm{Var}\left(\sum_{i = 1}^pv_iA_i\right)\Delta.
\end{align*}
Since the $A_i$ terms are in a variance and $A = xx^T - \Sigma$, we can drop $\Sigma$, since it is a constant. By relabeling, we can take $A = xx^T$, to obtain $\sum_{i = 1}^p v_i A_i = \left(x^{(1)}v^Tx, \dots, x^{(p)}v^Tx\right)^T$. Hence, we have $\mathrm{Var}\left(\sum_{i = 1}^pv_iA_i\right) = \left(v^T\mathrm{Cov}\left(x^{(i)}x, x^{(j)}x\right)v\right)_{i, j = 1}^p$. This denotes the matrix with $(i, j)$ entry given by $v^T\mathrm{Cov}\left(x^{(i)}x, x^{(j)}x\right)v$. Now for $i, k, l \in [p]$, using the assumptions on $x$ in \eqref{eq:add_x}, we have
\[
\mathrm{Var}\left(x^{(i)}x\right)_{kl} = \mathrm{Cov}\left(x^{(i)}x^{(k)}, x^{(i)}x^{(l)}\right) = \begin{cases} 
      0 & \text{if $k \neq l$} \\
      \mathrm{Var}\left(x^{(i)}x^{(k)}\right) & \text{if $k = l$},
   \end{cases}
\]
so $v^T\mathrm{Var}\left(x^{(i)}x\right)v = \sum_{k = 1}^pv_k^2\mathrm{Var}\left(x^{(i)}x^{(k)}\right)$. Also, for $i, j, k, l$, with $i \neq j$, we have 
\begin{align*}
\mathrm{Cov}\left(x^{(i)}x, x^{(j)}x\right)_{kl} = \mathrm{Cov}\left(x^{(i)}x^{(k)}, x^{(j)}x^{(l)}\right) = 0,
\end{align*}
where the subscript here denotes the $(k, l)$ entry. Therefore, we have $v^T\mathrm{Cov}\left(x^{(i)}x, x^{(j)}x\right)v = 0$ for $i \neq j$, so
\begin{align*}
\Delta^T\mathrm{\mathrm{Var}}\left(\sum_{i = 1}^pv_i A_i\right)\Delta = \sum_{i, k = 1}^p\Delta_i^2v_k^2\mathrm{Var}\left(x^{(i)}x^{(k)}\right) \geq \mathop{\min}\limits_{i, k}\mathrm{Var}\left(x^{(i)}x^{(k)}\right)||\Delta||_2^2 \geq C_1||\Delta||_2^2,
\end{align*}
Since by \eqref{eq:add_x}, we have $\sigma_2^2 \geq C_2 > 0$, we obtain $\mathrm{Var}(v^T \mathcal{L}(\theta), z)) \geq C_1||\Delta||_2^2 + C_2\tau_l$. From the proof of Lemma $4$ in \cite{RE}, we have
\begin{align*}
\mathbb{E}\left[\left((\nabla \mathcal{L}(\theta, z) - \nabla\mathcal{R}(\theta))^Tv\right)^4\right] \leq C_5\tau_u^2||\Delta||_2^4 + C_6 \leq C_7||\Delta||_2^4 + C_6,
\end{align*}
for some absolute constants $C_5, C_6, C_7 > 0$, since $\tau_u = \lambda_{\max}(\Sigma) \asymp 1$. Therefore, since $\tau_l = \lambda_{\min}(\Sigma) \asymp 1$, there is an absolute constant $\widetilde{C}_4$ such that for all unit vectors $v$ and $\theta \in \mathbb{R}^p$, we have
\begin{align*}
\mathbb{E}\left[\left((\nabla \mathcal{L}(\theta, z) - \nabla\mathcal{R}(\theta))^Tv\right)^4\right] \leq  \widetilde{C}_4(\mathrm{Var}(\nabla \mathcal{L}(\theta, z)^Tv))^2,
\end{align*}
which completes the proof.
\end{proof}

We shall use this result to bound $||\widehat{\mu} - \nabla\mathcal{R}(\theta)||_2$ using $\mathrm{\mathrm{Cov}}(\nabla \mathcal{L}(\theta, z))$, in order to explicitly construct the functions $\alpha$ and $\beta$ for our gradient estimators, i.e., to turn the output $\widehat{\mu}$ of Algorithm \ref{alg:HubGradTrunc} into a gradient estimator. We obtain this from the next lemma from \cite{RE}. We present its proof to show explicitly that we use Lemma \ref{lemma:covglm} and the bounded $4^{\text{th}}$ moments condition.
\begin{lemma}[\cite{RE}]
\label{lemma:geglmhb}
Consider the linear regression with squared error loss model from Example~\ref{sec:LR_Squared_Loss} with the conditions \eqref{eq:add_x}, with i.i.d.\ data $\mathcal{D}_n = \{z_i\}_{i = 1}^n = \{(x_i, y_i)\}_{i = 1}^n$ drawn from the Huber $\epsilon$-contamination model. Then Algorithm \ref{alg:HubGradTrunc} returns, for a fixed $\theta \in \mathbb{R}^p$, a gradient estimator $g$ such that
\begin{align*}
||g(\theta; \mathcal{D}_n, \widetilde{\zeta}) - \nabla\mathcal{R}(\theta)||_2 &\lesssim (\sqrt{\epsilon} + \gamma(\widetilde{n}, p, \widetilde{\zeta}, \epsilon))\sqrt{\log(p)}||\theta - \theta^*||_2\\
& \quad + (\sqrt{\epsilon} + \gamma(\widetilde{n}, p, \widetilde{\zeta}, \epsilon))\sigma_2\sqrt{\log(p)},
\end{align*}
with probability at least $1 - \widetilde{\zeta}$. Thus, $g$ is a gradient estimator with
\begin{align}
\label{eq:hua}
&\alpha(\widetilde{n}, \widetilde{\zeta}) \asymp (\sqrt{\epsilon} + \gamma(\widetilde{n}, p, \widetilde{\zeta}, \epsilon))\sqrt{\log(p)}, \\
\label{eq:hub}
&\beta(\widetilde{n}, \widetilde{\zeta}) \asymp (\sqrt{\epsilon} + \gamma(\widetilde{n}, p, \widetilde{\zeta}, \epsilon))\sigma_2\sqrt{\log(p)}.
\end{align}
\end{lemma}

\begin{proof}
By Lemma \ref{lemma:covglm}, the gradients have bounded $4^{\text{th}}$ moments, so we can use Lemma \ref{lemma:5}. Thus, there is an algorithm that returns for $(\widetilde{n}, \widetilde{\zeta})$ a $g(\theta)$, such that for $\theta \in \mathbb{R}^p$, with probability at least $1 - \widetilde{\zeta}$, we have
\begin{align*}
||g(\theta) - \nabla\mathcal{R}(\theta)||_2 \lesssim (\sqrt{\epsilon} + \gamma(\widetilde{n}, p, \widetilde{\zeta}, \epsilon))\sqrt{||\mathrm{\mathrm{Cov}}(\nabla \mathcal{L}(\theta, z))||_2\log(p)}.
\end{align*}
Using Lemma \ref{lemma:covglm} and bounding $||\mathrm{\mathrm{Cov}}(\nabla \mathcal{L}(\theta, z))||_2$, we obtain the desired result.
\end{proof}

Now we can finally present our applications to linear regression with squared error loss using projected gradient descent and Nesterov's method. We present the proof of the latter, since the projected gradient descent case is from \cite{RE}. Although the expressions will be tedious, we will care about scaling behaviors with $p$ and $\epsilon$ when $n \rightarrow \infty$.

\begin{lemma}[\cite{RE}]\label{lemma:Hu_eps_PGD}
Let $\mathcal{C} \subseteq \mathbb{R}^p$ and $\zeta \in (0, 1)$. Consider the linear regression with squared error loss model from Example~\ref{sec:LR_Squared_Loss} under the Huber $\epsilon$-contamination setting, assuming the conditions \eqref{eq:add_x}. Suppose $\theta^* \in \mathcal{C}$. Then there are absolute constants $C_1$ and $C_2$ such that, if $\gamma(\widetilde{n}, p, \widetilde{\zeta}, \epsilon) < \frac{\tau_l/C_1}{2\sqrt{\log(p)}}$ and $\epsilon < \left(\frac{\tau_l/C_1}{2\sqrt{\log(p)}} - \gamma(\widetilde{n}, p, \widetilde{\zeta}, \epsilon)\right)^2$, Algorithm \ref{alg:HubGradTrunc} generates a gradient estimator such that Algorithm \ref{alg:RobPGDNFW} for projected gradient descent, initialized at $\theta_0 \in \mathcal{C}$ with $\eta = \frac{2}{\tau_u + \tau_l}$, returns iterates $\{\theta_t\}_{t = 1}^T$ such that with probability at least $1 - \zeta$, we have for some $k < 1$ that
\begin{align}
\label{eq:glmreshubGD}
||\theta_t - \theta^*||_2 \lesssim||\theta_0 - \theta^*||_2k^t + \frac{\sigma_2\sqrt{\log(p)}}{1 - k}(\sqrt{\epsilon} + \gamma(\widetilde{n}, p, \widetilde{\zeta}, \epsilon)),
\end{align}
with
\begin{align*}
&\alpha(\widetilde{n}, \widetilde{\zeta}) = C_1(\sqrt{\epsilon} + \gamma(\widetilde{n}, p, \widetilde{\zeta}, \epsilon))\sqrt{\log(p)},\\
&\beta(\widetilde{n}, \widetilde{\zeta}) = C_2(\sqrt{\epsilon} + \gamma(\widetilde{n}, p, \widetilde{\zeta}, \epsilon))\sigma_2\sqrt{\log(p)}.
\end{align*}
\end{lemma}

\begin{remark}
Note that \cite{RE} ask for $\frac{\tau_l}{\sqrt{\log(p)}}$, since they need $\alpha < \tau_l$. Since we ask for $\alpha < \tau_l/2$, we only affect the lower bound on $\widetilde{n}$ by a factor of $2$. So, up to absolute constants, nothing changes.
\end{remark}

\begin{theorem}
\label{theorem:Hu_eps_Nester}
Let $\mathcal{C} = \mathbb{R}^p$ and $\zeta \in (0, 1)$. Consider the linear regression with squared error loss model from Example~\ref{sec:LR_Squared_Loss} under the Huber $\epsilon$-contamination setting, assuming the conditions \eqref{eq:add_x}. Suppose $1 < \frac{\tau_u}{\tau_l} < x^*$, where $x^* \approx 1.76759$ is the solution of the equation $f_1(x) = f_2(x)$ for $x \geq 1$, with these functions defined as before. Then there are absolute constants $C_1, C_2$, and $C_3$ such that, if \begin{align*}
\gamma(\widetilde{n}, p, \widetilde{\zeta}, \epsilon) < \frac{f_1\left(\frac{\tau_u}{\tau_l}\right)\tau_l/C_1}{\sqrt{\log(p)}}
\end{align*}
and
\begin{align*}
\left(\frac{f_1\left(\frac{\tau_u}{\tau_l}\right)\tau_l/C_1}{\sqrt{\log(p)}} - \gamma(\widetilde{n}, p, \widetilde{\zeta}, \epsilon)\right)^2 < \epsilon < \left(\frac{f_2\left(\frac{\tau_u}{\tau_l}\right)\tau_l/C_1}{\sqrt{\log(p)}} - \gamma(\widetilde{n}, p, \widetilde{\zeta}, \epsilon)\right)^2,
\end{align*}  
Algorithm \ref{alg:HubGradTrunc} generates a gradient estimator such that Algorithm \ref{alg:RobPGDNFW} for Nesterov's AGD initialized at $\theta_0, \theta_1 \in \mathcal{C}$, with $\eta = \frac{2}{\tau_u}$ and $\lambda = \frac{\sqrt{\tau_u} - \sqrt{\tau_l}}{\sqrt{\tau_u} + \sqrt{\tau_l}}$, returns iterates $\{\theta_t\}_{t = 1}^T$ such that with probability at least $1 - \zeta$, we have
\begin{align}
\label{eq:glmreshubAGD}
||\theta_t - \theta^*||_2 \leq \sqrt{\frac{2}{\tau_l}\left(\mathcal{R}(\theta_0) - \mathcal{R}(\theta^*)\right) + ||\theta_0 - \theta^*||_2^2}\left(1 - \sqrt{\frac{\tau_l}{\tau_u}}\right)^{t/2} + \left(\frac{\tau_u}{\tau_l}\right)^{1/4}\sqrt{\frac{R}{\tau_l}},
\end{align}
where 
\begin{align*}
&R = 2\left(\alpha^2 C_3 + \frac{\eta\alpha\beta}{1 - (1 + \lambda)k - \lambda k} + \beta\right)\\
& \qquad \cdot \left[\eta\beta + \left(\eta\alpha + \tau_u(1 + \eta)\left(2\lambda + \sqrt{\frac{\tau_l}{\tau_u}}\right)\right)\left(C_3\alpha + \frac{\eta\beta}{1 - (1 + \lambda)k - \lambda k}\right)\right],\\
&\alpha = \alpha(\widetilde{n}, \widetilde{\zeta}) = C_1(\sqrt{\epsilon} + \gamma(\widetilde{n}, p, \widetilde{\zeta}, \epsilon))\sqrt{\log(p)},\\
&\beta = \beta(\widetilde{n}, \widetilde{\zeta}) = C_2(\sqrt{\epsilon} + \gamma(\widetilde{n}, p, \widetilde{\zeta}, \epsilon))\sigma_2\sqrt{\log(p)},\\
&k = \frac{\tau_u - \tau_l + \alpha(\widetilde{n}, \widetilde{\zeta})}{\tau_u}.
\end{align*}
\end{theorem}

\begin{proof}
From Lemma \ref{lemma:geglmhb}, we have a gradient estimator $g(\theta)$ with functions $\alpha(\widetilde{n}, \widetilde{\zeta})$ and $\beta(\widetilde{n}, \widetilde{\zeta})$ as in the theorem hypothesis. What we assumed about $n$ and $\epsilon$ implies $f_1\left(\frac{\tau_u}{\tau_l}\right) < \frac{\alpha(\widetilde{n}, \widetilde{\zeta})}{\tau_l} < f_2\left(\frac{\tau_u}{\tau_l}\right)$, and we have mentioned after the end of Theorem \ref{theorem:AGD} that the stability assumption is satisfied, i.e., $\alpha(\widetilde{n}, \widetilde{\zeta}) < \tau_l/2$, since $\frac{\alpha(\widetilde{n}, \widetilde{\zeta})}{\tau_l} < f_2\left(\frac{\tau_u}{\tau_l}\right) \leq \frac{1}{2}$, as $\tau_u > \tau_l$. Then, for $R$ as in the theorem hypothesis, by Theorem \ref{theorem:AGD}, we obtain iterates $\{\theta_t\}_{t = 1}^T$ such that with probability at least $1 - \zeta$, we have
\begin{align*}
||\theta_t - \theta^*||_2 \leq \sqrt{\frac{2}{\tau_l}\left(\mathcal{R}(\theta_0) - \mathcal{R}(\theta^*)\right) + ||\theta_0 - \theta^*||_2^2}\left(1 - \sqrt{\frac{\tau_l}{\tau_u}}\right)^{t/2} + \left(\frac{\tau_u}{\tau_l}\right)^{1/4}\sqrt{\frac{R}{\tau_l}},
\end{align*}
with $C_3$ an absolute constant.
\end{proof}


\subsection{Comments and Comparisons in the Huber $\epsilon$-Contamination Setting}

Let us assume that $\sigma_2$ is an absolute constant. We already assumed in Section \ref{sec:Linear Regression} that $\lambda_{\min}(\Sigma)$ and $\lambda_{\max}(\Sigma)$ are absolute constants. We look at the linear regression with squared error loss model in Example~\ref{sec:LR_Squared_Loss}. We take the rate of convergence of the exponential term and the dependency of the error term on $p$ and $\epsilon$ into consideration. For projected gradient descent, the first term in inequality \eqref{eq:glmreshubGD} decays exponentially in $t$, with the contraction parameter $k$ that we defined before. The error term scales as $O\left(\sqrt{\epsilon \log(p)}\right)$, as $n \rightarrow \infty$, since in this case, $\gamma(\widetilde{n}, p, \widetilde{\zeta}, \epsilon) \rightarrow 0$. Also, we have a restriction on how small $n$ can be, given the upper bound on $\gamma(\widetilde{n}, p, \widetilde{\zeta}, \epsilon)$, and our contamination level has to be below a given threshold. The way this depends logarithmically on $p$ is due to the estimator from Lai et al.~\cite{Vpala}. As \cite{RE} states, the algorithm used is the only practical one for robust estimation in the case of general statistical models. Of course, for specific models, this error term could be brought down, but in the general setting, it appears that the best one can hope for is $O\left(\sqrt{\epsilon \log(p)}\right)$.

In contrast, Nesterov's AGD achieves a faster convergence rate, as stated in Remark~\ref{RemFaster}, but under the restriction that the smoothness and strong convexity parameters cannot be equal, and the smoothness parameter cannot exceed roughly $1.76$ times the strong convexity parameter. However, with this assumption, not only is the exponential decay with $t$ faster in inequality \eqref{eq:glmreshubAGD}, but the error term is as in the case of projected gradient descent when $n \rightarrow \infty$. To see this, the error term in our bound \eqref{eq:glmreshubAGD} scales like $\sqrt{R}$, with
\begin{align*}
R &= 2\left(\alpha^2 C_1 + \frac{\eta\alpha\beta}{1 - (1 + \lambda)k - \lambda k} + \beta\right)\\
& \quad \cdot \left[\eta\beta + \left(\eta\alpha + \tau_u(1 + \eta)\left(2\lambda + \sqrt{\frac{\tau_l}{\tau_u}}\right)\right)\left(C_1\alpha + \frac{\eta\beta}{1 - (1 + \lambda)k - \lambda k}\right)\right].
\end{align*}
Recall that $\alpha < \tau_l/2$. In the first term in the product, we have $\alpha^2 \le \tau_l\alpha$ and $\alpha\beta \le \tau_l\beta$. Hence, the first term is $O\left(\sqrt{\epsilon \log(p)}\right)$. In the second term, we have $\alpha\eta \le \tau_l\eta$, so the second term is also $O\left(\sqrt{\epsilon \log(p)}\right)$. Thus, we have $\sqrt{R} = O\left(\sqrt{\epsilon \log(p)}\right)$, and we perform the same as in the projected gradient descent method. 

Our method used in deriving the robust Nesterov's AGD in Theorem \ref{theorem:AGD} was an adaptation of the proof in \cite{Recht}. Other approaches might reduce the exponential decay further or relax the assumption on the smoothness and strong convexity parameters. Moreover, Nesterov's AGD case imposes more restrictions for the choices of $\epsilon$ and $n$. We are also asking for a lower bound on $\epsilon$. Also, its upper bound, without the square at least, is smaller than the projected gradient descent one (i.e., more restrictive), since $f_2(\frac{\tau_u}{\tau_l}) < \frac{1}{2}$. This is because $\tau_u > \tau_l$. Also, we have to choose a higher $n$ for Nesterov's AGD, again since $f_1(\frac{\tau_u}{\tau_l}) < \frac{1}{2}$. Overall, we trade off freedom of choosing some parameters for faster decay toward or close to $\theta^*$ in the AGD setting.

We can also analyze the effect of acceleration from an iteration complexity point of view. By \emph{iteration complexity} \cite{Wang_Iter_Comp, Richtarik_Iter_Comp}, we mean the iteration count $T$ as a function of $a > 0$, where $a$ is a desired upper bound error on $||\theta_T - \theta^*||$. Since the upper bounds in Lemma \ref{lemma:Hu_eps_PGD} and Theorem \ref{theorem:Hu_eps_Nester} have an error term that becomes $O\left(\sqrt{\epsilon \log(p)}\right)$ when $n \rightarrow \infty$, we run projected gradient descent and Nesterov's AGD so that the exponentially decaying term is $O\left(\sqrt{\epsilon \log(p)}\right)$, in the limit with $n$. This is in line with the reasoning in Remark \ref{remark:Nester_Iter_LB}, where we chose $T$ so that the exponentially decaying term is below the inescapable error. Hence, for projected gradient descent, we can choose $T = \log_{1/k}\left(1/\sqrt{\epsilon \log(p)}\right)$, and for Nesterov's AGD, we can choose $T = \log_{1/\rho}\left(1/\sqrt{\epsilon \log(p)}\right)$, where $\rho = \sqrt{1 - \sqrt{\frac{\tau_l}{\tau_u}}}$. Since, as explained in Remark~\ref{RemFaster}, we have $\sqrt{1 - \sqrt{\frac{\tau_l}{\tau_u}}} < k$, we see that acceleration translates into a better iteration complexity at the inescapable error level.


\section{Comparisons to Private SGD}
\label{AppSGD}

In this appendix, we compare our private accelerated Frank-Wolfe method to private SGD. Appendix \ref{sec:Comparison_SGD_Dist_Free} focuses on the distribution-free setting from Section \ref{sec:{Private ERM for Distribution-Free Data: Upper and Lower Bounds}}, while Appendix \ref{sec:Comparison_SGD_GLM} addresses the GLM setting from Section \ref{sec:Private Biased Parameter Estimation in GLMs: Upper and Lower Bounds for Excess Risk}.

We first introduce the private SGD approaches we will discuss. One will be from \cite{Bassily_SGD_l2}, and we will also consider the more efficient version for smooth (but not necessarily strongly convex) losses from \cite{DP_ERM_Faster}, despite the fact that they look at a regularized version of the problem. The main advantage of SGD is to reduce the number of gradient calls at each iteration, so it makes sense to not only compare convergence rates on the excess empirical risk, but also \emph{gradient complexities}, i.e., the total number of gradient calls in the whole iterative procedure. 

\begin{algorithm}
\caption{$\mathcal{A}_{\text{Noise - GD}}$: Differentially Private SGD (General Bounded Convex
Case)}
\label{alg:Priv_SGD}
\begin{algorithmic}[1]
\Function{$\mathcal{A}_{\text{Noise - GD)}}$}{Data space $\mathcal{E}$, $\mathcal{D}_n = \{z_1, \ldots, z_n\}$, loss function $\mathcal{L}(\theta, \mathcal{D}_n) = \frac{\sum_{i = 1}^n\mathcal{L}(\theta, z_i)}{n}$ (with $L_2$-Lipschitz constant for $\mathcal{L}$), $\epsilon, \delta$, bounded and convex set $\mathcal{C}$, learning rate function $\eta : [n^2] \rightarrow \mathbb{R}$}
\State Set noise variance $\sigma^2 = \frac{32L_2^2\log(n/\delta)\log(1/\delta)}{\epsilon^2}$.
\State Choose $\theta_0 \in \mathcal{C} \subset \mathbb{R}^p$ arbitrary.
    \For{$t = 0$ to $n^2 - 1$}
    \State Pick $d^{(t)}$ uniformly without replacement from $\mathcal{D}_n$.
        \State $\theta_{t + 1} = \mathcal{P}_{\mathcal{C}}\left(\theta_t - \eta_t\left(\nabla\mathcal{L}\left(\theta_t, d^{(t)}\right) + \xi_t\right)\right)$, where $\xi_t \sim N\left(0, \sigma^2I_p\right)$ and $\mathcal{P}_{\mathcal{C}}$ is the projection operator in the $\ell_2$-norm onto $\mathcal{C}$.
    \EndFor
    \State \textbf{return} $\theta_{n^2}$.
\EndFunction
\end{algorithmic}
\end{algorithm}

The private SGD algorithm from \cite{Bassily_SGD_l2} is provided in Algorithm~\ref{alg:Priv_SGD}. Note that it is $(\epsilon, \delta)$-DP for $\epsilon \in (0, 0.9]$ and $\delta \in (0, 1)$. The following result provides its utility guarantee:

\begin{lemma}[\cite{Bassily_SGD_l2}]\label{lemma:Util_Priv_SGD}
Let $p \geq 1$ and $0 < \epsilon \lesssim 1$, and let $\mathcal{C} \subseteq \mathbb{R}^P$ be a bounded, convex set. Let $\mathcal{E}$ be a data space and let $\mathcal{L}(\theta, z)$ be convex and $L_2$-Lipschitz in $\theta$, i.e., $\mathcal{L}(\theta_1, z) - \mathcal{L}(\theta_2, z)\leq L_2||\theta_1 - \theta_2||_2$, for any $\theta_1, \theta_2 \in \mathcal{C}$ and $z \in \mathcal{E}$. Then Algorithm \ref{alg:Priv_SGD}, with $\eta_t = \frac{||\mathcal{C}||_2}{\sqrt{t(n^2L_2 + p\sigma^2)}}$, returns $\theta_{n^2}$ such that
\begin{align*}
\mathbb{E}\left[\mathcal{L}(\theta_{n^2}, \mathcal{D}_n) - \mathop{\min}\limits_{\theta \in \mathcal{C}}\mathcal{L}(\theta, \mathcal{D}_n)\right] = O\left(\frac{L_2||\mathcal{C}||_2\log^{3/2}(n/\delta)\sqrt{p\log(1/\delta)}}{n\epsilon}\right).
\end{align*}
\end{lemma}

Note that, similar to \cite{NOPL}, Lemma~\ref{lemma:Util_Priv_SGD} assumes the data to be non-random. Additionally, \cite{Bassily_SGD_l2} present the optimality of their approach with a lower bound based on datasets $\mathcal{D}_n = \{d_i\}_{i = 1}^n$, with $d_i \in \{-1, 1\}^p$, for all $i \in [n]$. The upper bound rate in Lemma \ref{lemma:Util_Priv_SGD} is $\widetilde{O}\left(\frac{L_2||\mathcal{C}||_2\sqrt{p}}{n\epsilon}\right)$, with a gradient complexity of $n^2$.

Let us now turn our attention to the more efficient method from \cite{DP_ERM_Faster}, which also assumes the data are non-random. We do not include the algorithm here, because of the more extensive setup needed for it, but we provide its utility guarantees and gradient complexity. In \cite{DP_ERM_Faster}, the efficient version of the private SGD algorithm from \cite{Bassily_SGD_l2} is called \emph{DP-SVRG++}. It is important to note that they target a regularized version of the loss, namely $\mathcal{L}(\theta, \mathcal{D}_n) + reg(\theta)$, where $reg(\theta)$ is a regularizer, and they optimize over the whole of $\mathbb{R}^p$. They obtain a rate of $\widetilde{O}
\left(\frac{L_2\sqrt{p}}{n\epsilon}\right)$ on the expected excess regularized empirical risk, with a gradient complexity of $O\left(\frac{n\beta_{\mathcal{L}}\epsilon}{L_2\sqrt{p}} + n\log\left(\frac{n\epsilon}{L_2\sqrt{p}}\right)\right)$. Here, $L_2$ and $\beta_{\mathcal{L}}$ are the Lipschitz and smoothness parameters of $\mathcal{L}(\theta, \mathcal{D}_n)$, respectively. We will compare this to our accelerated Frank-Wolfe method, where we target the empirical risk and we minimize over an $\ell_2$-ball $\mathcal{C}$ centered at $0$. Hence, in order to use \emph{DP-SVRG++} for our purposes, we need to write the problem in our context using a ridge regularizer $reg(\theta) = \gamma_{\mathcal{C}}||\theta||_2^2$. Since we start from the constrained optimization over $\mathcal{C}$, we need to compute $\gamma_{\mathcal{C}}$ explicitly in order to use \emph{DP-SVRG++}, which cannot be done in practice. Also, because the bounds we obtain are on the excess empirical risk, while \cite{DP_ERM_Faster} derives theirs on the regularized version, it is not straightforward to compare the rates for the excess objectives. Our goal in the case of \cite{DP_ERM_Faster} is to look at gradient complexities. 

Some authors apply private SGD directly to population risk minimization rather than empirical risk minimization. In this setting, one wishes to minimize the excess risk $\mathcal{R}(\theta_T) - \mathop{\min}\limits_{\theta \in \mathcal{C}}\mathcal{R}(\theta)$, either with high probability or in expectation, where $\theta_T$ is the output of some $(\epsilon, \delta)$-DP procedure and $\mathcal{R}(\theta) = \mathbb{E}_z\left[\mathcal{L}(\theta, z)\right]$, for all $\theta$ in some convex set $\mathcal{C} \subseteq \mathbb{R}^p$. Bassily et al.~\cite{BasEtal19} consider the setting of differentiable, smooth, $L_2$-Lipschitz losses and convex sets $\mathcal{C}$ of bounded radius $M = \mathop{\max}\limits_{\theta \in \mathcal{C}}||\theta||_2$ (all in the $\ell_2$-norm). Using a private SGD method based initially on an empirical risk minimization approach and later taken to a population risk setting using the notion of uniform stability, they obtain a rate of $\widetilde{O}\left(\frac{1}{\sqrt{n}} + \frac{\sqrt{p}}{n\epsilon}\right)$ on the expected excess risk, with the expectation taken over $\theta_T$. This is also shown to be tight, but their method requires $O\left(\min\left\{n^{3/2}, n^{5/2}/p\right\}\right)$ gradient computations. Later, Feldman et al.~\cite{SCO_Feldman_2020} achieved the same optimal bound with $O\left(\min\left\{n, n^2/p\right\}\right)$ gradient computations, using a similar private SGD approach based on noisy empirical risk gradients, as in \cite{BasEtal19}. Additionally, Bassily et al.~\cite{BasEtal21} considered the setting of $\ell_q$-norms for $q \in (1, \infty]\setminus \{2\}$. Using the variance-reduced stochastic Frank-Wolfe method based on variance reduction from \cite{Var_Red_FW}, they obtain an upper bound of $\widetilde{O}\left(\sqrt{\frac{\kappa}{n}} + \frac{\kappa\sqrt{q}}{n\epsilon}\right)$, when $q \in (0, 2)$ (for which they also provide a lower bound). When $q \in (2, \infty)$, they obtain an upper bound of $\widetilde{O}\left(\frac{p^{1/2 - 1/q}}{\sqrt{n}} + \frac{p^{1 - 1/q}}{n\epsilon}\right)$. Here, $\kappa = \min\left\{1/(q - 1), 2\log(p)\right\}$.


\subsection{Comparisons in Section \ref{sec:{Private ERM for Distribution-Free Data: Upper and Lower Bounds}}}\label{sec:Comparison_SGD_Dist_Free}

We can compare the result of Lemma \ref{lemma:Util_Priv_SGD} with Theorem \ref{theorem:UPRidge}.
Consider the setting of Lemma \ref{lemma:Util_Priv_SGD}, with $\mathcal{C}$ being an $\ell_2$-ball of diameter $||\mathcal{C}||_2 = 2D > 0$, with $\mathcal{E} = \mathbb{B}_\infty(1) \times [-1, 1]$ and $\mathcal{D}_n = \{(x_i, y_i)\}_{i = 1}^n \subseteq \mathcal{E}^n$, where $\mathcal{L}$ is the squared error loss.

Firstly, note that Lemma \ref{lemma:Util_Priv_SGD} makes fewer assumptions than Theorem \ref{theorem:UPRidge}: Lemma \ref{lemma:Util_Priv_SGD} does not assume the loss to be smooth and does not have any conditions on the dataset $\mathcal{D}_n$. Under the particular setting involving the squared error loss mentioned above, as explained in the proof of Theorem \ref{theorem:UPRidge}, we have $L_2 \asymp \sqrt{p} + p||\mathcal{C}||_2$. Hence, in the setting mentioned above, the upper bound in Lemma \ref{lemma:Util_Priv_SGD} becomes $\widetilde{O}\left(\frac{(\sqrt{p} + p||\mathcal{C}||_2)||\mathcal{C}||_2\sqrt{p}}{n\epsilon}\right)$, the same as in Theorem \ref{theorem:UPRidge}. However, note that the overall gradient complexity of our accelerated Frank-Wolfe method in Algorithm \ref{alg:PrivFWERM} is better than the SGD approach in Algorithm \ref{alg:Priv_SGD}. This is because Algorithm \ref{alg:Priv_SGD} takes $T = n^2$ iterations, and at each iteration, they use one gradient call. Hence, their gradient complexity is $n^2$. In contrast, Algorithm \ref{alg:PrivFWERM} takes $T = O(\log(n))$ iterations to achieve the utility guarantee in Theorem \ref{theorem:UPRidge}, with $n$ gradient calls at each iteration. Hence, the gradient complexity of our method is $O(n\log(n))$.


Lastly, one can also consider the result in \cite{DP_ERM_Faster}. Compared to \cite{Bassily_SGD_l2}, they assume additionally that the loss is smooth. As mentioned earlier, it is not fair to consider a comparison of the convergence rates since \cite{DP_ERM_Faster} targets the excess regularized empirical risk (where the regularizer would be a ridge regularizer). Instead, we look at gradient complexities. Considering our setting in Theorem \ref{theorem:UPRidge}, note that a general tight bound for $\beta_{\mathcal{L}}$ would be $p$, since for $||x||_\infty \leq 1$, we have $||xx^T||_2^2 = ||x||_2^2 \leq p||x||_\infty \leq 1$. Since $L_2 \asymp \sqrt{p} + p||\mathcal{C}||_2$, the gradient complexity in \cite{DP_ERM_Faster} becomes $O\left(\frac{n\epsilon}{1 + \sqrt{p}||\mathcal{C}||_2} + n\log\left(\frac{n\epsilon}{(\sqrt{p} + p||\mathcal{C}||_2)\sqrt{p}}\right)\right)$.

Now assume $\epsilon$ is an absolute constant. If $p, ||\mathcal{C}||_2 \asymp 1$, the gradient complexity becomes $O\left(n + n\log\left(n\right)\right)$, which asymptotically is the same as the one in Theorem \ref{theorem:UPRidge}, i.e., $O(n\log(n))$. If we consider the context of the high-probability statement in Proposition~\ref{prop:achievab_UP_ERM}, with $n \geq \widetilde{\Omega}\left(p^{c_2}\right)$, $D^2(p) \asymp \sigma^2(p) \asymp \frac{1}{p}$, and $c_2 > \frac{5}{4}$, the gradient complexity in \cite{DP_ERM_Faster} is $O\left(n + n\log(n)\right)$ again. Hence, for $n$ and $p$ as in the context of Proposition~\ref{prop:achievab_UP_ERM}, our gradient complexity matches the one in \cite{DP_ERM_Faster}. Note again that the scaling of $n$ and $p$ in terms of $m \in \mathbb{N}$, with $m \rightarrow \infty$, that was used to achieve the lower bound in Theorem~\ref{theorem:LBRidge}, is a particular instance of the choice of $n$ in terms of $p$ in Proposition~\ref{prop:achievab_UP_ERM}. Hence, our method has the same asymptotic gradient efficiency as \cite{DP_ERM_Faster} in the context of the lower bound result, as well.


\subsection{Comparisons in Section \ref{sec:Private Biased Parameter Estimation in GLMs: Upper and Lower Bounds for Excess Risk}}\label{sec:Comparison_SGD_GLM}

Similar to our comparison in Appendix \ref{sec:Comparison_SGD_Dist_Free}, we can compare our upper bound results and the gradient complexities in Sections \ref{sec:Increasing C} and \ref{sec:Fixed C and the Lower Bound} with Algorithm \ref{alg:Priv_SGD} and its utility in Lemma \ref{lemma:Util_Priv_SGD}. We will analyze the results of Theorem \ref{theorem:ERM_UP_C_Inc} and Theorem \ref{theorem:ERM_GLM_UP}, which are also based on the accelerated Frank-Wolfe method in Algorithm \ref{alg:PrivFWERM}. Consider the setting of Lemma \ref{lemma:Util_Priv_SGD} with $\mathcal{C} = \mathbb{B}_2(D)$, $D > 0$, $||\theta^*||_2 - D > 0$, $\mathcal{E} = \mathbb{B}_2(L_x) \times [-K_y, K_y]$, $L_x, K_y \asymp 1$, $\mathcal{D}_n = \{(x_i, y_i)\}_{i = 1}^n \subseteq \mathcal{E}^n$, and $\mathcal{L}$ being the negative log likelihood loss. We only care about the scaling with $n$, and everything else involving $p$, $||\theta^*||_2$, and $c(\sigma)$ is treated as an absolute constant. We will consider the GLM setting from Section \ref{sec:Generalized Linear Models (GLM)}.

We start with Theorem \ref{theorem:ERM_UP_C_Inc}. Under the condition that $||\theta^*||_2 - D \asymp \frac{1}{n^{2/5}}$, and assuming the data follow a parametric GLM defined in Section \ref{sec:Generalized Linear Models (GLM)}, we can guarantee an upper bound on the excess empirical risk at rate $\widetilde{O}\left(\frac{1}{n^{4/5}\epsilon}\right)$, with high probability and for $n$ large enough. On the other hand, Lemma \ref{lemma:Util_Priv_SGD} guarantees an upper bound on the expected excess empirical risk at rate $\widetilde{O}\left(\frac{1}{n\epsilon}\right)$. Hence, if we only care about the upper bound rate, SGD  performs better. Note also that Lemma \ref{lemma:Util_Priv_SGD} makes fewer assumptions than Theorem \ref{theorem:ERM_UP_C_Inc}, in the sense that Lemma \ref{lemma:Util_Priv_SGD} does not assume the loss to be smooth and does not have any conditions on the dataset $\mathcal{D}_n$. However, we can also take the overall gradient complexity of Algorithm \ref{alg:Priv_SGD} and Algorithm \ref{alg:PrivFWERM} into account. The guarantee in Theorem \ref{theorem:ERM_UP_C_Inc} is based on $T = O\left(n^{2/5}\log(n)\right)$ iterations. Since at each iteration, we use $n$ gradient calls, the overall gradient complexity becomes $O\left(n^{7/5}\log(n)\right)$. The result in Lemma \ref{lemma:Util_Priv_SGD} is based on $T = n^2$ iterations, and one gradient call at each iteration. Hence, the gradient complexity becomes $n^2$. If we want a fair comparison that takes both the convergence rate and the gradient complexity into account, we can ask for the required number of samples needed in order to obtain an error below some fixed $a \in (0, 1)$, and then compare the gradient complexities in terms of $a$. The gradient complexity in Theorem \ref{theorem:ERM_UP_C_Inc} is accordingly $\widetilde{O}\left(\frac{1}{a^{7/4}}\right)$, while the one for Lemma \ref{lemma:Util_Priv_SGD} is $\widetilde{O}\left(\frac{1}{a^2}\right)$. Therefore, under a parametric GLM, provided the sample size is large enough and we optimize over an $\ell_2$-ball that increases toward $\theta^*$ at rate $O\left(\frac{1}{n^{4/5}}\right)$, the accelerated Frank-Wolfe approach has a better gradient efficiency than SGD. Note that one result is in expectation, while the other holds with high probability, but we ignore this difference in our comparison. 

Moving to Theorem \ref{theorem:ERM_GLM_UP}, suppose $||\theta^*||_2 - D \asymp 1$ and the data follow a parametric GLM defined in Section \ref{sec:Generalized Linear Models (GLM)}. We can guarantee an upper bound on the expected excess empirical risk at rate $\widetilde{O}\left(\frac{1}{n\epsilon}\right)$, for $n$ large enough. The same rate is guaranteed by Lemma \ref{lemma:Util_Priv_SGD}. We reiterate that Lemma \ref{lemma:Util_Priv_SGD} does not make any smoothness or distributional assumptions, as in Theorem \ref{theorem:ERM_GLM_UP}. If we instead consider gradient complexities of the two algorithms, Algorithm \ref{alg:PrivFWERM} takes $T \asymp \log(n)$ iterations in the context of Theorem \ref{theorem:ERM_GLM_UP}, and requires $n$ gradient computations at each iteration, resulting in a gradient complexity of $\Theta(n\log(n))$. The gradient complexity of Algorithm \ref{alg:Priv_SGD} is $n^2$. Thus, under a parametric GLM, provided the sample size is large enough and that we optimize over an $\ell_2$-ball $\mathcal{C}$ with absolute constant radius such that $\theta^* \notin \mathcal{C}$, the accelerated Frank-Wolfe method performs at the same rate in terms of $n$ as the SGD approach, up to logarithmic factors, but with a better gradient complexity.

We can also establish a comparison with \cite{DP_ERM_Faster}, which also assumes the loss is smooth. We only take the dependency on $n$ into account, and we take $\epsilon \asymp 1$. Once again, we only compare gradient complexities. Regardless of whether $D$ increases with $n$ toward $||\theta^*||_2$ or not, the gradient complexity in \cite{DP_ERM_Faster} is $O\left(n + n\log(n)\right)$. The gradient complexity in Theorem \ref{theorem:ERM_UP_C_Inc} is $O\left(n^{7/5}\log(n)\right)$, which is slightly worse than the one in \cite{DP_ERM_Faster}. The one in Theorem \ref{theorem:ERM_GLM_UP} is $O\left(n\log(n)\right)$, which is asymptotically the same as the one in \cite{DP_ERM_Faster}.

\bibliographystyle{plain}
\bibliography{REFS_HT_DP_REG.bib}

\begin{thebibliography}{10}

\bibitem{Deep_Learn_DP_Clip}
M.~Abadi, A.~Chu, I.~Goodfellow, H.~B. McMahan, I.~Mironov, K.~Talwar, and
  L.~Zhang.
\newblock Deep learning with differential privacy.
\newblock In {\em Proceedings of the 2016 ACM SIGSAC Conference on Computer and
  Communications Security}, pages 308--318, 2016.

\bibitem{Priv_SGD_l1_Geom}
H.~Asi, V.~Feldman, T.~Koren, and K.~Talwar.
\newblock Private stochastic convex optimization: {O}ptimal rates in $\ell_1$
  geometry.
\newblock In {\em International Conference on Machine Learning}, pages
  393--403. PMLR, 2021.

\bibitem{ROB_LR_OPT_POLY}
A.~Bakshi and A.~Prasad.
\newblock Robust linear regression: {O}ptimal rates in polynomial time.
\newblock In {\em Proceedings of the 53rd Annual ACM SIGACT Symposium on Theory
  of Computing}, pages 102--115, 2021.

\bibitem{Balakrish}
S.~Balakrishnan, S.~S. Du, J.~Li, and A.~Singh.
\newblock Computationally efficient robust sparse estimation in high
  dimensions.
\newblock In {\em Conference on Learning Theory}, pages 169--212. PMLR, 2017.

\bibitem{IMPROVE_DP}
B.~Balle and Y.-X. Wang.
\newblock Improving the {G}aussian mechanism for differential privacy:
  {A}nalytical calibration and optimal denoising.
\newblock In {\em International Conference on Machine Learning}, pages
  394--403. PMLR, 2018.

\bibitem{SUB_GAUSS_TRUNC}
M.~Barreto, O.~Marchal, and J.~Arbel.
\newblock Optimal sub-{G}aussian variance proxy for truncated gaussian and
  exponential random variables.
\newblock {\em arXiv preprint arXiv:2403.08628}, 2024.

\bibitem{PSHUBER_BARRON}
J.~T. Barron.
\newblock A general and adaptive robust loss function.
\newblock In {\em Proceedings of the IEEE/CVF Conference on Computer Vision and
  Pattern Recognition}, pages 4331--4339, 2019.

\bibitem{BasEtal19}
R.~Bassily, V.~Feldman, K.~Talwar, and A.~Thakurta.
\newblock Private stochastic convex optimization with optimal rates.
\newblock {\em Advances in Neural Information Processing Systems}, 32, 2019.

\bibitem{BasEtal21}
R.~Bassily, C.~Guzm{\'a}n, and A.~Nandi.
\newblock Non-{E}uclidean differentially private stochastic convex
  optimization.
\newblock In {\em Conference on Learning Theory}, pages 474--499. PMLR, 2021.

\bibitem{Bassily_SGD_l2}
R.~Bassily, A.~Smith, and A.~Thakurta.
\newblock Private empirical risk minimization: efficient algorithms and tight
  error bounds.
\newblock In {\em 2014 IEEE 55th Annual Symposium on Foundations of Computer
  Science}, pages 464--473. IEEE, 2014.

\bibitem{Beck_Proximal_2009}
A.~Beck and M.~Teboulle.
\newblock A fast iterative shrinkage-thresholding algorithm for linear inverse
  problems.
\newblock {\em SIAM Journal on Imaging Sciences}, 2(1):183--202, 2009.

\bibitem{SG_DEF}
S.~Boucheron, G.~Lugosi, and O.~Bousquet.
\newblock Concentration inequalities.
\newblock In {\em Summer School on Machine Learning}, pages 208--240. Springer,
  2003.

\bibitem{lemma311}
S.~Bubeck.
\newblock Convex optimization: {A}lgorithms and complexity.
\newblock {\em Foundations and Trends{\textregistered} in Machine Learning},
  8(3-4):231--357, 2015.

\bibitem{Cai_Cost_GLM}
T.~T. Cai, Y.~Wang, and L.~Zhang.
\newblock The cost of privacy in generalized linear models: {A}lgorithms and
  minimax lower bounds.
\newblock {\em arXiv preprint arXiv:2011.03900}, 2020.

\bibitem{Cost_Priv_Lin_Reg}
T.~T. Cai, Y.~Wang, and L.~Zhang.
\newblock The cost of privacy: {O}ptimal rates of convergence for parameter
  estimation with differential privacy.
\newblock {\em The Annals of Statistics}, 49(5):2825--2850, 2021.

\bibitem{CompIntr}
I.~Diakonikolas, G.~Kamath, D.~Kane, J.~Li, A.~Moitra, and A.~Stewart.
\newblock Robust estimators in high-dimensions without the computational
  intractability.
\newblock {\em SIAM Journal on Computing}, 48(2):742--864, 2019.

\bibitem{Stats_Info_Duchi}
J.~Duchi.
\newblock Lecture notes for {S}tatistics 311 / {E}lectrical {E}ngineering 377.
\newblock {\em URL: https://stanford. edu/class/stats311/Lectures/full\_notes.
  pdf. Last visited on}, 2:23, 2016.

\bibitem{AdaGrad_Duchi_2011}
J.~Duchi, E.~Hazan, and Y.~Singer.
\newblock Adaptive subgradient methods for online learning and stochastic
  optimization.
\newblock {\em Journal of Machine Learning Research}, 12(7), 2011.

\bibitem{Boost_DP}
C.~Dwork, G.~N. Rothblum, and S.~Vadhan.
\newblock Boosting and differential privacy.
\newblock In {\em 2010 IEEE 51st Annual Symposium on Foundations of Computer
  Science}, pages 51--60. IEEE, 2010.

\bibitem{SCO_Feldman_2020}
V.~Feldman, T.~Koren, and K.~Talwar.
\newblock Private stochastic convex optimization: {O}ptimal rates in linear
  time.
\newblock In {\em Proceedings of the 52nd Annual ACM SIGACT Symposium on Theory
  of Computing}, pages 439--449, 2020.

\bibitem{ACCFW}
D.~Garber and E.~Hazan.
\newblock Faster rates for the {F}rank-{W}olfe method over strongly-convex
  sets.
\newblock In {\em International Conference on Machine Learning}, pages
  541--549. PMLR, 2015.

\bibitem{Hampel}
F.~R. Hampel.
\newblock The influence curve and its role in robust estimation.
\newblock {\em Journal of the American Statistical Association},
  69(346):383--393, 1974.

\bibitem{Hoeff}
W.~Hoeffding.
\newblock Probability inequalities for sums of bounded random variables.
\newblock {\em The Collected Works of Wassily Hoeffding}, pages 409--426, 1994.

\bibitem{Huber2}
P.~J. Huber.
\newblock Robust regression: {A}symptotics, conjectures and monte carlo.
\newblock {\em The Annals of Statistics}, pages 799--821, 1973.

\bibitem{Huber1}
P.~J. Huber.
\newblock Robust estimation of a location parameter.
\newblock In {\em Breakthroughs in Statistics: Methodology and Distribution},
  pages 492--518. Springer, 1992.

\bibitem{Towards_Prac_DP_ERM}
R.~Iyengar, J.~P. Near, D.~Song, O.~Thakkar, A.~Thakurta, and L.~Wang.
\newblock Towards practical differentially private convex optimization.
\newblock In {\em 2019 IEEE Symposium on Security and Privacy (SP)}, pages
  299--316. IEEE, 2019.

\bibitem{Rev_FW_Prof_Free_Jaggi}
M.~Jaggi.
\newblock Revisiting {F}rank-{W}olfe: {P}rojection-free sparse convex
  optimization.
\newblock In {\em International Conference on Machine Learning}, pages
  427--435. PMLR, 2013.

\bibitem{(NEAR)}
P.~Jain and A.~G. Thakurta.
\newblock ({N}ear) dimension independent risk bounds for differentially private
  learning.
\newblock In {\em International Conference on Machine Learning}, pages
  476--484. PMLR, 2014.

\bibitem{CIGauss}
C.~Jin, P.~Netrapalli, R.~Ge, S.~M. Kakade, and M.~I. Jordan.
\newblock A short note on concentration inequalities for random vectors with
  sub-{G}aussian norm.
\newblock {\em arXiv preprint arXiv:1902.03736}, 2019.

\bibitem{EFF_PRIV_SCO}
C.~Jin, K.~Zhou, B.~Han, J.~Cheng, and T.~Zeng.
\newblock Efficient private sco for heavy-tailed data via averaged clipping.
\newblock {\em Machine Learning}, 113(11):8487--8532, 2024.

\bibitem{Comp_Thm_Kairouz}
P.~Kairouz, S.~Oh, and P.~Viswanath.
\newblock The composition theorem for differential privacy.
\newblock In {\em International Conference on Machine Learning}, pages
  1376--1385. PMLR, 2015.

\bibitem{Kamath_Priv_HT}
G.~Kamath, V.~Singhal, and J.~Ullman.
\newblock Private mean estimation of heavy-tailed distributions.
\newblock In {\em Conference on Learning Theory}, pages 2204--2235. PMLR, 2020.

\bibitem{Adam_2104}
D.~P. Kingma and J.~Ba.
\newblock Adam: A method for stochastic optimization.
\newblock {\em arXiv preprint arXiv:1412.6980}, 2014.

\bibitem{Vect_Bern}
J.~M. Kohler and A.~Lucchi.
\newblock Sub-sampled cubic regularization for non-convex optimization.
\newblock In {\em International Conference on Machine Learning}, pages
  1895--1904. PMLR, 2017.

\bibitem{Kuru_Nester_DP_2022}
N.~Kuru, S.~Ilker~Birbil, M.~Gurbuzbalaban, and S.~Yildirim.
\newblock Differentially private accelerated optimization algorithms.
\newblock {\em SIAM Journal on Optimization}, 32(2):795--821, 2022.

\bibitem{Vpala}
K.~A. Lai, A.~B. Rao, and S.~Vempala.
\newblock Agnostic estimation of mean and covariance.
\newblock In {\em 2016 IEEE 57th Annual Symposium on Foundations of Computer
  Science (FOCS)}, pages 665--674. IEEE, 2016.

\bibitem{Lehmann_Point_Estim}
E.~L. Lehmann and G.~Casella.
\newblock {\em Theory of Point Estimation}.
\newblock Springer Science \& Business Media, 2006.

\bibitem{Lera}
M.~Lerasle and R.~I. Oliveira.
\newblock Robust empirical mean estimators.
\newblock {\em arXiv preprint arXiv:1112.3914}, 2011.

\bibitem{NO_ROB_LR}
X.~Liu, P.~Jain, W.~Kong, S.~Oh, and A.~S. Suggala.
\newblock Near optimal private and robust linear regression.
\newblock {\em arXiv preprint arXiv:2301.13273}, 2023.

\bibitem{DP_ROB_HIGH_DIM}
X.~Liu, W.~Kong, and S.~Oh.
\newblock Differential privacy and robust statistics in high dimensions.
\newblock In {\em Conference on Learning Theory}, pages 1167--1246. PMLR, 2022.

\bibitem{LUGOSI_MEAN_ESTIM_HT}
G.~Lugosi and S.~Mendelson.
\newblock Mean estimation and regression under heavy-tailed distributions: {A}
  survey.
\newblock {\em Foundations of Computational Mathematics}, 19(5):1145--1190,
  2019.

\bibitem{Minsker}
S.~Minsker.
\newblock Geometric median and robust estimation in banach spaces.
\newblock {\em Bernoulli}, 21(4):2308--2335, 2015.

\bibitem{Nester}
Y.~Nesterov.
\newblock {\em Introductory Lectures on Convex Optimization: A Basic Course},
  volume~87.
\newblock Springer Science \& Business Media, 2013.

\bibitem{PENSIA_ROB_HT}
A.~Pensia, V.~Jog, and P.-L. Loh.
\newblock Robust regression with covariate filtering: {H}eavy tails and
  adversarial contamination.
\newblock {\em Journal of the American Statistical Association}, pages 1--12,
  2024.

\bibitem{FWS}
S.~Pokutta.
\newblock The {F}rank-{W}olfe algorithm: {A} short introduction.
\newblock {\em Jahresbericht der Deutschen Mathematiker-Vereinigung},
  126(1):3--35, 2024.

\bibitem{RE}
A.~Prasad, A.~S. Suggala, S.~Balakrishnan, and P.~Ravikumar.
\newblock Robust estimation via robust gradient estimation.
\newblock {\em Journal of the Royal Statistical Society Series B: Statistical
  Methodology}, 82(3):601--627, 2020.

\bibitem{Richtarik_Iter_Comp}
P.~Richt{\'a}rik and M.~Tak{\'a}{\v{c}}.
\newblock Iteration complexity of randomized block-coordinate descent methods
  for minimizing a composite function.
\newblock {\em Mathematical Programming}, 144(1):1--38, 2014.

\bibitem{Schmidt_2011_Conv}
M.~Schmidt, N.~Roux, and F.~Bach.
\newblock Convergence rates of inexact proximal-gradient methods for convex
  optimization.
\newblock {\em Advances in Neural Information Processing Systems}, 24, 2011.

\bibitem{Xu_ADMM_Nester_2021}
F.~Shang, T.~Xu, Y.~Liu, H.~Liu, L.~Shen, and M.~Gong.
\newblock Differentially private {ADMM} algorithms for machine learning.
\newblock {\em IEEE Transactions on Information Forensics and Security},
  16:4733--4745, 2021.

\bibitem{Is_Inter_Nec_DP_ERM}
A.~Smith, A.~Thakurta, and J.~Upadhyay.
\newblock Is interaction necessary for distributed private learning?
\newblock In {\em 2017 IEEE Symposium on Security and Privacy (SP)}, pages
  58--77. IEEE, 2017.

\bibitem{SGD_DP_2013_Upd}
S.~Song, K.~Chaudhuri, and A.~D. Sarwate.
\newblock Stochastic gradient descent with differentially private updates.
\newblock In {\em 2013 IEEE Global Conference on Signal and Information
  Processing}, pages 245--248. IEEE, 2013.

\bibitem{NOPL}
K.~Talwar, A.~Thakurta, and L.~Zhang.
\newblock Nearly optimal private {L}asso.
\newblock {\em Advances in Neural Information Processing Systems}, 28, 2015.

\bibitem{RMSprop_2012}
T.~Tieleman.
\newblock Lecture 6.5-rmsprop: Divide the gradient by a running average of its
  recent magnitude.
\newblock {\em COURSERA: Neural Networks for Machine Learning}, 4(2):26, 2012.

\bibitem{Vad17}
S.~Vadhan.
\newblock The complexity of differential privacy.
\newblock {\em Tutorials on the Foundations of Cryptography: Dedicated to Oded
  Goldreich}, pages 347--450, 2017.

\bibitem{Weak_Conv_Van_Der}
A.~W. Van Der~Vaart and J.~A. Wellner.
\newblock {\em Weak Convergence}.
\newblock Springer, 1996.

\bibitem{Prob_Stats_Chen}
M.~J. Wainwright.
\newblock {\em High-Dimensional Statistics: A Non-Asymptotic Viewpoint},
  volume~48.
\newblock Cambridge University Press, 2019.

\bibitem{DP_ERM_Faster}
D.~Wang, M.~Ye, and J.~Xu.
\newblock Differentially private empirical risk minimization revisited:
  {F}aster and more general.
\newblock {\em Advances in Neural Information Processing Systems}, 30, 2017.

\bibitem{Wang_Iter_Comp}
P.-W. Wang and C.-J. Lin.
\newblock Iteration complexity of feasible descent methods for convex
  optimization.
\newblock {\em The Journal of Machine Learning Research}, 15(1):1523--1548,
  2014.

\bibitem{Recht}
S.~J. Wright and B.~Recht.
\newblock {\em Optimization for Data Analysis}.
\newblock Cambridge University Press, 2022.

\bibitem{BOLT_ON_PRIV}
X.~Wu, F.~Li, A.~Kumar, K.~Chaudhuri, S.~Jha, and J.~Naughton.
\newblock Bolt-on differential privacy for scalable stochastic gradient
  descent-based analytics.
\newblock In {\em Proceedings of the 2017 ACM International Conference on
  Management of Data}, pages 1307--1322, 2017.

\bibitem{EFF_PRIV_SMOOTH}
J.~Zhang, K.~Zheng, W.~Mou, and L.~Wang.
\newblock Efficient private erm for smooth objectives.
\newblock {\em arXiv preprint arXiv:1703.09947}, 2017.

\bibitem{Var_Red_FW}
M.~Zhang, Z.~Shen, A.~Mokhtari, H.~Hassani, and A.~Karbasi.
\newblock One sample stochastic {F}rank-{W}olfe.
\newblock In {\em International Conference on Artificial Intelligence and
  Statistics}, pages 4012--4023. PMLR, 2020.

\end{thebibliography}

\end{document}